\documentclass[a4paper, pdftex]{article}

\usepackage[utf8]{inputenc}
\usepackage[T1]{fontenc}
\usepackage[pdftex]{graphicx}
\usepackage{amsmath}
\usepackage{amsfonts}
\usepackage{amsthm}
\usepackage{amssymb}
\usepackage{setspace}
\usepackage[format=hang]{caption}
\usepackage{thmtools}
\usepackage{thm-restate}
\usepackage{mathtools}
\usepackage{verbatim}
\usepackage[all, cmtip]{xy}
\usepackage{multirow}
\usepackage{bm}
\usepackage[textsize = tiny]{todonotes}
\usepackage{hyperref}
\hypersetup{colorlinks=true, citecolor=blue, linkcolor=blue, urlcolor=blue, filecolor=blue, bookmarksdepth=3, pdfencoding=unicode, psdextra}
\usepackage{bookmark}
\usepackage[capitalize,nameinlink,noabbrev]{cleveref}
\usepackage[hmargin=3cm,vmargin=3cm]{geometry}
\usepackage{quiver}
\usepackage[figuresleft]{rotating}
\usepackage{pdflscape}
\usepackage{enumitem}
\usepackage{pinlabel}
\usepackage{ifthen}
\usepackage{sidecap}
\usepackage{xparse}
\usepackage{adjustbox}

\usepackage[backend=bibtex, style=alphabetic, url=false, hyperref, backref, backrefstyle=none, maxbibnames=99, sorting=nyt]{biblatex}
\AtEveryBibitem{\clearlist{language}}
\RequirePackage{doi}

\newtheorem*{thm*}{Theorem}
\newtheorem{theorem}{Theorem}[section]
\newtheorem{corollary}[theorem]{Corollary}
\newtheorem{lemma}[theorem]{Lemma}

\newtheorem*{fact*}{Fact}
\newtheorem{proposition}[theorem]{Proposition}
\newtheorem{claim}[theorem]{Claim}

\newcounter{theoremalph}

\newtheorem{thmAlph}[theoremalph]{Theorem}
\theoremstyle{definition}
\newtheorem{question}[theorem]{Question}
\newtheorem{definition}[theorem]{Definition}

\theoremstyle{remark}
\newtheorem{remark}[theorem]{Remark}

\newtheorem{example}[theorem]{Example}
\newtheorem{construction}[theorem]{Construction}

\renewcommand{\ll}{\left\langle}
\newcommand{\rr}{\right\rangle}
\newcommand{\ls}{\left\{}
\newcommand{\rs}{\right\}}

\newcommand{\mcA}{\ensuremath{\mathcal{A}}}

\newcommand{\mcM}{\ensuremath{\mathcal{M}}}

\newcommand{\mbQ}{\ensuremath{\mathbb{Q}}}
\newcommand{\mbR}{\ensuremath{\mathbb{R}}}
\newcommand{\mbZ}{\ensuremath{\mathbb{Z}}}
\newcommand{\mbN}{\ensuremath{\mathbb{N}}}

\newcommand{\on}[1]{\operatorname{#1}}

\newcommand{\overbar}[1]{\mkern 2mu\overline{\mkern-2mu#1\mkern-2mu}\mkern 2mu}

\newcommand{\Sym}{\operatorname{Sym}}
\newcommand{\St}{\operatorname{St}}
\newcommand{\Stsymp}{\operatorname{St}^\omega}

\newcommand{\im}{\operatorname{im}}
\newcommand{\rk}{\operatorname{rk}}

\newcommand{\Ind}{\operatorname{Ind}}

\newcommand{\SL}[2]{\ensuremath{\operatorname{SL}_{#1}(#2)}}
\newcommand{\Sp}[2]{\ensuremath{\operatorname{Sp}_{#1}(#2)}}

\newcommand{\vcd}{\operatorname{vcd}}

\newcommand{\I}[1][n]{\mathcal{I}_{#1}}
\NewDocumentCommand {\Irel} { O{n} O{m} }{\mathcal{I}^{#2}_{#1}}
\newcommand{\Idel}[1][n]{\mathcal{I}_{#1}^{\delta}}
\NewDocumentCommand {\Idelrel} { O{n} O{m} }{\mathcal{I}^{\delta,#2}_{#1}}
\newcommand{\Isigdel}[1][n]{\mathcal{I}_{#1}^{\sigma,\delta}}
\NewDocumentCommand {\Isigdelrel} { O{n} O{m} }{\mathcal{I}^{\sigma,\delta,#2}_{#1}}
\newcommand{\IA}[1][n]{\mathcal{IA}_{#1}}
\NewDocumentCommand {\IArel} { O{n} O{m} }{\mathcal{IA}^{#2}_{#1}}
\newcommand{\IAA}[1][n]{\mathcal{IAA}_{#1}}
\NewDocumentCommand {\IAArel} { O{n} O{m} }{\mathcal{IAA}^{#2}_{#1}}
\newcommand{\IAAst}[1][n]{\mathcal{IAA}^*_{#1}}
\NewDocumentCommand {\IAAstrel} { O{n} O{m} }{\mathcal{IAA}^{*,#2}_{#1}}
\newcommand{\B}{\mathcal{B}}
\newcommand{\BA}{\mathcal{BA}}
\newcommand{\BAA}{\mathcal{BAA}}

\providecommand{\Link}{\ensuremath\mathsf{Link}}
\providecommand{\Linkhat}{\ensuremath\widehat{ \mathsf{Link}}}
\providecommand{\Star}{\ensuremath\mathsf{Star}}

\newcommand{\Bv}{\mathcal{BAA}}
\newcommand{\baseB}{\mathcal{B}(V')}
\newcommand{\baseBA}{\mathcal{BA}(V')}

\newcommand{\newproofphi}{\tilde{\phi}}
\newcommand{\newproofsphere}{\tilde{S}}
\newcommand{\newlemmaphi}{\phi'}   
\newcommand{\newlemmasphere}{S'}         
\newcommand{\lastindexexamples}{l}
   
\newcommand{\retraction}{\rho}
\newcommand{\rkfn}{\operatorname{rk}}
\newcommand{\badvertex}{\mathbf{s}}
\newcommand{\vertexvar}{\textbf{v}}
\newcommand{\sd}{\operatorname{sd}}
\newcommand{\Simp}{\operatorname{Simp}}
\newcommand{\new}{\operatorname{new}}

\newcommand{\skewname}{\operatorname{skew}}
\newcommand{\addname}{\operatorname{add}}

\providecommand{\Vr}{\ensuremath\mathsf{Vert}}
\providecommand{\Z}{\ensuremath\mathbb Z}
\providecommand{\Q}{\ensuremath\mathbb Q}
\providecommand{\OK}{\ensuremath\mathcal O_K}

\newcommand{\ApartmentModule}{A}
\newcommand{\SigmaAdditiveApartmentModule}{A^{\addname}}
\newcommand{\SkewApartmentModule}{A^{\skewname}}

\newcommand{\ApartmentMap}{\alpha}
\newcommand{\AdditiveMap}{\beta}
\newcommand{\SkewMap}{\gamma}

\newcommand{\sym}[1]{[\![ #1 ]\!]}

\newcommand{\SymplecticSummand}{U}
\newcommand{\CoxeterGroupTypeB}{\mathcal{W}}
\newcommand{\CoxeterGenerators}{\mathcal{S}}
\newcommand{\GroupElement}{\phi}

\newcommand{\VertexSet}[1]{\mathcal{V}_{#1}}

\title{A presentation of symplectic Steinberg modules and cohomology of $\Sp{2n}{\mbZ}$}
\author{Benjamin Br\"uck, Peter Patzt, Robin J. Sroka}
\date{}

\begin{document}
\maketitle

\begin{abstract}
Borel--Serre proved that the integral symplectic group $\Sp{2n}\Z$ is a virtual duality group of dimension $n^2$ and that the symplectic Steinberg module $\St^\omega_n(\Q)$ is its dualising module. This module is the top-dimensional homology of the Tits building associated to $\Sp{2n}\Q$.
We find a presentation of this Steinberg module and use it to show that the codimension-1 rational cohomology of $\Sp{2n}\Z$ vanishes for $n \geq 2$, $H^{n^2 -1}(\Sp{2n}\Z;\Q) \cong 0$. 
Equivalently, the rational cohomology of the moduli stack $\mathcal A_n$ of principally polarised abelian varieties of dimension $2n$ vanishes in the same degree.
Our findings suggest a vanishing pattern for high-dimensional cohomology in degree $n^2-i$, similar to the one conjectured by Church--Farb--Putman for special linear groups.
\end{abstract}

\tableofcontents

\section{Introduction} 

Let $R$ be a commutative ring and $(\vec e_1,\vec f_1, \dots, \vec e_n,\vec f_n)$ an ordered basis of $R^{2n}$. The \emph{(standard) symplectic form} $\omega$ on $R^{2n}$ is the antisymmetric bilinear form given by
\[ \omega(\vec e_i,\vec e_j) = \omega(\vec f_i,\vec f_j) = 0\quad\text{and}\quad \omega(\vec e_i,\vec f_j) = \delta_{ij},\]
where $\delta_{ij}$ is the Kronecker delta. The group of invertible $2n\times 2n$ matrices $A$ with entries in $R$ with the property
\[ \omega(A\vec v,A\vec w) = \omega(\vec v,\vec w)\]
for all $\vec v,\vec w \in R^{2n}$ is called the \emph{symplectic group of $R$} and is denoted by $\Sp{2n}{R}$.

These groups show up in many areas of mathematics. In particular, the simple Lie group $\Sp{2n}{\mathbb R}$ and its lattice $\Sp{2n}{\mbZ}$ are ubiquitous in geometry and topology. In this paper, we consider the group cohomology of $\Sp{2n}\Z$ with trivial rational coefficients.

Borel--Serre \cite{BS} proved that $\Sp{2n}\Z$ has virtual cohomological 
dimension $\vcd(\Sp{2n}{\mbZ}) = n^2$, which implies that 
\begin{equation}
	\label{eq:borel-serre-vcd}
	H^j(\Sp{2n}\Z; \mbQ) \cong 0\quad \text{for $j>n^2$.}
\end{equation}		
It follows from results of Gunnells \cite{gun2000} (see Brück--Patzt--Sroka \cite[Chapter 5]{Sroka2021} and Br\"uck--Santo Rego--Sroka \cite{Brueck2022c}) that the ``top-dimensional'' rational cohomology of $\Sp{2n}{\mbZ}$ is trivial,
\begin{equation}
	\label{eq_codim0_vanishing}
	H^{n^2}(\Sp{2n}{\mbZ}; \mbQ) \cong 0\quad \text{for $n\ge1$.}
\end{equation}
In this work, we show a vanishing result in ``codimension-1'':

\begin{thmAlph}\label{thmA}
The rational cohomology of $\Sp{2n}{\mbZ}$ vanishes in degree $n^2-1$,
\[ H^{n^2-1}(\Sp{2n}{\mbZ}; \mbQ) \cong 0 \text{, if }n\ge 2.\]
\end{thmAlph}

Previously, this was only known for $n\leq 4$ by work of Igusa \cite{igusa1962}, Hain \cite{hain2002} and Hulek-Tommasi \cite{ht2012} (see \cref{tab:cohomologySp}); our proof applies for $n\geq 3$. 

Our findings suggest a more general vanishing pattern in the high-dimensional rational cohomology of symplectic groups, see \cref{sec_CFPconjecture}. 
Since the rational cohomology of $\Sp{2n}{\mbZ}$ is isomorphic to that of the \emph{moduli stack $\mathcal A_n$ of principally polarised abelian varieties of dimension $2n$}, \autoref{thmA} furthermore implies that 
\begin{equation*}
	H^{n^2-1}(\mcA_n; \mbQ) \cong H^{n^2-1}(\Sp{2n}{\mbZ}; \mbQ) \cong 0 \text{ if } n \geq 2.
\end{equation*}
This connection to algebraic geometry and number theory is discussed in \cref{sec_Ag_connection}. The next two subsections explain how \autoref{thmA} is related to a question of Putman \cite{PutmanQ}, which our main result answers (see \autoref{thmB}).

\subsection{Borel--Serre duality}
The method we use to get a handle on this high-dimensional cohomology of $\Sp{2n}{\mbZ}$ is Borel--Serre duality.
Borel--Serre \cite{BS} proved that
\[ H^{n^2 -i}( \Sp{2n}\Z; M) \cong H_i(\Sp{2n}\Z; \St^\omega_n(\Q)\otimes M)\]
for all rational coefficient modules $M$ and all codimensions $i$. Here $\St^\omega_n(\Q)$ is the 
\emph{symplectic Steinberg module} that is defined in 
the following way.

Let $T^\omega_n = T^\omega_n(\mbQ)$ be the simplicial complex whose vertices are the nonzero isotropic subspaces $V$ of $\mbQ^{2n}$ (the subspaces on which the form $\omega$ vanishes), and whose $k$-simplices are flags $ V_0 \subsetneq V_1 \subsetneq \dots \subsetneq V_k.$
This is the rational \emph{Tits building} of type $\mathtt{C}_n$. 
It was proved to be $(n-1)$-spherical by Solomon--Tits \cite{ST},
so in particular, its reduced homology is concentrated in dimension $n-1$. This homology group is what we call the symplectic Steinberg module
\[ \St^\omega_n = \St^\omega_n(\mbQ) \coloneqq \widetilde H_{n-1}(T^\omega_n;\Z).\]

The result of Borel--Serre allows one to compute the codimension-1 cohomology of $\Sp{2n}{\mbZ}$ by using the isomorphism
\begin{equation}
\label{eq_Borel_Serre}
	H^{n^2-1}(\Sp{2n}\Z;\Q) \cong H_1(\Sp{2n}\Z; \St^\omega_n\otimes\mbQ).
\end{equation}
In order to compute the right hand side of this equation, one needs a good understanding of the $\Sp{2n}{\mbZ}$-module $\St^\omega_n$. We obtain this by finding a \emph{presentation of this module}.

\subsection{Presentation of \texorpdfstring{$\St^\omega_n$}{St\unichar{"005E}\unichar{"1D714}}}
\begin{figure}
\begin{center}
\includegraphics{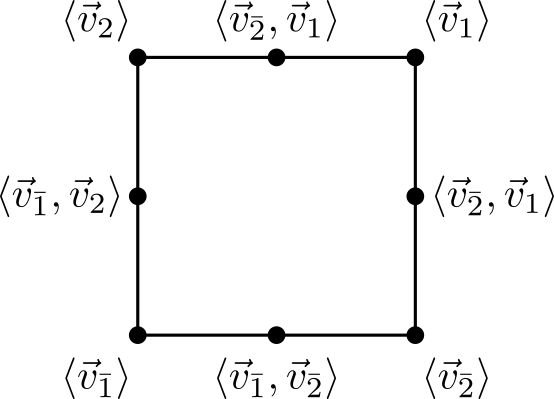}
\end{center} 
\caption{An apartment in $T^\omega_2$ corresponding to a symplectic basis $\vec v_1,\vec v_{\bar{1}},\vec v_2,\vec v_{\bar{2}}\in \mbQ^4$.} \label{fig:symplecticapartment}
\end{figure}

Let $(\vec v_1,\vec v_{\bar{1}}, \dots, \vec v_n,\vec v_{\bar{n}})$ be an ordered 
symplectic basis of $\mbQ^{2n}$, i.e.~the sequence of columns of a matrix in $\Sp{2n}{\mbQ}$.\footnote{This is equivalent to saying that $\vec v_1,\vec v_{\bar{1}}, \dots, \vec v_n,\vec v_{\bar{n}}$ form a basis of $\mbQ^{2n}$ and $\omega(\vec v_i,\vec v_{\bar{i}}) = 1$ for 
		all $i$ and all other pairings are $0$.}
Consider the full subcomplex of $T^\omega_n$ on the isotropic subspaces 
spanned by nonempty subsets of $\{\vec v_1,\vec v_{\bar{1}}, \dots, \vec v_n,\vec v_{\bar{n}}\}$ as illustrated in \autoref{fig:symplecticapartment} for $n = 2$.
This subcomplex is isomorphic to the barycentric subdivision of the boundary of an $n$-dimensional cross-polytope and called a (symplectic) \emph{apartment} of the building $T^\omega_n$.
It induces a nonzero homology class in $\St^\omega_n = \widetilde H_n(T^\omega_n;\Z)$ that we denote by $[v_1, v_{\bar{1}}, \dots, v_n,v_{\bar{n}}]$ and call a \emph{symplectic apartment class}. Solomon--Tits \cite{ST} proved that the apartment classes generate $\St^\omega_n$.

We call an apartment that comes from an \emph{integral} basis $(\vec v_1,\vec v_{\bar{1}}, \dots, \vec v_n,\vec v_{\bar{n}})$, i.e.~the columns of a matrix in $\Sp{2n}{\mbZ}$, an \emph{integral symplectic apartment}. Gunnells \cite{gun2000} showed that $\St^\omega_n$ is generated by its integral symplectic apartment classes\footnote{In fact, Gunnells shows more generally that $\St^\omega_n(K)$ is generated by integral apartments, where $K$ is a number field with Euclidean ring of integers $\OK$.}; see \cite{bruecksroka2023} for an alternative proof.
Our following \autoref{thmB} gives a presentation of $\St^\omega_n$ in terms of this generating set.

To state it, let $\sym{n} \coloneqq \{ 1, \bar{1}, \dots, n, \bar{n}\}$, where $\bar{\bar{a}} = a$. We denote by $\CoxeterGroupTypeB_n$ the group of all bijections $\pi\colon  \sym{n} \to \sym{n}$ such that $\pi(\bar{a})= \overline{\pi(a)}$. This is the group of signed permutations, the Weyl group associated to $\on{Sp}_{2n}$. It comes with a standard generating set of simple reflections $\CoxeterGenerators$ and we write $\on{len}_\CoxeterGenerators(\pi)$ for the word length of $\pi\in \CoxeterGroupTypeB_n$ with respect to $\CoxeterGenerators$ (for more details, see \cref{eq:coxeter_details} et seq.).

\begin{thmAlph}
	\label{thmB}
	The symplectic Steinberg module $\Stsymp_n(\Q)$ has 
	the following presentation as an abelian group:
	\begin{enumerate}[leftmargin=*]
		\item[] \textbf{Generators:} Formal symbols $[v_1, v_{\bar{1}}, \dots, v_n,v_{\bar{n}}]$, where $(v_1, v_{\bar{1}}, \dots, v_n,v_{\bar{n}})$ is an ordered set of lines in 
		$\mbZ^{2n}$ such that, for some choice of $\vec v_i\in \mbZ^{2n}$ with $v_i = \langle \vec v_i \rangle$, $i \in \sym{n}$, the tuple
		$(\vec v_1,\vec v_{\bar{1}}, \dots, \vec v_n, \vec v_{\bar{n}})$ is a symplectic basis of $\mbZ^{2n}$.
		\item[] \textbf{Relations:} For each symplectic basis $(\vec v_1, \vec v_{\bar{1}}, \dots, \vec v_n, \vec v_{\bar{n}})$  of $\mbZ^{2n}$:
		\begin{enumerate}[leftmargin=*]
			\item $[v_1, v_{\bar{1}}, \dots, v_n, v_{\bar{n}}] = 
			(-1)^{\on{len}_{\CoxeterGenerators}(\pi)}\cdot 
			[v_{\pi(1)},v_{\pi(\bar{1})},\dots, v_{\pi(n)}, v_{\pi(\bar{n})}] \,\forall \pi \in \CoxeterGroupTypeB_n$;
			\item $[v_1, v_{\bar{1}}, \dots, v_n, v_{\bar{n}}] = [v_1, \langle\vec v_1 +\vec 
			v_{\bar{1}} \rangle, \dots, v_n, v_{\bar{n}}]+[\langle \vec v_1+\vec v_{\bar{1}} \rangle, 
			v_{\bar{1}}, \dots, v_n, v_{\bar{n}}]$;
			\item $[v_1, v_{\bar{1}}, \dots, v_n, v_{\bar{n}}] = [v_1,\langle \vec v_{\bar{1}} - 
			\vec v_{\bar{2}} \rangle, \langle \vec v_1 + \vec v_2 \rangle, v_{\bar{2}}, \dots, 
			v_n, v_{\bar{n}}]$\\
			$\phantom{[v_1, v_{\bar{1}}, \dots, v_n, v_{\bar{n}}] =}+ [\langle \vec v_{\bar{1}}-\vec 
			v_{\bar{2}} \rangle , v_2, \langle\vec v_1 + \vec v_2\rangle, v_{\bar{1}}, \dots, 
			v_n, v_{\bar{n}}]$.
		\end{enumerate}
	\end{enumerate}
	The action of $\Sp{2n}{\mbZ}$ is given by $\GroupElement \cdot [v_1, v_{\bar{1}}, \dots, v_n, v_{\bar{n}}] = [\GroupElement(v_1), 
		\GroupElement(v_{\bar{1}}), \dots, \GroupElement(v_n), \GroupElement(v_{\bar{n}})].$
\end{thmAlph}

\autoref{thmB} is the main result of this article and, using Borel--Serre duality, \autoref{thmA} is a rather direct consequence of it. \autoref{thmB} answers a question of Putman \cite{PutmanQ} by establishing a ``Bykovski\u{\i} presentation'' for the symplectic Steinberg module: Relation (a) and (b) of the presentation are visibly analogous to the relations in Bykovski\u{\i}'s presentation for the Steinberg module of $\SL{n}{\mbQ}$ \cite{byk2003} (see also \cite[Theorem B]{CP}). However, a new type of relation -- Relation (c) -- shows up in the symplectic setting. Ongoing work of Br\"uck--Santos Rego--Sroka investigates these similarities and differences in the broader framework of Chevalley groups.

\subsection{Connection to other work}
\label{sec_intro_connection_to_others}

\subsubsection{Church--Farb--Putman Conjecture for \texorpdfstring{$\SL{n}{\mbZ}$}{SL(\unichar{"2124})}}
\label{sec_CFPconjecture}

Church--Farb--Putman \cite{cfp2014} conjectured that
\begin{equation}
\label{eq_CFP_conjecture}
H^{\binom n2-i}(\SL n\Z;\Q) \cong 0 \quad \text{for $n\ge i+2$,}
\end{equation}
where $\binom n2 = \vcd(\SL{n}{\mbZ})$.
This conjecture proposes a generalisation of a result by Lee--Szczarba \cite{ls1976} that states
\[ H^{\binom n2}(\SL n\Z;\Q) \cong 0 \quad \text{for $n\ge 2$.}\]
Later the conjecture was also proved for $i=1$ by Church--Putman \cite{CP} and for $i=2$ by Br\"uck--Miller--Patzt--Sroka--Wilson \cite{Brueck2022}. 
Motivated by this conjecture, there have been numerous results that have significantly improved the understanding of the high-dimensional cohomology of $\SL{n}{R}$ for different rings $R$ \cite{cfp2019,CP, MPWY:Nonintegralitysome,Kupers2022, Brueck2022}.  The motivation for this article is to investigate the following $\Sp{2n}{\mbZ}$-analogue of Church--Farb--Putman's conjecture.

\begin{question}
	\label{q:vanishing-question}
	Does it hold that $H^{n^2-i}(\Sp{2n}{\mbZ}; \mbQ) \cong 0 \text{, if } n \geq i + 1$?
\end{question}

In a lot of ways, the story for $\SL n\Z$ is parallel to the one of $\Sp{2n}\Z$: The duality result of Borel--Serre also applies to this group and yields an isomorphism
\begin{equation*}
	H^{\binom n2 - i}(\SL n\Z;\Q) \cong H_i(\SL n\Z; \St_n(\Q)\otimes \mbQ).
\end{equation*}
This gives access to high-dimensional cohomology via a partial resolution of the Steinberg module $\St_n(\Q)$ associated to the special linear group, that is the top reduced homology of the corresponding Tits building of type $\mathtt{A}_{n-1}$ over $\mbQ$. 
Lee--Szczarba's result \cite{ls1976} about the vanishing of $H^{\binom n2}(\SL n\Z;\Q)$ can also be deduced from the fact that 
$\St_n(\Q)$ is generated by integral apartments, which was proved by 
Ash--Rudolph \cite{ar1979}. 
Church--Putman \cite{CP} used a particularly nice presentation of $\St_n(\Q)$ 
to show that $H^{\binom n2-1}(\SL n\Z;\Q) \cong 0$. This presentation was first established by Bykovski\u{\i} \cite{byk2003}; Church--Putman \cite{CP} gave a new proof using simplicial complexes. Similarly, Br\"uck--Miller--Patzt--Sroka--Wilson \cite{Brueck2022} obtained a three term resolution of $\St_n(\mbQ)$ to show the vanishing result in codimension $i = 2$.

In that sense, our result can be seen as an $\on{Sp}_{2n}$-counterpart of \cite{CP}.
In technical terms however, the different combinatorics of special linear and symplectic groups lead to a significantly higher complexity in the present article (see \cref{sec_intro_relation_other_complexes}).

\subsubsection{Computations for small \texorpdfstring{$n$}{n}, and the moduli spaces \texorpdfstring{$\mcA_n$}{A\unichar{"005F}n} and \texorpdfstring{$\mcM_n$}{M\unichar{"005F}n}}
\label{sec_Ag_connection}
For small $n$, the rational cohomology of $\Sp{2n}\Z$ was previously computed. We summarise in \autoref{tab:cohomologySp} what is known for $n\le 4$.
In particular, the answer to \autoref{q:vanishing-question} is known to be affirmative in this case,
\[ H^{n^2-i}(\Sp{2n}\Z;\Q) \cong 0 \quad \text{for $n\ge i+1$ if $n\le 4$}.\]

\begin{table}
\begin{center}
\begin{tabular}{c|rrrrrrrrrrrrrrr}
$n\backslash j$ & 0&1&2&3&4&5&6&7&8&9&10&11&12&$\ge13$\\
\hline
1&1&0\\
2&1&0&1&0&0\\
3&1&0&1&0&1&0&2&0&0&0\\
4&1&0&1&$\ge0$&$\ge1$&$\ge0$&$\ge2$&$\ge0$&$\ge1$&$\ge0$&$\ge1$&0&2&0 \\
\end{tabular}
\end{center}
\caption{The $j$-th Betti numbers of $\Sp{2n}\Z$ for $n\le 4$. For $n=1$, $\Sp2\Z = \SL2\Z$. The $n = 2$ row follows from work of Igusa \cite{igusa1962}, see also Lee--Weintraub \cite[Corollary 5.2.3 et seq.]{LW}. The $n = 3$ row is due to Hain \cite[Theorem 1]{hain2002}, and Hulek--Tommasi \cite[Corollary 3]{ht2012} derive the $n=4$ row.}
\label{tab:cohomologySp}
\end{table}

Using methods different from the ones used in this work, the rational cohomology groups $H^{*}(\Sp{2n}\Z;\Q)$ can be studied by recognising them as the  rational cohomology groups of a finite dimensional quasiprojective variety: The \emph{coarse moduli space of principally polarised abelian varieties of dimension $2n$} is the quotient $\mathcal A_n^{\operatorname{coarse}} :=\Sp{2n}\Z \backslash \mathbb H^n$
of the Siegel upper half plane
\[ \mathbb H^n = \{ A \in \operatorname{Mat}_{n\times n}(\mathbb C) \mid A^T = A \text{ and } \im A >0\}\]
of symmetric complex $n\times n$ matrices whose imaginary part is positive definite. 
As $\Sp{2n}\Z$ acts on $\mathbb H^n$ properly discontinuously with finite stabiliser and $\mathbb H^n$ 
is contractible, there is an isomorphism
\[ H^j(\mathcal A_n^{\operatorname{coarse}} ; \Q) \cong H^j(\Sp{2n}\Z;\Q) \cong H^j( \mathcal A_n; \Q).\]
The moduli space $\mathcal A_n^{\operatorname{coarse}}$ admits a weight filtration.
 Brandt--Bruce--Chan--Melo--Moreland--Wolfe \cite[Theorem A]{BBC:TopWeightRational} computed the top-weight rational cohomology for $n\le 7$. In particular, they showed that
\[ \operatorname{Gr}^W_{n^2+n}H^{n^2-i}(\mathcal A_n;\Q) \cong 0 \quad \text{for $n\ge i+1$ if $n\le 7$},\]
which provides some evidence for \autoref{q:vanishing-question}, and that
\begin{equation*}
H^{n^2-n}(\mathcal A_n;\Q) \cong H^{n^2 - n}(\Sp{2n}\Z;\Q) \not\cong 0 \quad\text{if $n\le 7$}.
\end{equation*}
This leads us to the following refinement of \autoref{q:vanishing-question}.
\begin{question}
\label{question_cohomology_vanishing_codimensions}
Is $j = n^2-n$ the largest degree for which $H^j(\Sp{2n}\Z;\Q) = H^j(\mcA_n;\Q)$ is nonzero?
\end{question}

\autoref{question_cohomology_vanishing_codimensions} is a variation of \cite[Question 7.1.1 and 7.1.3]{BBC:TopWeightRational}, which asks about triviality and non-triviality of certain graded pieces of these cohomology groups and relates this to stable cohomology classes in the Satake or Baily--Borel compactification of $\mcA_n^{\operatorname{coarse}}$. 

If \autoref{question_cohomology_vanishing_codimensions} had a positive answer, it would show a stark difference between the high-dimensional cohomology of $\mcA_n$
and $\mcM_n$, the moduli space of genus $n$ curves: The rational cohomology of $\mcM_n$ is the same as that of the mapping class group $\on{MCG}(\Sigma_n)$ of a genus-$n$ surface and Harer \cite{harer1986} showed that its virtual cohomological dimension is given by $\vcd(\on{MCG}(\Sigma_n)) = 4n-5$. Church--Farb--Putman \cite{cfp2014} conjectured a vanishing pattern similar to \cref{eq_CFP_conjecture} for these cohomology groups. While Church--Farb--Putman \cite{CFP:rationalcohomologymapping} and Morita--Sakasai--Suzuki \cite{moritasakasaisuzuki2013} proved that the rational cohomology of $\mcM_n$ vanishes in its virtual cohomological dimension, Chan--Galatius--Payne \cite{Chan2021} and subsequently Payne--Willwacher \cite{Payne2021} found many non-trivial classes in degrees close to the vcd\footnote{Their results show that $H^{4n-5-i}(\mcM_n;\mbQ)$ is non-trivial for $i\in \ls 1,3,4,6,7,9,10,11,13,14 \rs$ and $n$ sufficiently large. See \cite[Table 1]{Payne2021}.}, thereby disproving the conjecture. 

From a group theoretic perspective, i.e.~comparing $\Sp{2n}{\mbZ}$ to $\on{MCG}(\Sigma_n)$, we note that the virtual cohomological dimension of $\Sp{2n}{\mbZ}$ grows quadratically, $\vcd(\Sp{2n}{\mbZ}) = n^2$ \cite{BS}, while that of $\on{MCG}(\Sigma_n)$, $\vcd(\on{MCG}(\Sigma_n)) = 4n-5$ \cite{harer1986}, grows linearly in $n$. The difference in the behaviour of the codimension-1 cohomology, i.e.~comparing \autoref{thmA} and the non-vanishing result of \cite{Chan2021}, suggests that cohomologically $\Sp{2n}{\mbZ}$ is ``closer'' to the arithmetic group $\SL{n}{\mbZ}$ than to the mapping class group $\on{MCG}(\Sigma_n)$. This might not be as surprising, given that $\SL{n}{\mbZ}$ and $\Sp{2n}{\mbZ}$ are both examples of Chevalley groups over Euclidean rings of integers. Br\"uck--Santo Rego--Sroka 
studied the top-dimensional rational cohomology of such groups. See 
\cite[Question 1.2]{Brueck2022c} for how both the conjecture of 
Church--Farb--Putman \cref{eq_CFP_conjecture} and 
\autoref{q:vanishing-question} might fit into a more 
general framework. In codimension-1, this framework is currently investigated by Br\"uck--Santos Rego--Sroka.

\subsection{Simplicial complexes}
\label{sec_intro_simplicial_complexes}
The key step in our proof of \autoref{thmB} is to consider certain simplicial complexes and show that they are highly connected. These complexes encode the structure of integral apartments in the building $T^\omega_n$ and the relations between them. The proof of the following connectivity result takes up a majority of the paper and constitutes its main technical achievement.

\begin{thmAlph}\label{thm_connectivity_IAA}
	For all $n \geq 1$, the simplicial complex $\IAA$ is $n$-connected.
\end{thmAlph}

We extract the presentation of $\Stsymp_n(\mbQ)$ from this connectivity result.
The construction of $\IAA$ is rather involved, but the intuition is that $\IAA$ contains ``polyhedral cells'' that encode the relations in \autoref{thmB}. 
These are depicted in \autoref{fig_relation_polytopes}. To prove \autoref{thm_connectivity_IAA}, we build on connectivity results about similar complexes that already appeared in the literature \cite{put2009, CP, Brueck2022, bruecksroka2023} (see \cref{table_overview_all_complexes} for a detailed overview).

\begin{figure}
	\begin{center}
		\includegraphics{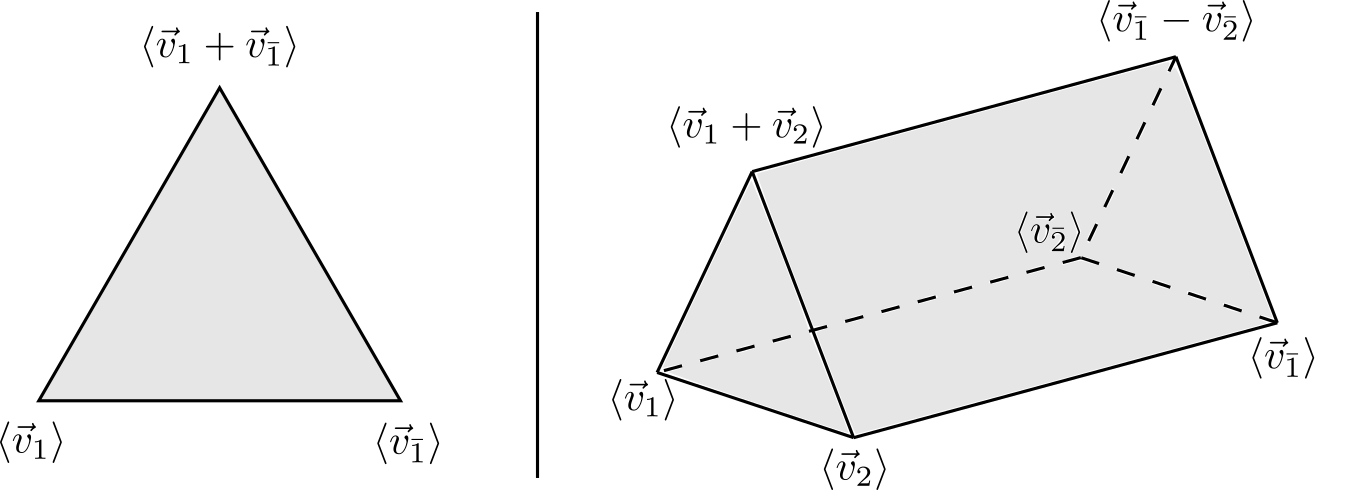}
	\end{center} 
	\caption{Subcomplexes of $\IAA$ encoding relations between integral apartment classes. The left is a $2$-simplex in $\IAA[1]$ corresponding to relation (b) in \autoref{thmB}. The right is a union of $3$-simplices in $\IAA[2]$ and corresponds to relation (c). Relation (a) comes from symmetries of the underlying apartment (see \cref{figure_apartment_intro}).} \label{fig_relation_polytopes}
\end{figure}

\subsubsection{Comparison with previously studied complexes and related works}
\label{sec_intro_relation_other_complexes}
The idea of using highly connected simplicial complexes to study Steinberg modules is due to Church--Farb--Putman \cite{cfp2019}. We draw inspiration from this and, in particular, from Church--Putman's approach \cite{CP} to Bykovski\u{\i}'s presentation for the Steinberg modules of special linear groups. 
Nevertheless, obtaining the presentation in \autoref{thmB} from the connectivity result in \autoref{thm_connectivity_IAA} is more difficult here than in the setting of special linear groups and uses inductive methods (see \cref{sec:presentation}).

Our proof of \autoref{thm_connectivity_IAA} relies on connectivity results that have been established in the special linear group setting \cite{cfp2019, CP, Brueck2022}. We use these results by considering embeddings of $\SL{n}{\mbZ}$ and $\SL{2}{\mbZ}$ in $\Sp{2n}{\mbZ}$ and of $\Sp{2n}{\mbZ}$ into $\SL{2n}{\mbZ}$.\footnote{These embeddings e.g.~show up in \cref{sec_retraction}, \cref{sec_subcomplex_W}, \cref{lem_induction_beginning}, \cref{prop:spanmap}, \autoref{claim:induction-beginning}.} 
But in comparison to previous works, several new difficulties arise in the context of  the present article:

The first difficulty is the shear complexity of our complex $\IAA$. Similar to $\IAA$, the complex $\BA_n$ studied in \cite{CP} is a complex whose simplices are given by certain admissible sets of lines in $\mbZ^n$ (\cref{def:B-and-BA}). However, while $\BA_n$ has only \emph{two} types of such admissible sets, the complex $\IAA$ has \emph{twelve} (\cref{def:IAA}). 
One reason for this is that $\IAA$ contains the complex $\BAA_n$ as subcomplexes. $\BAA_n$ was used in \cite{Brueck2022} to study the \emph{codimension-2} cohomology of $\SL{n}{\mbZ}$; nevertheless $\IAA$ only gives access to the \emph{codimension-1} cohomology of the symplectic group $\Sp{2n}\Z$.

Another difference to the setting of $\on{SL}_n$ is that while \emph{simplicial} structures show up naturally in the type-$\mathtt{A}$ combinatorics of the special linear group, this is less true for the type-$\mathtt{C}$ combinatorics of the symplectic group. While we decided to still work with simplicial complexes in this text, other \emph{polyhedral} cell types show up naturally in our complexes (see e.g.\ the ``prism'' on the right of \cref{fig_relation_polytopes} and \cref{rem_polyhedral_complexes}). Taking this into account leads to another increase in complexity.

A conceptual difficulty for the symplectic group is also the question of how to find suitable highly connected $\Sp{2n}{\mbZ}$-complexes:
The simplicial complexes $\BA_n$ and $\BAA_n$ used in \cite{CP} and \cite{Brueck2022} to study the codimension-1 and -2 cohomology of $\SL n\mbZ$ are very similar to Voronoi tesselations of symmetric spaces (see Elbaz-Vincent--Gangle--Soulé \cite{EGS13}). 
Such tesselations are in general not available for the symplectic group. In low rank however, MacPherson--McConnell \cite{MM93} gave a CW-structure for an equivariant retract of the symmetric space $\mathbb H_2 = \Sp{4}{\mbR}/U(2)$.
Our simplicial complex $\IAA[2]$ has some similarity to their complex.\footnote{In particular, our minimal $\sigma^2$ simplex (see \cref{def:IAA-simplices} and \cref{fig:IA2apartment}) corresponds to their ``red square'' and our prism (see \cref{def_prism} and right-hand side \cref{fig_relation_polytopes}) made from three individual simplices corresponds to their ``hexagon''.}

\subsubsection{Overview: Constructing simplicial complexes for studying \texorpdfstring{$\St_n^\omega$}{St\unichar{"005E}\unichar{"1D714}}}
\label{sec_intro_construction_simplicial_complexes}
This subsection gives a rough overview of how the simplicial complexes that appear in this article are constructed.

\begin{figure}
	\begin{center}
		\includegraphics{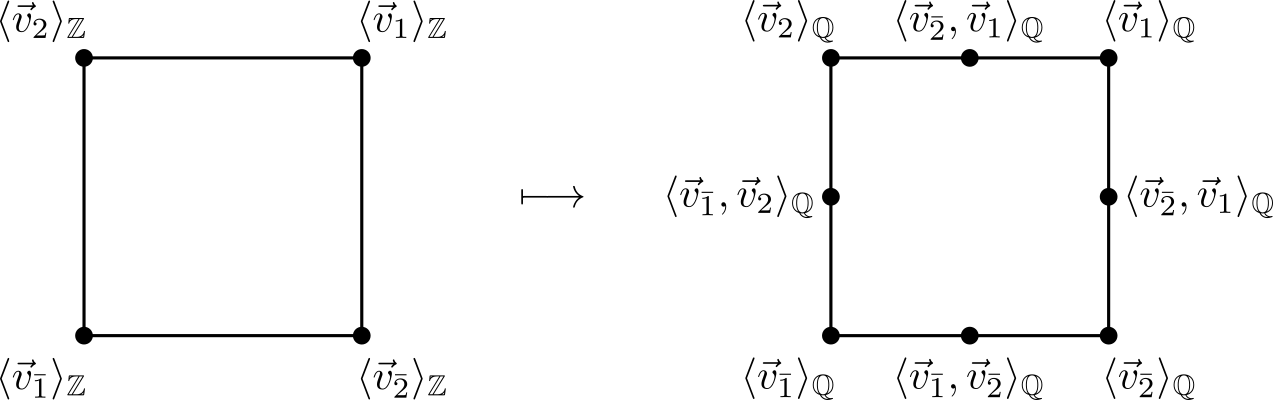}
	\end{center} 
	\caption{The right-hand side shows an (integral) apartment in $T^\omega_2$, the left-hand side a corresponding part of $\I[2]$. Here, $\vec v_1, \vec v_{\bar{1}}, \vec v_2, \vec v_{\bar{2}}$ denotes an integral symplectic basis of $\mbQ^4$.}
	\label{figure_apartment_intro}
\end{figure}

We begin with the following simplicial complex $\I$ that is defined in detail in \autoref{sec_defIA}. Let the vertices of $\I$ be the lines in $\mbZ^{2n}$ that are spanned by primitive vectors $\vec v\in \Z^{2n}$ (i.e.~the greatest common divisor of the entries of $\vec v$ is $1$). A $k$-simplex is formed by $(k+1)$ such lines if they span a rank-$(k+1)$ direct summand of $\Z^{2n}$ that is isotropic with respect to $\omega$. This simplicial complex was first considered by Putman \cite{put2009} and closely related complexes play a key role in work of van der Kallen--Looijenga  \cite{vanderkallenlooijenga2011sphericalcomplexesattachedtosymplecticlattices}.

The barycentric subdivision of $\I$ maps simplicially to $T^\omega_n$ by taking the $\mbQ$-spans of the involved summands of $\mbZ^{2n}$. This  induces a surjective map 
\begin{equation*}
	\widetilde H_{n-1}(\I) \twoheadrightarrow \widetilde H_{n-1}(T^\omega_n) = \St^\omega_n
\end{equation*}
because every integral apartment has a preimage in $\I$ (see \cref{figure_apartment_intro}). But $\I$ has certain $(n-1)$-dimensional homology classes that are sent to zero in $\St^\omega_n$. These actually come from apartment classes of the Steinberg module of the integral special linear group. We fill these homology classes by augmenting $\I$ with additional simplices to obtain a simplicial complex $\IAA^{(0)}$ (see \autoref{sec_int_cpx}) whose $(n-1)$-dimensional homology is isomorphic to $\St^\omega_n$. The construction of $\IAA^{(0)}$ is related to certain complexes $\Idel$, $\BA_n$ and $\BAA_n$ that were introduced by Putman \cite{put2009}, Church--Putman \cite{CP}, and Brück--Miller--Patzt--Sroka--Wilson \cite{Brueck2022} respectively, see \cref{sec_definitions}.

\begin{figure}
	\begin{center}
		\includegraphics{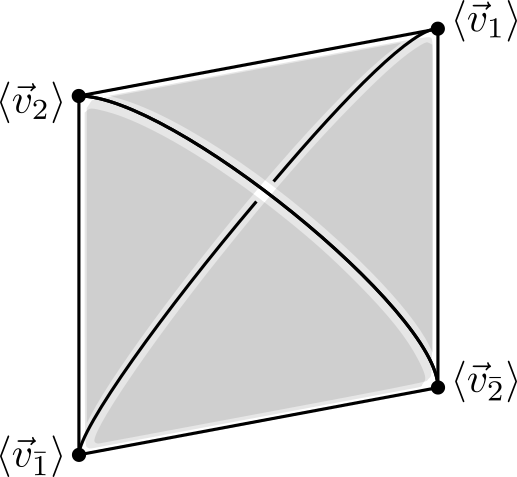}
	\end{center} 
	\caption{Four $2$-simplices in $\IAA[2]^{(1)}$ that ``fill in'' the square spanned by an integral symplectic basis. The square is filled in twice, once by the two triangles in the back and once by the two in the front. The result is a pillow-shaped 2-sphere. It is filled in by a 3-simplex in $\IAA[2]^{(2)}$.} \label{fig:IA2apartment}
\end{figure}

We then glue ``non-isotropic'' simplices to $\IAA^{(0)}$ that are allowed to contain one symplectic pair each. This yields a simplicial complex $\IAA^{(1)}$ (see \autoref{sec_int_cpx}) where all of the apartments from $\I$ are filled in (see \autoref{fig:IA2apartment}). In fact, the relative homology $H_n(\IAA^{(1)},\IAA^{(0)})$ surjects onto $\St^\omega_n$, which (re)proves that the integral symplectic apartment classes form a generating set. This uses that $\IAA^{(1)}$ is $(n-1)$-connected, whereas $\IAA^{(0)}$ and $\I$ are only $(n-2)$-connected. The complex $\IAA^{(1)}$ shares many features with a simplicial complex $\Isigdel$ defined by Putman \cite{put2009} and a related complex $\IA$ studied in recent work by Brück--Sroka \cite{bruecksroka2023}, see \cref{sec_definitions}.

To now get a presentation of $\St^\omega_n$, we have to understand the different ways in which we filled the apartments (see \cref{fig:IA2apartment}) and add additional simplices to $\IAA^{(1)}$ that encode the relations between them (see \cref{fig_relation_polytopes}). This finally leads us to the definition of a complex $\IAA^{(2)}$ (see \autoref{sec_int_cpx}), which we prove to be $n$-connected. \autoref{thmB} is then extracted from the nested triple of highly connected complexes $(\IAA^{(2)}, \IAA^{(1)}, \IAA^{(0)})$ by an induction argument.

The complex $\IAA$ appearing in \cref{thm_connectivity_IAA} (see \cref{sec_definition_IAA}) is a subcomplex of $\IAA^{(2)}$ that has only those augmentations that we need to prove its high connectivity. We use \cref{thm_connectivity_IAA} to deduce the above connectivity results.

\subsection{Structure of the paper}

We begin the article with preliminaries on symplectic linear algebra and simplicial complexes in \cref{sec_prelims}.

\cref{sec_definitions} to  \cref{sec_normal_form_spheres} essentially all aim at proving \autoref{thm_connectivity_IAA}.
This starts in \cref{sec_definitions}, where we define $\IAA$ and several other simplicial complexes.  
In \cref{sec_structure_of_links}, we describe the links, i.e.~the local structure of these complexes. 
\cref{sec:regular-maps} first gives background material on combinatorial manifolds and then describes the combinatorial structure of the maps that we use to study the connectivity properties of the complexes in this paper.
In \cref{sec_retraction}, we explain how to reduce the complexity (measured by a natural number that we call the rank) of maps $\phi\colon  S\to \IAA$ from combinatorial spheres to the complex $\IAA$. This involves constructing retractions on certain well-behaved subcomplexes of links of vertices in these complexes. \cref{sec_highly_connected_subcomplexes} is a collection of results that show that many subcomplexes of $\IAA$ are highly connected. These are mostly drawn from existing literature or straight-forward consequences of such.
Finally, \cref{sec:thm_connectivity_IAA} describes an induction that is the core of the proof of \autoref{thm_connectivity_IAA}. It relies on an assumption on the link structure in maps $\phi\colon  S\to \IAA$ from combinatorial spheres to $\IAA$. 
\cref{sec_normal_form_spheres} shows that this assumption can always be made. It is rather technical and concerned with manipulations of the combinatorial structure of such maps.

In \cref{sec:presentation}, we infer further connectivity results from \autoref{thm_connectivity_IAA}. These are then used to prove \autoref{thmB} by obtaining a partial resolution of the symplectic Steinberg module $\Stsymp_n$.
\cref{sec_cohomology_vanishing} shows that, using Borel--Serre duality, \autoref{thmA} is a straight-forward consequence of the partial resolution given by \autoref{thmB}.

The appendix contains \cref{table_overview_all_complexes}, which gives an overview of the different complexes in this work.

\subsubsection{Your personal reading guide}
Admittedly, this paper is long. The following suggestions provide shortcuts to the main results.

\paragraph{The codimension-0 case (16 pages)}
The short article \cite{bruecksroka2023} gives a geometric proof of Gunnells' result that the symplectic Steinberg module is generated by integral apartment classes. It showcases many of the techniques used in the present article and could be helpful to get a first impression.

\paragraph{Fast track to \autoref{thmA} ($\sim$6 pages)}
If you would like to see as fast as possible why the codimension-1 cohomology of $\Sp{2n}{\mbZ}$, or equivalently $\mcA_n$, vanishes, you should first glance at the symplectic linear algebra notation in \cref{sec_prelim_symplectic_LA}. Then look at \cref{def_building_over_Z}, which defines the symplectic building $T^\omega_n$. This should already allow you to understand the partial resolution of the Steinberg module $\Stsymp_n = \widetilde H_{n-1}(T^\omega_n;\mbZ)$ obtained in \cref{thm:stpres} (the relevant modules and morphisms are defined in the text starting after \cref{thm:stpres} and ending with \cref{def:skew-apartment-module}). Accepting that the sequence from \cref{thm:stpres} is exact, you can proceed to \cref{sec_cohomology_vanishing}, where we show that this partial resolution together with Borel--Serre duality imply the cohomology vanishing. This is a standard homological algebra argument.

\paragraph{Fast track to \autoref{thmB} ($\sim$23 pages)}
This could be the right option if you would like to understand how we obtain a presentation of the Steinberg module $\Stsymp_n$ but avoid the explicit combinatorial topology in the proof of \autoref{thm_connectivity_IAA}.
Start by reading the basics about symplectic linear algebra in \cref{sec_prelim_symplectic_LA} and about simplicial complexes in \cref{sec_prelims_simplicial_complexes}. Then jump to \cref{sec:presentation}, which contains the proof of \autoref{thmB}. In \cref{subsec:highly-connected-complex-IAA-zero-one-two}, the simplicial complexes $\IAA^{(i)}$ that were already mentioned in \cref{sec_intro_simplicial_complexes} above are defined. For these definitions, you will need to have a look at the simplex types given in \cref{def:IA-simplices}, \cref{def:BAA-simplices} and \cref{def:IAA-simplices}, but not more from \cref{sec_definitions}. In \cref{sec_connectivity_intermediate}, we show that the complexes $\IAA^{(i)}$ are highly connected. This is where we use the work from \hyperref[sec_structure_of_links]{Sections \ref{sec_structure_of_links}} to \ref{sec_normal_form_spheres}, in particular \autoref{thm_connectivity_IAA}. But if you take these on faith, the rest of \cref{sec:presentation} is mostly algebraic.

\paragraph{Fast track to \autoref{thm_connectivity_IAA} ($\sim$27 pages)}
If you would like to get an idea of what the simplicial complexes in this article look like and why they are highly connected, you could do the following.
Start with the preliminaries on symplectic linear algebra in \cref{sec_prelim_symplectic_LA} and have a look at the conventions about simplicial complexes in \cref{sec_prelims_simplicial_complexes}.
Continue with the definition of the simplicial complexes in \cref{sec_definitions}; there are quite a few types of augmentations in these, \cref{table_overview_complexes_definition} and \cref{table_overview_all_complexes} might help you to keep an overview.
Then skim through \cref{sec:regular-maps} and at the very least notice that we define a notion of regularity for maps (\cref{def_regular_maps}, \cref{def_reg_homotopies}).
In \cref{sec_retraction}, it is sufficient for now to read the introduction and in particular look at \cref{IAA-retraction}.
From there, proceed to \cref{sec:thm_connectivity_IAA}, which contains an induction that is the heart of the connectivity argument proving \autoref{thm_connectivity_IAA}. The induction beginning is in \cref{sec_induction_beginning}. It uses the fact that $\SL{2}{\mbZ} = \Sp{2}{\mbZ}$ and relies on the connectivity of complexes $\B_n$, $\BA_n$ and $\BAA_n$ that were studied in the context of the special linear group. The induction step is in \cref{sec_induction_step} and is a ``cut out'' argument for bad simplices. The regularity notions are crucial here, so you might need to go back to \cref{sec:regular-maps} to get a better understanding of this regularity. 
The induction step also relies on the fact that the relevant homotopy classes can be represented by maps having an ``isolation'' property (\cref{prop_isolating-rank-r-vertices}). The proof that this is true requires some technical work in \cref{sec_normal_form_spheres}, but if you are willing to take it on faith, you have already reached your goal.

\subsection{Acknowledgements}
We would like to thank Paul Gunnells, Jeremy Miller, Andrew Putman and Dan Yasaki for helpful conversations.
We are grateful about comments on a preliminary version by Jeremy Miller. 
We thank Matthew  Cordes for language advise and Samir Canning for comments about the connection to $\mcA_g$. Robin J.\ Sroka would like to thank his PhD advisor Nathalie Wahl for many fruitful and clarifying conversations.

All three authors were supported by the Danish National Research Foundation (DNRF92,  DNRF151).
Peter Patzt was supported in part by a Simons collaboration grant and the European Research Council under the European Union’s Seventh Framework Programme ERC Grant agreement ERC StG 716424 - CASe, PI Karim Adiprasito.
Robin J.\ Sroka was supported by the European Research Council (ERC grant agreement No.772960), by NSERC Discovery Grant A4000 as a Postdoctoral Fellow at McMaster University, and by the Swedish Research Council under grant no.\ 2016-06596 while the author was in residence at Institut Mittag-Leffler in Djursholm, Sweden during the semester \emph{Higher algebraic structures in algebra, topology and geometry}.

\section{Preliminaries}
\label{sec_prelims}

\subsection{Symplectic linear algebra}
\label{sec_prelim_symplectic_LA}
We begin with setting up the notation for linear algebra over the integers and recalling some basic facts about symplectic forms in this context. More details can be found in textbooks such as \cite{Milnor1973}.

Let $\mbZ^{2n}$ be equipped with the standard symplectic form $\omega = \omega_n$. We denote its standard symplectic basis by $\{\vec e_1, \ldots , \vec e_{n}, \vec f_{1}, \ldots, \vec f_n\} = \{\vec e_1, \vec f_1,\ldots , \vec e_{n},\vec f_{n}\}$. That is, $\mbZ^{2n} = \langle \vec e_1, \vec f_1,\ldots , \vec e_{n},\vec f_{n} \rangle$ and for $i,j \in \{1, \dots, n\}$, we have 
\begin{equation*}
	\omega(\vec e_i, \vec e_j) = \omega(\vec f_i, \vec f_j) = 0 , \quad
	\omega(\vec e_i, \vec f_j) =\omega(\vec f_j, \vec e_i) = 0  \text{ for } i \neq j,\quad
	\omega(\vec e_i, \vec f_i) = - \omega(\vec f_i, \vec e_i) = 1.
\end{equation*}

More generally, we call a set $\{\vec v_1, \vec w_1,\ldots , \vec v_{n},\vec w_{n}\}\subset \mbZ^{2n}$ a \emph{symplectic basis} if there is a symplectic (i.e.~form-preserving) automorphism $\mbZ^{2n} \to \mbZ^{2n}$ that sends $\vec v_i$ to $\vec e_i$ and $\vec w_i$ to $\vec f_i$ for all $1\leq i \leq n$. 

If $V \subseteq \mbZ^{2n}$ is a submodule, we define the \emph{symplectic complement} of $V$ as
\begin{equation*}
	V^{\perp} = \ls \vec v' \in \mbZ^{2n} \mid \omega(\vec v', \vec v) = 0 \text{ for all } v\in V  \rs.
\end{equation*}

Note that $\mbZ^{2n}$ is canonically embedded in $\mbQ^{2n}$ and that the form $\omega$ on $\mbZ^{2n}$ is the restriction of a corresponding symplectic form on $\mbQ^{2n}$. Almost all notions defined in this section have counterparts in the rational setting. However, we mostly restrict ourselves to the setting of $\mbZ^{2n}$ in this article. In particular, if we talk about the ``span'' $\langle \vec v_1, \ldots, \vec v_m\rangle$  of a set of vectors, we mean its $\mbZ$-span if we do not explicitly say something else.

\begin{definition}
Let $V \subseteq \Z^{2n}$ be a submodule. Then we call $V$
\begin{itemize}
\item a \emph{(direct) summand} if $\mbZ^{2n} = V \oplus W $ for some submodule $W\subseteq \mbZ^{2n}$;
\item an \emph{isotropic summand} if $V$ is a direct summand and $V\subseteq V^{\perp}$, that is $\omega(\vec v_1, \vec v_2) = 0$ for all $\vec v_1, \vec v_2 \in V$;
\item a \emph{symplectic summand} if $V$ is a direct summand and there is a form-preserving isomorphism $(V,\omega|_{V})\cong (\mbZ^{2k},\omega_k)$, where $\omega_k$ is the standard symplectic form on $\mbZ^{2k}$. The number $k$ is then called the \emph{genus} of the symplectic summand.
\end{itemize}
\end{definition}

The following two lemmas contain standard facts about direct summands. Their (elementary) proofs can e.g.~be found in \cite[Section 2.2]{CP}.

\begin{lemma}
\label{lem_direct_summands_basics}
If $V,W\subseteq \mbZ^{2n}$ are direct summands and $V\subseteq W$, then $V$ is a direct summand of $W$.
\end{lemma}

\begin{lemma}
\label{lem_subspace_summand}
Let $V$ be an isotropic subspace of $\mbQ^{2n}$. Then $V \cap \mbZ^{2n} \subseteq \mbZ^{2n}$ is an isotropic summand of $\mbZ^{2n}$.
\end{lemma}

In fact, in the definition of symplectic summands, the condition that $V$ be a direct summand is redundant. This follows from the following lemma, which is a consequence of \cite[Lemma I.3.1]{Milnor1973}.

\begin{lemma}
\label{lem_split_off_symplectic_summand}
Let $V\subseteq\mbZ^{2n}$ be a submodule such that $(V,\omega|_{V})\cong (\mbZ^{2k},\omega_k)$. Then $V$ is a (symplectic) summand. Furthermore, $V^\perp$ is a symplectic summand of genus $n-k$ and there is a symplectic isomorphism
\begin{equation*}
	(V\oplus V^\perp, \omega|_{V} \oplus \omega|_{V^\perp}) \to (\mbZ^{2n}, \omega).
\end{equation*}
\end{lemma}

In this article, rank-1 summands of $\mbZ^{2n}$, which we sometimes refer to as \emph{lines}, are especially important. They form the vertices of the simplicial complexes that we consider later on.
A vector $\vec v \in \mbZ^{2n}$ is called \emph{primitive} if its $\mbZ$-span is a rank-1 summand of $\mbZ^{2n}$.\footnote{If we express $\vec v$ as a row vector $\vec v = (x_1, \ldots, x_{2n})$, this is equivalent to saying that $\on{gcd}(x_1, \ldots, x_{2n} ) = 1$.} In this case, we write  $v = \ll \vec v\rr$ for this summand. Similarly, given a rank-1 summand $v$ of $\mbZ^{2n}$, we write $\vec v$ for some choice of primitive vector in $v$. There are only two choices for this. The other choice is $-\vec v$.
Using this notation, we write 
\begin{equation*}
\omega(v_1,v_2) = \pm r
\end{equation*}
if $r = \omega(\vec v_1, \vec v_2)$. If $r$ is $0$ here, we also drop the $\pm$-sign.

The following is a consequence of \cref{lem_split_off_symplectic_summand}, see \cite[Theorem I.3.2]{Milnor1973}.
\begin{lemma}
\label{lem_symplectic_pair}
Let $\vec v, \vec w \in \mbZ^{2n}$. If $\omega(\vec v, \vec w)\in \ls -1, 1\rs$, then both $\vec v$ and $\vec w$ are primitive and $\langle \vec v, \vec w \rangle$ is a symplectic summand of $\mbZ^{2n}$.
\end{lemma}
We call vectors $\vec v, \vec w$ satisfying the conditions of \cref{lem_symplectic_pair} a \emph{symplectic pair}. We also use this formulation for the pair of corresponding lines $v$, $w$.

Roughly speaking, the next lemma says that a set of vectors that spans a direct summand and looks like a partial symplectic basis can actually be extended to a symplectic basis. It is a standard fact that can e.g.~be proved using \cite[Lemma I.2.6]{Milnor1973}.
\begin{lemma}
\label{lem_extend_to_symplectic_basis}
Let $V\subseteq \mbZ^{2n}$ be a direct summand and 
\begin{equation*}
	B = \ls \vec v_1, \ldots, \vec v_m, \vec v_{m+1}, \vec w_{m+1}, \ldots, \vec v_{m+k}, \vec w_{m+k} \rs
\end{equation*}
a basis of $V$ such that for all $i,j\in \ls 1, \ldots, n \rs$
\begin{equation*}
	\omega(\vec v_i, \vec v_j) = \omega(\vec w_i, \vec w_j) = 0 , \quad
	\omega(\vec v_i, \vec w_j) =\omega(\vec w_j, \vec v_i) = 0  \text{ for } i \neq j,\quad
	\omega(\vec v_i, \vec w_i) = - \omega(\vec w_i, \vec v_i) =1.
\end{equation*}
Then $B$ can be extended to a symplectic basis of $\mbZ^{2n}$.
\end{lemma}

We in particular use \cref{lem_extend_to_symplectic_basis} for the case where $B$ is either a symplectic pair or the basis of an isotropic summand $V$. We also need the following strengthening of \cref{lem_extend_to_symplectic_basis}.

\begin{lemma}
\label{lem_primitive_for_form}
Let $\vec e\in \mbZ^{2(m+n)}$ such that $\vec e \not\in \langle \vec e_1, \ldots, \vec e_m \rangle$ and $\omega(\vec e,\vec e_i) = 0$ for all $1\leq i \leq m$. Then there are $\vec e' \in \mbZ^{2(m+n)}$ and $a \in \mbZ$ such that 
\begin{itemize}
\item $\omega(\vec e',\vec e_i) = 0$ for all $1\leq i \leq m$,
\item $\ls \vec e_1, \ldots, \vec e_m, \vec e' \rs$ can be extended to a symplectic basis of $\mbZ^{2(m+n)}$ and
\item for all $\vec v\in \mbZ^{2(m+n)}$ such that $\omega(\vec v, \vec e_i) = 0$ for all $1\leq i \leq m$, we have $\omega(\vec e,\vec v) = a \omega(\vec e', \vec v)$.
\end{itemize}
\end{lemma}
\begin{proof}
As $\vec e \not\in \langle \vec e_1, \ldots, \vec e_m \rangle$ and $\omega(e,e_i) = 0$ for all $1\leq i \leq m$, the $\mbQ$-span $\langle \vec e_1, \ldots,\vec  e_m, \vec e \rangle_{\mbQ}$ is an $(m+1)$-dimensional isotropic subspace of $\mbQ^{2(m+n)}$. It contains $\langle \vec e_1, \ldots,\vec  e_m \rangle_{\mbQ}$ as an $m$-dimensional subspace. This implies that 
\begin{equation*}
	V \coloneqq \langle \vec e_1, \ldots,\vec  e_m, \vec e \rangle_{\mbQ} \cap \mbZ^{2(m+n)}
\end{equation*}
is a rank-$(m+1)$ isotropic direct summand of $\mbZ^{2(m+n)}$ that contains 
\begin{equation*}
	\langle \vec e_1, \ldots, \vec e_m \rangle_{\mbQ}\cap \mbZ^{2(m+n)} = \langle \vec e_1, \ldots, \vec e_m \rangle \subseteq V
\end{equation*}
as a direct summand of rank $m$.
Hence, we can choose $\vec e' \in V$ such that $\ls \vec e_1, \ldots,\vec  e_m, \vec e'  \rs$ is basis (over $\mbZ$) for $V$. We claim that $\vec e'$ satisfies the desired properties.
To see this, first observe that $\ls \vec e_1, \ldots,\vec  e_m, \vec e'  \rs$ is a basis for the isotropic direct summand $V\subseteq \mbZ^{2(m+n)}$. Hence by \cref{lem_extend_to_symplectic_basis}, it can be extended to a symplectic basis of $\mbZ^{2(m+n)}$.
Now let $\vec v\in \mbZ^{2(m+n)}$ such that $\omega(\vec v, \vec e_i) = 0$ for all $1\leq i \leq m$.
As $\vec e$ is contained in $V$, we can write it as
\begin{equation*}
	\vec e = \sum_{i=1}^m a_i \vec e_i +a \vec e'.
\end{equation*}
Using that $\omega(\vec v, \vec e_i) = 0$ for all $1\leq i \leq m$, we have
\begin{equation*}
	\omega(\vec e, \vec v)  = \omega\left(\sum_{i=1}^m a_i \vec e_i +a \vec e', \vec v \right) = \sum_{i=1}^m a_i \omega(\vec e_i, \vec v) + a \omega(\vec e', \vec v) = a \omega(\vec e', \vec v). \qedhere
\end{equation*}
\end{proof}

\subsection{Simplicial complexes}
\label{sec_prelims_simplicial_complexes}
Most of this article is concerned with studying connectivity properties of simplicial complexes. 
Usually, we consider a simplicial complex $X$ as a collection of subsets, the set of simplices, of a set $\Vr(X)$, the vertex set of $X$, such that the set of simplices is closed under passing to subsets.
However, we often do not distinguish between $X$ and its geometric realisation $|X|$ if what is meant seems clear from the context. In particular, we associate topological properties to simplicial complexes and e.g.~say that $X$ is \emph{$n$-connected} if $|X|$ is, i.e.~if $\pi_k(|X|)$ is trivial for all $k\leq n$. We use the convention that $(-1)$-connected means that $X$ is non-empty.

If $X$ and $Y$ are simplicial complexes, we denote their simplicial join by $X\ast Y$.
If $\Delta$ is a simplex of $X$, we write $\Link_X(\Delta)$ for the \emph{link} of $\Delta$ in $X$, i.e.~the complex consisting of all simplices $\Theta$ in $X$ such that $\Theta \cap \Delta = \emptyset$ and $\Theta \cup \Delta$ is a simplex in $X$. We write $\Star_X(\Delta) = \Delta \ast \Link_X(\Delta)$ for the \emph{star} of $\Delta$.
We say that $\Theta$ is a \emph{face} of the simplex $\Delta$ if $\Theta \subseteq \Delta$; note that this containment need not be proper, so in particular every simplex is a face of itself.

If $X$ is a simplicial complex and $Y\subseteq X$ is a subcomplex, then we denote by $X\setminus Y$ the \emph{set} of simplices of $X$ that are not contained in $Y$. Note that this set is a priori not a simplicial complex.

A simplicial complex $X$ is (homotopy) \emph{Cohen--Macaulay of dimension $n$} if it has dimension $n$, is $(n-1)$-connected and for every simplex $\Theta$, the link $\Link_X(\Theta)$ is $(n-\dim(\Theta)-2)$-connected.

\subsubsection{Standard link argument} \label{subsec:standard-link-argument}
Recall that a map $f\colon  X \to Y$ between topological spaces is \emph{$n$-connected} if the induced map on homotopy groups $\pi_k (f)\colon  \pi_k(X) \to \pi_k(Y)$
is an isomorphism for all $k< n$ and a surjection for $k=n$.

The following results are contained in \cite[Section 2.1]{hatchervogtmann2017tethers} and are frequently used in this article. Let $X$ be a simplicial complex and let $Y \subseteq X$ be a subcomplex.

\begin{definition}\label{def:standard-link-argument-bad-simplices}
  A set of simplices $B \subset X\setminus Y$ is a set of \emph{bad simplices} if the following two conditions hold for any simplex $\Delta$ of $X$:
  \begin{enumerate}
  \item If no face of $\Delta$ is contained in $B$, then $\Delta$ is contained in $Y$.
  \item If two faces $\Theta_1$ and $\Theta_2$ of $\Delta$ are in $B$, then $\Theta_1 \cup \Theta_2$ is also in $B$.
  \end{enumerate}
\end{definition}

Recall that also in the above definition, a ``face'' need not be proper.

\begin{definition}\label{def:standard-link-argument-good-link}
  Let $B$ be a set of bad simplices and assume that $\Delta \in B$ is a bad simplex. Then we let $\Link^{\mathrm{good}}_X(\Delta)$ denote the subcomplex of $\Link_X(\Delta)$ containing all simplices $\Theta \in \Link_X(\Delta)$ that satisfy the following condition:
  $$\text{Any bad face of } \Theta \ast \Delta \text{ is contained in } \Delta.$$
\end{definition}

\begin{lemma}
\label{lem:standard-link-argument}
  Let $X$ be a simplicial complex and let $Y \subseteq X$ be a subcomplex. Assume that $X\setminus Y$ has a set of bad simplices $B$. Then:
  \begin{enumerate}
  \item \label{it_link_argument_Y_to_X}If $Y$ is n-connected and $\Link^{\mathrm{good}}_X(\Delta)$ is $(n-\dim(\Delta)-1)$-connected for all $\Delta \in B$, then $X$ is n-connected.
  \item \label{it_link_argument_connected_map}If $\Link^{\mathrm{good}}_X(\Delta)$ is $(n-\dim(\Delta)-1)$-connected for all $\Delta \in B$, then the inclusion $Y\hookrightarrow X$ is $n$-connected.
  \end{enumerate}
\end{lemma}
\begin{proof}
\cref{it_link_argument_Y_to_X} is exactly \cite[Corollary 2.2(a)]{hatchervogtmann2017tethers}.

\cref{it_link_argument_connected_map} is just a slightly stronger version of \cite[Corollary 2.2]{hatchervogtmann2017tethers} and also quickly follows from the results of Hatcher and Vogtmann: By \cite[Proposition 2.1]{hatchervogtmann2017tethers}, the assumptions here imply that the relative homotopy groups $\pi_k(X,Y)$ vanish for all $k\leq n$. The claim then follows from the long exact sequence of homotopy groups.
\end{proof}

\subsubsection{Zeeman's relative simplicial approximation theorem}

We use the following version of simplicial approximation, which is due to Zeeman.

\begin{theorem}[\cite{Zeeman1964}] \label{zeeman-relative-simplicial-approximation}
	Let $K, M$ be finite simplicial complexes and $L$ a subcomplex of $K$. Let $f\colon  |K| \to |M|$ be a continuous map such that the restriction $f|_L$ is a simplicial map from $L$ to $M$. Then there exists a subdivision $K'$ of $K$ containing $L$ as a subcomplex and a simplicial map $g\colon  K' \to M$ such that $g|_L = f|_L$ and $g$ is homotopic to $f$ keeping $L$ fixed.
\end{theorem}

\section{Definition of complexes and augmentations}
\label{sec_definitions}
The goal of this section is to define the complex $\IAA$ appearing in the main technical result of this work, \autoref{thm_connectivity_IAA}. We start by collecting the definitions of the simplicial complexes $\B_n$, $\BA_n$ and $\BAA_n$, that play an important role in recent work on special linear groups \cite{CP, Brueck2022}, as well as the simplicial complexes $\I$, $\Isigdel$ and $\IA$, that have been used to study Torelli and symplectic groups \cite{put2009, bruecksroka2023}. We then introduce the new complex $\IAA$. Lastly, we define relative versions of all these spaces, e.g.\ $\BA_n^m$, $\Irel$, and $\IAArel$, that depend on two non-negative integers $m$ and $n$.

\cref{table_overview_complexes_definition} gives an overview of the definitions and \cref{table_overview_all_complexes} shows the connectivity properties of these complexes that are studied in later sections.

\subsection{The complexes \texorpdfstring{$\I$, $\IA$, $\B_n$ and $\BA_n$}{I, IA, B, and BA}}\label{sec_defIA}
Let $\VertexSet{n}$ be the set 
\begin{equation}
\label{eq_V_n}
\VertexSet{n} \coloneqq \ls v \subseteq \mbZ^{2n} \,\middle|\, v \text{ is a rank-1 summand of } \mbZ^{2n}\rs.
\end{equation}
This is already the vertex set of the complex $\IAA$ that we want to show to be highly connected in \autoref{thm_connectivity_IAA}. We now start working towards its definition.

Given a line $v \in \VertexSet{n}$, recall that we denote by $\vec v$ the choice of one of the two primitive vectors $\{\vec v, -\vec v\}$ contained in $v$.

\begin{definition} \label{def:IA-simplices}
  Let $\VertexSet{n}$ be as in \cref{eq_V_n}. A subset $\Delta = \{v_0, \dots, v_k\} \subset \VertexSet{n}$
  of $k+1$ lines\footnote{Note that $\Delta$ is a set without a preferred ordering, so all the statements here as well as in \cref{def:BAA-simplices} and \cref{def:IAA-simplices} are to be understood for an appropriate choice of indices for the $v_i$.} $v_i = \langle \vec v_i \rangle$ is called
  \begin{enumerate}
  \item a \emph{standard} simplex, if $\langle \vec v_i \mid 0 \leq i \leq k \rangle$ is an isotropic rank-$(k+1)$ summand of $\mbZ^{2n}$;
  \item a \emph{2-additive} simplex, if $k \geq 2$, $	\vec v_0 = \pm \vec v_1 \pm  \vec v_2$,\footnote{Here as well as in \cref{def:BAA-simplices} and \cref{def:IAA-simplices}, the $\pm$ in these equations are to be understood as ``for some choice of signs''.}
  and $\Delta \setminus \{v_0\}$ is a standard simplex;
  \item a $\sigma$ simplex, if $k \geq 1$, $\omega(v_{k}, v_{k-1}) = \pm 1$, $\omega(v_{k}, v_i) = 0$ for $i \not= k-1$
  and $\Delta \setminus \{v_k\}$ is a standard simplex;
  \item a \emph{mixed} simplex, if $k \geq 3$, $\Delta \setminus \{v_0\}$ is a $\sigma$ simplex and $\Delta \setminus \{v_k\}$ is a 2-additive simplex.
  \end{enumerate}
\end{definition}

\begin{example}
Let $\{ \vec e_1, \ldots, \vec e_n, \vec f_1, \ldots, \vec f_n \}$ be the standard symplectic basis of $\mbZ^{2n}$ and let $1\leq {\lastindexexamples} \leq n$.
\begin{enumerate} 
\item $\ls e_1,\ldots, e_{\lastindexexamples} \rs$ is a standard simplex;
\item $\ls \langle \vec e_1 + \vec e_2 \rangle, e_1, e_2, \ldots, e_{\lastindexexamples}  \rs$ is a 2-additive simplex if ${\lastindexexamples} \geq 2$;
\item $\ls e_1, \ldots, e_{\lastindexexamples}, f_{\lastindexexamples}  \rs$ is a $\sigma$ simplex;
\item $\ls \langle \vec e_1 + \vec e_2 \rangle, e_1, e_2, \ldots, e_{\lastindexexamples}, f_{\lastindexexamples}  \rs$ is a mixed simplex if ${\lastindexexamples} \geq 3$.
\end{enumerate}
\end{example}

We are now ready to introduce the simplicial complexes $\I, \Idel, \Isigdel$ and $\IA$. The first three of these were used by Putman to study the Torelli group \cite{put2009}. A complex that is closely related to $\I$ also appears in work of van der Kallen--Looijenga \cite{vanderkallenlooijenga2011sphericalcomplexesattachedtosymplecticlattices}. The complex $\IA$ was used in \cite{bruecksroka2023} and \cite[Chapter 5]{Sroka2021} to study the cohomology of symplectic groups.

\begin{definition} \label{def:I-and-IA}
  The simplicial complexes $\I, \Idel,\Isigdel$ and $\IA$ have $\VertexSet{n}$ as their vertex set and
  \begin{enumerate}
  \item the simplices of $\I$ are all standard;
  \item the simplices of $\Idel$ are all either standard or 2-additive;
  \item the simplices of $\Isigdel$ are all either standard, 2-additive or $\sigma$;
  \item the simplices of $\IA$ are all either standard, 2-additive, $\sigma$ or mixed.
  \end{enumerate}
	If $\SymplecticSummand \subseteq \mbZ^{2n}$ is a symplectic summand and $X \in \{\I[], \Idel[], \Isigdel[], \IA[]\}$, we let $X(\SymplecticSummand)$ be the full subcomplex of $X_n$ on the set of rank-$1$ summands of $\SymplecticSummand$ that we denote by $\VertexSet{n} \cap \SymplecticSummand$.
\end{definition}

Next, we introduce the complexes $\B_n$ and $\BA_n$. These complexes were defined and studied by Church--Putman in \cite{CP}. Complexes that are closely related to $\B_n$ appear in \cite{Maazen79, vanderkallen1980homologystabilityforlineargroups, cfp2019}.

\begin{definition} \label{def:B-and-BA}
	Let $V$ be an isotropic summand of $\mbZ^{2n}$.
  \begin{enumerate}
  \item Let $\B(V)$ be the simplicial complex with vertex set 
  $$\VertexSet{n} \cap V = \ls v \subseteq V \,\middle|\, v \text{ is a rank-1 summand of } V \rs$$
   and in which all simplices are standard in the sense of \cref{def:IA-simplices}.
  \item Let $\BA(V)$ be the simplicial complex on the same vertex set as $\B(V)$ and in which all simplices are standard or 2-additive in the sense of \cref{def:IA-simplices}.
  \end{enumerate}
	If $V = \langle \vec e_1, \dots, \vec e_n \rangle \subseteq \mbZ^{2n}$, we write $\B_n \coloneqq \B(V)$ and $\BA_n \coloneqq \BA(V)$.
\end{definition}

\subsection{The complex \texorpdfstring{$\BAA_n$}{BAA}}
The complex $\BAA_n$, defined below, was introduced by Brück--Miller--Patzt--Sroka--Wilson \cite{Brueck2022} and used to study the codimension-2 cohomology of $\SL{n}{\mbZ}$.

\begin{definition} \label{def:BAA-simplices}
	Let $\VertexSet{n}$ be as in \cref{eq_V_n}. A subset $\Delta = \{v_0, \dots, v_k\} \subset \VertexSet{n}$ of $k+1$ lines is called
  \begin{enumerate}
  \item a \emph{3-additive} simplex, if $k \geq 3$,
    $\vec v_0 = \pm \vec v_1 \pm \vec v_2 \pm \vec v_3$
    and $\Delta \setminus \{v_0\}$ is a standard simplex;           
  \item a \emph{double-triple} simplex, if $k \geq 4$, $\vec v_0 = \pm \vec v_2 \pm \vec v_3$, $\vec v_1 = \pm \vec v_2 \pm \vec v_4$
    and $\Delta \setminus \{v_0, v_1\}$ is a standard simplex;
  \item a \emph{double-double} simplex, if $k \geq 5$, $\vec v_0 = \pm \vec v_2 \pm \vec v_3$, $\vec v_1 = \pm \vec v_4 \pm \vec v_5$
     and $\Delta \setminus \{v_0, v_1\}$ is a standard simplex.
  \end{enumerate}
\end{definition}

\begin{example}
Let $\{ \vec e_1, \ldots, \vec e_n, \vec f_1, \ldots, \vec f_n \}$ be the standard symplectic basis of $\mbZ^{2n}$ and $1 \leq {\lastindexexamples} \leq n$.
\begin{enumerate} 
\item $\ls \langle \vec e_1 + \vec e_2 +\vec e_3 \rangle, e_1, e_2, e_3, \ldots, e_{\lastindexexamples}  \rs$ is a 3-additive simplex if ${\lastindexexamples} \geq 3$;
\item $\ls \langle \vec e_1 + \vec e_2 \rangle, \langle \vec e_1 + \vec e_3 \rangle, e_1, e_2, e_3, \ldots, e_{\lastindexexamples}  \rs$ is a double-triple simplex if ${\lastindexexamples} \geq 3$;
\item $\ls \langle \vec e_1 + \vec e_2 \rangle, \langle \vec e_3 + \vec e_4 \rangle, e_1, \dots , e_4, \ldots, e_{\lastindexexamples}  \rs$ is a double-double simplex if ${\lastindexexamples} \geq 4$.
\end{enumerate}
\end{example}

\begin{definition} \label{def:BAA}
  Let $V$ be an isotropic summand of $\mbZ^{2n}$. We define $\BAA(V)$ to be the simplicial complex with vertex set 
  $$\VertexSet{n} \cap V = \{v \subseteq V  \mid v \text{ is a rank-1 summand of } V \}$$
  and in which all simplices are either
  \begin{enumerate}
  \item standard or 2-additive simplices in the sense of \cref{def:IA-simplices} or
  \item 3-additive, double-triple or double-double simplices in the sense of \cref{def:BAA-simplices}.
  \end{enumerate}
	If $V = \langle \vec e_1, \dots, \vec e_n \rangle \subseteq \mbZ^{2n}$, we write $\BAA_n \coloneqq \BAA(V)$.
\end{definition}

\subsection{The complex \texorpdfstring{$\IAA$}{IAA}}
\label{sec_definition_IAA}
We now turn to the definition of the new complex $\IAA$ that we study in this work and show to be $n$-connected (\autoref{thm_connectivity_IAA}). The next definition describes the new simplices appearing in $\IAA$.
 
\begin{definition} \label{def:IAA-simplices}
	Let $\VertexSet{n}$ be as in \cref{eq_V_n}.
  A subset $\Delta = \{v_0, \dots, v_k\} \subset \VertexSet{n}$ of $k+1$ lines is called
  \begin{enumerate}
   \item a $\sigma^2$ simplex, if $k\geq 3$,  $\omega(v_{k-1}, v_{k-3})  = \omega(v_{k}, v_{k-2}) = \pm 1$, $\omega(v_{i}, v_{j}) = 0$ otherwise
  and $\Delta \setminus \{v_{k-1}, v_k\}$ is a standard simplex;
  \item a \emph{skew-additive} simplex, if $k\geq 2$, $\omega(v_{k}, v_{0}) = \omega(v_{k}, v_{1}) = \pm 1$, $\omega(v_{i}, v_{j}) = 0$ otherwise
 	  and $\Delta \setminus \{v_k\}$ is a standard simplex;
 	\item a \emph{2-skew-additive} simplex, if $k\geq 3$, $\vec v_0 = \pm \vec v_1 \pm \vec v_2$, 
 	\begin{equation*}
 	\omega(v_{k}, v_{0}) = \omega(v_{k}, v_{1}) = \pm 1,
 	\end{equation*}
 	$\omega(v_{i}, v_{j}) = 0$ otherwise
  	and $\Delta \setminus \{v_0, v_k\}$ is a standard simplex;
  \item a \emph{skew-$\sigma^2$} simplex if $k\geq 3$,
  \begin{equation*}
  	\omega(v_{k-1}, v_{k-3}) = \omega(v_{k}, v_{k-3}) = \omega(v_{k}, v_{k-2}) = \pm 1,
  \end{equation*}
	$\omega(v_{i}, v_{j}) = 0$ otherwise
  	and $\Delta \setminus \{v_{k-1}, v_k\}$ is a standard simplex;
  	\item a \emph{$\sigma$-additive} simplex, if $k\geq 2$,  $\vec v_k = \pm \vec v_{k-1} \pm \vec v_{k-2}$,
  	\begin{equation*}
  		\omega(v_{k-2}, v_{k-1}) = \omega(v_{k-2}, v_{k}) = \omega(v_{k-1}, v_{k}) = \pm 1,
  	\end{equation*}
	$\omega(v_{i}, v_{j}) = 0$ otherwise
 	  and $\Delta \setminus \{v_{k-1}, v_k\}$ is a standard simplex.
  \end{enumerate}
\end{definition}

\begin{example}
	\label{expl:simplices-iaa}
Let $\{ \vec e_1, \ldots, \vec e_n, \vec f_1, \ldots, \vec f_n \}$ be the standard symplectic basis of $\mbZ^{2n}$ and let $1\leq {\lastindexexamples} \leq n$.
\begin{enumerate} 
\item $\{e_{\lastindexexamples}, \ldots, e_2, e_1, f_2, f_1\}$ is a $\sigma^2$ simplex if ${\lastindexexamples} \geq 2$;
\item $\{e_1, e_2, \ldots, e_{\lastindexexamples}, \langle \vec f_1- \vec f_2\rangle \}$ is a skew-additive simplex if ${\lastindexexamples} \geq 2$;
\item $\{\langle \vec e_1+ \vec e_2 \rangle, e_1, e_2, \ldots, e_{\lastindexexamples}, \langle \vec f_1- \vec f_2 \rangle\}$ is a 2-skew-additive simplex if ${\lastindexexamples} \geq 2$;
\item $\{e_{\lastindexexamples}, \ldots, e_2, e_1, f_2, \langle \vec f_1 - \vec f_2 \rangle\}$ is a skew-$\sigma^2$ simplex if ${\lastindexexamples} \geq 2$;
\item $\{e_1, \ldots, e_{{\lastindexexamples}}, f_{\lastindexexamples}, \langle \vec e_{\lastindexexamples}+ \vec f_{\lastindexexamples} \rangle\}$ is a $\sigma$-additive simplex if ${\lastindexexamples} \geq 1$.
\end{enumerate}
\end{example}

\begin{definition} \label{def:IAA} \label{def:IAAst}
  The simplicial complexes $\IAAst$ and $\IAA$ have $\VertexSet{n}$ as their vertex sets. The simplices of $\IAAst$ are the ones introduced in \cref{def:IA-simplices} and \cref{def:BAA-simplices}, the simplices of $\IAA$ are the ones introduced in \cref{def:IA-simplices}, \cref{def:BAA-simplices} and \cref{def:IAA-simplices}. 
  If $\SymplecticSummand \subseteq \mbZ^{2n}$ is a symplectic summand, we let  $\IAAst[](\SymplecticSummand)$ and $\IAA[](\SymplecticSummand)$ denote the full subcomplexes of $\IAAst$ and $\IAA$, respectively, on the set $\VertexSet{n} \cap \SymplecticSummand$ of rank-$1$ summands of $\SymplecticSummand$.
\end{definition}

The definitions of the complexes introduced in this and the previous two subsections are summarised in \cref{table_overview_complexes_definition}.	

\begin{table}
\begin{center}
\begin{tabular}{l|p{9.8cm}}
Complex & Simplex types \\
\hline
$\I$ 	& standard \\
$\Idel$ 	& standard, 2-additive \\
$\Isigdel$ 	& standard, 2-additive, $\sigma$ \\
$\IA$ 	& standard, 2-additive, $\sigma$, mixed \\
$\IAAst$ 	& standard, 2-additive, $\sigma$, mixed, 3-additive, double-triple, double-double \\
$\IAA$ 	& standard, 2-additive, $\sigma$, mixed, 3-additive, double-triple, double-double, $\sigma^2$, skew-additive, 2-skew-additive, skew-$\sigma^2$, $\sigma$-additive  \\
$\B_n$ 	& standard \\
$\BA_n$ 	& standard, 2-additive\\
$\BAA_n$ 	& standard, 2-additive, 3-additive, double-triple, double-double
\end{tabular}
\end{center}
\caption{The different complexes and their simplex types. The complexes starting with $\mathcal{I}$ are defined over $(\mbZ^{2n}, \omega)$, the ones starting with $\mathcal{B}$ over an isotropic summand of rank $n$. \cref{table_overview_all_complexes} also shows their connectivity properties.}
\label{table_overview_complexes_definition}
\end{table}

\subsection{Relative complexes}

In this final subsection, we introduce relative versions of the complexes introduced above (and listed in \cref{table_overview_complexes_definition}). These relative versions depend on two non-negative integers $m$ and $n$, e.g.\ $\BA_n^m$, $\Irel$, and $\IAArel$ and will be used to inductively prove that $\IAA$ is $n$-connected. 

Throughout this subsection we let 
\begin{equation}
	n,m \geq 0
	\tag{Standing assumption}
\end{equation}
be non-negative integers and consider the symplectic module $\mbZ^{2(m+n)}$ of genus $m+n$ equipped with its symplectic standard basis $\{\vec e_1, \ldots, \vec e_{m+n}, \vec f_1, \ldots, \vec f_{m+n}\}$.

\begin{definition}
\label{def_linkhat}
Let $X_{m+n}$ denote one of the complexes defined above,
$$X_{m+n} \in \{\B_{m+n}, \BA_{m+n}, \BAA_{m+n}, \I[m+n], \Idel[m+n], \Isigdel[m+n], \IA[m+n], \IAAst[m+n], \IAA[m+n]\}.$$
\begin{enumerate}
\item \label{def_Xmn}We denote by $X_n^m\subseteq \Link_{X_{m+n}}(\ls e_1, \ldots, e_m \rs)$
the full subcomplex on the set of vertices $v$ satisfying the following.
  \begin{enumerate}
  \item \label{condition_one} $\vec v \notin \langle \vec e_1, \dots, \vec e_m \rangle$.
  \item \label{condition_two} For $1\leq i \leq m$, we have $\omega(e_i , v) = 0$, i.e.~there is no $\sigma$ edge between $v$ and one of the vertices of $\{e_1, \ldots, e_m\}$.
  \end{enumerate}
  Note that $X^0_n = X_n$.
\item Let $\Delta = \{v_0, \dots, v_k\}$ be a simplex of $X_n^m$. We denote by $
	\Linkhat_{X_n^m}(\Delta) \subseteq \Link_{X_n^m}(\Delta)$
 the full subcomplex  on the set of vertices $v$ satisfying the following.
  \begin{enumerate}
  \item $\vec v \notin \langle \vec e_1, \ldots, \vec e_m, \vec v_0, \dots, \vec v_k \rangle$
  \item For $0\leq i \leq k$, we have $\omega(v_i , v) = 0$.
  \end{enumerate}
  \item If $\SymplecticSummand$ is a symplectic summand of $\mbZ^{2(m+n)}$ that contains $\langle \vec e_1, \dots, \vec e_m \rangle$ and $X \not\in \{\B, \BA, \BAA\}$, we denote by 
  $ 	X^m(\SymplecticSummand)\subseteq \Link_{X(\SymplecticSummand)}(\ls e_1, \ldots, e_m \rs)$
    the full subcomplex on the set of vertices satisfying \cref{condition_one} and \cref{condition_two}. 
  \item If $V$ is an isotropic summand of $\mbZ^{2(m+n)}$ that contains $\langle \vec e_1, \dots, \vec e_m \rangle$ and $X \in \{\B, \BA, \BAA\}$, we denote by
$X^m(V)\subseteq \Link_{X(V)}(\ls e_1, \ldots, e_m \rs)$
  the full subcomplex on the set of vertices satisfying \cref{condition_one} and \cref{condition_two}.
\end{enumerate}

\end{definition}
\begin{remark}
	\label{rem:about-relative-bbabaa-complexes}
	If $X \in \{\B, \BA, \BAA\}$, then the second condition in the first two items above is trivially satisfied. Indeed, by \cref{def:B-and-BA} and \cref{def:BAA}, any two vertices $v, v' \in X(V)$ satisfy $\omega (\vec v, \vec v') = 0$ because these are lines that are contained in the isotropic summand $V$ of $\mbZ^{2(m+n)}$.
\end{remark}

\begin{example}
 Let $m+n \geq 4$, $n\geq 1$, $\{ \vec e_1, \ldots, \vec e_{m+n}, \vec f_1, \ldots, \vec f_{m+n}\}$ be the standard symplectic basis of $\mbZ^{2(m+n)}$ and $X_{m+n}$ as in \cref{def_linkhat}.
	\begin{enumerate}
		\item $\langle \vec e_2 + \vec e_3 \rangle$ is a vertex in $X^m_n$ if $0 \leq m \leq 2$ but not if $m \geq 3$.
		\item  Let $X \notin \{\B, \BA, \BAA\}$. Then $\langle \vec f_2 + \vec f_3 \rangle$ is a vertex in $X^m_n$ if $0 \leq m \leq 1$ but not if $m \geq 2$.
		\item Let $X \notin \{\B, \I[]\}$. Then $\langle \vec e_3 + \vec e_{4} \rangle$ is a vertex in $\Link_{X^m_n}(e_{4})$ and $\Linkhat_{X^m_n}(e_{4})$ if $0 \leq m \leq 2$. However, for $m = 3$ it is a vertex in the complex $\Link_{X^m_n}(e_{4})$ but not in $\Linkhat_{X^m_n}(e_{4})$.
		\item Let $X \notin \{\I[], \Idel[], \B, \BA, \BAA\}$. Then $f_{m+1}$ is a vertex in $\Link_{X^m_n}(e_{m+1})$, but not a vertex in $\Linkhat_{X^m_n}(e_{m+1})$.
	\end{enumerate}
\end{example}

\subsubsection{Simplex types in the relative complexes}
\label{subsec:simplex-types-in-relative-complexes}

Let $\Delta$ be a simplex in $\IAArel$. In this subsection, we introduce naming conventions that we use throughout this work to refer to simplices in e.g.\ $\IAArel$, $\Link_{\IAArel}(\Delta)$, $\Linkhat_{\IAArel}(\Delta)$ and all other relative complexes.

\begin{definition}
\label{def_relative_simplex_types}
Let $\Delta'$ be a simplex of $\Link_{\IAArel}(\Delta)$ and let 
$$\tau \in 
\begin{rcases}
\begin{dcases}
	\text{standard, 2-additive, 3-additive, double-triple, double-double,}\\
	\text{mixed, } \sigma, \sigma^2, \text{ skew-additive, 2-skew-additive, skew-}\sigma^2, \sigma\text{-additive}
\end{dcases}
\end{rcases}
$$
 be one of the simplex types defined the previous subsections. We say that $\Delta'$ is a \emph{simplex of type $\tau$} or a \emph{$\tau$ simplex} in $\Link_{\IAArel}(\Delta)$ if the \emph{underlying simplex} $	\{e_1, \dots, e_m\} \cup \Delta \cup \Delta'$
in $\IAA[m+n]$ is a simplex of type $\tau$.
\end{definition}

\begin{remark}
	\label{rem:relative-simplex-types-other-complexes}
	\cref{def_relative_simplex_types} also makes sense for simplices $\Delta'$ in
	\begin{itemize}
		\item $\IAArel$ using $\Delta = \emptyset$ and the convention that $\Link_{\IAArel}(\emptyset) = \IAArel$;
		\item $\Linkhat_{\IAArel}(\Delta)$ and any other subcomplex $X \subset \IAArel$;
		\item $\B^m_n, \BA^m_n$ and $\BAA^m_n$ by the previous item. However, we note that in these three complexes, we can only talk about simplices of type 
		$$\tau \in \{\text{standard, 2-additive, 3-additive, double-triple, double-double}\}.$$
	\end{itemize}
\end{remark}

\begin{remark}
	\label{rem:forgetting_symplectic_information}
	``Forgetting'' the symplectic form $\omega$ and considering $\mbZ^{2(m+n)}$ as an abelian group with trivial form, \cref{def:B-and-BA}, \cref{def:BAA} and \cref{def_linkhat} make sense for the isotropic summand $V = \mbZ^{2(m+n)}$ and yield complexes $\B^m_{2n+m}, \BA^m_{2n+m}$ and $\BAA^m_{2n+m}$ containing e.g.\ $\Irel, \IArel$ and $\IAArel$ as subcomplexes. Using the naming convention above, ``forgetting'' the symplectic form via these inclusions of subcomplexes, e.g.\ 
	$\IAArel \hookrightarrow \BAA^m_{2n+m},$
	changes the simplex types as shown in \cref{table_forgetting_symplectic_info}.
\end{remark}
\begin{table}
\begin{center}
\begin{tabular}{c|c}
Simplex type in $\IAArel$ & Simplex type in $\BAA^m_{2n+m}$ \\ \hline
standard & standard\\
2-additive&  2-additive\\
3-additive& 3-additive\\
double-triple& double-triple\\
double-double&  double-double\\
mixed& 2-additive\\
$\sigma$ & standard\\
$\sigma^2$& standard\\
skew-additive&standard\\
2-skew-additive& 2-additive\\
skew-$\sigma^2$ & standard\\
$\sigma$-additive & 2-additive
\end{tabular}
\end{center}
\caption{Types of simplices in $\IAArel$ and the types of their images in $\BAA^m_{2n+m}$ under the inclusion $\IAArel \hookrightarrow \BAA^m_{2n+m}$ that ``forgets'' the symplectic form.}
\label{table_forgetting_symplectic_info}
\end{table}

\begin{example}
	\label{expl:relative-simplex-types}
	This example illustrates \cref{def_relative_simplex_types} and \cref{rem:relative-simplex-types-other-complexes}. Let $n,m \geq 3$ and $\{ \vec e_1, \ldots, \vec e_{m+n}, \vec f_1, \ldots, \vec f_{m+n}\}$ be the standard symplectic basis of $\mbZ^{2(m+n)}$.
	\begin{enumerate} 
		\item $\{e_{m+3}, \langle \vec e_{1} + \vec e_{2} + \vec e_{m+3} \rangle\}$ is $3$-additive in $\IAArel$, and $\{e_{m+3}, \langle \vec e_{m+1} + \vec e_{m+2} + \vec e_{m+3} \rangle\}$ is $3$-additive in $\Linkhat_{\IAArel}(\{e_{m+1}, e_{m+2}\})$;
		\item $\ls \langle \vec e_{1} + \vec e_{m+2} \rangle, \langle \vec e_{1} + \vec e_{m+3} \rangle, e_{m+2}, e_{m+3}\rs$ is a double-triple simplex in $\IAArel$, and\\ $\ls \langle \vec e_{m+1} + \vec e_{m+2} \rangle, \langle \vec e_{m+1} + \vec e_{m+3} \rangle, e_{m+2}, e_{m+3} \rs$ is a double-triple simplex in $\Linkhat_{\IAArel}(e_{m+1})$;
		\item $\{e_{m+1}, f_{m+1}, e_{m+2}, f_{m+2}\}$ is a $\sigma^2$ simplex in $\IAArel$, and $\{e_{m+2}, f_{m+2}\}$ is a $\sigma^2$ simplex in $\Linkhat_{\IAArel}(\{e_{m+1}, f_{m+1}\})$;
		\item $\{\langle \vec e_{m+1} + \vec e_{m+2} \rangle, e_{m+1}, e_{m+2}, \langle \vec f_{m+1} - \vec f_{m+2} \rangle\}$ is a 2-skew-additive simplex in $\IAArel$, and $\{e_{m+1}, e_{m+2}, \langle \vec f_{m+1} - \vec f_{m+2} \rangle\}$ is a 2-skew-additive simplex in $\Linkhat_{\IAArel}(\langle \vec e_{m+1} + \vec e_{m+2} \rangle)$.
	\end{enumerate}
\end{example}

\begin{definition} 
\label{def_minimal_simplices_additive_core}
Let $\Delta'$ be a simplex in $\Link_{\IAArel}(\Delta)$.
\begin{enumerate}
  \item The simplex $\Delta'$ is called a \emph{minimal} simplex of type $\tau$ if $\Delta'$ is a simplex of type $\tau$ in the sense of \cref{def_relative_simplex_types} and if it does not contain a proper face that also is of this type.
  \item The \emph{augmentation core} of $\Delta'$ is the (possibly empty) unique minimal face of $\Delta'$ that is of the same type as $\Delta'$.
  \end{enumerate}
\end{definition}

\begin{remark}
	\label{rem:minimal-simplices-additive-core}
	As in \cref{rem:relative-simplex-types-other-complexes}, this defines the notion of minimal simplex and augmentation core for simplices in $\IAArel$, $\Linkhat_{\IAArel}(\Delta)$ and any other subcomplex $X \subset \IAArel$.
\end{remark}

\begin{example}
	Any simplex in \cref{expl:relative-simplex-types} is minimal and hence its own augmentation core. None of the simplices in \cref{expl:simplices-iaa} are minimal if ${\lastindexexamples} > 2$, their augmentation cores are $\{e_2, e_1, f_2, f_1\}$, $\{e_1, e_2, \langle \vec f_1- \vec f_2\rangle\}$, $\{\langle \vec e_1+ \vec e_2 \rangle, e_1, e_2, \langle \vec f_1- \vec f_2 \rangle\}$, $\{e_2, e_1, f_2, \langle \vec f_1 - \vec f_2 \rangle\}$ and $\{e_{{\lastindexexamples}}, f_{\lastindexexamples}, \langle \vec e_{\lastindexexamples}+ \vec f_{\lastindexexamples} \rangle\}$, respectively.
\end{example}

We note that any simplex in the non-relative complex $\IAA[m+n]$ has an augmentation core in the sense of \cref{def_minimal_simplices_additive_core}, and that it is the empty set if and only if the simplex is standard.

\begin{definition}
	\label{def:internal-external-s-related-simplices}
	Let $\Delta'$ be a simplex in $\Link_{\IAArel}(\Delta)$. Let $\Theta \in \IAA[m+n]$ denote the augmentation core (\cref{def_relative_simplex_types}) of the underlying simplex $\{e_1, \dots, e_m\} \cup \Delta \cup \Delta'$
	of $\Delta'$ in the non-relative complex $\IAA[m+n]$.
	\begin{enumerate}
		\item $\Delta'$ is called an \emph{external} simplex if $\Theta \cap (\{e_1, \dots, e_m\} \cup \Delta) \neq \emptyset$.
		\item $\Delta'$ is called a \emph{$\Delta$-related} simplex if $\Theta \cap \Delta \neq \emptyset$.
		\item $\Delta'$ is called an \emph{internal} simplex if $\Theta \subseteq \Delta'$, i.e.\ $\Theta \cap (\{e_1, \dots, e_m\} \cup \Delta) = \emptyset$.
	\end{enumerate}
	Note that in particular, every $\Delta$-related simplex is external.
\end{definition}

\begin{remark}
\label{rem:internal-external-s-related-simplices}
As in \cref{rem:relative-simplex-types-other-complexes}, there is an obvious way in which the notion of internal, external and (for subcomplexes of links) $\Delta$-related simplices can be used to refer to simplices of $\IAArel$, $\Linkhat_{\IAArel}(\Delta)$ and other subcomplexes $X \subset \IAArel$.
\end{remark}

\begin{example}
	This example illustrates \cref{def:internal-external-s-related-simplices} by classifying the simplices discussed in \cref{expl:relative-simplex-types}. 
	\begin{enumerate}
		\item The simplex $\{e_{m+3}, \langle \vec e_{1} + \vec e_{2} + \vec e_{m+3} \rangle\}$ in $\IAArel$ is external (but not $\Delta$-related). The simplex $\{e_{m+3}, \langle \vec e_{m+1} + \vec e_{m+2} + \vec e_{m+3} \rangle\}$ in the complex $\Linkhat_{\IAArel}(\{e_{m+1}, e_{m+2}\})$ is $\{e_{m+1}, e_{m+2}\}$-related (and external).
		\item The simplex $\ls \langle \vec e_{1} + \vec e_{m+2} \rangle, \langle \vec e_{1} + \vec e_{m+3} \rangle, e_{m+2}, e_{m+3}\rs$ in $\IAArel$ is an external (but not $\Delta$-related) simplex. $\ls \langle \vec e_{m+1} + \vec e_{m+2} \rangle, \langle \vec e_{m+1} + \vec e_{m+3} \rangle, e_{m+2}, e_{m+3} \rs$ is an $e_{m-1}$-related (and external) simplex in the complex $\Linkhat_{\IAArel}(e_{m+1})$.
		\item The simplex $\{e_{m+1}, f_{m+1}, e_{m+2}, f_{m+2}\}$ in $\IAArel$ is internal. The simplex $\{e_{m+2}, f_{m+2}\}$ in $\Linkhat_{\IAArel}(\{e_{m+1}, f_{m+1}\})$ is $\{e_{m+1}, f_{m+1}\}$-related (and external).
		\item The simplex $\{\langle \vec e_{m+1} + \vec e_{m+2} \rangle, e_{m+1}, e_{m+2}, \langle \vec f_{m+1} - \vec f_{m+2} \rangle\}$ in $\IAArel$ is internal. The simplex $\{e_{m+1}, e_{m+2}, \langle \vec f_{m+1} - \vec f_{m+2} \rangle\}$ in $\Linkhat_{\IAArel}(\langle \vec e_{m+1} + \vec e_{m+2} \rangle)$ is $\langle \vec e_{m+1} + \vec e_{m+2} \rangle$-related (and external).
	\end{enumerate}
\end{example}

\section{Structure of links}
\label{sec_structure_of_links}

In this section we describe and study the links of various simplices in the complexes $\Irel$, $\IArel$ and $\IAArel$, which we defined in the previous section.
Throughout, $n$ and $m$ are natural numbers such that $n\geq 1$ and $m\geq 0$.

\subsection{Links in \texorpdfstring{$\Irel$ and $\Idelrel$}{I and I\unichar{"005E}\unichar{"03B4}}}

We start by describing the links of simplices in $\Irel$.

\begin{lemma}
\label{link_Irel_standard}
Let $\Delta$ be a standard simplex in $\Irel$. Then
$\Link_{\Irel}(\Delta) \cong \Irel[n-(\dim(\Delta)+1)][m+(\dim(\Delta) + 1)].$
\end{lemma}
\begin{proof}
If $\Delta = \ls v_0, \ldots, v_k \rs$ is a standard simplex, then $\ls\vec e_1,\ldots, \vec e_m,  \vec v_0, \ldots, \vec v_k \rs$ spans an isotropic summand of rank $m+k+1$ of $\mbZ^{2(m+n)}$.
Hence by \cref{lem_extend_to_symplectic_basis}, it can be extended to a  symplectic basis $\ls \vec e_1,\ldots, \vec e_m, \vec v_0, \ldots, \vec v_k \rs \cup B$
of $\mbZ^{2(m+n)}$.
Let $\phi\colon  \mbZ^{2(m+n)} \to \mbZ^{2(m+n)}$ be a symplectic isomorphism such that $\phi(\vec e_i) = \vec e_i$ for all $1\leq i \leq m$ and $\phi(\vec v_j) = \vec e_{m+j}$ for all $0\leq i \leq k$.
It is easy to check that $\phi$ induces the isomorphism $\Link_{\Irel}(\Delta) \cong \Irel[n-(k+1)][m+(k + 1)]$.
\end{proof}

The following lemma is an immediate consequence of the definition of 2-additive simplices (see \cref{def:IA-simplices}).
\begin{lemma}
\label{link_Idelrel_2add}
Let $\Delta = \ls v_0, \ldots, v_k \rs$ be a 2-additive simplex in $\Idelrel$ such that $v_0 = \langle \vec v_1 + \vec v_2 \rangle$ or $v_0 = \langle \vec e_i + \vec v_1 \rangle$ for some $1 \leq i \leq m$. Then $\Delta' = \ls v_1, \ldots, v_k \rs$ is a standard simplex in $\Irel$ and 
$$\Link_{\Idelrel}(\Delta) = \Link_{\Irel}(\Delta').$$
\end{lemma}

\subsection{Links in \texorpdfstring{$\IAAstrel$ and $\IAArel$}{IAA* and IAA}}
We describe the link of some minimal simplices in $\IAAstrel$ and $\IAArel$.
 
\begin{lemma}
\label{lem_link_of_vertex}
Let $v$ be a vertex of $\IAAstrel$.
Then there is a symplectic isomorphism $\phi\colon  \mbZ^{2(m+n)}\to \mbZ^{2(m+n)}$ fixing $e_1, \dots, e_m$ and mapping $v$ to $e_{m+1}$ that induces the following commuting diagram.
\begin{equation}
\label{eq_isos_link_vertex}
\begin{aligned}
\xymatrix{
	\Linkhat_{\IAAstrel}( v) \ar[r]^{\cong} \ar@{^{(}->}[d] &  \ar@{^{(}->}[d] \IAAstrel[n-1][m+1] \\
	\Linkhat_{\IAArel}(v)  \ar[r]^{\cong} & \IAArel[n-1][m+1]
}
\end{aligned}
\end{equation}
Furthermore, these isomorphisms preserve the type of every simplex, i.e.~they send standard simplices to standard simplices, 2-additive simplices to 2-additive simplices etc.
\end{lemma}
\begin{proof}
Using \cref{lem_extend_to_symplectic_basis}, we can extend $\ls\vec e_1,\ldots, \vec e_m,  \vec v \rs$ to a symplectic basis $\ls\vec e_1,\ldots, \vec e_m,  \vec v \rs \cup B$ of $\mbZ^{2(m+n)}$. 
Let $\phi\colon   \mbZ^{2(m+n)} \to \mbZ^{2(m+n)}$ be a symplectic isomorphism such that $\phi(\vec e_i) = \vec e_i$ for all $1\leq i \leq m$ and $\phi(\vec v) = \vec e_{m+1}$ . Using \cref{def_linkhat}, we find that
\begin{equation*}
	\Linkhat_{\IAAstrel}(v) = \Linkhat_{\Linkhat_{\IAAst[m+n]}(\ls e_1, \ldots, e_m \rs)}(v) = \Linkhat_{\IAAst[m+n]}(\ls e_1, \ldots, e_m,v \rs)
\end{equation*}
and
\begin{equation*}
	\IAAstrel[n-1][m+1]  = \Linkhat_{\IAAst[m+n]}(\ls e_1, \ldots, e_m, e_{m+1} \rs);
\end{equation*}
Similar identifications holds for $\IAArel$. It is easy to see that $\phi$ induces the desired isomorphisms of simplicial complexes.
\end{proof}

We next want to prove a similar, but slightly stronger statement for links of minimal $\sigma$ simplices. This requires the following definition:

\begin{definition}
\label{def_rank_and_ranked_complexes}
Let $m,n \geq 0$ and consider the symplectic module $(\mbZ^{2(m+n)}, \omega)$ with the symplectic basis $\{\vec e_1, \vec f_1, \dots, \vec e_{m+n}, \vec f_{m+n}\}$. 
We denote by
\begin{align*}
	\rkfn(-)\colon   \mbZ^{2(m+n)}&\to \mbZ \\
	\vec v &\mapsto \omega(\vec e_{m+n},\vec v)
\end{align*}
the projection onto the $\vec f_{m+n}$-coordinate.
We denote by $\bar v \in \{\vec v, - \vec v\}$ the choice\footnote{This choice is unique if $|\rkfn(\vec v)|>0$.} of a primitive vector in $v$ whose $\vec f_{m+n}$-coordinate is non-negative, i.e.\ that satisfies $$\rkfn(\bar v) \geq 0.$$

If $X$ is a subcomplex of $\IAArel$, this induces a function
\begin{align*}
	\rkfn(-)\colon   \Vr(X)&\to \mbN \\
	v &\mapsto |\rkfn(\vec v)| = \rk(\bar{v}) = |\omega(e_{m+n},v)|
\end{align*}
that sends every vertex of $X$ to the absolute value of the $\vec f_{m+n}$-coordinate of some (hence any) primitive vector $\vec v \in v$.
\begin{enumerate}
	\item We say that a vertex $v \in \Vr(X)$ has \emph{rank} $\rkfn(v)$.
	\item For $R \in \mbN $, we denote by $X^{< R}$ the full subcomplex of $X$ on all vertices $v \in \Vr(X)$ of rank less than $R$, $\rkfn(v) < R$.
	\item For $R\in \mbN \cup \ls \infty \rs$, we denote by $X^{\leq R}$ the full subcomplex of $X$ on all vertices $v \in \Vr(X)$ of rank less than or equal to $R$, $\rkfn(v) \leq R$. In particular, $X^{\leq \infty} \coloneqq X$.
\end{enumerate}
\end{definition}

\begin{lemma}
\label{lem_isom_link_for_making_regular}
Let $\Delta$ be a $\sigma$ edge in $\IAAstrel$ and $R\in \mbN \cup \ls \infty \rs$. Then, for some $R'\in \mbN \cup\ls \infty \rs$, there is commutative diagram of the following form.
\begin{equation}
\label{eq_isos_link_sigma}
\begin{aligned}
\xymatrix{
	(\Link_{\IAAstrel}(\Delta))^{\leq R} \ar@{^{(}->}[d] \ar[r]^<<<<<{\cong} & (\Idelrel[n-1])^{\leq R'} \ar@{^{(}->}[d]\\
	(\Linkhat_{\IAArel}(\Delta))^{\leq R}  \ar[r]^<<<<<{\cong} & (\Isigdelrel[n-1])^{\leq R'} 
}
\end{aligned}
\end{equation}
The isomorphisms in this diagram send $\sigma$ simplices to standard simplices, mixed simplices to 2-additive simplices and $\sigma^2$ simplices to $\sigma$ simplices.
\end{lemma}
\begin{proof}
Let $\Delta = \{v, w\}$ and extend $\{\vec v, \vec w, \vec e_1,\ldots, \vec e_m \}$ to a symplectic basis $\{\vec v, \vec w, \vec e_1,\ldots, \vec e_m \} \cup B$ of $\mbZ^{2(m+n)}$ (as before, this is possible by \cref{lem_extend_to_symplectic_basis}). 
Let $\phi\colon   \mbZ^{2(m+n)} \to \mbZ^{2(m+n)}$ be a symplectic isomorphism fixing $\vec e_1, \dots, \vec e_m$, mapping $\Delta = \{v, w\}$ to $\ls e_{m+n}, f_{m+n} \rs$ and restricting to an isomorphism 
\begin{equation*}
	\bar{\phi}\colon  \langle \vec e_1,\ldots, \vec e_m \cup B \rangle \to \mbZ^{2(m+n-1)} = \langle \vec e_1,\ldots, \vec e_m, \ldots, \vec e_{m+n-1}, \vec f_1, \ldots, \vec f_{m+n-1} \rangle.
\end{equation*}

We claim that $\bar{\phi}$ induces compatible isomorphisms $\Link_{\IAAstrel}(\Delta) \cong \Idelrel[n-1]$ and $\Linkhat_{\IAArel}(\Delta)  \cong \Isigdelrel[n-1]$: Note that $\Link_{\IAAstrel}(\Delta)$ and $\Linkhat_{\IAArel}(\Delta)$ have the same vertex set. It consists of all vertices of $\IAAstrel$ that are contained in the symplectic complement $\langle \Delta \rangle^\perp = \langle \{\vec e_1,\ldots, \vec e_m\} \cup B \rangle$. Using \cref{def_relative_simplex_types} and \cref{rem:relative-simplex-types-other-complexes}, we see that all simplices in  $\Link_{\IAAstrel}(\Delta)$ are either $\sigma$ or mixed simplices and the only additional simplices in $\Linkhat_{\IAArel}(\Delta)$ are of type $\sigma^2$. From this, the isomorphisms $\Link_{\IAAstrel}(\Delta) \cong \Idelrel[n-1]$ and $\Linkhat_{\IAArel}(\Delta)  \cong \Isigdelrel[n-1]$ follow.

We now explain how to determine $R'$, and discuss the two restrictions $(\Link_{\IAAstrel}(\Delta))^{\leq R} \cong (\Idelrel[n-1])^{\leq R'}$ and $(\Linkhat_{\IAArel}(\Delta))^{\leq R}  \cong (\Isigdelrel[n-1])^{\leq R'}$: Let $R \in \mbN \cup \{\infty\}$. Let $z$ be a vertex of one of the complexes on the left-hand side of \cref{eq_isos_link_sigma}. Then it holds that
$$\rkfn(z) = |\omega(\vec e_{m+n},\vec z)| \leq R.$$
The isomorphisms defined in the previous paragraph map the complexes on the left-hand side of \cref{eq_isos_link_sigma} to the full subcomplexes of $\Idelrel[n-1]$ and $\Isigdelrel[n-1]$, respectively, spanned by all vertices $z$ satisfying 
$$|\omega(\phi(\vec e_{m+n}), \vec z)| \leq R.$$
We need to find a suitable $R'$ and show that these subcomplexes are isomorphic to $(\Idelrel[n-1])^{\leq R'}$ and $(\Isigdelrel[n-1])^{\leq R'}$, respectively.

For this, let $\overline{\phi(\vec e_{m+n})}$ be the orthogonal projection of $\phi(\vec e_{m+n})$ to $\mbZ^{2(m+n-1)}$ and let $\omega_{\mbZ^{2(m+n-1)}}$ denote the restriction of the symplectic form to $\mbZ^{2(m+n-1)}$. For any $z \subseteq \mbZ^{2(m+n-1)}$, it holds that 
\begin{equation*}
	|\omega(\phi(\vec e_{m+n}), \vec z)|\leq R  \Longleftrightarrow |\omega_{\mbZ^{2(m+n-1)}}(\overline{\phi(\vec e_{m+n})}, \vec z)|\leq R.
\end{equation*}
However, $\overline{\phi(\vec e_{m+n})}$ need not be primitive in $\mbZ^{2(m+n-1)}$; e.g.\ if $\phi(\vec e_{m+n}) = \vec e_{m+n}$, then $\overline{\phi(\vec e_{m+n})} = \vec 0$.

We consider two cases. Firstly, assume that $\overline{\phi(\vec e_{m+n})}\in \langle \vec e_1, \ldots, \vec e_m \rangle$. 
Then we have
\begin{equation*}
	\omega(\overline{\phi(e_{m+n})}, z)=0 \text{ for all vertices } z \text{ of } \Idelrel[n-1]\text{ and }\Isigdelrel[n-1], 
\end{equation*}
since $\phi(\vec e_i, \vec z) = 0$ for all $1\leq i \leq m$. Hence, the claim is true for $R' \coloneqq \infty$.
If this is not the case, we can apply \cref{lem_primitive_for_form} to $\vec e = \overline{\phi(\vec e_{m+n})}$ because
\begin{equation*}
	\omega_{\mbZ^{2(m+n-1)}}(\overline{\phi(\vec e_{m+n})}, \vec e_i) =  \omega(\phi(\vec e_{m+n}),\vec e_i) = \omega(\phi(\vec e_{m+n}), \phi(\vec e_i))= 0
\end{equation*}
for $1\leq i \leq m$. This yields $\vec e'\in \mbZ^{2(m+n-1)}$ that is isotropic to $\vec e_1, \ldots, \vec e_m $ and such that $\vec e_1, \ldots, \vec e_m, \vec e'$ can be extended to a symplectic basis $B'$ of $\mbZ^{2(m+n-1)}$. Furthermore, there is an $a\in \mbN$ such that for all vertices $z \in \Idelrel[n-1], \, \Isigdelrel[n-1]$, we have 
\begin{equation*}
	\omega_{\mbZ^{2(m+n-1)}}(\overline{\phi(\vec e_{m+n})}, \vec z) = a \omega_{\mbZ^{2(m+n-1)}}(\vec e', \vec z).
\end{equation*}
So we get
\begin{equation*}
	|\omega_{\mbZ^{2(m+n-1)}}(\overline{\phi(\vec e_{m+n})}, \vec z)|\leq R  \Longleftrightarrow |\omega_{\mbZ^{2(m+n-1)}}(\vec e', \vec z)| \leq \lfloor R/a\rfloor \eqqcolon R' .
\end{equation*}
Let $\psi\colon   \mbZ^{2(m+n-1)} \to \mbZ^{2(m+n-1)}$ be an isomorphism that sends $B'$ to the standard basis and such that $\psi(\vec e_i) = \vec e_i$ and $\psi(\vec e) = \vec e_{m+n-1}$. Then $\psi \circ \bar{\phi}$ induces the desired isomorphisms.
\end{proof}

\begin{lemma}
\label{link_IAArel_sigma2}
\label{link_IAArel_skew_sigma2}
\label{link_IAArel_sigma_additive}
Let $\Delta$ be a minimal $\sigma^2$, skew-$\sigma^2$ or $\sigma$-additive simplex in $\IAArel$. Then 
$$\Link_{\IAArel}(\Delta) \cong \Irel[n-(\dim(\Delta)-1)][m].$$
\end{lemma}
\begin{proof}
First observe that these types of simplices can only occur as internal simplices (see \cref{def:internal-external-s-related-simplices} and \cref{def_linkhat}).
Hence, minimal $\sigma^2$ and skew-$\sigma^2$ simplices are 3-dimensional of the form $\ls v_0, v_1, v_2, v_3 \rs$ and minimal $\sigma$-additive are two dimensional of the form $\ls v_0, v_1, v_2\rs$. In either case, such a minimal simplex $\Delta$ determines a symplectic summand $\langle \Delta \rangle$ of $\mbZ^{2(m+n)}$ of genus $\dim(\Delta)-1$. 
The lines $e_1, \ldots, e_m$ are contained in the symplectic complement $\langle \Delta \rangle^\perp \subseteq \mbZ^{2(m+n)}$ of this summand and we can, similarly to the proof of \cref{lem_isom_link_for_making_regular}, find a symplectic isomorphism $\langle \Delta \rangle^\perp \to \mbZ^{2(m+n-(\dim(\Delta)-1))}$ that restricts to the identity on $e_1,\ldots, e_m$. This isomorphism of symplectic spaces induces the desired isomorphism of simplicial complexes. 
\end{proof}

\begin{lemma}
\label{link_IAAstrel_doubletriple_doubledouble}
Let $\Delta$ be a minimal double-triple or double-double simplex in $\IAArel$. Then 
\begin{equation*}
\Link_{\IAArel}(\Delta) = \Link_{\IAAstrel}(\Delta) \cong \Irel[n-(\dim(\Delta)-1)][m+(\dim(\Delta)- 1)].
\end{equation*}
\end{lemma}
\begin{proof}
We start by noting that double-triple or double-double simplices $\Delta$ can also occur as external simplices (see \cref{def:internal-external-s-related-simplices}). We can order the vertices of $\Delta$ such that $\Delta = \ls v_0, \ldots, v_k \rs$ where $k = \dim(\Delta) \geq 2$ and $\ls v_2, \ldots, v_{k}\rs$ is a standard simplex in $\IAArel$. Similar to the proof of \cref{link_Irel_standard}, we can extend $\ls \vec e_1, \ldots, \vec e_m, \vec v_2, \ldots, \vec v_{k}\rs$ to a symplectic basis of $\mbZ^{2(m+n)}$ and define a symplectic isomorphism $\mbZ^{2(m+n)} \to \mbZ^{2(m+n)}$ that maps $\ls \vec e_1, \ldots, \vec e_m, \vec v_2, \ldots, \vec v_{k}\rs$ to $\ls \vec e_1, \ldots, \vec e_{m+(k-1)}\rs$. This map then induces the desired isomorphism of simplicial complexes $\Link_{\IAArel}(\Delta) \cong \Irel[n-(k-1)][m+(k- 1)]$.
\end{proof}

Lastly, we also describe the links of minimal external 2-additive simplices in $\IAAstrel$. We compare them with links in $\IArel$.
\begin{lemma}
\label{lem_compare_linkhatIA_and_linhatIAA}
Let $v \in \IAAstrel$ be a vertex of rank $R = \rkfn(v)>0$ and $1\leq i \leq m$.
Then $\Linkhat_{\IArel}^{<R}(v)$ is a subcomplex of $\Linkhat_{\IAAstrel}^{<R}(\{v,\langle \vec v \pm \vec e_i\rangle\})$
and every simplex of $\Linkhat_{\IAAstrel}^{<R}(\{v,\langle \vec v \pm \vec e_i\rangle\})$ that is not contained in $\Linkhat_{\IArel}^{<R}(v)$ is of type double-triple.
\end{lemma}
The above lemma can be proved by going through all types of simplices in $\IAAstrel$ and $\IArel$. This is is not complicated and very similar to \cite[Lemma 5.8]{Brueck2022}, so we omit the proof here.

\subsection{\texorpdfstring{$\Linkhat$s}{Linkhats} in \texorpdfstring{$\IArel$ and $\IAAstrel$}{IA and IAA*}}

The following observation about $\IArel$ can be shown like \cref{lem_link_of_vertex}.

\begin{lemma}
	\label{link_IArel_vertex}
	Let $v$ be a vertex of $\IArel$. Then $\Linkhat_{\IArel}(v)\cong \IArel[n-1][m+1]$.
\end{lemma}

Finally, we need the following lemma about $\IAAstrel$.

\begin{lemma}
	\label{equality_link_linkhat}
	Let $v \in \IAAstrel$ be a vertex of rank $R = \rkfn(v)>0$.
	\begin{enumerate}
		\item \label{it_eq_link_linkhat_IAA_v}If $w\in\Link^{<R}_{\IAAstrel}(v)$, then $\ls v,w \rs$ is a $\sigma$ simplex or $w\in \Linkhat^{<R}_{\IAAstrel}(v) $. 
		\item \label{it_eq_link_linkhat_IAA_2add}$\Link_{\IAAstrel}^{<R}(\ls v,\langle \vec v \pm \vec e_i\rangle\rs) = \Linkhat_{\IAAstrel}^{<R}(\ls v,\langle \vec v \pm \vec e_i\rangle\rs)$ for all $1\leq i \leq m$. 
	\end{enumerate}
\end{lemma}

\begin{proof}
	We first consider \cref{it_eq_link_linkhat_IAA_2add}.
	Let $w\in \Link_{\IAAstrel}^{<R}(\ls v,\langle \vec v \pm \vec e_i\rangle\rs)$. It suffices to verify that 
	\begin{equation*}
		w\in \Linkhat_{\IAAstrel}^{<R}(\ls v,\langle \vec v \pm \vec e_i\rangle\rs)
	\end{equation*}
	since $\Linkhat_{\IAAstrel}^{<R}(\ls v,\langle \vec v \pm \vec e_i\rangle\rs)$ is a full subcomplex of $\Link_{\IAAstrel}^{<R}(\ls v,\langle \vec v \pm \vec e_i\rangle\rs)$.
	Following \cref{def_linkhat}, there are two things to check: Firstly, that we have $\vec w\not\in \langle \vec e_1, \ldots, \vec e_m, \vec v \rangle $ and secondly, that $\omega(w,v) = \omega(w,\langle \vec v \pm \vec e_i\rangle) = 0$.
	The first condition follows exactly as in \cite[proof of Lemma 5.7]{Brueck2022}.
	The second condition follows because $v$ and $\langle \vec v \pm \vec e_i\rangle$ are contained in the augmentation core of a 2-additive face of $\ls v,\langle \vec v \pm \vec e_i\rangle\rs \cup \ls e_1, \ldots, e_m \rs$. This implies that all lines $w$ that form a simplex with them in $\IAAst[m+n]$ must be isotropic to them.
	
	The proof of \cref{it_eq_link_linkhat_IAA_v} is easier: Assume that $\ls v,w \rs$ is not a $\sigma$ simplex. Then $\omega(v,w) = 0$. So the only thing to check is that $\vec w\not\in \langle \vec e_1, \ldots, \vec e_m, \vec v \rangle $ and this follows again as in \cite[proof of Lemma 5.7]{Brueck2022}.
\end{proof}

\section{Regular maps} \label{sec:regular-maps}
In order to prove \autoref{thm_connectivity_IAA}, i.e.~to show that $\pi_k(\IAA) = 0$ for all $k\leq n$, we need to study maps from spheres into $\IAA$ and its relative versions $\IAArel$ as well as homotopies between such maps.
In order to control the behaviour of these, we will only work with certain simplicial maps $S^k\to \IAArel$ from triangulated spheres into these complexes and we will only allow certain ``regular'' homotopies between such maps. Related ideas have been used by Putman to prove that $\Isigdelrel$ is $(n-1)$-connected (see \cite[Proposition 6.11]{put2009}, and \cite[Section 3]{bruecksroka2023}, which fixes some small gaps in Putman's argument).

The aim of this section is to introduce the necessary definitions and properties of the types of maps that we will be working with.

\subsection{Combinatorial manifolds}
We start by introducing definitions and elementary properties of combinatorial manifolds. We mostly stick to the notation used in  \cite[Section 6.1]{put2009}. For general references about this topic, see \cite{Rourke1982} and \cite{Hudson1969}.

In what follows, we assume that for $k<0$, the empty set is both a $k$-ball and a $k$-sphere.

\begin{definition}
\label{def_combinatorial_manifolds}
Let $k$ be a natural number. We define the notion of a \emph{combinatorial $k$-manifold} inductively as follows:
\begin{itemize}
\item Every $0$-dimensional simplicial complex is a combinatorial $0$-manifold.
\item A $k$-dimensional simplicial complex $M$ is called a \emph{combinatorial $k$-manifold} if for every simplex $\Delta$ of $M$, the link $\Link_M(\Delta)$ is a combinatorial $(k-\dim(\Delta)-1)$-manifold whose geometric realisation is either homeomorphic to a sphere or to a ball.
\end{itemize}

If $M$ is a combinatorial $k$-manifold, we denote by $\partial M$ the subcomplex consisting of all simplices $\Delta$ such that $\dim(\Delta)<k$ and $|\Link_M(\Delta)|$ is a ball.

We say that $M$ is a \emph{combinatorial $k$-sphere} or a \emph{combinatorial $k$-ball} if its geometric realisation is homeomorphic to a $k$-sphere or a $k$-ball.
\end{definition}

The following lemma collects some elementary properties of combinatorial manifolds. For the proofs, see \cite{Rourke1982}, in particular \cite[Exercise 2.21]{Rourke1982} and \cite[Proof of Lemma 1.17]{Hudson1969}.
\begin{lemma}
Let $M$ be a combinatorial $k$-manifold. Then 
\begin{enumerate}
\item the geometric realisation $|M|$ is a topological manifold;
\item the boundary $\partial M$ is a combinatorial $(k-1)$-manifold;
\item the geometric realisation of the boundary is the boundary of the geometric realisation, $|\partial M| = \partial |M|$.
\end{enumerate}
\end{lemma}
In particular, these properties yield the following descriptions of the boundaries of combinatorial balls and spheres.
\begin{corollary}
\label{cor_boundaries_of_combinatorial_balls_spheres}
If $B$ is a combinatorial $k$-ball, then $\partial B$ is a combinatorial $(k-1)$-sphere.

If $S$ is a combinatorial $k$-sphere, then $\partial S = \emptyset$, so for all simplices $\Delta$ of $S$, the link $\Link_S(\Delta)$ is a combinatorial $(k-\dim(\Delta)-1)$-sphere.
\end{corollary}

The next lemma is an easy consequence of the definition of $\partial M$.

\begin{lemma}
\label{lem_boundary_join}
Let $M_1$ and $M_2$ be combinatorial $k_1$- and $k_2$-balls or -spheres.
\begin{enumerate}
\item \label{it_join_ball}If at least one of $M_1$, $M_2$ is a ball, then the join $M_1 \ast M_2$ is a combinatorial $(k_1+k_2+1)$-ball.
\item \label{it_join_sphere}If both $M_1$ and $M_2$ are spheres, then the join $M_1 \ast M_2$ is a combinatorial $(k_1+k_2+1)$-sphere.
\end{enumerate}
Furthermore, for $B, B'$ combinatorial balls and $S, S'$ combinatorial spheres, we have 
\begin{align*}
	&\partial (B \ast B') = (\partial B \ast B') \cup (B \ast \partial B'),\\
	&\partial (S \ast B) = S \ast \partial B,\\
	&\partial (S \ast S') = \emptyset.
\end{align*}
\end{lemma}
\begin{proof}
\cref{it_join_ball} and \cref{it_join_sphere} follow from \cite[Lemma 1.13]{Hudson1969} (that our definition of combinatorial balls and spheres agrees with the one in \cite{Hudson1969} follows from \cite[Corollary 1.16 and Remark on p.26]{Hudson1969}).

If $\Delta = \Delta_1 \ast \Delta_2$ is a simplex of $M_1\ast M_2$, where $\Delta_1\in M_1$ and $\Delta_2\in M_2$, we have
\begin{equation*}
	\Link_{M_1\ast M_2}(\Delta) = \Link_{M_1}(\Delta_1)\ast \Link_{M_2}(\Delta_2).
\end{equation*} 
Hence, the statement about the boundaries follows from \cref{it_join_ball} and \cref{it_join_sphere}. (Note that $\partial D^0 = \emptyset = \partial S^0$ and $\emptyset \ast X = X$ for all simplicial complexes $X$.)
\end{proof}

Note that the above statement does not generalise to arbitrary combinatorial manifolds. (E.g.\ the join of three points with a singleton is not a manifold.)

We will frequently use the following consequence of \cref{it_join_ball} of \cref{lem_boundary_join}:
\begin{corollary}
\label{cor_boundary_star}
Let $S$ be a combinatorial $k$-sphere and $\Delta$ a simplex of $S$.
Then $\Star_{S}(\Delta) = \Delta\ast \Link_S(\Delta)$ is a combinatorial $k$-ball and
\begin{equation*}
	\partial \Star_{S}(\Delta) = \partial \Delta\ast \Link_S(\Delta).
\end{equation*}
\end{corollary}
\begin{proof}
This immediately follows from \cref{lem_boundary_join} because $\Link_S(\Delta)$ is a $(k-1)$-sphere, therefore $\partial\Link_S(\Delta) = \emptyset$ as observed above.
\end{proof}

The next lemma allows us to restrict ourselves to combinatorial spheres and balls when investigating the homotopy groups of simplicial complexes.
\begin{lemma}
\label{lem_simplicial_approximation_combinatorial}
Let $X$ be a simplicial complex and $k\geq 0$.
\begin{enumerate}
\item \label{it_simplicial_approximation_combinatorial_sphere}Every element of $\pi_k(X)$ can be represented by a simplicial map $\phi\colon   S\to X$, where $S$ is a combinatorial $k$-sphere.
\item \label{it_simplicial_approximation_combinatorial_ball}If $S$ is a combinatorial $k$-sphere and $\phi\colon  S\to X$ is a simplicial map such that $|\phi|$ is nullhomotopic, then there is a combinatorial $(k+1)$-ball $B$ with $\partial B = S$ and a simplicial map $\psi\colon   B\to X$ such that $\psi|_{S} = \phi$.
\end{enumerate}
\end{lemma}
This can be shown using Zeeman's simplicial approximation (\cref{zeeman-relative-simplicial-approximation}), see \cite[Lemma 6.4]{put2009}. The key observation for it is that subdivisions of combinatorial manifolds are again combinatorial manifolds.

Recall that if $M$ is a simplicial complex and $C\subseteq M$ a subcomplex, we write $M\setminus C$ for the subset (not sub\emph{complex}) of $M$ consisting of all simplices  that are not contained in $C$.

\begin{lemma}
\label{lem_combinatorial_ball_replace_same_boundary}
Let $M$ be a combinatorial $k$-manifold and $C\subseteq M$ a subcomplex that is a combinatorial $k$-ball. Let $C'$ be a combinatorial $k$-ball such that 
\begin{equation*}
	C'\cap M = \partial C = \partial C'.
\end{equation*}
Then $M' \coloneqq (M\setminus C) \cup C'$ is a combinatorial $k$-manifold with $|M|\cong |M'|$. (See \cref{figure_replace_ball_same_boundary}.)
\end{lemma}
\begin{figure}
\begin{center}
\includegraphics{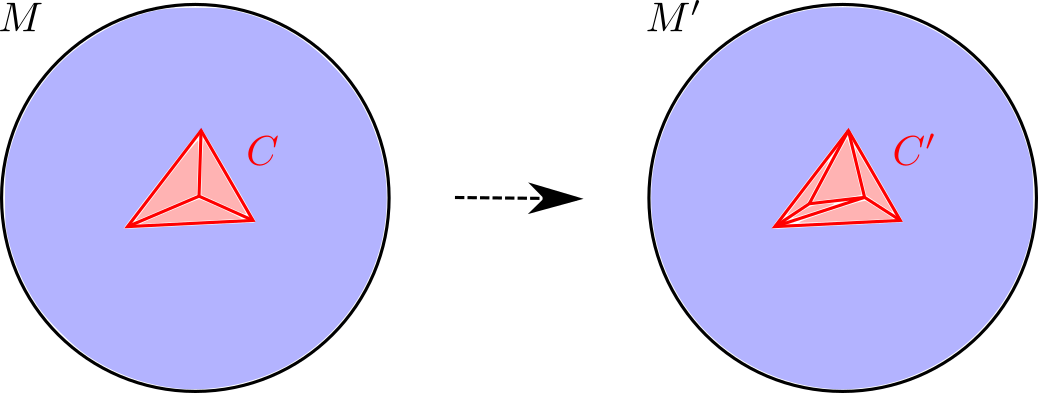}
\end{center}
\caption{Replacing $C$ by $C'$ to obtain $M'$ from $M$ in the setting of \cref{lem_combinatorial_ball_replace_same_boundary}, $k=2$.}
\label{figure_replace_ball_same_boundary}
\end{figure}
\begin{proof}
First observe that topologically, $|M'|$ is obtained from $|M|$ by removing the maximal-dimensional ball $|C|$ and attaching a ball $|C'|$ of the same dimension along the same boundary $\partial |C| = \partial |C'|$. Clearly, this does not change the homeomorphism type of $|M|$, i.e.~we have $|M|\cong |M'|$.

Hence, it suffices to show that $M'$ is indeed a \emph{combinatorial} manifold.
We proof this claim by induction over $k$.
The induction beginning are the cases $k\leq 0$, which are trivial.

Now assume $k>0$ and that the claims holds for all $l<k$.
Let $\Delta$ be a simplex of $M'$. We need to show that $\Link_{M'}(\Delta)$ is a combinatorial $(k-\dim(\Delta)-1)$-sphere or -ball. 
If $\Delta$ is contained in $M'\setminus C'$, then 
\begin{equation}
\label{eq_observation_for_replacement_stars}
	\Link_{M'}(\Delta) = \Link_M(\Delta),
\end{equation}
so the claim follows because $M$ is a combinatorial manifold.
Similarly, if $\Delta$ is contained in $C'\setminus(C'\cap M)$, then $\Link_{M'}(\Delta) = \Link_{C'}(\Delta)$, so the statement follows because $C'$ is a combinatorial ball.
So we can assume that $\Delta\subseteq C' \cap M$, which by assumption means that
$\Delta\subseteq \partial C = \partial C'$.
As $C$ and $C'$ are combinatorial $k$-balls, both $\Link_{C}(\Delta)$ and $\Link_{C'}(\Delta)$ are combinatorial $(k-\dim(\Delta)-1)$-balls.
Furthermore, $\Link_{M}(\Delta)$ is a combinatorial $(k-\dim(\Delta)-1)$-ball or sphere and it is not hard to see that $\Link_{M'}(\Delta)$ is obtained from $\Link_{M}(\Delta)$ by replacing $\Link_{C}(\Delta)$ with $\Link_{C'}(\Delta)$. 
So by the induction hypothesis, $\Link_{M'}(\Delta)$ is again a combinatorial $(k-\dim(\Delta)-1)$-ball or sphere, respectively.
It follows that $M'$ is a combinatorial $k$-manifold.
\end{proof}

\begin{remark}
\label{rem_subdivision_combinatorial_manifolds}
In particular, \cref{lem_combinatorial_ball_replace_same_boundary} implies that subdivisions of combinatorial manifolds are again combinatorial manifolds.
\end{remark}

We also need the following variant of \cref{lem_combinatorial_ball_replace_same_boundary}, which is depicted in \cref{figure_replace_ball_bigger_boundary}. We state it separately to simplify references later on, but omit the proof as it is entirely parallel to the one of \cref{lem_combinatorial_ball_replace_same_boundary}.
\begin{lemma}
\label{lem_combinatorial_ball_replace_bigger_boundary}
Let $M$ be a combinatorial $k$-ball and $C\subseteq M$ a subcomplex that is a combinatorial $k$-ball. Let $C'$ be a combinatorial $k$-ball such that 
\begin{equation*}
	C'\cap M \subseteq \partial C \cap \partial C' \text{ and } \partial C \subseteq \partial M \cup \partial C'.
\end{equation*}
Let $M' \coloneqq (M\setminus C) \cup C'$. Then if $|M'|$ is a $k$-ball, in fact $M'$ is a \emph{combinatorial} $k$-ball.
\end{lemma}

\begin{figure}
\begin{center}
\includegraphics{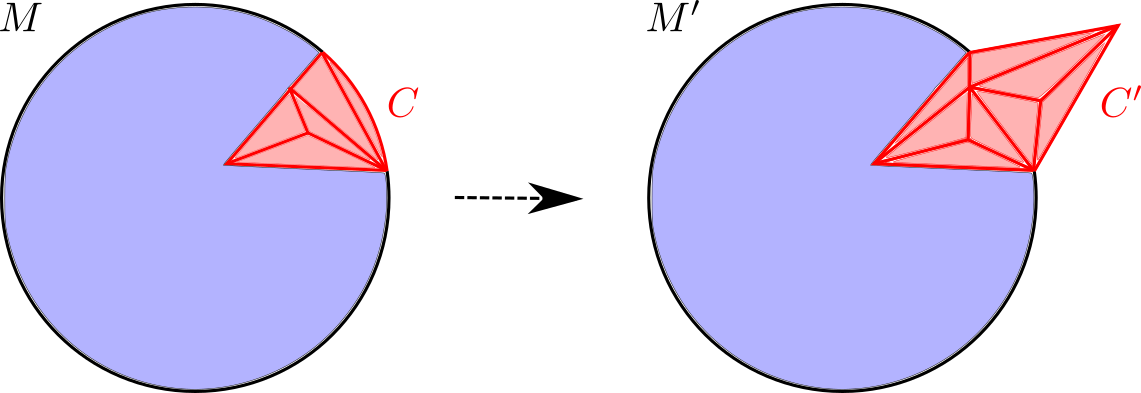}
\end{center}
\caption{Replacing $C$ by $C'$ to obtain $M'$ from $M$ in the setting of \cref{lem_combinatorial_ball_replace_bigger_boundary}, $k=2$.}
\label{figure_replace_ball_bigger_boundary}
\end{figure}

\begin{definition}
If $M'$ is obtained from $M$ as in \cref{lem_combinatorial_ball_replace_same_boundary} or \cref{lem_combinatorial_ball_replace_bigger_boundary}, We say that \emph{$M'$ is obtained from $M$ by replacing $C$ with $C'$}.
\end{definition}

We record the observation from \cref{eq_observation_for_replacement_stars} in the proof of \cref{lem_combinatorial_ball_replace_same_boundary} for later reference:
\begin{lemma}
\label{lem_replacement_stars}
If $M'$ is obtained from $M$ by replacing $C$ with $C'$ and $\Theta$ is a simplex in $M'\setminus C'$, we have $\Star_{M'}(\Theta) = \Star_M(\Theta)$.
\end{lemma}

\subsection{Cross maps and regularity}
\label{subsec:cross-maps-and-regularity}

\cref{lem_simplicial_approximation_combinatorial} above allows us to restrict to maps $M\to \IAArel$ from combinatorial manifolds to $\IAArel$ for proving \autoref{thm_connectivity_IAA}. We will further restrict the class of maps by asking that they satisfy certain regularity conditions, which we introduce in this subsection. The notion of regularity presented here extends the concept of $\sigma$-regularity as defined by Putman, see \cref{def_sigma_cross_map}. In \cref{rem_polyhedral_complexes} at the end of this subsection, we briefly describe some intuition for these regularity conditions.

\begin{definition}[Cross polytopes and prisms]
\label{def_prism} \
\begin{enumerate}
\item We denote by $C_k$ the join of $k$ copies of $S^0$. (This is the boundary of the $k$-dimensional cross polytope and a combinatorial $(k-1)$-sphere by \cref{rem_cross_polytope_combinatorial}.). We interpret $C_k$ for $k\leq 0$ as the empty set.
\item We denote by $P_3$ the simplicial complex with the vertices $x_1$, $x_2$, $x_{12}$, $y_1$, $y_2$, $y_{12}$ that is given as the union of the three 3-simplices $\ls x_1, x_2, x_{12}, y_{12} \rs $, $\ls x_1, x_2, y_1, y_{12} \rs $ and $\ls x_1, y_1, y_2, y_{12} \rs $. This complex is the 3-dimensional ``prism'' depicted in \cref{figure_prism}.
\end{enumerate}
\end{definition}

\begin{remark}
\label{rem_cross_polytope_combinatorial}
 \cref{it_join_sphere} of \cref{lem_boundary_join} implies that $C_k$ is a combinatorial $(k-1)$-sphere.
\end{remark}

\begin{figure}
	    \begin{center}\vspace{40px}
    \begin{picture}(200,120)
    \put(0,10){\includegraphics[scale=.8]{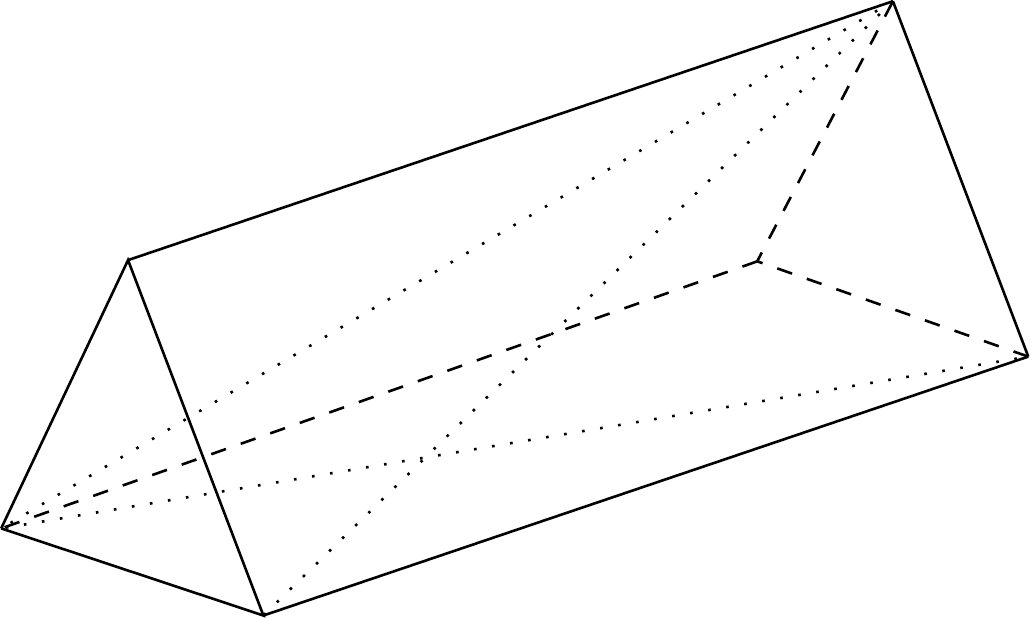}}
    
     \put(18,97){$x_{12}$}
     \put(55,2){$x_2$}
     \put(-12,30){$x_1$}     
    
    \put(200,160){$y_{12}$}
     \put(180,95){$y_2$}
     \put(245,70){$y_1$}        
    
     \end{picture}
	\end{center}
	\caption{$P_3$.}
	\label{figure_prism}
\end{figure}

The following is easy to verify.
\begin{lemma}
\label{lem_prism_combinatorial_manifold}
The complex $P_3$ is a combinatorial manifold.
\end{lemma}

In this work, it is helpful to think of prisms $P$ contained in a simplicial complex $M$ as one cell. We use the following notation to describe neighbourhoods of prisms $P$ in $M$.

\begin{definition} \label{def:link-and-star-of-subcomplex}
	If $P$ is a subcomplex of a simplicial complex $M$, then $\dots$
	\begin{enumerate}
		\item $\dots$ the \emph{link of $P$ in $M$} is
		$\Link_M(P) = \bigcap_{\Delta \text{ simplex of } P} \Link_M(\Delta);$
		\item $\dots$ the \emph{star of $P$ in $M$} is
		$\Star_M(P) = P \ast \Link_M(P).$
	\end{enumerate}
\end{definition}

Note that if $P$ is a simplex, this definition agrees with the usual definition of link and star in simplicial complexes.

\begin{definition}[Cross maps]\label{def:cross-maps}
Let $\Delta^1$, $\Delta^2$ denote an abstract 1- and 2-simplex, respectively.

\begin{enumerate}
\item  A \emph{$\sigma^2$ cross map} is a simplicial map $\phi\colon   \Delta^1 \ast \Delta^1 \ast C_{k-2} \to \IAArel$ with the following property: Let  $x_1,y_1 \ldots, x_k,y_k$ be the vertices of $\Delta^1 \ast \Delta^1 \ast C_{k-2}$. Then there is a symplectic summand of $\mbZ^{2(m+n)}$ with a symplectic basis $\ls \vec v_1, \vec w_1, \ldots, \vec v_k, \vec w_k \rs$ such that $\phi(x_i)= v_i$ and $\phi(y_i) = w_i$ for all $i$.
\item A \emph{prism cross map} is a simplicial map $\phi\colon   P_3 \ast C_{k-2} \to \IAArel$ with the following property: Let $\ls x_1, x_2, x_{12}, y_1, y_2, y_{12} \rs $ be the vertices of $P_3$ as in \cref{def_prism} and $x_3, y_3, \ldots, x_k,y_{k}$ be the vertices of $C_{k-2}$. There is a symplectic summand of $\mbZ^{2(m+n)}$ with a symplectic basis $\ls \vec v_1, \vec w_1, \ldots, \vec v_k, \vec w_k \rs$ such that 
\begin{gather*}
\phi(x_1) = v_1,\, \phi(x_2) = v_2,\, \phi(x_{12}) = \langle \vec v_1 + \vec v_2 \rangle,\\
 \phi(y_1) = w_1, \, \phi(y_2) = w_2,\, \phi(y_{12}) = \langle \vec w_1 - \vec w_2 \rangle, \\
\text{and } \phi(x_i)= v_i, \, \phi(y_i)= w_i \text{ for } i\in \ls 3,\ldots, k \rs.
\end{gather*}
\item A \emph{$\sigma$-additive cross map} is a simplicial map $\phi\colon   \Delta^2 \ast C_{k-1} \to \IAArel$ with the following property: Let  $xy_1,x_1,y_1, \ldots, x_k,y_k$ be the vertices of $\Delta^2 \ast C_{k-1}$. Then there exists a symplectic summand of $\mbZ^{2(m+n)}$ with a symplectic basis $\ls \vec v_1, \vec w_1, \ldots, \vec v_k, \vec w_k \rs$ such that $\phi(x_i)= v_i$, $\phi(y_i) = w_i$ for all $i$ and $\phi(xy_1) = \langle \vec v_1 + \vec w_1 \rangle$.
\item An \emph{external 2-skew-additive cross map} is a simplicial map $\phi\colon   \Delta^2 \ast C_{k-1} \to \IAArel$ with the following property: Let $xy_1, x_1, y_1, \dots, x_k, y_k$ be the vertices of $\Delta^2 \ast C_{k-1}$. Then there exists a symplectic summand of $\mbZ^{2(m+n)}$ with a symplectic basis $\{\vec v_1, \vec w_1, \dots, \vec v_k, \vec w_k\}$ such that $\phi(x_i) = v_i$, $\phi(y_i) = w_i$ for all $i$ and
	$$\phi(xy_1) = \langle \vec v_1 \pm \vec \vertexvar \rangle \text{ for some } \vertexvar \in \{e_1, \dots, e_m\}.$$
\end{enumerate}
\end{definition}

We observe the following facts about cross maps:

\begin{remark}
\label{rem_cross_maps_elementary_properties}
\begin{enumerate}
\item All cross maps defined above are isomorphisms onto their images.
\item If $\{\vec v_1, \vec w_1, \dots, \vec v_k, \vec w_k\}$ is the symplectic basis in the image of a cross map, then it is compatible with $\{\vec e_1, \dots, \vec e_m\}$. That is, $\{\vec v_1, \vec w_1, \dots, \vec v_k, \vec w_k, \vec e_1, \dots, \vec e_m\}$
can be extended to a symplectic basis of $\mbZ^{2(m+n)}$.
\item The link of the prism $P_3$ in the domain $P_3 \ast C_{k-2}$ of a prism cross map (as defined in \cref{def:link-and-star-of-subcomplex}) is equal to $C_{k-2}$ and its star is the whole domain,
\begin{equation*}
	\Link_{P_3 \ast C_{k-2}}(P_3) = C_{k-2} \text{ and }\Star_{P_3 \ast C_{k-2}}(P_3) = P_3 \ast C_{k-2}.
\end{equation*}
\item \label{it_prism_cross_image}A prism cross map $\phi$ sends $P_3$ to the union of the following three simplices: 
\begin{itemize}
\item the 2-skew-additive simplex $\ls v_1, v_2, \langle \vec v_1 + \vec v_2 \rangle, \langle \vec w_1 - \vec w_2 \rangle \rs $,
\item the skew-$\sigma^2$ simplex $\ls v_1, v_2, w_1, \langle \vec w_1 - \vec w_2 \rangle \rs$ and
\item the 2-skew-additive simplex $\ls v_1, w_1, w_2, \langle \vec w_1 - \vec w_2 \rangle \rs$.
\end{itemize}
This image contains two skew-additive simplices, namely $\{ v_1, v_2, \langle \vec w_1 -\nolinebreak \vec w_2 \rangle \} $ and $\{ v_1, w_1, \langle \vec w_1 - \vec w_2 \rangle \}$. 
It is a combinatorial 3-ball whose boundary $\partial \phi(P_3)$ is the union the following eight simplices, which are all contained in $\IAAstrel$:
\begin{align*}
	\ls v_1, v_2, \langle \vec v_1 + \vec v_2 \rangle\rs & \text{ --- 2-additive simplex}\\
	\ls v_1, \langle \vec v_1 + \vec v_2 \rangle, \langle  \vec w_1 - \vec w_2 \rangle \rs & \text{ --- $\sigma$ simplex}\\
	\ls v_2, \langle \vec v_1 + \vec v_2 \rangle, \langle  \vec w_1 - \vec w_2 \rangle \rs &\text{ --- $\sigma$ simplex}\\
	\ls v_1, v_2, w_1 \rs &\text{ --- $\sigma$ simplex}\\
	\ls v_2, w_1, \langle \vec w_1 - \vec w_2 \rangle \rs &  \text{ --- $\sigma$ simplex}\\
	\ls v_1, w_1, w_2 \rs & \text{ --- $\sigma$ simplex}\\
	\ls v_1, w_2, \langle \vec w_1 - \vec w_2 \rangle \rs & \text{ --- $\sigma$ simplex}\\
	\ls w_1, w_2, \langle \vec w_1 - \vec w_2 \rangle \rs & \text{ --- 2-additive simplex}
\end{align*}
\end{enumerate}
\end{remark}

Spelling out the definitions, we can describe the simplex types in the images of cross maps:

\begin{lemma}
\label{lem_cross_maps_simplex_types}
Let $\phi\colon   C \to \IAArel$ be a cross map. 
Then $C\cong \phi(C)$ is a combinatorial ball and:
\begin{enumerate}
\item If $\phi$ is a $\sigma^2$ cross map, then 
	\begin{itemize}
		\item $\phi(\partial  C)$ contains only standard and $\sigma$ simplices;
		\item $\phi(C \setminus \partial  C)$ contains only $\sigma^2$ simplices.
	\end{itemize}
\item If $\phi$ is a prism cross map, then 
	\begin{itemize}
		\item $\phi(\partial  C)$ contains only standard, $\sigma$ and (internal) 2-additive simplices;
		\item $\phi(C \setminus \partial  C)$ contains only (internal) skew-additive, (internal) 2-skew-additive and skew-$\sigma^2$ simplices.
	\end{itemize}
\item If $\phi$ is a $\sigma$-additive cross map, then 
	\begin{itemize}
		\item $\phi(\partial  C)$ contains only standard and $\sigma$ simplices;
		\item $\phi(C \setminus \partial  C)$ contains only $\sigma$-additive simplices.
	\end{itemize}
\item If $\phi$ is an external 2-skew-additive cross map, then 
	\begin{itemize}
		\item $\phi(\partial  C)$ contains only standard, $\sigma$ and (external) 2-additive simplices;
		\item $\phi(C \setminus \partial  C)$ contains only (external) 2-skew-additive simplices.
	\end{itemize}
\end{enumerate}
\end{lemma}
\begin{proof}
We can write $C = \Sigma \ast C_{i}$ as in the definition of the corresponding cross map. By \cref{lem_boundary_join}, this is a combinatorial ball and we have
\begin{equation}
\label{eq_boundary_cross_map_domain}
\partial C = \partial \Sigma \ast C_{i} \cup \Sigma \ast \partial  C_{i} = \partial \Sigma \ast C_{i},
\end{equation}
where we use that $\partial C_{i}$ is a combinatorial sphere and hence has empty boundary $\partial C_{i} = \emptyset$.
Spelling out the definitions of the cross maps, we obtain the claimed descriptions of $\phi(\partial C)$ (for the case of prism cross maps, see \cref{it_prism_cross_image} of \cref{rem_cross_maps_elementary_properties}).

\cref{eq_boundary_cross_map_domain} also implies that if a simplex of $C$ does not lie in $\partial C$, it must be of the form of the form $\Delta \ast \Theta$, where $\Delta$ is a face of $\Sigma$ that is not contained in $\partial \Sigma$ and $\Theta$ is contained in $C_i$. The image $\phi(\Delta \ast \Theta)$ is then a simplex of the same type as $\phi(\Delta)$. Again going through the definitions, the claimed descriptions of $\phi(C \setminus \partial  C)$ follow.
\end{proof}

Noting that $\partial \phi(C) = \phi(\partial  C)$, we get the following corollary.

\begin{corollary}
\label{lem_boundary_cross_map}
Let $\phi\colon   C \to \IAArel$ be a $\sigma^2$-, prism-, $\sigma$-additive or external 2-skew-additive cross map. Then $\partial \phi(C) \subseteq \IAAstrel$.
\end{corollary}

\begin{definition}[Regular maps]
\label{def_regular_maps}
Let $M$ be a combinatorial manifold. A simplicial map $\phi\colon   M \to \IAArel$ is called $\dots$
  \begin{enumerate}
    \item $\dots$ \emph{$\sigma^2$-regular} if the following holds: If $\Delta$ is simplex of $M$ such that $\phi(\Delta)$ is a minimal (i.e.~3-dimensional) $\sigma^2$ simplex, then $\phi|_{\Star_M({\Delta})}$ is a $\sigma^2$ cross map.
    \item $\dots$ \emph{weakly prism-regular} if the following holds: Let $\Delta$ be a simplex of $M$ such that $\phi(\Delta)$ is a minimal skew-additive, 2-skew-additive or skew-$\sigma^2$ simplex. Then one of the following two cases holds:
    \begin{enumerate}
    	\item \label{wprism-condition-a} There exists a unique subcomplex $P \cong P_3$ of $M$ such that $\Delta \subseteq P$, $\Link_M(P) = \Link_M(\Delta')$ for any of the three maximal simplices\footnote{The intuition behind this condition is that in a (weakly) prism-regular map, a prism $P$, although it is the union of three maximal simplices, should rather be considered as a single -- non-simplicial -- cell.} $\Delta'$ of $P$, and $\phi|_{\Star_M(P)}$ is a prism cross map.
    	\item \label{wprism-condition-b} The simplex $\Delta$ has dimension $2$, 
    	$$\phi(\Delta) = \{v_0, v_1, \langle \vec v_0 \pm \vec \vertexvar \rangle \} \text{ for some } \vertexvar \in \{e_1, \dots, e_m\},$$
    	is an \emph{external} 2-skew-additive simplex with $\omega(\vec v_0, \vec v_1) = \pm 1$, and $\phi|_{\Star_M(\Delta)}$ is an external 2-skew-additive cross map.
    \end{enumerate}
    \item $\dots$ \emph{prism-regular} if the following holds: Let $\Delta$ be a simplex of $M$ such that $\phi(\Delta)$ is a minimal skew-additive, 2-skew-additive or skew-$\sigma^2$ simplex. Then the condition in \cref{wprism-condition-a} is satisfied.
    \item $\dots$ \emph{$\sigma$-additive-regular} if the following holds: If $\Delta$ is a simplex of $M$ such that $\phi(\Delta)$ is a minimal (i.e.~2-dimensional) $\sigma$-additive simplex, then $\phi|_{\Star_M({\Delta})}$ is a $\sigma$-additive cross map.
    \item $\dots$ \emph{weakly regular} if $\phi$ is $\sigma^2$-regular, weakly prism-regular and $\sigma$-additive-regular.
    \item $\dots$ \emph{regular} if $\phi$ is $\sigma^2$-regular, prism-regular and $\sigma$-additive-regular.
  \end{enumerate}
\end{definition}

Note that, compared to the definition of a regular map, the regularity notion for prisms is slightly less restrictive in the definition of a weakly regular map. In particular, regularity implies weak regularity:

\begin{lemma}\label{lem:regular-implies-weakly-regular}
	Every regular map $\phi\colon   M \to \IAArel$ is weakly regular.
\end{lemma}

\begin{lemma}
\label{lem_cross_maps_intersections}
Let $\phi\colon   M \to \IAArel$ be a weakly regular map.
Let $C$ and $C'$ be maximal subcomplexes of $M$ such that the restrictions $\phi|_C$ and $\phi|_{C'}$ are $\sigma^2$-, prism-, $\sigma$-additive- or external 2-skew-additive cross maps. 
Then either $C$ and $C'$ coincide or they only intersect in their boundaries,
\begin{equation*}
C=C' \text{ or } C \cap C' \subseteq \partial C \cap \partial C'.
\end{equation*}
\end{lemma}
\begin{proof}
It is enough to show that a simplex $\Delta$ in $C \setminus \partial C$ can only be contained in $C'$ if $C=C'$: By symmetry, the same argument then also show that a simplex in $C' \setminus \partial C'$ can only be contained in $C$ if $C=C'$. This is equivalent to the claim.

By \cref{lem_cross_maps_simplex_types}, the simplex $\Delta\in C \setminus \partial C$ can only be contained in $C'$ if $\phi|_C$ and $\phi|_{C'}$ are cross maps of the same type.
It is easy to verify that this implies $C = C'$. 
\end{proof}

\begin{definition}[Regular homotopies]
\label{def_reg_homotopies}
Let $0\leq k \leq n$ and let $X$ be a subcomplex of $\IAArel$.
\begin{enumerate}
\item For $i = 1,2$, let $\phi_i\colon   S_i \to\IAAstrel$ be simplicial maps from combinatorial $k$-spheres $S_i$. We say that $\phi_1$ and $\phi_2$ are \emph{(weakly) regularly homotopic (in $X$)} if there are a combinatorial manifold $M$ homeomorphic to $S^k \times [0,1]$ with $\partial M = S_1 \sqcup S_2$ and a (weakly) regular map $\Psi\colon   M \to X\subseteq \IAArel$ such that $\Psi|_{S_i} = \phi_i$ for $i=1,2$.
\item Let $\phi\colon   S \to\IAAstrel$ be a simplicial map from a combinatorial $k$-sphere $S$. We say that $\phi$ is \emph{(weakly)  regularly nullhomotopic (in $X$)} if there is a combinatorial ball $B$ with $\partial B = S$ and a (weakly)  regular map $\Psi\colon   B \to X\subseteq \IAArel$ such that $\Psi|_{S} = \phi$.
\end{enumerate}
\end{definition}

\begin{remark}
\label{rem_regular_homotopies_notation_combinatorial}
Note that we only define (weak) regular homotopies between maps whose domains are \emph{combinatorial} spheres of the same dimension. So if $S$ is a combinatorial $k$-sphere, $\phi\colon   S \to \IAAstrel$ a simplicial map and we say ``$\phi$ is (weakly) regularly homotopic to $\newlemmaphi\colon   \newlemmasphere \to \IAAstrel$'', then we always assume that $\newlemmasphere$ is also a combinatorial $k$-sphere and $\newlemmaphi$ is a simplicial map.
\end{remark}

The notions of regular homotopies and nullhomotopies are clearly special cases of the corresponding topological notions. I.e.~if $\phi$ and $\newlemmaphi$ are weakly regularly homotopic, then $|\phi|$ and $|\newlemmaphi|$ are homotopic; and if $\phi$ is weakly regularly nullhomotopic, then $|\phi|$ is nullhomotopic.
Furthermore, these notions are compatible with one another (cf.~\cite[Lemma 6.18.2b]{put2009}):
\begin{lemma}
\label{lem_regular_homotopic_and_nullhomotopic}
If $\phi$ and $\phi'$ are (weakly) regularly homotopic and $\phi'$ is (weakly) regularly nullhomotopic, then so is $\phi$.
\end{lemma}

For constructing regular homotopies, we mostly use the following lemma. It can be proved like \cite[Lemma 6.18.3]{put2009}.
\begin{lemma}
\label{lem_regular_homotopy_by_balls}
Let $S$ be a combinatorial $k$-sphere and $\phi\colon   S\to \IAAstrel$ a simplicial map. 
Let $B$ be a combinatorial $(k+1)$-ball and $\Psi\colon   B \to \IAArel$ a (weakly) regular map. 
Assume that $
\partial B = D_1 \cup D_2$
is the union of two combinatorial $k$-balls $D_1$, $D_2$ such that 
\begin{equation*}
	D_1 \cap D_2 = \partial D_1 = \partial D_2.
\end{equation*}
Furthermore, assume that $S \cap B = D_1$ and $\Psi|_{D_1} = \phi|_{D_1}$ (see \cref{figure_replace_map}).

Let $\newlemmasphere$ be the combinatorial $k$-sphere that is obtained from $S$ by replacing $D_1$ with $D_2$ and $\newlemmaphi: \newlemmasphere\to \IAArel$ the simplicial map defined by $\newlemmaphi|_{D_2} = \Psi|_{D_2}$ and $\newlemmaphi|_{\newlemmasphere \setminus D_2} = \phi$. Then $\phi$ is (weakly) regularly homotopic to $\newlemmaphi$.

We say that \emph{$\newlemmaphi$ is obtained from $\phi$ by replacing $\phi|_{D_1}$ by $\Psi|_{D_2}$}.
\end{lemma}

\begin{figure}
\begin{center}
\includegraphics{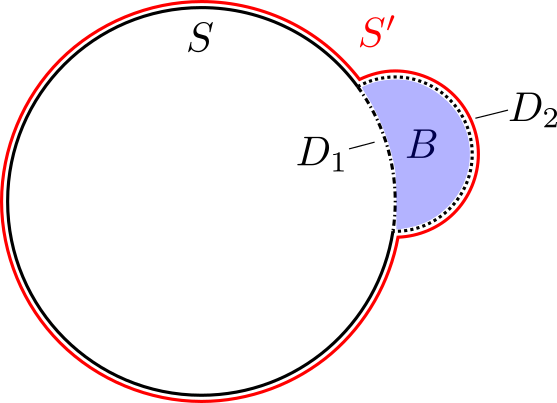}
\end{center}
\caption{Replacing $\phi\colon   S\to \IAAstrel$ by $\phi'\colon  S'\to \IAAstrel$ for $k=1$.}
\label{figure_replace_map}
\end{figure}

Using \cref{lem_replacement_stars}, we obtain the following.
\begin{lemma}
\label{lem_replacement_stars_maps}
If $\newlemmaphi$ is obtained from $\phi$ by replacing $\phi|_{D_1}$ by $\Psi|_{D_2}$ and $\Theta$ is a simplex in $\newlemmasphere\setminus D_2 = S \setminus D_1$, we have
$\newlemmaphi|_{\Star_{\newlemmasphere}}(\Theta) = \phi|_{\Star_{S}}(\Theta)$.
\end{lemma}

We finish with a comment about the motivation for introducing the regularity conditions above.

\begin{remark}
\label{rem_polyhedral_complexes}
Much of this work is concerned with showing that the complex $\IAA$ is highly connected, i.e.\ proving \autoref{thm_connectivity_IAA}. When we eventually do this in \cref{sec:thm_connectivity_IAA} and \cref{sec_normal_form_spheres}, we consider $\sigma$-regular maps $\phi\colon   S \to \IAA$, which are introduced in \cref{def_sigma_cross_map}, and regular homotopies $\psi\colon   B \to \IAA$, in the sense of \cref{def_reg_homotopies}, between these. 
Throughout this work and in these arguments in particular, it is useful to think of the cross maps contained in $\phi$ and $\psi$ as being defined on a single polyhedral cell that consists of several simplices. For example, a prism cross map $\psi|_{P_3 \ast C_{k-2}}\colon   P_3 \ast C_{k-2} \to \IAA$ contained in a regular homotopy $\psi$ can be thought of as a map defined on a single polyhedral cell that for $k = 2$ has the shape of the prism depicted in \cref{figure_prism}.
It might be possible to formulate the arguments presented in this work in a category of suitable polyhedral cell complexes. The disadvantage of this might be that the arguments are less parallel to the ones in previous work \cite{cfp2019,CP,Brueck2022,put2009}.
\end{remark}

\section{Reducing the rank}
\label{sec_retraction} 

In this section, we study certain subcomplexes $X$ of $\Irel, \IArel, \IAAstrel$ and $\IAArel$.
We develop tools that are applied in later sections (see \cref{sec_highly_connected_subcomplexes} and \cref{sec:thm_connectivity_IAA}) to prove that various simplicial complexes, including $\IAArel$, are highly connected.

 Recall from \cref{def_rank_and_ranked_complexes} that the rank of a vertex $v$ of $X$ is defined as the absolute value of the $\vec f_{m+n}$-coordinate of some (hence any) primitive vector $\vec v \in v$, $\rkfn(v) = |\omega(\vec e_{m+n}, \vec v)|$, and that we denote by $X^{< R}$ the full subcomplex of $X$ on all vertices $v$ satisfying $\rkfn(v) < R$. This yields a filtration of $X$,
$$X(W) = X^{< 1} \subset \dots \subset X^{< R} \subset \dots \subset X,$$
interpolating between $X(W)$, the full subcomplex of $X$ on all vertices $v$ contained in the submodule
$$W = \ll \vec e_1, \vec f_1, \dots, \vec e_{m+n-1}, \vec f_{m+n-1}, \vec e_{m+n} \rr,$$
and the simplicial complex $X$. Our goal is to develop techniques that allow us to ``reduce the rank'', i.e.~to construct a map $\phi'\colon   M \to X^{< R}$ from a given map $\phi\colon   M \to X$. An additional difficulty is to do this in such a way that desirable properties such as regularity (see \cref{def_regular_maps}) are preserved. More precisely, we will show the following:

\begin{proposition} \label{IAA-retraction}
    Let $n \geq 2$, $m \geq 0$, and $R > 0$. Let $\badvertex$ be a vertex of $\IAArel$ with $\rkfn(\badvertex) = R$. Assume we are given a commutative diagram of simplicial maps
  $$
  \begin{tikzcd}
    S^{k-1} \arrow[r, "\phi|"] \arrow[d, hook] & \Linkhat_{\IAArel}^{<R}(\badvertex) \arrow[d, hook]\\
    D^k \arrow[r, "\phi"] & \Linkhat_{\IAArel}(\badvertex)
  \end{tikzcd}
  $$
  such that $\phi$ is weakly regular\footnote{A precise definition for what this means is given in \cref{def:weak-regularity-in-link}.}. Then there exists a combinatorial ball $D^k_{\new}$ with boundary sphere $S^{k-1}$ and a commutative diagram of simplicial maps
  $$
  \begin{tikzcd}
    S^{k-1} \arrow[r, "\phi|"] \arrow[d, hook] & \Linkhat_{\IAArel}^{<R}(\badvertex)\\
    D^k_{\new} \arrow[ur, "\psi", swap] & 
  \end{tikzcd}
  $$
  such that $\psi$ is weakly regular.
\end{proposition}

We work our way to this result through several subsections.
Recall that $\Irel \subset \IArel \subset \IAAstrel \subset \IAArel$. We start by studying subcomplexes of $\Irel$ in the first subsection. Then we move to $\IArel, \IAAstrel$ and $\IAArel$, respectively. The results in each subsection are extensions of the ones from the previous subsection. A key idea in the arguments here is that forgetting the symplectic form yields inclusion maps
\begin{itemize}
\item $\Irel \hookrightarrow \B_{2n+m}^{m}$,
\item $\IArel \hookrightarrow \BA_{2n+m}^{m}$,
\item $\IAAstrel \hookrightarrow \BAA_{2n+m}^{m}$ and
\item $\IAArel \hookrightarrow \BAA^m_{2n+m}$
\end{itemize}
(compare \cref{rem:forgetting_symplectic_information}). The complexes on the right of these arrows were studied in the context of high-dimensional rational cohomology of the special linear group $\SL{n}{\mbZ}$. This is why, up to technicalities, ``restricting'' along the first three inclusion maps allows us to deduce our results for $\Irel$, $\IArel$ and $\IAAstrel$ from work of Maazen \cite{Maazen79}, Putman \cite{put2009}, Church--Putman \cite{CP} and Brück--Miller--Patzt--Sroka--Wilson \cite{Brueck2022}.
When ``reducing the rank'' in $\IAArel$, we additionally need to ensure that certain regularity properties are preserved. Overcoming this extra difficulty for $\IAArel$ requires a careful study of cross maps and is our main contribution here.
\newline

\subsection{Reducing the rank in \texorpdfstring{$\Irel$}{I}}

In this subsection, we explain how ideas due to Maazen \cite{Maazen79} and Church--Putman \cite[Section 4.1]{CP} generalise to the symplectic setting, and how they can be used to ``reduce the rank'' in $\Irel$. Similar ideas were used by Putman in the proof of \cite[Proposition 6.13.2]{put2009}. We use the notation introduced in \cref{def_rank_and_ranked_complexes} and frequently apply the notions defined in \cref{subsec:simplex-types-in-relative-complexes}. The following construction explains how to ``reduce the rank'' of vertices; it plays a key role in throughout this section.

\begin{construction} \label{def:I-retraction}
	Let $m,n \geq 0$, $R > 0$ and $X^m_n$ be $\IAArel$ or any other complex in \cref{def_linkhat}. Let $\Delta$ be a standard simplex of $X^m_n$ such that some vertex $\badvertex \in \Delta$ satisfies $\rkfn(\badvertex) = R$. Let $v \in \Vr(\Linkhat_{X^m_n}(\Delta))$ be a vertex. Recall that $\bar v \in \{\vec v, - \vec v\}$ is a primitive vector in $v$ satisfying $\rkfn(\bar v) \geq 0$ (see \cref{def_rank_and_ranked_complexes}). Then we define $\retraction(v) \in \Vr(\Linkhat_{X^m_n}^{< R}(\Delta))$ to be the vertex of $\Linkhat_{X^m_n}^{< R}(\Delta)$ given by
	$$\retraction(v) \coloneqq \ll \bar v - a \bar \badvertex \rr$$
	where $a \in \mbZ$ is chosen such that $\rkfn(\bar v - a \bar \badvertex) \in [0, R)$ and determined by the Euclidean algorithm.
\end{construction}

The goal of this subsection is to show that for $X^m_n = \Irel$ the map $\retraction$ between the vertex set of $\Linkhat_{\Irel}(\Delta)$ and $\Linkhat_{\Irel}^{< R}(\Delta)$ is simplicial.

\begin{proposition}\label{I-retraction}
  Let $n \geq 0$, $m \geq 0$ and $R > 0$. Let $\Delta$ be a simplex of $\Irel$ such that some vertex $\badvertex$ of $\Delta$ satisfies $\rkfn(\badvertex) = R$. Then the map $\rho$ defined in \cref{def:I-retraction} is a simplicial retraction
  $$\retraction \colon   \Linkhat_{\Irel}(\Delta) \twoheadrightarrow \Linkhat^{< R}_{\Irel}(\Delta).$$
\end{proposition}

\cref{I-retraction} is a  symplectic analogue of the following result by Church--Putman.

\begin{lemma}[{\cite[Lemma 4.5]{CP}}] \label{B-retraction}
  Let $n \geq 0$, $m \geq 0$, and $R > 0$. Let $\Delta$ be a simplex of $\B^m_{2n+m}$ such that some vertex $\badvertex$ of $\Delta$ satisfies $\rkfn(\badvertex) = R$. Then the map $\rho$ defined in \cref{def:I-retraction} is a simplicial retraction
  $$\retraction \colon   \Linkhat_{\B^m_{2n+m}}(\Delta) \twoheadrightarrow \Linkhat^{< R}_{\B^m_{2n+m}}(\Delta).$$
\end{lemma}

To deduce \cref{I-retraction} from \cref{B-retraction}, we study the effect of \cref{def:I-retraction} on symplectic information. This is the content of the next lemma.

\begin{definition} \label{def:non-additive-simplex}
  Let $n,m \geq 0$ and $R > 0$. Let $\Delta$ be a standard simplex of $\IAArel$ such that some vertex $\badvertex$ of $\Delta$ satisfies $\rkfn(\badvertex) = R$. A simplex $\Delta' \in \Linkhat_{\IAArel}(\Delta)$ is called \emph{non-additive} if it is of type standard, $\sigma$, $\sigma^2$, skew-additive, or skew-$\sigma^2$.
\end{definition}

Observe that in the setting of \cref{def:non-additive-simplex}, $\Delta'$ is non-additive if and only if its image under the inclusion $\Linkhat_{\IAArel}(\Delta) \hookrightarrow \Linkhat_{\BAA^m_{2n+m}}(\Delta)$ is a standard simplex. The next lemma is similar to \cite[Lemma 6.15]{put2009}.

\begin{lemma}
  \label{retraction-on-non-additive-simplices}
  Let $n \geq 0$, $m \geq 0$ and $R > 0$. Let $\Delta$ be a standard simplex of $\IAArel$ such that some vertex $\badvertex$ of $\Delta$ satisfies $\rkfn(\badvertex) = R$. Assume that $\Delta' \in \Linkhat_{\IAArel}(\Delta)$ is a non-additive simplex of type 
  $$\tau \in \{\text{standard}, \sigma, \sigma^2, \text{skew-additive}, \text{skew-}\sigma^2\}.$$
  Then the set of vertices $\retraction(\Delta')$ forms a non-additive simplex of type $\tau$ and of the same dimension as $\Delta'$ in $\IAArel$. Here, $\retraction$ is the map of sets defined in \cref{def:I-retraction}.
\end{lemma}
\begin{proof}
  By \cref{def:non-additive-simplex} and \autoref{B-retraction}, it follows that $\retraction(\Delta')$ is a standard simplex in $\B^m_{2n+m}$. The proof of \cite[Lemma 4.5]{CP} shows that $\retraction(\Delta')$ has the same dimension as $\Delta'$. Hence, we only need to see that the retraction map preserves the symplectic relations. By the definition of $\Linkhat_{\IAArel}(\Delta)$, it holds that 
  $$\langle \vec v \mid v \in \Delta' \rangle \subset \badvertex^\perp,$$
  where $\badvertex^\perp$ is the symplectic complement of $\badvertex$ in $\mbZ^{2(m+n)}$.
  This implies that
  $$\omega(\bar v, \bar v') = \omega(\bar v - a \bar \badvertex, \bar v' - a' \bar \badvertex)$$
  for all $v, v' \in \Delta'$ and any choice of $a, a' \in \Z$. Again by the definition of $\Linkhat_{\IAArel}(\Delta)$ and since $\Delta$ is a standard simplex of $\IAArel$, it holds that
  $$\omega(\bar u, \bar v - a \bar \badvertex) = 0$$ for any $v \in \Delta$, $u \in \Delta'$ and any $a \in \Z$. These three observations together yield that $\retraction(\Delta')$ is a non-additive simplex in $\IAArel$ of the same type as $\Delta'$.
\end{proof}
  
\begin{proof}[Proof of \autoref{I-retraction}]
  Consider the inclusion $\Irel \hookrightarrow \B^m_{2n+m}$. It restricts to the inclusion
  $$\Linkhat_{\Irel}(\Delta) \hookrightarrow  \Linkhat_{\B^m_{2n+m}}(\Delta).$$
  \autoref{B-retraction} yields a simplicial retraction
  $\retraction\colon  \Linkhat_{\B^m_{2n}}(\Delta) \twoheadrightarrow \Linkhat^{<R}_{\B^m_{2n+m}}(\Delta)$
  defined as in \cref{def:I-retraction}.
  By \autoref{retraction-on-non-additive-simplices}, it restricts to a retraction
  \begin{equation*}
  	\retraction\colon  \Linkhat_{\Irel}(\Delta) \twoheadrightarrow \Linkhat^{<R}_{\Irel}(\Delta).\qedhere
  \end{equation*}
\end{proof}

\subsection{Reducing the rank in \texorpdfstring{$\Isigdelrel$}{I\unichar{"005E}\{\unichar{"03C3}, \unichar{"03B4}\}}}

In this subsection, we explain how ideas due to Church--Putman \cite[Section 4.4]{CP} generalise to the symplectic setting, and how they can be used to ``reduce the rank'' in $\Isigdelrel$. Similar ideas are used by Putman in the proof of \cite[Proposition 6.13.4]{put2009}. We continue to use the notation introduced in \cref{def_rank_and_ranked_complexes} and frequently apply the notions defined in \cref{subsec:simplex-types-in-relative-complexes} to refer to simplices in $\Linkhat_{\Isigdelrel}(\badvertex)$. The goal of this subsection is to prove the following symplectic analogue of Church--Putman \cite[Proposition 4.17]{CP}.

\begin{proposition}\label{Isigdel-retraction} 
  Let $n \geq 1$, $m \geq 0$ and let $R > 0$. Let $\badvertex$ be a vertex of $\Isigdelrel$ such that $\rkfn(\badvertex) = R$. Then there exists a subdivision $\sd(\Linkhat_{\Isigdelrel}(\badvertex))$ of $\Linkhat_{\Isigdelrel}(\badvertex)$ containing $\Linkhat_{\Isigdelrel}^{< R}(\badvertex)$ as a subcomplex and a simplicial retraction
  $$\retraction \colon  \sd(\Linkhat_{\Isigdelrel}(\badvertex)) \twoheadrightarrow \Linkhat^{< R}_{\Isigdelrel}(\badvertex)$$
  that extends the retraction introduced in \cref{I-retraction}.
\end{proposition}

\begin{construction} \label{def:Isigdel-subdivision}
  The following is based on \cite[p.1022]{CP}. Let $\Delta = \{v_0, \dots, v_k\}$ be an \emph{internal} 2-additive simplex in $\Linkhat_{\Isigdelrel}(\badvertex)$ with augmentation core given by $\Theta = \{v_0, v_1, v_2\}$. We may assume that $\bar v_{0} = \bar v_{1} + \bar v_{2}$. In the setting of \cref{Isigdel-retraction}, $\Delta$ is called \emph{carrying} if
  $$\left\lfloor \frac{\rkfn(\bar v_{0})}{R} \right\rfloor \neq \left\lfloor \frac{\rkfn(\bar v_{1})}{R} \right\rfloor + \left\lfloor \frac{\rkfn(\bar v_{2})}{R} \right\rfloor.$$
  The subdivision $\sd(\Linkhat_{\Isigdelrel}(\badvertex))$ of $\Linkhat_{\Isigdelrel}(\badvertex)$ is obtained by placing a new vertex $t(\Theta)$ at the barycentre of every carrying \emph{minimal} 2-additive simplex $\Theta$. This subdivides $\Theta$ into three simplices $\Theta_1$, $\Theta_2$ and $\Theta_3$ involving $t(\Theta)$. The subdivision is then extended to every carrying 2-additive simplex $\Delta = \Theta \ast \Delta'$ with augmentation core $\Theta$ by subdividing $\Delta$ into three simplices $\Theta_1 \ast \Delta'$, $\Theta_2 \ast \Delta'$ and $\Theta_3 \ast \Delta'$. In particular, only internal 2-additive simplices with carrying augmentation core get subdivided when passing from $\Linkhat_{\Isigdelrel}(\badvertex)$ to $\sd(\Linkhat_{\Isigdelrel}(\badvertex))$.
\end{construction}

\begin{construction} \label{def:Isigdel-retraction}
  On vertices $v \in \Linkhat_{\Isigdelrel}(\badvertex)$, the retraction $\retraction$ in \cref{Isigdel-retraction} agrees with the retraction for $\Irel$ (see \cref{def:I-retraction}). Recall that $\bar v \in \{\vec v, - \vec v\}$ denotes a primitive vector in $v$ satisfying $\rkfn(\bar v) \geq 0$ (see \cref{def_rank_and_ranked_complexes}). Then $\retraction(v) = \ll \bar v - a \bar \badvertex \rr$ where $a \in \mbZ$ is chosen such that $$\rkfn(\bar v - a \bar \badvertex) \in [0, R).$$
  On the new vertices $t(\Theta) \in \sd(\Linkhat_{\Isigdelrel}(\badvertex))$ sitting at the barycentres of   minimal carrying internal 2-additive simplices $\Theta = \{v_0, v_1, v_2\} \in \Linkhat_{\Isigdelrel}(\badvertex)$ (see \cref{def:Isigdel-subdivision}), the vertices $\{v_0, v_1, v_2\}$ are used to define $\retraction(t(\Theta))$ as follows: As in \cref{def:Isigdel-subdivision}, we may assume that $\bar v_{0} = \bar v_{1} + \bar v_{2}$. Furthermore, if $\Theta$ is carrying one has $\rkfn(\bar v_{1}), \rkfn(\bar v_{2}) > 0$ (see \cite[p. 1022, l.11 et seq.]{CP}). Therefore, $v_0$ is the \emph{unique} vertex in $\Theta$ maximising $\rkfn(-)$ among $\{v_0, v_1, v_2\}$. Pick an \emph{arbitrary} index $i \in \{1,2\}$. Then the value of $\retraction$ at $t(\Theta)$ is defined as
  $$\retraction(t(\Theta)) \coloneqq \langle \overbar{\retraction(v_i)} - \bar \badvertex \rangle.$$
  One can check that $0 < \rkfn(\overbar{\retraction(v_i)}) < R$ (see \cite[p.1023, l.26 et seq.]{CP}). This implies that 
  $$\rkfn(\overline{\retraction(t(\Theta))}) = R - \rkfn(\overline{\retraction(v_i)}) < R,$$ 
  i.e.\ that $\retraction(t(\Theta)) \in \Linkhat_{\Isigdelrel}^{< R}(\badvertex)$.
 \end{construction}

\cref{Isigdel-retraction} is a consequence of the following.

\begin{lemma}[{ \cite[Proposition 4.17]{CP}}] \label{BA-retraction}
  Let $n \geq 1$, $m \geq 0$ and let $R > 0$. Let $\badvertex$ be a vertex of $\BA^m_{2n+m}$ such that $\rkfn(\badvertex) = R$. Then there exists a subdivision $\sd(\Linkhat_{\BA^m_{2n+m}}(\badvertex))$ of $\Linkhat_{\BA^m_{2n+m}}(\badvertex)$ containing $\Linkhat_{\BA^m_{2n+m}}^{< R}(\badvertex)$ as a subcomplex and a simplicial retraction
  $$\retraction \colon  \sd(\Linkhat_{\BA^m_{2n+m}}(\badvertex)) \twoheadrightarrow \Linkhat^{< R}_{\BA^m_{2n+m}}(\badvertex)$$
  extending the retraction introduced in \cref{B-retraction}.
\end{lemma}

\begin{construction}\label{def:BA-subdivision-and-retraction}
  In \cref{BA-retraction}, the subdivision $\sd(\Linkhat_{\BA^m_{2n+m}}(\badvertex))$ of $\Linkhat_{\BA^m_{2n+m}}(\badvertex)$ and the retraction $\retraction$ are defined as discussed in \cref{def:Isigdel-subdivision} and \cref{def:Isigdel-retraction}.
\end{construction}

Combining this result with \autoref{retraction-on-non-additive-simplices} yields \autoref{Isigdel-retraction}.

\begin{proof}[Proof of \autoref{Isigdel-retraction}]
  The inclusion $\Isigdelrel \hookrightarrow \BA^m_{2n+m}$ restricts to an inclusion $\Linkhat_{\Isigdelrel}(\badvertex) \hookrightarrow  \Linkhat_{\BA^m_{2n+m}}(\badvertex)$. By \autoref{BA-retraction}, there is a simplicial retraction
  $$\retraction\colon  \sd(\Linkhat_{\BA^m_{2n+m}}(\badvertex)) \twoheadrightarrow \Linkhat^{< R}_{\BA^m_{2n+m}}(\badvertex).$$
  Let $\sd(\Linkhat_{\Isigdelrel}(\badvertex))$ be the restriction of the subdivision $\sd(\Linkhat_{\BA^m_{2n+m}}(\badvertex))$ of $\Linkhat_{\BA^m_{2n+m}}(\badvertex)$ to the subcomplex $\Linkhat_{\Isigdelrel}(\badvertex)$. Restriction yields a simplicial map
  $$\retraction\colon  \sd(\Linkhat_{\Isigdelrel}(\badvertex)) \to \Linkhat^{< R}_{\BA^m_{2n+m}}(\badvertex).$$
  We claim that the image of each simplex $\Delta$ in $\sd(\Linkhat_{\Isigdelrel}(\badvertex))$ is contained in $\Linkhat^{< R}_{\Isigdelrel}(\badvertex)$. \autoref{retraction-on-non-additive-simplices} implies that if $\Delta \in \Linkhat_{\BA^m_{2n+m}}(\badvertex)$ is a standard or $\sigma$ simplex then so is $\retraction(\Delta)$. If $\Delta$ is a (possibly carrying) $2$-additive simplex, then the image $\retraction(\sd(\Delta))$ consists of standard and 2-additive simplices in $\Linkhat^{< R}_{\BA^m_{2n+m}}(\badvertex)$. By definition (see \cref{def:Isigdel-retraction}), every vertex in $\retraction(\sd(\Delta))$ is contained in the \emph{isotropic} summand $\ll \Delta \cup \{\vec \badvertex, \vec e_1, \dots, \vec e_m\} \rr_{\mbZ}$. In particular, the simplices in $\retraction(\sd(\Delta))$ form standard or $2$-additive simplices in $\Linkhat^{<R}_{\Isigdelrel}(\badvertex)$, and cannot be of $\sigma$-type. It follows that restricting the codomain yields a well-defined map
   $$\retraction\colon  \sd(\Linkhat_{\Isigdelrel}(\badvertex)) \to \Linkhat^{<R}_{\Isigdelrel}(\badvertex).$$
   This concludes the proof of the proposition, because this map is by definition the identity on $\Linkhat^{< R}_{\Isigdelrel}(\badvertex)$ (see \cref{def:Isigdel-retraction}). \qedhere
\end{proof}

\subsection{Reducing the rank in \texorpdfstring{$\IArel$}{IA}}

In this subsection, we explain how one can extend \cref{Isigdel-retraction} to $\IArel$. We continue to use the notation introduced in \cref{def_rank_and_ranked_complexes} and frequently apply the notions defined in \cref{subsec:simplex-types-in-relative-complexes} to refer to simplices in $\Linkhat_{\IArel}(\badvertex)$. The goal of this subsection is to prove the following.

\begin{proposition}\label{IA-retraction} 
  Let $n \geq 1$, $m \geq 0$ and let $R > 0$. Let $\badvertex$ be a vertex of $\IArel$ such that $\rkfn(\badvertex) = R$. Then there exists a subdivision $\sd(\Linkhat_{\IArel}(\badvertex))$ of $\Linkhat_{\IArel}(\badvertex)$ containing $\Linkhat_{\IArel}^{< R}(\badvertex)$ as a subcomplex and a simplicial retraction
  $$\retraction \colon  \sd(\Linkhat_{\IArel}(\badvertex)) \twoheadrightarrow \Linkhat^{< R}_{\IArel}(\badvertex)$$
  extending the retraction introduced in \cref{Isigdel-retraction}.
\end{proposition}

Recall that $\IArel$ is obtained from $\Isigdelrel$ by attaching mixed simplices. For the proof of \cref{IA-retraction}, we hence only need to verify that the retraction for $\Isigdelrel$ (see \cref{Isigdel-retraction}) extends over mixed simplices.

\begin{construction}
	\label{def:IA-subdivision}
	The following is based on \cite[p.1022]{CP}. We explain how the subdivision $\sd(\Linkhat_{\Isigdelrel}(\badvertex))$, described in \cref{def:Isigdel-subdivision}, can be extended over mixed simplices to obtain the subdivision $\sd(\Linkhat_{\IArel}(\badvertex))$ of $\Linkhat_{\IArel}(\badvertex)$ in \cref{Isigdel-retraction}. For this, let $\Delta$ be mixed simplex in $\Linkhat_{\IArel}(\badvertex)$. Then $\Delta = \Theta \ast \Delta'$, where $\Theta$ is a minimal 2-additive simplex and $\Delta'$ is a $\sigma$ simplex. If $\Theta$ is internal minimal 2-additive and carrying in the sense of \cref{def:Isigdel-subdivision}, it is subdivided into three simplices $\Theta_1$, $\Theta_2$ and $\Theta_3$ in $\sd(\Linkhat_{\Isigdelrel}(\badvertex))$ by placing a new vertex $t(\Theta)$ at its barycentre. Hence, to extend  $\sd(\Linkhat_{\Isigdelrel}(\badvertex))$ to $\sd(\Linkhat_{\IArel}(\badvertex))$, we subdivide such \emph{carrying} mixed simplices $\Delta$ into three simplices $\Theta_1 \ast \Delta'$, $\Theta_2 \ast \Delta'$ and $\Theta_3 \ast \Delta'$.
\end{construction}

\begin{construction}
	\label{def:IA-retraction}
	On the vertices $v \in \Linkhat_{\IArel}(\badvertex)$ and the new vertices $t(\Theta) \in \sd(\Linkhat_{\IArel}(\badvertex))$ (see \cref{def:IA-subdivision}) the retraction $\retraction$ in \cref{IA-retraction} is defined exactly as described in \cref{def:Isigdel-retraction}.
\end{construction}

We are now ready to prove \cref{IA-retraction}.

\begin{proof}[Proof of \cref{IA-retraction}]
  Consider the inclusion $\IArel \hookrightarrow \BA^m_{2n+m}.$ It restricts to an inclusion $\Linkhat_{\IArel}(\badvertex) \hookrightarrow  \Linkhat_{\BA^m_{2n+m}}(\badvertex)$. By \autoref{BA-retraction}, there is a simplicial retraction
  $$\retraction\colon  \sd(\Linkhat_{\BA^m_{2n+m}}(\badvertex)) \twoheadrightarrow \Linkhat^{< R}_{\BA^m_{2n+m}}(\badvertex).$$
  In the proof of \cref{Isigdel-retraction}, we explained why $\retraction$ restricts to a retraction 
  $$\retraction\colon  \sd(\Linkhat_{\Isigdelrel}(\badvertex)) \to \Linkhat^{< R}_{\Isigdelrel}(\badvertex).$$
  Hence, it suffices to focus on mixed simplices $\Delta \in \Linkhat_{\IArel}(\badvertex)$. Note that $\Delta$ can be written as $\Delta = \Theta \ast \Delta'$, where $\Theta$ is a minimal 2-additive simplex and $\Delta'$ is a $\sigma$ simplex. In particular, forgetting the symplectic information yields a 2-additive simplex $\Delta \in \Linkhat_{\BA^m_{2n+m}}(\badvertex)$ that might be subdivided when passing to $\sd(\Linkhat_{\BA^m_{2n+m}}(\badvertex))$ (see \cref{def:BA-subdivision-and-retraction}). We know that $\retraction(\sd(\Delta)) \in \Linkhat^{< R}_{\BA_{2n+m}^m}(\badvertex)$, and need to check that $\retraction(\sd(\Delta)) \in \Linkhat^{< R}_{\Isigdelrel}(\badvertex)$.
  
  To see this, we first note that by \cref{retraction-on-non-additive-simplices} the image $\retraction(\Delta')$ of $\Delta'$ is a $\sigma$ simplex of the same dimension as $\Delta'$.
  
  If $\Theta$ is not carrying, then $\retraction(\Delta) = \retraction(\Theta) \ast \retraction(\Delta')$ is either a standard simplex or a $2$-additive with augmentation core $\retraction(\Theta)$ in $\Linkhat^{< R}_{\BA_{2n+m}^m}(\badvertex)$ (see \cite[p.1021-1023, Claim 1-4]{CP}). Therefore, we conclude that $\retraction(\Delta) = \retraction(\Theta) \ast \retraction(\Delta')$ is either a $\sigma$ simplex or a mixed simplex in $\Linkhat_{\IArel}^{<R}(\badvertex)$.
  
  If $\Theta$ is carrying, then $\Theta$ is subdivided into three simplices $\Theta_{1}$, $\Theta_{2}$, $\Theta_{3}$ when passing to $\sd(\Linkhat_{\BA^m_{2n+m}}(v))$ (exactly as described in \cref{def:Isigdel-subdivision}). These three simplices have the property that $\retraction(\Theta_{i})$ is either internal $2$-additive or $\badvertex$-related $2$-additive (compare with \cite[Claim 5, p.1023-1024]{CP}) for $1 \leq i \leq 3$. Therefore, $\retraction(\Theta_{i} \ast \Delta') = \retraction(\Theta_{i}) \ast \retraction(\Delta')$ is a mixed simplex for $1 \leq i \leq 3$. We conclude that the retraction for $\BA^m_n$ restricts to a simplicial retraction
   \begin{equation*}
   	\sd(\Linkhat_{\IArel}(\badvertex)) \twoheadrightarrow \Linkhat^{< R}_{\IArel}(\badvertex). \qedhere
   \end{equation*}
\end{proof}

\subsection{Reducing the rank in \texorpdfstring{$\IAAstrel$}{IAA*}}

In this subsection, we explain how work of Brück--Miller--Patzt--Sroka--Wilson \cite[Section 3]{Brueck2022} generalises to the symplectic setting, and how it can be used to ``reduce the rank'' in $\IAAstrel$. We continue to use the notation introduced in \cref{def_rank_and_ranked_complexes} and frequently apply the notions defined in \cref{subsec:simplex-types-in-relative-complexes} to refer to simplices in $\Linkhat_{\IAAstrel}(\badvertex)$. The goal of this subsection is to prove the following extension of \cref{IA-retraction}.

\begin{proposition} \label{IAAst-retraction}
  Let $n \geq 1$, $m \geq 0$, and let $R > 0$. Let $\badvertex$ be a vertex of $\IAAstrel$ such that $\rkfn(\badvertex) = R$. Then there exists a subdivision $\sd(\Linkhat_{\IAAstrel}(\badvertex))$ of $\Linkhat_{\IAAstrel}(\badvertex)$ containing $\Linkhat_{\IAAstrel}^{< R}(\badvertex)$ as a subcomplex and a simplicial retraction
  $$\retraction \colon  \sd(\Linkhat_{\IAAstrel}(\badvertex)) \twoheadrightarrow \Linkhat^{< R}_{\IAAstrel}(\badvertex)$$
  extending the retraction introduced in \cref{Isigdel-retraction}.
\end{proposition}

\cref{IAAst-retraction} is a consequence of the following.

\begin{lemma}[{\cite[Proposition 3.39.]{Brueck2022}}] \label{BAA-retraction}
  Let $n \geq 1$, $m \geq 0$ and $R > 0$. Let $\badvertex$ be a vertex of $\BAA_{2n+m}^m$ such that $\rkfn(\badvertex) = R$. Then there exists a subdivision $\sd(\Linkhat_{\BAA_{2n+m}^m}(\badvertex))$ of $\Linkhat_{\BAA_{2n+m}^m}(\badvertex)$ containing $\Linkhat_{\BAA_{2n+m}^m}^{< R}(\badvertex)$ and a simplicial retraction
  $$\retraction \colon  \sd(\Linkhat_{\BAA_{2n+m}^m}(\badvertex)) \twoheadrightarrow \Linkhat_{\BAA_{2n+m}^m}^{< R}(\badvertex)$$
  extending the retraction introduced in \cref{BA-retraction}.
\end{lemma}

\begin{construction}\label{def:IAAst-subdivision}
	The following is based on \cite[Section 3]{Brueck2022}. In \cref{IAAst-retraction}, the subdivision $\sd(\Linkhat_{\IAAstrel}(\badvertex))$ of $\Linkhat_{\IAAstrel}(\badvertex)$ has the property that only certain ``carrying'' 2-additive, mixed, 3-additive, double-double and double-triple simplices are subdivided.  Carrying 2-additive and mixed simplices are defined exactly as in \cref{def:Isigdel-subdivision} and \cref{def:IA-subdivision}. The relevant properties of the other carrying simplices are discussed in \cref{def:IAAst-retraction} below. Using the inclusion $\Linkhat_{\IAAstrel}(\badvertex) \hookrightarrow  \Linkhat_{\BAA_{2n+m}^m}(\badvertex)$, the subdivision $\sd(\Linkhat_{\IAAstrel}(\badvertex))$ is obtained as a restriction of the subdivision $\sd(\Linkhat_{\BAA_{2n+m}^m}(\badvertex))$ in \cref{BAA-retraction} to $\Linkhat_{\IAAstrel}(\badvertex)$. The subdivision $\sd(\Linkhat_{\IArel}(\badvertex))$ of $\Linkhat_{\IArel}(\badvertex)$ occurring in \cref{IA-retraction} is a subcomplex of $\sd(\Linkhat_{\IAAstrel}(\badvertex))$ (compare with \cite[Proposition 3.39.]{Brueck2022}).
\end{construction}

\begin{construction}
	\label{def:IAAst-retraction}
	The following is based on \cite[Section 3]{Brueck2022}. The restriction of the map $\retraction$ in \cref{IAAst-retraction} to $\sd(\Linkhat_{\IArel}(\badvertex))$ yields the retraction in \cref{IA-retraction} (compare \cite[Proposition 3.39.]{Brueck2022}). Furthermore, if $\Delta$ is a 2-additive, 3-additive, double-double or double-triple simplex, then $\Delta = \Theta \ast \Delta'$, where $\Theta$ is the augmentation core of $\Delta$ and $\Delta'$ is a standard simplex. If $\Theta$ is ``carrying'', then the subdivision $\sd(\Delta)$ of $\Delta$ is of the form $\sd(\Delta) = \sd(\Theta)\ast \Delta'$ for some subdivision of $\sd(\Theta)$ of $\Theta$. In this case, it holds that $\retraction(\sd(\Delta)) = \retraction(\sd(\Theta)) \ast \retraction(\Delta')$ and that every vertex $\retraction(v)$ contained in $\retraction(\sd(\Theta))$ satisfies $\retraction(v) \subset \ll \Theta \cup \{\vec \badvertex, \vec e_1, \dots, \vec e_m\} \rr.$
\end{construction}

Combining \cref{BAA-retraction} with \autoref{retraction-on-non-additive-simplices}, we can prove \cref{IAAst-retraction}.

\begin{proof}[Proof of \autoref{IAAst-retraction}]
  The inclusion $\IAAstrel \hookrightarrow \BAA_{2n+m}^m$ restricts to an inclusion $$\Linkhat_{\IAAstrel}(\badvertex) \hookrightarrow  \Linkhat_{\BAA_{2n+m}^m}(\badvertex).$$ By \autoref{BAA-retraction}, there is a simplicial retraction
  $$\retraction\colon  \sd(\Linkhat_{\BAA_{2n+m}^m}(\badvertex)) \twoheadrightarrow \Linkhat^{<R}_{\BAA_{2n+m}^m}(\badvertex).$$
  Let $\sd(\Linkhat_{\IAAstrel}(\badvertex))$ be the restriction of the subdivision $\sd(\Linkhat_{\BAA_{2n+m}^m}(\badvertex))$ of $\Linkhat_{\BAA_{2n+m}^m}(\badvertex)$ to the subcomplex $\Linkhat_{\IAAstrel}(\badvertex)$. Restricting the domain of $\retraction$ to $\sd(\Linkhat_{\IAAstrel}(\badvertex))$ yields a simplicial map
  $$\retraction\colon  \sd(\Linkhat_{\IAAstrel}(\badvertex)) \to \Linkhat^{<R}_{\BAA_{2n+m}^m}(\badvertex).$$
  We claim that the image $\retraction(\Delta)$ of each simplex $\Delta$ in $\sd(\Linkhat_{\IAAstrel}(\badvertex))$ is contained in $\Linkhat^{<R}_{\IAAstrel}(\badvertex)$. By \autoref{IA-retraction} and because the retraction in \cref{BAA-retraction} is an extension of the one in \cref{BA-retraction} (see \cref{def:IAAst-subdivision} and \cref{def:IAAst-retraction}), this holds for all simplices except possibly 3-additive, double-double and double-triple simplices. If $\Delta$ is a (possibly ``carrying'') simplex of this type in $\Linkhat_{\IAAstrel}(\badvertex)$, then the image $\retraction(\sd(\Delta))$ consists of standard, $2$-additive, $3$-additive, double-double and double-triple simplices in $\Linkhat^{<R}_{\BAA_{2n+m}^m}(\badvertex)$. By definition (see \cref{def:IAAst-retraction}), every vertex in $\retraction(\sd(\Delta))$ is contained in the \emph{isotropic} summand $\ll \Delta \cup \{\vec \badvertex, \vec e_1, \dots, \vec e_m\} \rr$ of $\mbZ^{2(m+n)}$. In particular, the simplices contained in $\retraction(\sd(\Delta))$ form standard, $2$-additive, $3$-additive, double-double or double-triple simplices in $\Linkhat^{<R}_{\IAAstrel}(\badvertex)$. It follows that restricting the codomain yields a well-defined map
  $$\retraction\colon  \sd(\Linkhat_{\IAAstrel}(\badvertex)) \to \Linkhat^{<R}_{\IAAstrel}(\badvertex).$$
  This completes the proof.
\end{proof}

\subsection{Reducing the rank in \texorpdfstring{$\IAArel$}{IAA}}

In this final subsection, we explain how one can ``reduce the rank'' in $\IAArel$ by proving \cref{IAA-retraction}. In contrast to the previous subsections, it is important that we are able to reduce the rank of maps $\phi\colon  M \to \Linkhat_{\IAArel}(\badvertex)$ in such a way that weak regularity properties are preserved. This is used later, in the proof that the complex $\IAArel$ is highly connected, to construct regular maps in $\IAArel$; see the proof of \cref{prop_cut_out_bad_vertices}. Throughout this subsection, we make the standing assumption that
\begin{equation*}
	n\geq 2,\, m\geq 0 \text{ and } R > 0.
	\tag{Standing assumption}
\end{equation*}
In \cref{def_regular_maps}, we explained what weak regularity means for maps with codomain $\IAArel$. The following definition clarifies what we mean by a weakly regular map with codomain $\Linkhat_{\IAArel}(\badvertex)$.

\begin{definition} \label{def:weak-regularity-in-link}
	Let $M$ be a combinatorial manifold, $\badvertex \in \IAArel$ be a vertex and consider a simplicial map $\phi\colon  M \to \Linkhat_{\IAArel}(\badvertex)$. Then $\phi$ is called \emph{weakly regular} if for some (hence any) symplectic isomorphism $\mbZ^{2(m+n)} \to \mbZ^{2(m+n)}$ fixing $e_1, \dots, e_m$ and sending $\badvertex$ to $e_{m+1}$, the induced map
	$$\phi\colon  M \to \Linkhat_{\IAArel}(\badvertex) \cong \IAArel[n-1][m+1]$$
	is weakly regular in the sense of \cref{def_regular_maps}.
\end{definition}

Equivalently, one can adapt the definition of cross maps and regularity in \cref{sec:regular-maps} to maps $\phi\colon  M \to \Linkhat_{\IAArel}(\badvertex)$ as follows:
\begin{enumerate}
	\item In the definition of cross maps (see \cref{def:cross-maps}), replace $\IAArel$ by $\Linkhat_{\IAArel}(\badvertex)$, and allow $\vertexvar \in \{e_1, \dots, e_m, \badvertex\}$ in the part concerning \emph{external 2-skew-additive} cross maps.
	\item In the definition of regular maps (see \cref{def_regular_maps}), replace $\IAArel$ by $\Linkhat_{\IAArel}(\badvertex)$, and allow $\vertexvar \in \{e_1, \dots, e_m, \badvertex\}$  in \cref{wprism-condition-b} of the part concerning \emph{weakly prism-regular}.
\end{enumerate}

\begin{remark}
	\label{rem:problematic-simplices}
	In contrast to \cref{I-retraction}, \cref{Isigdel-retraction} and \cref{IA-retraction}, the authors have not been able to extend the a retraction in \cref{IAAst-retraction} over a subdivision of $\Linkhat_{\IAArel}(\badvertex)$. In particular, we did not find a suitable definition on the following simplices: Let $\Delta = \{v_0, v_1, v_2, v_3\}$ be a minimal internal 2-skew-additive simplex in $\Linkhat_{\IAArel}(\badvertex)$. The 2-additive face $\{v_0, v_1, v_2\}$ of $\Delta$ contains a unique vertex $\hat v \in \{v_0, v_1, v_2\}$ such that $\{\hat v, v_3\}$ is not a $\sigma$ edge. We say that a 2-skew-additive simplex with augmentation core $\Delta$ is \emph{inessential} if $\rkfn(\hat v)$ is the unique maximum of $\rkfn(-)$ on $\{v_0, v_1, v_2\}$.
\end{remark}

The proof of \cref{IAA-retraction} consists of three steps: In Step 1, we replace $\phi$ by a ``nice'' weakly regular map $\phi'$ with $\phi|_{S^{k-1}}=\phi'|_{S^{k-1}}$ . This is to avoid inessential simplices (see \cref{rem:problematic-simplices}), i.e.\ the image of $\phi'$ does not contain any such simplices. In Step 2, we explain how one can extend the retraction for $\IAAst$ (see \cref{IAAst-retraction}) over the image of cross maps that occur in such ``nice'' weakly regular maps $\phi'$. In Step 3, this  and Zeeman's relative simplicial approximation theorem (\cref{zeeman-relative-simplicial-approximation}) are used to construct $\psi$ from $\phi'$.

\subsubsection{Step 1: Essential prisms in weakly regular maps}

In this subsection, we explain how one can replace a weakly regular map $\phi$ into $\Linkhat_{\IAArel}(\badvertex)$ by a ``nice'' weakly regular map $\phi'$, whose image contains no inessential simplices (see \cref{rem:problematic-simplices}). The weak regularity of $\phi$ (see \cref{def:weak-regularity-in-link}) implies that every inessential 2-skew-additive simplex in image of $\phi$ has to be contained in the image of a prism cross map in $\phi$. This leads us to the next definition.

\begin{definition} \label{def:niceprism} Let $\badvertex \in \IAArel$.
  \begin{enumerate}
  \item An internal 2-skew-additive simplex  in $\Linkhat_{\IAArel}(\badvertex)$ with augmentation core given by $\Delta = \{v_0, v_1, v_2, v_3\}$ is called \emph{essential} if there exists a vertex $m^v$ of the 2-additive face $\{v_0, v_1, v_2\}$ of $\Delta$ such that $\{m^v, v_3\}$ is a $\sigma$ edge in $\Delta$ and $\rkfn(m^v)$ is the maximum value $\rkfn(-)$ takes on $\{v_0, v_1, v_2\}$. Otherwise, the simplex is \emph{inessential} (compare with \cref{rem:problematic-simplices}).
  \item Let $\phi\colon  D^k \to \Linkhat_{\IAArel}(\badvertex)$ be weakly regular. We call a pair $(P, \phi|_{P})$, where $P \subset D^k$ is a subcomplex, a \emph{prism in $\phi$} if $\phi|_{\Star_{D^k}(P)}$ is a prism cross map (compare with \cref{def_regular_maps}).
  \item Let $\phi\colon  D^k \to \Linkhat_{\IAArel}(v)$ be a weakly regular map and let $(P, \phi|_P)$ be a prism in $\phi$. We say that $(P, \phi|_P)$ is \emph{essential} if both internal 2-skew-additive simplices in $\phi(P)$ are essential (see \cref{figure_max_path} for an equivalent, pictorial description of this condition). Otherwise, $(P, \phi|_P)$ is called \emph{inessential}.
  \end{enumerate}
\end{definition}

\begin{figure}
\begin{center}
\includegraphics{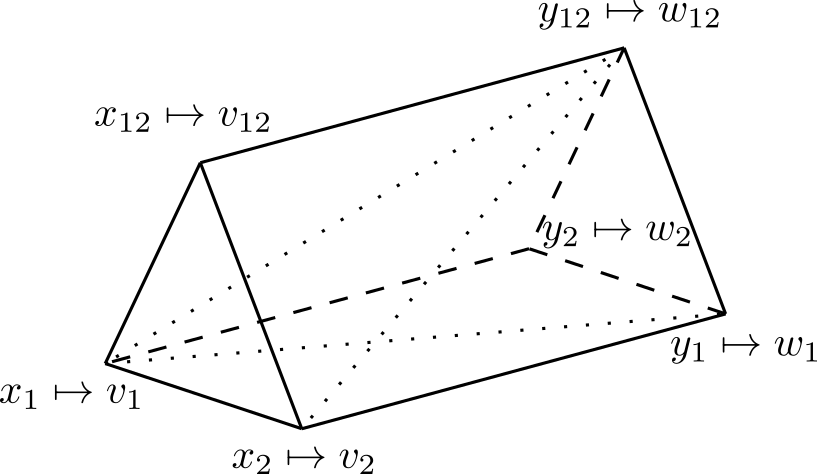}
\end{center}
\caption{Example of a prism $(P,\phi)$ in a weakly regular map, where the dotted lines in $P$ are mapped to $\sigma$ edges. $(P,\phi)$ is essential if the maximum of $\rkfn(-)$ on $\{v_1,v_2,v_{12}\}$ and $\{w_1,w_2,w_{12}\}$ is attained on $\{v_1, v_2\}$ and $\{w_1, w_{12}\}$, respectively. Note that $\{v_1, v_2, w_1, w_{12}\}$ are exactly the vertices of the $\sigma$ edge path in $\phi(P)$.}
\label{figure_max_path}
\end{figure}

A weakly regular map $\phi\colon  D^k \to \Linkhat_{\IAArel}(\badvertex)$ is ``nice'' if every prism in $\phi$ is essential. The following lemma allows us to assume this property.

\begin{lemma} \label{lem:nice-regular-replacements}
  Let $\badvertex \in \IAArel$ be a vertex with $\rkfn(\badvertex) = R > 0$. Given a commutative diagram of simplicial maps
  $$
  \begin{tikzcd}
    S^{k-1} \arrow[r, "\phi|"] \arrow[d, hook] & \Linkhat_{\IAArel}^{<R}(\badvertex) \arrow[d, hook]\\
    D^k \arrow[r, "\phi"] & \Linkhat_{\IAArel}(\badvertex)
  \end{tikzcd}
  $$
  such that $\phi$ is (weakly) regular (see \cref{def_regular_maps}). Then there is a (weakly) regular map $\phi'\colon  D^k \to \Linkhat_{\IAArel}(\badvertex)$ such that $\phi'|_{S^{k-1}} = \phi|_{S^{k-1}}$ and such that every prism in $\phi'$ is essential.
\end{lemma}

\begin{proof}
  Let $(P_3, \phi|_{P_3})$ be a prism in $\phi$ that is inessential. We show how to remove $(P_3, \phi|_{P_3})$ from $\phi$ without introducing a new inessential prism.
  Since $\phi$ is (weakly) regular, $\phi|_{\Star_{D^k}(P_3)}$ is a prism cross map and, using the notation in \cref{def:cross-maps}, we may write $\Star_{D^k}(P_3) = P_3 \ast \Link_{D^k}(P_3) = P_3 \ast C_{k-3}$. \cref{cor_boundary_star} and the fact that $\partial C_{k-3} = \emptyset$ imply that $\Star_{D^k}(P_3)$ is a $k$-ball with boundary given by 
\begin{equation*}
	\partial \Star_{D^k}(P_3) = \partial P_3 \ast C_{k-3}.
\end{equation*} 
As in \cref{def_prism}, we denote the vertices of $P_3$ by  $ x_1, x_2, x_{12}, y_1, y_2, y_{12} $ and their images by $v_1, v_2, v_{12}, w_1, w_2, w_{12}$, where $P_3$ is union of the three 3-simplices 
\begin{equation}
\label{eq_simplices_P_3}
\ls x_1, x_2, x_{12}, y_{12} \rs , \ls x_1, x_2, y_1, y_{12} \rs, \ls x_1, y_1, y_2, y_{12} \rs 
\tag{$P_3$}
\end{equation}
(see \cref{figure_max_path}).
There are three $\sigma$ edges in $\phi(P_3)$ and they form a path connecting the vertices $\ls w_1, v_1,w_{12}, v_2\rs = \ls v_1,v_2, w_1, w_{12} \rs$.

Let $m^v\in \{v_1, v_2, v_{12}\} = \Theta^v$ and $m^w \in \{w_1, w_2, w_{12}\} = \Theta^w$ be two vertices, whose rank $\rkfn(m^v)$ and $\rkfn(m^w)$ is maximal among $\Theta^v$ and $\Theta^w$ respectively. If both $m^v$ and $m^w$ can be chosen to be in $\ls v_1,v_2, w_1, w_{12} \rs$, then $P_3$ is essential. Hence, we can assume that at least one of them is not contained in this set. Without loss of generality\footnote{In order to assume this, we might need to relabel the vertices of $P_3$ using the isomorphism $x_1 \leftrightarrow y_{12}$, $x_2 \leftrightarrow y_1$, $x_{12}\leftrightarrow y_1$.}, assume that this is the case for $m^w$. Then, it holds that
$m^w = w_2$.
We need to consider two cases: Either
\begin{equation*}
	m^v \in \ls v_1, v_2 \rs \text{ or } m^v = v_{12}.
\end{equation*}
We first discuss the case $m^v \in \ls v_1, v_2 \rs$ and comment on the other case at the end of the proof.

Let $D$ denote the simplicial complex with the same vertices as $P_3$ that is the union of the \emph{four} 3-simplices 
\begin{equation}
\label{eq_simplices_P_3_Delta}
	\{x_1, x_2, x_{12}, y_{12}\}, \{x_1, x_2, y_2, y_{12}\}, \{x_2, y_1, y_2, y_{12}\} \text{ and } \{x_1, x_2, y_1, y_2\}
	\tag{$D$}
\end{equation}
(see \cref{figure_splitting_sigma_edge_off_a_prism}).
\begin{figure}
\begin{center}
\includegraphics{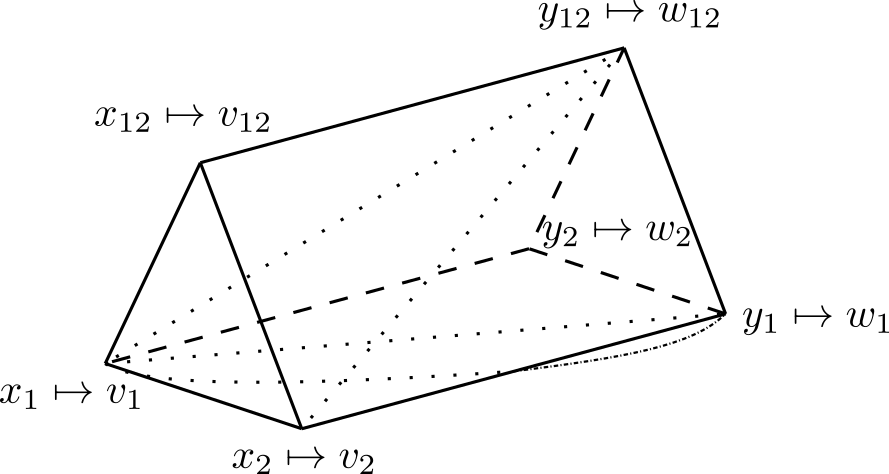}
\end{center}
\caption{The simplices in $D$. The upper prism is $P_3'$. The edge $\ls x_1,y_1 \rs$ is drawn curved at the bottom. It is the only edge of $D$ that is not contained in $P_3'$.}
\label{figure_splitting_sigma_edge_off_a_prism}
\end{figure}

Note that the subcomplex $P_3' \subset D$ that is the union of the first three simplices
\begin{equation}
\label{eq_simplices_P_3prime}
\{x_1, x_2, x_{12}, y_{12}\}, \{x_1, x_2, y_2, y_{12}\}, \{x_2, y_1, y_2, y_{12}\} 
\tag{$P_3'$}
\end{equation}
is again a prism.
The complex $D$ is a 3-ball and we have $\partial D = \partial P_3$.
Consequently,
\begin{equation*}
	C \coloneqq D \ast C_{k-3}
\end{equation*}
is a $k$-ball that has the same vertices as $\Star_{D^k}(P_3)$ and the same boundary 
\begin{equation*}
	\partial C = \partial (D \ast C_{k-3}) = \partial P_3 \ast C_{k-3} = \partial \Star_{D^k}(P_3).
\end{equation*}

As $C$ has the same vertices as $\Star_{D^k}(P_3)$, we can define a map
\begin{equation*}
	\psi\colon  C \to \Linkhat_{\IAArel}(\badvertex)
\end{equation*}
by setting it to be equal to $\phi|_{\Star_{D^k}(P_3)}$ on these vertices.
To see that this actually defines a simplicial map, one checks that $\psi$ sends the four maximal simplices of $D$ (see \cref{eq_simplices_P_3_Delta}) to a 2-skew-additive simplex, a skew-$\sigma^2$ simplex, a 2-skew-additive simplex and a $\sigma^2$-simplex, respectively. The first three of these simplices form the image of the new prism $\psi(P_3')$. It is not hard to see that $\Star_{C}(P_3') = P_3' \ast C_{k-3}$ and that the restriction of $\psi$ to this is a prism cross map. Similarly, the restriction of $\psi$  to $\{x_1, x_2, y_1, y_2\} \ast C_{k-3}$ is a $\sigma^2$ cross map. It follows that $\psi$ is regular.
There is only one prism in $\psi$, namely $(P_3',\psi_{P_3'})$.

Recall that in the first case $m^v \in \ls v_1, v_2 \rs$. It follows that, under this assumption, the prism $(P_3',\psi_{P_3'})$ is essential because the vertices $m^v$ and $m^w = w_2$ are contained in the $\sigma$ edge path in $\psi(P_3')$ (see \cref{figure_splitting_sigma_edge_off_a_prism}).

As $\psi$ agrees with $\phi$ on $\partial \Star_{D^k}(P_3) = \partial C$, we can alter $\phi$ by replacing the $k$-ball $\Star_{D^k}(P_3)$ in $D^k$ with the $k$-ball $C$ and replacing $\phi|_{\Star_{D^k}(P_3)}$ with $\psi$. The result is a new map $D^k\to \Linkhat_{\IAArel}(\badvertex)$. It has one less inessential prism than $\phi$ because the replacement removed $(P_3, \phi_{P_3})$ and only introduced the essential prism $(P_3',\psi_{P_3'})$. Furthermore, the new map is still (weakly) regular because $\psi$ is regular and domains of cross maps can only intersect in their boundaries (\cref{lem_cross_maps_intersections}). It agrees with $\phi$ on $\partial
D^k = S^{k-1}$ because $S^{k-1}$ can intersect $C$ only in its boundary $\partial C = \partial \Star_{D^k}(P_3)$, where $\psi$ agrees with $\phi$.

It remains to discuss the second case $m^v = v_{12}$. Observe that, under this assumption, the two vertices $m^v$ and $m^w = w_2$ are not contained in the $\sigma$ edge path in image of the prism $(P_3',\psi_{P_3'})$ constructed in the first case. To resolve this, we apply the same procedure as above again, now to $(P_3',\psi_{P_3'})$ instead of $(P_3,\phi_{P_3})$: We define $D'$ as the union of the four 3-simplices
\begin{equation}
\label{eq_simplices_Dprime}
	\{x_1, x_2, x_{12}, y_2\}, \{x_2, x_{12}, y_2, y_{12}\}, \{x_2, y_1, y_2, y_{12}\} \text{ and } \{x_1, x_{12}, y_2, y_{12}\}.
	\tag{$D'$}
\end{equation}
The first three of these simplices form a new prism $P_3''$ and $\partial D' = P_3'$. Just as before, we obtain a regular map $\psi'\colon D' \ast C_{k-3} \to \Linkhat_{\IAArel}(\badvertex)$. Note that the prism $(P_3'',\psi'|{P_3''})$ is essential because the vertices $m^v = v_{12}$ and $m^w = w_{2}$ are contained in the $\sigma$ edge path in $\psi'(P_3'')$. The rest of the argument works as above. For an overview of the replacement process, see \cref{figure_P3_and_the_primes}.

\begin{figure}
\begin{center}
\includegraphics[scale=0.75]{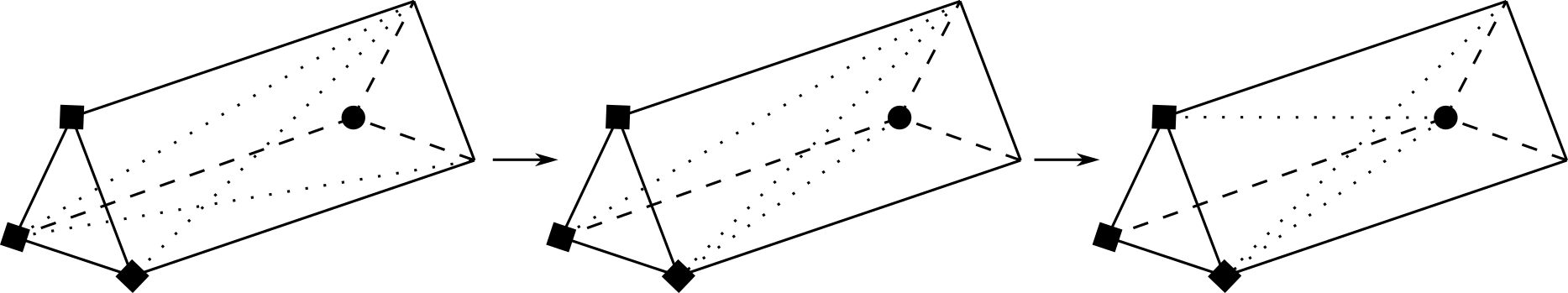}
\end{center}
\caption{The prisms $P_3$, $P_3'$ and $P_3''$ occurring in the proof of \cref{lem:nice-regular-replacements}. The round dot marks the vertex that is mapped to $m^w = w_2$, the squared dots mark the three vertices that can get mapped to $m^v$.}
\label{figure_P3_and_the_primes}
\end{figure}

 We explained how to replace $\phi$ with a map that has one less inessential prism. Iterating this procedure and successively removing all inessential prisms, we obtain the desired map $\phi'$.
\end{proof}

\subsubsection{Step 2: Reducing the rank of cross maps} 

In order to prove \cref{IAA-retraction}, we need to explain how one can reduce the rank of weakly regular maps $\phi\colon  D^k \to \Linkhat_{\IAArel}(\badvertex)$ in which every prism is essential (compare \cref{lem:nice-regular-replacements}). To do this, we show that the retraction for $\IAAstrel$ (see \cref{IAAst-retraction}) extends over all, except inessential, simplices and record the effect of this retraction on cross maps, i.e.\ the building blocks of weak regularity (see \cref{def_regular_maps}). We continue to use the notation introduced in \cref{def_rank_and_ranked_complexes} and frequently apply the notions defined in \cref{subsec:simplex-types-in-relative-complexes} to refer to simplices in $\Linkhat_{\IAArel}(\badvertex)$.

We start by explaining why the retraction for $\IAAstrel$ (see \cref{IAAst-retraction}) extends over all simplices of non-additive type (see \cref{def:non-additive-simplex}). 

\begin{lemma}\label{lem:IAAst-retraction-extends-to-sigma2-skew-additive-skew-sigma2}
  Let $\badvertex \in \IAArel$ be a vertex with $\rkfn(\badvertex) = R > 0$. The composition
  $$\retraction\colon  \sd(\Linkhat_{\IAAstrel}(\badvertex)) \to \Linkhat^{< R}_{\IAAstrel}(\badvertex) \hookrightarrow \Linkhat^{< R}_{\IAArel}(\badvertex) $$
  of the retraction defined in \cref{IAAst-retraction} and the inclusion map extends simplicially over all $\sigma^2$, skew-additive and skew-$\sigma^2$  simplices $\Delta \in \Linkhat_{\IAArel}(\badvertex)$.
\end{lemma}

\begin{proof}
  Let $\tau \in \{\sigma^2, \text{skew-additive}, \text{skew-}\sigma^2\}$ in the following. Let $\Delta$ be a simplex of type $\tau$ in $\Linkhat_{\IAArel}(\badvertex)$. Then $\Delta = \Theta \ast \Delta'$ for a minimal $\Theta$ of type $\tau$ and a standard simplex $\Delta'$. Recall from \cref{def:IAAst-subdivision} that the only simplices that were subdivided to obtain $\sd(\Linkhat_{\IAAstrel}(\badvertex))$ from $\Linkhat_{\IAAstrel}(\badvertex)$ are 2-additive, mixed, 3-additive, double-double and double-triple simplices.
  \begin{itemize}
  \item Let $\tau \in \{\sigma^2, \text{skew-additive}\}$. It follows that $\partial \Theta \ast \Delta'$ is a subcomplex of the subdivision $\sd(\Linkhat_{\IAAstrel}(\badvertex))$. The retraction on $\partial \Theta \ast \Delta'$ is hence defined as described in \cref{retraction-on-non-additive-simplices}. Note that $\Delta$ is a non-additive simplex in the sense of \cref{def:non-additive-simplex}. Hence, \cref{retraction-on-non-additive-simplices} implies that $\retraction$ extends over $\Delta$ and that $\retraction(\Delta) = \retraction(\Theta) \ast \retraction(\Delta') \in \Linkhat^{< R}_{\IAArel}(\badvertex)$ is again a simplex of type $\tau$ of the same dimension as $\Delta$.
  \item Let $\tau = \text{skew-}\sigma^2$. Then $\partial \Theta \ast \Delta'$ consists of standard, $\sigma$ and skew-additive simplices. The retraction is defined on these simplices by the previous item. Now, $\Delta$ is again a non-additive simplex in the sense of \cref{def:non-additive-simplex}. Hence, \cref{retraction-on-non-additive-simplices} implies that $\retraction$ extends over $\Delta$ and that $\retraction(\Delta) = \retraction(\Theta) \ast \retraction(\Delta') \in \Linkhat^{< R}_{\IAArel}(\badvertex)$ is a simplex of type skew-$\sigma^2$ of the same dimension as $\Delta$. \qedhere
  \end{itemize}
\end{proof}

	The next lemma shows that the retraction extends over external 2-skew-additive simplices.

\begin{lemma}\label{lem:IAAst-retraction-extends-to-external-2-skew-additive}
	Let $\badvertex \in \IAArel$ be a vertex with $\rkfn(\badvertex) = R > 0$. The composition
	$$\retraction\colon  \sd(\Linkhat_{\IAAstrel}(\badvertex)) \to \Linkhat^{< R}_{\IAAstrel}(\badvertex) \hookrightarrow \Linkhat^{< R}_{\IAArel}(\badvertex) $$
	of the retraction defined in \cref{IAAst-retraction} and the inclusion map extends simplicially over all external 2-skew-additive $\Delta = \Theta \ast \Delta' \in \Linkhat_{\IAArel}(\badvertex)$, where $\Theta$ is a minimal 2-skew-additive simplex of the form $\Theta = \{v_0 , \langle \vec v_0 \pm \vec \vertexvar \rangle, v_1\}$ for $\vertexvar \in \{e_1, \dots, e_m, \badvertex\}$, $\omega(\vec v_0, \vec v_1) = \pm 1$, and $\Delta'$ is a standard simplex.
\end{lemma}

\begin{proof}
	Note that $\partial \Theta \ast \Delta'$ consists of external 2-additive and $\sigma$ simplices. Since neither of these simplices can be carrying in $\Linkhat_{\IAAstrel}(\badvertex)$ (see \cref{def:IAAst-subdivision}), it follows that $\partial \Theta \ast \Delta'$ is contained in $\sd(\Linkhat_{\IAAstrel}(\badvertex))$. Forgetting symplectic information $\Delta = \Theta \ast \Delta' \in \Linkhat_{\BAA_{2n+m}}(\badvertex)$ is an external 2-additive, hence non-carrying, simplex. It follows from \cref{BA-retraction} (compare \cite[p.1021, Claim 1, Claim 2 and Claim 3]{CP}), that $\rho(\Delta) = \rho(\Theta) \ast \rho(\Delta') \in \Linkhat_{\BAA_{2n+m}}^{<R}(\badvertex)$ is a standard or (external) 2-additive simplex with augmentation core $\rho(\Theta)$. We need to check that $\rho(\Delta) = \rho(\Theta) \ast \rho(\Delta')$ also forms a simplex in $\Linkhat^{< R}_{\IAArel}(\badvertex)$.
	
	Since $\Delta'$ is a standard simplex in $\Linkhat_{\IAAstrel}(\badvertex)$, \cref{retraction-on-non-additive-simplices} implies that $\rho(\Delta')$ is a standard simplex of the same dimension in $\Linkhat^{< R}_{\IAArel}(\badvertex)$. We can hence focus on $\rho(\Theta)$.
	
	Firstly, assume that $\vertexvar = \badvertex$. Then either $\retraction(\langle \vec v_0 \pm \vec \badvertex \rangle) = \retraction(v_0)$, or $\retraction(v_0) = v_0$ and $\retraction(\langle \vec v_0 \pm \vec \badvertex \rangle) = \langle \overrightarrow{\retraction(v_0)} \pm \vec \badvertex \rangle$ (see \cite[p.1021-1022, Proof of Claim 3]{CP}). In the first case $\retraction(\Theta) = \{\retraction(v_0), \retraction(v_1)\}$ is a $\sigma$-edge contained in $\Linkhat_{\IAAstrel}^{<R}(\badvertex)$. In the second case, $\retraction(\Theta) = \{\retraction(v_0) = v_0, \retraction(\langle \vec v_0 \pm \vec \badvertex \rangle) = \langle \vec v_0 \pm \vec \badvertex \rangle,  \retraction(v_1)\}$ is a 2-skew-additive of dimension 2 in $ \Linkhat^{< R}_{\IAArel}(\badvertex)$. We conclude that $\rho(\Delta) = \rho(\Theta) \ast \rho(\Delta')$ is either a $\sigma$ simplex or a $\badvertex$-related 2-skew-additive simplex in $ \Linkhat^{< R}_{\IAArel}(\badvertex)$.
	
	Secondly, we observe that if $\vertexvar = e_i$ for some $1 \leq i \leq m$, then $\retraction(\langle \vec v_0 \pm \vec e_i \rangle) = \langle \overrightarrow{\retraction(v_0)} \pm e_i \rangle$. Therefore, $\retraction(\Theta) = \{\retraction(v_0), \retraction(\langle \vec v_0 \pm \vec e_i \rangle) = \langle \overrightarrow{\retraction(v_0)} \pm \vec e_i \rangle,  \retraction(v_1)\}$ and $\rho(\Delta)=\rho(\Theta) \ast \rho(\Delta')$ are external 2-skew-additive simplices in $ \Linkhat^{< R}_{\IAArel}(\badvertex)$.
\end{proof}

Our next goal is to extend the retraction over $\sigma$-additive simplices. To do this, we need to subdivide all ``carrying'' $\sigma$-additive simplices as explained in the following remark.

\begin{construction}
	\label{def:sigma-additive-subdivision-and-retraction}
	Let $\badvertex \in \IAArel$ with $\rkfn(\badvertex) = R > 0$ and $\Delta \in \Linkhat_{\IAArel}(\badvertex)$ $\sigma$-additive simplex. Forgetting symplectic information yields an \emph{internal} 2-additive simplex $\Delta \in \Linkhat_{\BA^m_{2n+m}}(\badvertex)$ which might be carrying in the sense of \cref{def:Isigdel-subdivision}. Define $\sd(\Delta)$ and the retraction $\retraction$ on $\sd(\Delta)$ exactly as for such internal 2-additive simplices. I.e.\ if $\Delta$ is \emph{carrying}, $\sd(\Delta)$ is the union of three simplices (see \cref{def:Isigdel-subdivision}) on which the retraction is defined as in \cref{def:Isigdel-retraction}. Otherwise, $\sd(\Delta) = \Delta$ and the value of $\retraction(\Delta)$ is determined by the values on the vertex set as in \cref{IAAst-retraction}.
\end{construction}

\begin{lemma} \label{lem:IAAst-retraction-extends-to-sigma-additive}
  Let $\badvertex \in \IAArel$ with $\rkfn(\badvertex) = R > 0$. The composition
  $$\retraction\colon  \sd(\Linkhat_{\IAAstrel}(\badvertex)) \to \Linkhat^{< R}_{\IAAstrel}(\badvertex) \hookrightarrow \Linkhat^{< R}_{\IAArel}(\badvertex) $$
of the retraction defined in \cref{IAAst-retraction} and the inclusion map extends simplicially over the subdivision $\sd(\Delta)$ of all $\sigma$-additive simplices $\Delta \in \Linkhat_{\IAArel}(\badvertex)$. Here, $\sd(\Delta)$ and the value of $\rho$ on $\sd(\Delta)$ are defined as in \cref{def:sigma-additive-subdivision-and-retraction}.
\end{lemma}

\begin{proof}
  Let $\Delta$ be a $\sigma$-additive simplex in $\Linkhat_{\IAArel}(\badvertex)$. Then $\Delta = \Theta \ast \Delta'$ for a minimal $\sigma$-additive simplex $\Theta$ and a standard simplex $\Delta'$. Note that $\partial \Theta \ast \Delta'$ consists of $\sigma$ and standard simplices. Hence it follows that $\partial \Theta \ast \Delta'$ is contained in the subdivision $\sd(\Linkhat_{\IAAstrel}(\badvertex))$, even in the subcomplex $\sd(\Linkhat_{\Isigdelrel}(\badvertex))$. The retraction for $\IAAstrel$ is an extension of the retraction on $\Isigdelrel$. 
  We recall that the retraction on $\sd(\Linkhat_{\Isigdelrel}(\badvertex))$ was defined using the inclusion 
  $$\Isigdelrel \hookrightarrow \BA^m_{2n+m}.$$
 Forgetting symplectic information, $\Delta$ corresponds to an \emph{internal} 2-additive simplex in $\BA^m_{2n+m}$. We  check that the retraction for $\BA^m_{2n+m}$ (see \cref{BA-retraction}) can be used to extend the retraction for $\IAAstrel$ over $\sigma$-additive simplices: Let $\Theta = \{v_0,v_1,v_2\}$.
  
  First assume that $\Theta$ is not a carrying internal 2-additive simplex in $\Linkhat_{\BA^m_{2n+m}}(\badvertex)$ (in the sense of \cref{def:Isigdel-subdivision}). Then $\retraction(\Delta) = \retraction(\Theta) \ast \retraction(\Delta')$ is an internal 2-additive simplex in $\Linkhat_{\BA^m_{2n+m}}^{< R}(\badvertex)$ with additive core $\retraction(\Theta)$ (see \cite[p.1022, Claim 4]{CP}). Arguing as in \cref{retraction-on-non-additive-simplices} it follows that $\retraction(\Theta)$ forms a $\sigma$-additive simplex in $\Linkhat_{\IAArel}^{< R}(\badvertex)$, because $\Theta$ is a minimal $\sigma$-additive simplex in $\Linkhat_{\IAArel}(\badvertex)$ by assumption, $\retraction(v_i) = \ll \bar v_{i} - a_i \bar \badvertex \rr$ for some $a_i$ and $\omega(\bar \badvertex, \bar v_i) = 0$ for all $i$.
  
  Now assume that $\Theta$ is a carrying internal 2-additive simplex in $\Linkhat_{\BA^m_{2n+m}}(\badvertex)$ (in the sense of \cref{def:Isigdel-subdivision}). Following Church--Putman's construction, we subdivide $\Delta = \Theta \ast \Delta'$ into three simplices $\Theta_i \ast \Delta$ with 
  $$\Theta_i = \{t(\Theta)\} \cup \{v_j \mid 0 \leq j \leq 2 \text{ and } j \neq i \}$$
(compare with \cref{def:Isigdel-subdivision}) and define the retraction on $t(\Theta)$ exactly as in \cref{def:Isigdel-retraction}. Without loss of generality, let us assume that $v_0 = \langle \bar v_1 + \bar v_2\rangle$ is the unique vertex in $\Theta$ that maximises $\rkfn(-)$ and that $\retraction(t(\Theta)) = \langle \overbar{\retraction(v_0)} - \bar \badvertex \rangle$. It follows from \cite[p.1023-1024, Claim 5]{CP} that in $\BA^m_{2n+m}$, these three simplices are mapped to one internal 2-additive simplex $\retraction(\Theta_1) \ast \retraction(\Delta')$ and two $\badvertex$-related 2-additive simplices $\retraction(\Theta_i) \ast \retraction(\Delta')$ for $i = 0$ and $i = 2$. The image of the $\Theta_i$ has the following form:
    \begin{itemize}
    \item $\retraction(\Theta_1) = \{\retraction(v_0), \retraction(t(\Theta)), \retraction(v_2)\} = \{ \retraction(v_0), \ll \overrightarrow{\retraction(v_0)} - \overrightarrow{\retraction(v_2)} \rr, \retraction(v_2)\}$\\
      (compare with $\pi(\beta)$ on \cite[p.1024, Claim 5]{CP}).
    \item $\retraction(\Theta_0) = \{\retraction(t(\Theta)), \retraction(v_1), \retraction(v_2)\} = \{\ll \overrightarrow{\retraction(v_1)} - \vec {\badvertex} \rr, \retraction(v_1),  \retraction(v_2)\}$\\
      (compare with $\pi(\alpha)$ on \cite[p.1024, Claim 5]{CP}).
    \item $\retraction(\Theta_2) = \{\retraction(v_0), \retraction(v_1), \retraction(t(\Theta))\} = \{\retraction(v_0), \retraction(v_1), \ll \overrightarrow{\retraction(v_1)} - \vec {\badvertex} \rr\}$\\
      (compare with $\pi(\gamma)$ on \cite[p.1024, Claim 5]{CP}).
    \end{itemize}
Recall that $\Theta = \{v_0,v_1,v_2\}$ forms a minimal $\sigma$-additive simplex in the complex $\Linkhat_{\IAArel}(\badvertex)$. Hence, $\omega(\vec v_i, \vec v_j) = \pm 1$ for any $0 \leq i \neq j \leq 2$. Taking this symplectic information into account, we conclude that $\retraction(\Theta_1 \ast \Delta) = \retraction(\Theta_1) \ast \retraction(\Delta')$ is $\sigma$-additive, and that $\retraction(\Theta_i \ast \Delta) = \retraction(\Theta_i) \ast \retraction(\Delta')$ for $i = 0$ and $i = 2$ are external 2-skew-additive in $\Linkhat_{\IAArel}^{< R}(\badvertex)$. Hence, $\retraction$ extends over $\sd(\Delta)$. \qedhere
\end{proof}

The three previous lemmas and their proofs have the following consequence.

\begin{corollary}\label{lem:retraction-sigma2-sigma-additive-external-2-skew-additive-cross-map} 
  Let $\badvertex \in \IAArel$ be a vertex with $\rkfn(\badvertex) = R > 0$ and let $\retraction$ denote the extension of the retraction of $\IAAstrel$ over $\sigma^2$, skew-additive, skew-$\sigma^2$, $\sigma$-additive and external 2-skew-additive simplices constructed in \cref{lem:IAAst-retraction-extends-to-sigma2-skew-additive-skew-sigma2}, \cref{lem:IAAst-retraction-extends-to-external-2-skew-additive}, \cref{def:sigma-additive-subdivision-and-retraction} and \cref{lem:IAAst-retraction-extends-to-sigma-additive}. Assume that $\phi$ is a $\dots$
  \begin{enumerate}
  \item $\dots$ $\sigma^2$ cross map $\phi\colon  \Delta^1 \ast \Delta^1 \ast C_{k-2} \to \Linkhat_{\IAArel}(\badvertex)$;
  \item $\dots$ external 2-skew-additive cross map $\phi\colon  \Delta^2  \ast C_{k-1} \to  \Linkhat_{\IAArel}(\badvertex)$; or
  \item $\dots$ $\sigma$-additive cross map $\phi\colon  \Delta^2  \ast C_{k-1} \to \Linkhat_{\IAArel}(\badvertex)$.
  \end{enumerate} 
  Then 
  \begin{enumerate}
  \item $\psi = \retraction \circ \phi\colon  \Delta^1 \ast \Delta^1 \ast C_{k-2} \to \Linkhat_{\IAArel}^{< R}(\badvertex)$ is a $\sigma^2$ cross map;
  \item $\psi = \retraction \circ \phi\colon  \Delta^2  \ast C_{k-1} \to \Linkhat_{\IAArel}^{< R}(\badvertex)$ has image in $\Linkhat_{\IAAstrel}^{<R}(\badvertex)$ or is an external 2-skew-additive cross map; or
  \item $\psi = \retraction \circ \phi\colon  \sd(\Delta^2)  \ast C_{k-1} \to \Linkhat_{\IAArel}^{< R}(\badvertex)$ is a weakly regular map.\\
   Here, $\sd(\Delta^2)$ is the subdivision of the 2-simplex $\Delta^2$ into three simplices obtained by placing the new vertex $t(\Delta^2)$ at the barycentre of $\Delta^2$ if $\phi(\Delta^2)$ is carrying and $\sd(\Delta^2) = \Delta^2$ otherwise (compare with \cref{def:sigma-additive-subdivision-and-retraction}).
  \end{enumerate}
\end{corollary}

\begin{remark}
	Note that a $\sigma$-additive cross map $\phi$ might only yield a \emph{weakly regular} map $\psi = \retraction \circ \phi$. The reason for this is that if the minimal $\sigma$-additive simplex in the image of $\phi$ is carrying, then the image of $\psi = \retraction \circ \phi$ contains two external 2-skew-additive simplices. This has been made explicit at the end of the proof of \cref{lem:IAAst-retraction-extends-to-sigma-additive}.
\end{remark}

The next and final lemma shows that the retraction extends over the image of any prism cross map whose prism is essential (see \cref{def:niceprism}). In light of \cref{lem:nice-regular-replacements}, this suffices for our purposes. We use the following notation, which is illustrated in \cref{figure_P3_diamond}.

\begin{figure}
\begin{center}
\includegraphics{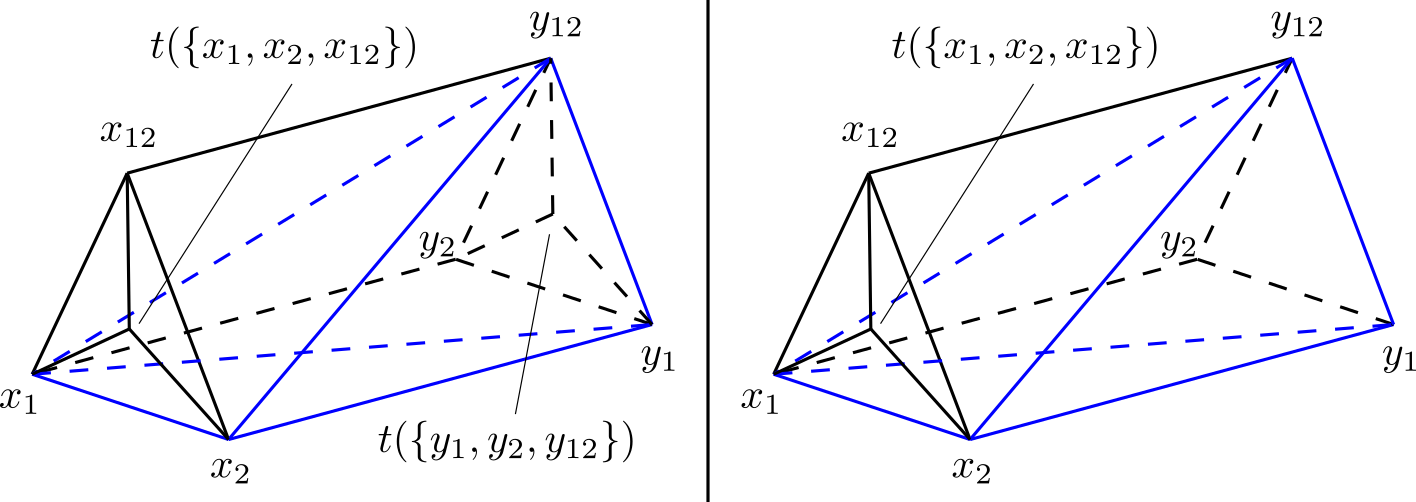}
\end{center}
\caption{The subdivision $\sd(P_3^\diamond)$, on the left for the case where both $\phi(\{x_1, x_2, x_{12}\})$ and $\phi(\{y_1, y_2, y_{12}\})$ are carrying, on the right for the case where only $\phi(\{x_1, x_2, x_{12}\})$ is carrying. Blue lines mark the edges of the only 3-dimensional simplex in $\sd(P_3^\diamond)$.}
\label{figure_P3_diamond}
\end{figure}
\begin{definition} \label{def:prism-diamond}
	Let $P_3$ be the prism on the vertex set $\ls x_1, x_2, x_{12}, y_1, y_2, y_{12} \rs$ introduced in \cref{def_prism}. Recall that $P_3$ is the union of the three 3-simplices $\ls x_1, x_2, x_{12}, y_{12} \rs $, $\ls x_1, x_2, y_1, y_{12} \rs $ and $\ls x_1, y_1, y_2, y_{12} \rs$ (see \cref{figure_prism}).
	\begin{enumerate}
		\item We let $P_3^\diamond$ be the subcomplex of $P_3$ obtained by removing the two simplices $\ls x_1, x_2, x_{12}, y_{12} \rs $ and $\ls x_1, y_1, y_2, y_{12} \rs$ that are mapped to 2-skew-additive simplices by a prism cross map. All other simplices, including all proper faces of $\ls x_1, x_2, x_{12}, y_{12} \rs $ and $\ls x_1, y_1, y_2, y_{12} \rs$, are contained in $P_3^\diamond$.
	\end{enumerate}
	Let $\badvertex \in \IAArel$ and assume that $\phi\colon  P_3 \ast C_{k-2} \to \Linkhat_{\IAArel}(\badvertex)$ is a prism cross map (see \cref{def:cross-maps}). Then $\phi(\{x_1, x_2, x_{12}\})$ and $\phi(\{y_1, y_2, y_{12}\})$ are internal 2-additive simplices in the complex $\Linkhat_{\IAArel}(\badvertex)$ that might be carrying or not.
	\begin{enumerate}
		\setcounter{enumi}{1}
		\item We denote by $\sd(P_3^\diamond)$ the subdivision of $P_3^\diamond$ obtained by placing new vertices $t(\ls x_1, x_2, x_{12} \rs)$ and $t(\{y_1, y_2, y_{12}\})$ at the barycentre of $\ls x_1, x_2, x_{12} \rs$ and $\{y_1, y_2, y_{12}\}$, respectively, depending on whether $\phi(\{x_1, x_2, x_{12}\})$,\\ $\phi(\{y_1, y_2, y_{12}\})$, neither or both are carrying.
	\end{enumerate}
\end{definition}

\begin{lemma} \label{lem:retraction-nice-prism-cross-map}
  Let $\badvertex \in \IAArel$ be a vertex with $\rkfn(\badvertex) = R > 0$ and let $\retraction$ denote the extension of the retraction for $\IAAstrel$ over $\sigma^2$, skew-additive, skew-$\sigma^2$, $\sigma$-additive and external 2-skew-additive simplices constructed in \cref{lem:IAAst-retraction-extends-to-sigma2-skew-additive-skew-sigma2}, \cref{lem:IAAst-retraction-extends-to-external-2-skew-additive} and \cref{lem:IAAst-retraction-extends-to-sigma-additive}.
  Assume that
  $$\phi\colon  P_3 \ast C_{k-2} \to \Linkhat_{\IAArel}(\badvertex)$$
  is a prism cross map such that $(P_3, \phi|_{P_3})$ is essential (in the sense of \cref{def:niceprism}). Then the simplicial map
  $$\psi|_{P_3^\diamond \ast C_{k-2}} = \retraction \circ \phi|_{\sd(P_3^\diamond) \ast C_{k-2}}\colon  \sd(P_3^\diamond) \ast C_{k-2} \to \Linkhat_{\IAArel}^{< R}(\badvertex)$$
  can be extended to a {weakly regular} map
  $$\psi\colon  \sd(P_3) \ast C_{k-2} \to \Linkhat_{\IAArel}^{< R}(\badvertex),$$
  where $\sd(P_3)$ is a subdivision of $P_3$ containing $\sd(P_3^\diamond)$.
\end{lemma}

\begin{remark}
	\label{rem:extendingoveressentialsimplices}
	The proof of \cref{lem:retraction-nice-prism-cross-map} actually shows that the retraction $\retraction$ extends over all essential internal 2-skew-additive simplices in $\Linkhat_{\IAArel}(\badvertex)$. I.e.\ the extension over such simplices, constructed in the proof of \cref{lem:retraction-nice-prism-cross-map}, does not rely on the ambient prism cross map.
\end{remark}

\begin{proof}
  Throughout this proof we use the notation introduced in \cref{def:prism-diamond} and, for prism cross maps, in \cref{def:cross-maps}. We need to explain how to extend $\psi|_{P_3^\diamond \ast C_{k-2}}$ over the two simplices $\ls x_1, x_2, x_{12}, y_{12} \rs $ and $\ls x_1, y_1, y_2, y_{12} \rs$ of $P_3$, which are not contained in $\sd(P_3^\diamond)$ and whose image under $\phi$ are essential internal 2-skew-additive simplices. We focus on $\Delta = \ls x_1, x_2, x_{12}, y_{12} \rs $ and write $\Theta = \ls x_1, x_2, x_{12} \rs$ for the face with the property that $\phi(\Theta)$ is an internal 2-additive simplex in $\Linkhat_{\IAArel}(\badvertex)$. The extension over $\ls x_1, y_1, y_2, y_{12} \rs$ is completely analogous.

  There are two cases depending on whether the image of $\Theta$,
  $$\phi(\Theta) = \{v_1, v_2, v_{12} = \ll \vec v_1 \pm \vec v_2 \rr\},$$
   is a carrying 2-additive simplex in $\Linkhat_{\IAArel}(\badvertex)$ (see \cref{def:Isigdel-subdivision}) or not.

First assume that $\phi(\Theta)$ is not carrying. Then it follows from \cref{def:Isigdel-subdivision} that the simplex $\phi(\Theta) \in \Linkhat_{\IAAstrel}(\badvertex)$ is not subdivided when passing from $\Linkhat_{\IAAstrel}(\badvertex)$ to $\sd(\Linkhat_{\IAAstrel}(\badvertex))$. In this case \cite[p. 1022, Claim 4]{CP} implies that $(\retraction \circ \phi)(\Theta)$ is an internal 2-additive simplex. Taking the symplectic information into account (see \cref{def:cross-maps}), we conclude that
    $$(\retraction \circ \phi)(\Delta) = \Bigl\{\retraction(v_1), \retraction(v_2), \retraction(v_{12}) = \ll \overrightarrow{\retraction(v_1)} \pm \overrightarrow{\retraction(v_2)} \rr, \retraction(w_{12})\Bigr\}$$
is an internal 2-skew-additive simplex. Therefore, for any simplex $\Delta' \in C_{k-2}$, the image $(\retraction \circ \phi)(\Delta \ast \Delta')$ is a 2-skew-additive simplex and $\psi|_{P_3^\diamond \ast C_{k-2}}$ extends over $\Delta$.\\

\begin{figure}
\begin{center}
\includegraphics{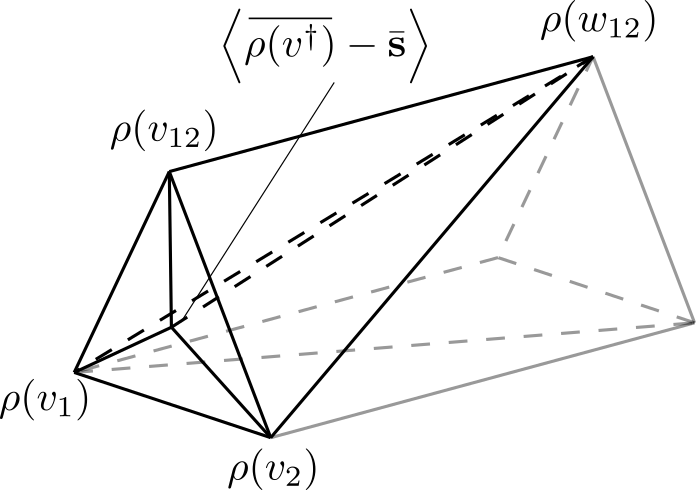}
\end{center}
\caption{The image of $\sd(P_3)$ in the case where $\phi(\Theta)$ is carrying and $v^\dagger = v_{12}$.}
\label{figure_retraction_prism_step2_image}
\end{figure}

Now assume that $\phi(\Theta)$ is carrying. This situation is depicted in \cref{figure_retraction_prism_step2_image}. In this case, it follows that the simplex $\phi(\Theta) \in \Linkhat_{\IAAstrel}(\badvertex)$ is subdivided into three simplices when passing from $\Linkhat_{\IAAstrel}(\badvertex)$ to $\sd(\Linkhat_{\IAAstrel}(\badvertex))$ (see \cref{def:Isigdel-subdivision}).  Recall that the definition of the retraction on the subdivision of a \emph{carrying} simplex depends on which vertex of $\phi(\Theta) = \{v_1, v_2, v_{12}\}$ maximises $\rkfn(-)$ (see \cref{def:Isigdel-retraction}). I.e.\ if $\phi(\Theta)$ is carrying, then there exists a \emph{unique} vertex $m \in \{v_1, v_2, v_{12}\}$ such that 
\begin{equation*}
\rkfn(m) \text{ is maximal in }\{\rkfn(v_1), \rkfn(v_2), \rkfn(v_{12})\}
\end{equation*}
(compare \cite[p. 1022, l.11 et seq.]{CP}).
Since $(P_3, \phi|_{P_3})$ is essential, $\phi(\Delta)$ is an essential internal 2-skew-additive simplex in the sense of \cref{def:niceprism}. We therefore have $m = v_1$ or $m = v_2$ because
\begin{equation}
	\label{eq:essentialconsequence}
	v_{12} \text{ is not contained in a $\sigma$ edge in $\phi(\Delta)$}
	\tag{$E$}
\end{equation}
(compare \cref{figure_max_path}).
We assume that $m = v_1$ (without loss of generality, compare \cref{rem:extendingoveressentialsimplices}).  Then, we have that $\bar v_1 = \bar v_2 + \bar v_{12}$ and $\overbar{\retraction(v_1)} = \overbar{\retraction(v_2)} + \overbar{\retraction(v_{12})} - \bar \badvertex$ (compare \cite[p. 1024, l.6]{CP}). Furthermore, it holds that $\rkfn(\retraction(v_2)), \rkfn(\retraction(v_{12})) > 0$ (compare \cref{def:Isigdel-retraction} and \cite[p. 1023, l.26 et seq.]{CP}). 
Let $v^\dag \in \{v_2, v_{12}\}$ denote the \emph{arbitrary} choice made in the definition of the retraction $\retraction$ (see \cref{def:Isigdel-retraction}). Using \cref{def:Isigdel-retraction}, the image of the new vertex $t(\Theta)$ in $\sd(\Theta)$ under the retraction $\retraction$ is given by $$\retraction(t(\Theta)) = \ll \overbar{\retraction(v^\dag)} - \bar \badvertex \rr.$$
 This implies that the image of the subdivided simplex $\sd(\Theta)$ under the map $\psi|_{\sd(P_3^\diamond)} = \retraction \circ \phi$ is the union of the following three 2-additive simplices:
    \begin{itemize}
      \item $\psi|_{\sd(P_3^\diamond)}(\Theta_1) = \Bigl\{\ll \overbar{\retraction(v^\dag)} - \bar \badvertex \rr, \retraction(v_2), \retraction(v_{12})\Bigr\}$, which is $\badvertex$-related 2-additive;
      \item $\psi|_{\sd(P_3^\diamond)}(\Theta_2) = \Bigl\{\ll \overbar{\retraction(v^\dag)} - \bar \badvertex \rr, \retraction(v_1), \retraction(v_{12}) \Bigr\}$, which is $\badvertex$-related 2-additive if $v^\dag = v_{12}$ and internally 2-additive otherwise;
      \item $\psi|_{\sd(P_3^\diamond)}(\Theta_3) = \Bigl\{\ll \overbar{\retraction(v^\dag)} - \bar \badvertex \rr, \retraction(v_1), \retraction(v_2) \Bigr\}$, which is $\badvertex$-related 2-additive if $v^\dag = v_2$ and internally 2-additive otherwise.
    \end{itemize}
    
    We now reduce the case $v^\dag = v_2$ to the case $v^\dag = v_{12}$, and then explain how to extend $\psi|_{P_3^\diamond \ast C_{k-2}}$.

    \textbf{Reduction to the case $\boldsymbol{v^\dag = v_{12}}$:} Assume that $v^\dag = v_2$. By \cite[Lemma 3.40]{Brueck2022}, we know that
    $$\Omega_2 = \{\retraction(v_2), \retraction(v_{12}), \retraction(v_1)\} \ast \{\ll \overbar{\retraction(v_2)} - \bar \badvertex \rr\}$$
    and
    $$\Omega_{12} = \{\retraction(v_2), \retraction(v_{12}), \retraction(v_1)\} \ast \{\ll \overbar{\retraction(v_{12})} - \bar \badvertex \rr\}$$
    are two double-triple in $\Linkhat_{\IAArel}(\badvertex)$ sharing the 3-additive facet 
    $\{\retraction(v_2), \retraction(v_{12}), \retraction(v_1)\}$
    (see \cref{figure_retraction_prism_step2_reduction_to_v12}).
\begin{figure}
\begin{center}
\includegraphics{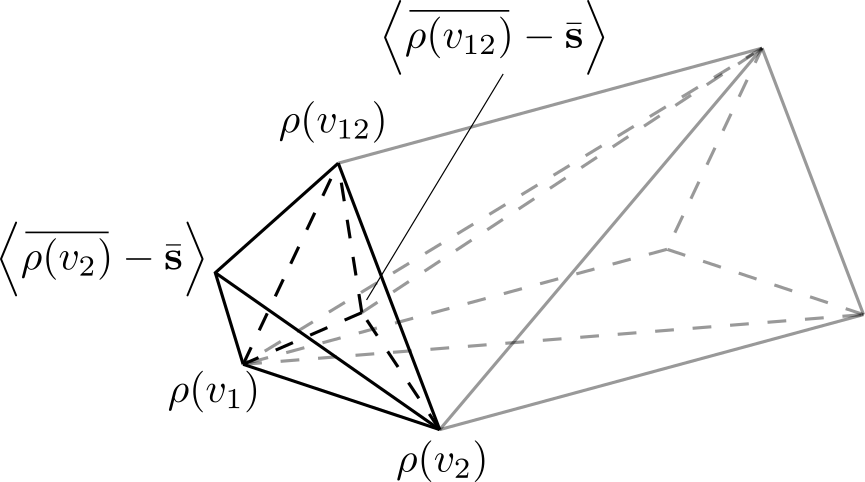}
\end{center}
\caption{Reduction to the case $v^\dag = v_{12}$. Black lines mark the corresponding two double-triple simplices in $\Linkhat_{\IAArel}(\badvertex)$.}
\label{figure_retraction_prism_step2_reduction_to_v12}
\end{figure}
    This implies that $\Omega_2 \ast \Omega$ and $\Omega_{12} \ast \Omega$ are double-triple simplices for every simplex $\Omega \in C_{k-2}$. Therefore, we can extend $\psi|_{P_3^\diamond \ast C_{k-2}}$ over $(\Omega_2  \cup \Omega_{12}) \ast \Omega$ for any simplex $\Omega \in C_{k-2}$. And hence, we may assume that $v^\dag = v_{12}$.

    \textbf{The case $\boldsymbol{v^\dag = v_{12}}$:}
     To see that $\retraction$ extends over $\Delta = \ls x_1, x_2, x_{12}, y_{12} \rs $, we start by subdividing the simplex $\Delta \ast \Omega$ for any $\Omega \in C_{k-2}$ in a way that is compatible with the subdivision $\sd(\Theta)$ of the face $\Theta = \ls x_1, x_2, x_{12}\rs \subset \Delta$ into the simplices $\Theta_1, \Theta_2, \Theta_3$ discussed above. Let $\Omega$ be any simplex in $C_{k-2}$. Then we subdivide $\Delta \ast \Omega$ into the following three simplices (compare \cref{figure_retraction_prism_step2_image}):
    \begin{itemize}
    \item $\Delta_1 = \Theta_1 \ast \{y_{12}\} \ast \Omega$,
    \item $\Delta_2 = \Theta_2 \ast \{y_{12}\} \ast \Omega$ and
    \item $\Delta_3 = \Theta_3 \ast \{y_{12}\} \ast \Omega.$
    \end{itemize}
    Note that $\psi|_{P_3^\diamond \ast C_{k-2}} = \retraction \circ \phi$ is defined on the vertex set of these simplices. We claim that their image in $\Linkhat^{< R}_{\IAArel}(\badvertex)$ always forms a simplex. This is exactly where we use the assumption $v^\dag = v_{12}$, i.e.\ for $v^\dag = v_2$ this claim is wrong. The key point is that $\omega(\vec v_{12}, \vec w_{12}) = 0$ while $\omega(\vec v_2, \vec w_{12}) = \pm 1$ and $\omega(\vec v_1, \vec w_{12}) = \pm 1$ (compare \cref{eq:essentialconsequence} and \cref{figure_max_path}). Computing the image of each simplex, we find:
    \begin{align*}
      \psi(\Delta_1) = \psi(\Theta_1) \ast \psi(\{y_{12}\}) \ast \psi(\Omega) 
      = \Bigl\{\ll \overbar{\retraction(v_{12})} - \bar \badvertex \rr, \retraction(v_2), \retraction(v_{12})\Bigr\} \ast \Bigl\{\retraction(w_{12})\Bigr\} \ast \psi(\Omega),
    \end{align*}
    which is an $\badvertex$-related mixed simplex.
    \begin{align*}
      \psi(\Delta_2) = \psi(\Theta_2) \ast \psi(\{y_{12}\}) \ast \psi(\Omega)
      = \Bigl\{\ll \overline{\retraction(v_{12})} - \bar \badvertex \rr, \retraction(v_1), \retraction(v_{12})\Bigr\} \ast \Bigl\{\retraction(w_{12})\Bigr\} \ast \psi(\Omega),
    \end{align*}
    which is an $\badvertex$-related mixed simplex.
    \begin{align*}
      \psi(\Delta_3) &= \psi(\Theta_3) \ast \psi(\{y_{12}\}) \ast \psi(\Omega)
      = \Bigl\{\ll \overline{\retraction(v_{12})} - \bar \badvertex \rr, \retraction(v_1), \retraction(v_{2})\Bigr\} \ast \Bigl\{\retraction(w_{12})\Bigr\} \ast \psi(\Omega)\\
      &= \Bigl\{\ll \overline{\retraction(v_{1})} - \overline{\retraction(v_{2})} \rr, \retraction(v_1), \retraction(v_{2})\Bigr\} \ast \Bigl\{\retraction(w_{12})\Bigr\} \ast \psi(\Omega),
    \end{align*}
    which is a 2-skew-additive. To check these statements, we used:
    \begin{itemize}
    \item $\omega(\overbar{\retraction(v^\dag)} - \bar \badvertex, \overbar {\retraction(w_{12})}) = 0$ if $v^\dag = v_{12}$. Note that this fails if $v^\dag = v_{2}$.
    \item The identity $\ll \overbar{\retraction(v_{12})} - \bar \badvertex \rr = \ll \overbar{\retraction(v_{1})} - \overbar{\retraction(v_{2})} \rr$, which follows from the fact that $\overbar {\retraction(v_1)} = \overbar{\retraction(v_2)} + \overbar{\retraction(v_{12})} - \bar \badvertex$.
    \end{itemize}
    This completes the proof of the fact that the simplicial map
    $$\psi|_{\sd(P_3^\diamond) \ast C_{k-2}} = \retraction \circ \phi\colon  \sd(P_3^\diamond) \ast C_{k-2} \to \Linkhat_{\IAArel}^{< R}(\badvertex)$$
    can be extended to a simplicial map
    $$\psi\colon  \sd(P_3) \ast C_{k-2} \to \Linkhat_{\IAArel}^{< R}(\badvertex).$$
\begin{figure}
\begin{center}
\includegraphics{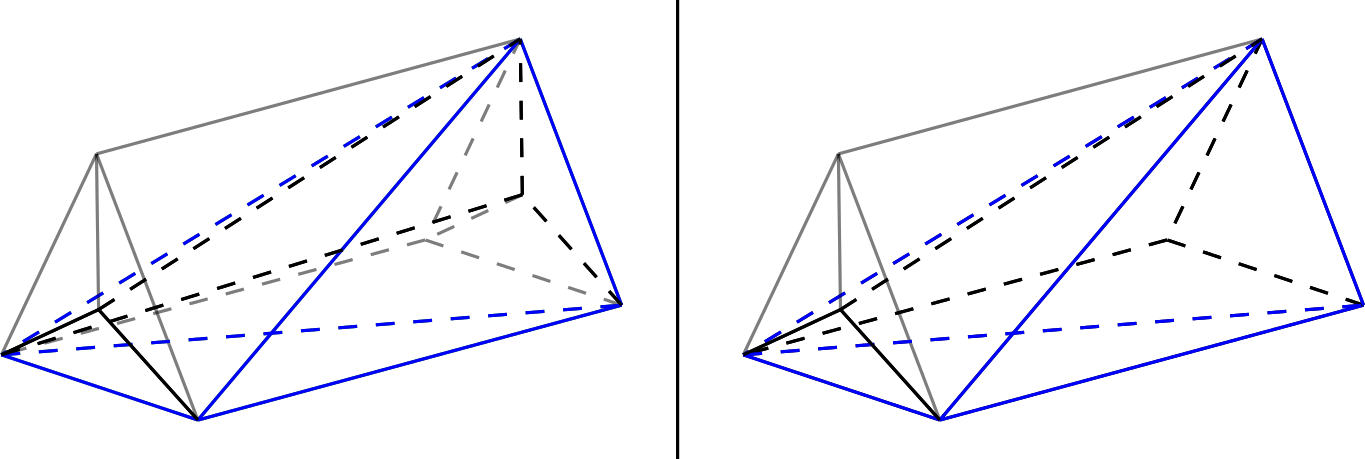}
\end{center}
\caption{The final picture of $\psi(\sd(P_3))$ for the case $v^\dag = v_{12}$. On the left for the case where both $\phi(\{x_1, x_2, x_{12}\})$ and $\phi(\{y_1, y_2, y_{12}\})$ are carrying, on the right for the case where only $\phi(\{x_1, x_2, x_{12}\})$ is carrying. The black and blue lines mark the image of a prism $P$ in $\sd(P_3)$, where the blue lines mark the edges of the skew-$\sigma^2$ simplex that is already contained in $\psi|(\sd(P_3^\diamond))$.}
\label{figure_retraction_prism_step2_reduction_final_situation}
\end{figure}
  It remains to check that the extension $\psi$ is weakly regular. For this, we note that the image of $\psi$ contains exactly two 2-skew-additive simplices and one skew-$\sigma^2$ simplex (which is already contained in the image of $\psi|_{\sd(P_3^\diamond) \ast C_{k-2}}$), see \cref{figure_retraction_prism_step2_reduction_final_situation}. By construction these three simplices are the image of a prism $P$ in $\sd(P_3)$, which has the property that $\Star_{\sd(P_3) \ast C_{k-2}}(P) = P \ast C_{k-2}$ and $\psi|_{P \ast C_{k-2}}$ is a prism cross map. Since the image of $\psi$ does not contain external 2-skew-additive simplices, $\psi$ is prism-regular. Since it also does not contain any $\sigma^2$ or $\sigma$-additive simplices, we conclude that $\psi$ is weakly regular (compare \cref{def_regular_maps}).
\end{proof}

\subsubsection{Step 3: Proof of \texorpdfstring{\cref{IAA-retraction}}{Proposition 6.1}} \label{subsec:proof-of-IAA-retraction}

The proof of \cref{IAA-retraction} relies on Zeeman's relative simplicial approximation theorem (\cref{zeeman-relative-simplicial-approximation}).

\begin{proof}[Proof of \cref{IAA-retraction}]
	We use the notation introduced in the statement of \cref{IAA-retraction} and, invoking Step 1 (i.e.\ \cref{lem:nice-regular-replacements}), may assume that every prism in the weakly regular map
	$$\phi\colon  D^k \to \Linkhat_{\IAArel}(\badvertex)$$
	is essential.
	
	The following construction relies on the fact that the domains of cross maps in $\phi$ can only intersect at their boundary spheres (see \cref{lem_cross_maps_intersections}). \cref{lem_cross_maps_simplex_types} implies that the only cross maps $\Sigma \ast C_{i} \to \Linkhat_{\IAArel}(\badvertex)$ in $\phi$ for which the image of the boundary sphere $\partial (\Sigma \ast C_{i})$ can contain carrying simplices in $\Linkhat_{\IAArel}(\badvertex)$ are prism cross maps. Indeed, let $\phi|_{\Star_{D^k}(P)}\colon  \Star_{D^k}(P) \to \Linkhat_{\IAArel}(\badvertex)$ be a 
	prism cross map in $\phi$. Writing $\Star_{D^k}(P) = P \ast C_{k-3}$, we have $\partial \Star_{D^k}(P) = \partial (P \ast C_{k-3}) = (\partial P) \ast C_{k-3}$, and two simplices in $\partial P$ might get mapped to carrying internal 2-additive simplices by $\phi$.
	
	We start the construction of $\psi$ by removing the interior of the domain of any cross map in $\phi$ from $D^k$. We denote the resulting subcomplex of $D^k$ by $K$. Note that the boundary sphere of the domain of any cross map in $\phi$ is contained in $K$ and that the boundary sphere $\partial D^k = S^{k-1}$ of $D^k$ is also contained in $K$. We denote by $L$ the subcomplex of $K$ obtained as the union of all of these spheres. For each prism cross map in $\phi$, consider the boundary sphere $(\partial P) \ast C_{k-3}$ of its domain. Let $\Theta_x$ and $\Theta_y$ denote the two simplices in $\partial P$ that get mapped to internal 2-additive simplices by $\phi$. Subdivide $\Theta_x$ and $\Theta_y$ by placing a new vertex at their barycentre depending on whether $\phi(\Theta_x)$, $\phi(\Theta_y)$, neither or both are carrying to obtain $\sd(\partial P)$. For each prism cross map in $\phi$, replace $(\partial P) \ast C_{k-3}$ by $\sd(\partial P) \ast C_{k-3}$ in $L$ to obtain $L'$ and let $K'$ be the coarsest subdivision of $K$ containing $L'$.
	
	Next, we use that weak regularity and \cref{lem_cross_maps_simplex_types} imply that the restriction $\phi|_K$ has image in $\Linkhat_{\IAAstrel}(\badvertex)$, i.e.\
	$$\phi|_K\colon  K \to \Linkhat_{\IAAstrel}(\badvertex).$$	
	Since $$|\sd(\Linkhat_{\IAAstrel}(\badvertex))| \cong |\Linkhat_{\IAAstrel}(\badvertex)|,$$ we obtain a continuous map $\phi|_K\colon  |K| \to |\sd(\Linkhat_{\IAAstrel}(\badvertex))|$ that is simplicial except on simplices of $K$ whose image under $\phi$ was carrying in $\Linkhat_{\IAAstrel}(\badvertex)$ (see \cref{def:IAAst-subdivision}). Identifying $K$ with $K'$, we obtain a continuous map
	$$\phi'|_{K'}\colon  |K'| \to |\sd(\Linkhat_{\IAAstrel}(\badvertex))|$$
	that is simplicial on the subcomplex $L'$ of $K'$.	Composing this map with the simplicial retraction $\retraction$ for $\sd(\Linkhat_{\IAAstrel}(\badvertex))$ constructed in \cref{IAAst-retraction} yields a continuous map
	$$\retraction \circ \phi'|_{K'}\colon  |K'| \to |\Linkhat_{\IAAstrel}^{< R}(\badvertex)|$$
	that is simplicial on the subcomplex $L'$ of $K'$. An application of Zeeman's simplicial approximation theorem (\cref{zeeman-relative-simplicial-approximation}) now provides us with a subdivision $K''$ of $K'$ containing $L'$ and a simplicial map $$\psi|_{K''}\colon  K'' \to \Linkhat_{\IAAstrel}^{< R}(\badvertex)$$ that agrees with $\retraction \circ \phi'|_{L'}$ on $L' \subseteq K''$.
	
	Finally, we inspect the domains of cross maps in $\phi$ and use $\psi|_{K''}$ described above, \cref{lem:retraction-sigma2-sigma-additive-external-2-skew-additive-cross-map} and \cref{lem:retraction-nice-prism-cross-map} to construct the weakly regular map 
	$$\psi\colon  D^k_{\new} \to \Linkhat_{\IAArel}^{< R}(\badvertex).$$
	\begin{enumerate}
		\item Consider a $\sigma^2$ cross map $\phi|_{\Star_{D^k}(\Delta)}\colon  \Star_{D^k}(\Delta) \to \Linkhat_{\IAArel}(\badvertex)$ in $\phi$. Then \cref{lem:retraction-sigma2-sigma-additive-external-2-skew-additive-cross-map} implies that $\retraction \circ \phi|_{\Star_{D^k}(\Delta)}$ is a $\sigma^2$ cross map. Since $\partial \Star_{D^k}(\Delta)$ is contained in $L'$ and $\psi|_{K''}$ agrees with $\retraction \circ \phi|_{\Star_{D^k}(\Delta)}$ on $\partial \Star_{D^k}(\Delta)$, it follows that we can glue $\Star_{D^k}(\Delta)$ to $K''$ and extend $\psi|_{K''}$ by $\retraction \circ \phi|_{\Star_{D^k}(\Delta)}$.
		\item Let $\phi|_{\Star_{D^k}(\Delta)}\colon  \Star_{D^k}(\Delta) \to \Linkhat_{\IAArel}(\badvertex)$ be an external 2-skew-additive cross map  in $\phi$. Then \cref{lem:retraction-sigma2-sigma-additive-external-2-skew-additive-cross-map} implies that $\retraction \circ \phi|_{\Star_{D^k}(\Delta)}$ is either contained in $\Linkhat_{\IAAstrel}^{< R}(\badvertex)$ or an external 2-skew-additive cross map. Since $\partial \Star_{D^k}(\Delta)$ is contained in $L'$ by definition and $\psi|_{K''}$ agrees with $\retraction \circ \phi|_{\Star_{D^k}(\Delta)}$ on $\partial \Star_{D^k}(\Delta)$, it follows that we can glue $\Star_{D^k}(\Delta)$ to $K''$ and extend $\psi|_{K''}$ by $\retraction \circ \phi|_{\Star_{D^k}(\Delta)}$.
		\item Consider a $\sigma$-additive cross map $\phi|_{\Star_{D^k}(\Delta)}\colon  \Star_{D^k}(\Delta) \to \Linkhat_{\IAArel}(\badvertex)$ in $\phi$, and write the domain as $\Star_{D^k}(\Delta) = \Delta \ast C_{k-2}$. Note that $\retraction \circ \phi|_{\Star_{D^k}(\Delta)}$ might not be simplicial because $\phi(\Delta)$ could be carrying in the sense of \cref{def:sigma-additive-subdivision-and-retraction}. However, \cref{lem:retraction-sigma2-sigma-additive-external-2-skew-additive-cross-map} implies that $$\retraction \circ \phi|_{\sd(\Delta) \ast C_{k-2}}\colon  \sd(\Delta) \ast C_{k-2} \to \Linkhat_{\IAArel}^{<R}(\badvertex)$$ is a weakly regular. Noting that $$\partial \Star_{D^k}(\Delta) = (\partial \Delta) \ast C_{k-2} = (\partial \sd(\Delta)) \ast C_{k-2} = \partial (\sd(\Delta) \ast C_{k-2}),$$
		that $\partial \Star_{D^k}(\Delta)$ is contained in $L'$ by definition and that $\psi|_{K''}$ agrees with $\retraction \circ \phi|_{\sd(\Delta) \ast C_{k-2}}$ on $\partial \Star_{D^k}(\Delta)$, we can glue $\sd(\Delta) \ast C_{k-2}$ to $K''$ and extend $\psi|_{K''}$ by $\retraction \circ \phi|_{\sd(\Delta) \ast C_{k-2}}$.
		\item Consider a 
		prism cross map $\phi|_{\Star_{D^k}(P)}\colon  \Star_{D^k}(P) \to \Linkhat_{\IAArel}(\badvertex)$ in $\phi$, and write the domain as $\Star_{D^k}(P) = P \ast C_{k-3}$. By assumption $(P, \phi|_P)$ is essential, hence we can invoke \cref{lem:retraction-nice-prism-cross-map} to obtain a weakly regular map $$\psi|_{\sd(P) \ast C_{k-3}}\colon  \sd(P) \ast C_{k-3} \to \Linkhat_{\IAArel}^{<R}(\badvertex).$$
		The boundary $\partial \sd(P)$ of $\sd(P)$ (see \cref{def:prism-diamond}) is exactly the complex $\sd(\partial P)$, i.e.\ the subdivision of $\partial P$ that we used to define $L'$. Hence,
		$$\partial (\sd(P) \ast C_{k-3}) = \sd(\partial P) \ast C_{k-3}$$
		is contained in $L'$ by construction and $\psi|_{K''}$ agrees with $\psi|_{\sd(P) \ast C_{k-3}}$ on $\partial (\sd(P) \ast C_{k-3})$. It follows that we can glue $\sd(P) \ast C_{k-3}$ to $K''$ and extend $\psi|_{K''}$ by $\psi|_{\sd(P) \ast C_{k-3}}$.
	\end{enumerate}

	Using items 1.-4.\ above, we can extend $\psi|$ to a simplicial map 
	$$\psi'\colon  D^k_{\new} \to \Linkhat_{\IAArel}^{< R}(\badvertex).$$
	The domain $D^k_{\new}$ of this map is a combinatorial $k$-ball because it is constructed as a subdivision of $D^k$ (see \cref{rem_subdivision_combinatorial_manifolds}). Furthermore, the map is weakly regular and agrees with $\phi$ on $S^{k-1} \subseteq L'$ by construction. This completes the proof.
\end{proof}

\section{Highly connected subcomplexes}
\label{sec_highly_connected_subcomplexes}

In this section we collect and prove auxiliary connectivity results about Tits buildings and subcomplexes of $\IAA$. We use these later to show that $\IAA$ is highly connected, i.e.~to prove \autoref{thm_connectivity_IAA}.

\subsection{The complexes \texorpdfstring{$\B_n^m$, $\BA_n^m$ and $\BAA_n^m$}{B, BA and BAA}}

In \cref{def:B-and-BA}, \cref{def:BAA}, \cref{def_linkhat} and \cref{rem:about-relative-bbabaa-complexes} we introduced the complexes $\B^m_n$, $\BA^m_n$ and $\BAA^m_n$. These were first defined in \cite{CP,Brueck2022} to study the Steinberg module of $\SL{n}{\mbQ}$: A complex closely related to $\B_n^m$ was used by Church--Farb--Putman \cite{cfp2019} to  construct a generating set, which was first described by Ash--Rudolph in \cite{ar1979}. The complex $\BA_n^m$ was used by Church--Putman \cite{CP} to construct a presentation, which was first described by Bykovski\u{\i} in \cite{byk2003} and for $n = 2$ by Manin in \cite{manin1972}. The complex $\BAA_n^m$ was used by Brück--Miller--Patzt--Sroka--Wilson \cite{Brueck2022} to understand the two-syzygies, the relations among the relations. The following two theorems summarise the main connectivity theorems contained in \cite{CP, Brueck2022}.

\begin{theorem}[{\cite[Theorem 4.2 and Theorem C']{CP}}]
	\label{thm_connectivity_B_BA} Let $n \geq 1$ and $m \geq 0$.
	\begin{enumerate}
		\item $\B_n^m$ is $(n-2)$-connected and Cohen--Macaulay of dimension $n-1$.
		\item $\BA_n^m$ is $(n-1)$-connected. If $m+n \geq 2$, then $\BA_n^m$ is Cohen--Macaulay of dimension $n$. If $m+n \leq 2$, then $\BA_n^m$ is contractible.
	\end{enumerate}
\end{theorem}

\begin{theorem}[{\cite[Theorem 2.10 and Theorem 2.11]{Brueck2022}}] \label{connectivity_BAA} Let $n \geq 1$ and $m \geq 0$.
	The complex $\BAA_n^m$ is $n$-connected. If $m+n \geq 3$, then $\BAA_n^m$ is Cohen--Macaulay of dimension $n+1$. If $m+n \leq 2$, then $\BAA_n^m$ is contractible.
\end{theorem}

\subsection{The symplectic Tits building}
\label{sec_building}
In the introduction, we defined the rational Tits building of type $\mathtt{C}_n$ as the order complex of the poset of nonzero isotropic subspaces of $\mbQ^{2n}$.
If $V \subseteq \mbQ^{2n}$ is such a subspace, then by \cref{lem_subspace_summand}, $V\cap \mbZ^{2n}$ is an isotropic summand of $\mbZ^{2n}$. On the other hand, given an isotropic summand $V' \subseteq \mbZ^{2n}$, the tensor product $V' \otimes \mbQ$ gives an isotropic subspace of $\mbQ^{2n}$. This yields isomorphisms between the poset of isotropic subspaces of $\mbQ^{2n}$ and the poset of isotropic summands of $\mbZ^{2n}$. We use the latter as the definition of the Tits building for the remainder of this article.

\begin{definition}
\label{def_building_over_Z}
Let $n\geq 1$. We denote the poset of nonzero isotropic summands of $\mbZ^{2n}$ by $T^\omega_n$ and call it the \emph{symplectic Tits building} (of rank $n$).
\end{definition}

We usually do not distinguish between this poset and its order complex.
The following is a version of the Solomon--Tits Theorem.

\begin{theorem}[{\cite{ST}}]
\label{Solomon-Tits}
The symplectic Tits building $T^\omega_n$ is Cohen--Macaulay of dimension $n-1$. 
\end{theorem}

\subsection{\texorpdfstring{$W$}{W}-restricted subcomplexes of \texorpdfstring{$\Irel$, $\Idelrel$ and $\IAAstrel$}{I, I\unichar{"005E}\unichar{"03B4} and IAA*}}
\label{sec_subcomplex_W}

Let $m,n \in \mbN$. In this subsection, we study subcomplexes of $\Irel$, $\Idelrel$ and $\IAAstrel$  whose vertex sets are restricted using the submodule
$$W = W_{m+n} = \ll \vec e_1, \vec f_1 \ldots, \vec e_{m+n-1}, \vec f_{m+n-1}, \vec e_{m+n} \rr \subseteq \mbZ^{2(m+n)}.$$
The results here build on ideas contained in \cite{put2009}, \cite[Chapter 5]{Sroka2021} and \cite{bruecksroka2023}.

\begin{definition}
\label{def_W_complexes}
	Let $X \in \{\I[], \Idel[], \IAAst[]\}$. We define $X_n^m(W)$ to be the full subcomplex of $X^m_n$ on the set of vertices that are contained in $W_{m+n}$.
	
	 Similarly, we define $T^{\omega}_{m+n}(W)$ to be the subposet of the symplectic Tits building $T^{\omega}_{m+n}$ of summands contained in $W_{m+n}$. We let
	 \begin{equation*}
	 	T^{\omega,m}_n(W) = T^{\omega}_{m+n}(W)_{> \ll e_1, \dots, e_m \rr}
\end{equation*}
denote the upper link of the summand $\ll e_1, \dots, e_m \rr$, i.e.~the subposet of all $V\in T^{\omega}_{m+n}(W)$ such that $\ll e_1, \dots, e_m \rr < V$.
\end{definition}

The following proposition concerns the connectivity properties of these $W$-restricted subcomplexes.

\begin{proposition} \label{connectivity_subcomplex_W}
	Let $n \geq 1$ and $m \geq 0$.
	\begin{enumerate}
		
		\item \label{it_irelw_conn} $\Irel(W)$ is Cohen--Macaulay of dimension $n-1$.
		\item \label{it_idelrelw_conn} $\Idelrel(W)$ is $(n-1)$-connected.
		\item \label{it_iaastrelw_conn}$\IAAstrel(W)$ is $n$-connected.
	\end{enumerate}  
\end{proposition}

The connectivity results of \cref{it_irelw_conn} and \cref{it_idelrelw_conn} of \cref{connectivity_subcomplex_W} are essentially \cite[Proposition 6.13 (1) and (3)]{put2009}. Putman informed us that their proof contains some small gaps. These gaps were fixed in \cite[Theorem 3.5, Item 1 and 3]{bruecksroka2023}. The Cohen--Macaulay property of $\Irel(W)$ has been verified in \cite[Lemma 3.14]{bruecksroka2023}. We are hence left with checking \cref{it_iaastrelw_conn} of \cref{connectivity_subcomplex_W}. For this, we apply the same strategy as in \cite[Theorem 3.5, Item 3, and Corollary 3.16]{bruecksroka2023}. We start by defining an intermediate simplicial complex.

\begin{definition}
  The simplicial complex $\int_W\BAA^m_n$ has as vertices the rank-1 summands $v$ of $W$ such that $ \ll \vec e_1, \dots, \vec e_m, \vec v \rr$ is an isotropic rank-$(m+1)$ summand of $\Z^{2(m+n)}$. A collection $\Delta$ of such lines forms a simplex if there exists $V \in T^{\omega,m}(W)$ such that $\Delta$ is a simplex in $\BAA^m(V )$.
\end{definition}

\begin{lemma} \label{lem:intwbaa-connectivity}
  Let $n \geq 1$ and $m \geq 0$. Then $\int_W\BAA^m_n$ is $n$-connected.
\end{lemma}

\begin{proof}
There is a poset map
  \begin{align*}
  s\colon  \Simp(\int_W\BAA^m_n) &\longrightarrow T^{\omega, m}_n(W)\\
   \Delta &\longmapsto \langle \vec e_1, \dots, \vec e_m \rangle + \langle \Delta \rangle,
   \end{align*}
  where $\Simp(\int_W\BAA^m_n)$ denotes the simplex poset of $\int_W\BAA^m_n$.
  The target $T^{\omega, m}_n(W)$ of this poset map is a contractible Cohen--Macaulay poset by \cite[Lemma 3.9]{bruecksroka2023}.
Hence, a result by van der Kallen--Looijenga\footnote{For the $n$ in \cite[Corollary 2.2]{vanderkallenlooijenga2011sphericalcomplexesattachedtosymplecticlattices}, choose what is $(n+1)$ in the notation of the present article; define their map $t$ by $t(V)=\dim(V) - m + 2$} {\cite[Corollary 2.2]{vanderkallenlooijenga2011sphericalcomplexesattachedtosymplecticlattices}} implies that it suffices to check that for every  $V \in T^{\omega, m}_n(W)$, the poset fibre $s_{\leq V}$ is  $(\dim(V) - m)$-connected.
This follows from  \autoref{connectivity_BAA} because $s_{\leq V} \cong \Simp(\BAA^m(V))$.
\end{proof}

The following lemma shows that $\IAAstrel(W)$ is $n$-connected, which finishes the proof of \cref{connectivity_subcomplex_W}. For its proof, we use that $\IAAstrel(W)$ is obtained from $\int_W\BAA^m_n$ by attaching $\sigma$ simplices along highly connected links.

\begin{lemma}
\label{IAAW_highly_connected}
Let $n \geq 1$ and $m \geq 0$. Let $S$ a combinatorial $k$-sphere for $k \leq n$ and consider a simplicial map $\phi\colon  S \to \IAAstrel(W)$. Then $\phi$ is weakly regularly nullhomotopic in $\IAAstrel(W)$.
\end{lemma}

\begin{proof}
  We start with the following observation: If $\psi\colon  D \to \IAAstrel(W)$ is any nullhomotopy of $\phi$, it follows from \cref{it_simplicial_approximation_combinatorial_ball} of \cref{lem_simplicial_approximation_combinatorial} that we may assume that $D$ is a combinatorial $(k+1)$-ball. Because $\IAAstrel(W)$ does not contain any simplices of the type listed in \autoref{def:IAA-simplices}, it follows that any such nullhomotopy $\psi$ is regular and, in particular, weakly regular. Hence, we only need to prove that $\IAAstrel(W)$ is $n$-connected. 
To see this, we apply the standard link argument as explained in \cref{subsec:standard-link-argument}: By \cref{lem:intwbaa-connectivity}, we know that the subcomplex $X_0 = \int_W \BAA^m_n$ of $X_1 = \IAAstrel(W)$ is $n$-connected. Let $B$ be the set of all minimal $\sigma$ simplices $\Delta = \{v,w\}$ in $X_1$. This set is a set of bad simplices in the sense of \cref{def:standard-link-argument-bad-simplices}. Following \cref{def:standard-link-argument-good-link}, we find that $\Link^{\mathrm{good}}_{X_1}(\Delta) = \Link_{\IAAstrel(W)}(\Delta) = \I[]^{\delta,m}(W \cap \ll \Delta \rr^\perp)$ for $\Delta = \{v, w\} \in B$. Using that there is an isomorphism
\begin{equation*}
	(W \cap \ll \Delta \rr^\perp, \omega|_{W \cap \ll \Delta \rr^\perp}) \cong \left(\langle \vec e_1, \vec f_1, \ldots,\vec e_{m+n-2}, \vec f_{m+n-2}, \vec e_{m+n} \rangle , \omega|_{\langle \vec e_1, \ldots, \vec f_{m+n-2}, \vec e_{m+n} \rangle }\right)
\end{equation*}
that preserves $\vec e_1, \ldots, \vec e_m$, it follows that $$\I[]^{\delta,m}(W \cap \ll \Delta \rr^\perp) \cong \I[]^{\delta,m}(\ll e_1, f_1 \ldots,e_{m+n-2}, f_{m+n-2}, e_{m+n} \rr).$$ 
The complex on the right hand side is $(n-2)$-connected by \cref{it_idelrelw_conn} of \cref{connectivity_subcomplex_W}. Hence, \cref{it_link_argument_Y_to_X} of \cref{lem:standard-link-argument} implies that $X_1$ is $(n-2)$-connected.
\end{proof}

\subsection{Rank-\texorpdfstring{$R$}{R} subcomplex of \texorpdfstring{$\Irel$ and $\Idelrel$}{I and I\unichar{"005E}\unichar{"03B4}}}

Let $R \in \mbN \cup \{\infty\}$. In this subsection, we collect results about the full subcomplex $(\Irel)^{\leq R}$ and $(\Idelrel)^{\leq R}$ of $\Irel$ and $\Idelrel$, respectively, on the set of vertices of rank at most $R$ (see \cref{def_rank_and_ranked_complexes}). 

The first lemma is an easy consequence of \cite[Proposition 6.13 (2)]{put2009} using \cref{link_Irel_standard} or \cite[Proposition 1.2]{vanderkallenlooijenga2011sphericalcomplexesattachedtosymplecticlattices}. An alternative proof can be found in \cite[Lemma 3.15]{bruecksroka2023}.

\begin{lemma}
	\label{lem:Irel-cohen-macaulay}
	Let $n,m \geq 0$. Then $\Irel$ is Cohen--Macaulay of dimension $(n-1)$.
\end{lemma}

The next lemma is a refinement of the previous one, and an immediate consequence of the proof of \cite[Proposition 6.13 (2)]{put2009}. For the convenience of the reader, we include a short argument using \cref{lem:Irel-cohen-macaulay} and \cref{I-retraction}.

\begin{lemma}
\label{lem_connectivity_Irel}
Let $n,m \geq 0$. For all $R\in \mbN \cup \ls \infty \rs$, the complex $(\Irel)^{\leq R}$ is $(n-2)$-connected.
\end{lemma}
\begin{proof} 
  We perform a standard link argument as explained in \cref{subsec:standard-link-argument}, using induction on $R$ and the retraction for $\Irel$ constructed in \cref{I-retraction}. For $R = 0$, $(\Irel)^{\leq 0} = \Irel(W)$ and the claim follows from \cref{it_irelw_conn} of \cref{connectivity_subcomplex_W}. For $R = \infty$, $(\Irel)^{\leq \infty} = \Irel$  and the claim follows from \cref{lem:Irel-cohen-macaulay}. Assume for the induction that $X_R = (\Irel)^{\leq R}$ is $(n-2)$-connected. 
  Consider the complex $X_{R+1} = (\Irel)^{\leq R+1}$ containing $X_R$ as a subcomplex. Let $B \subset X_{R+1} - X_{R}$ be the set of simplices $\Delta = \{v_0, \dots, v_k\}$ with the property that $\rkfn(v_i) = R+1$ for all $i$. This is a set of bad simplices in the sense of \cref{def:standard-link-argument-bad-simplices}. Following \cref{def:standard-link-argument-good-link}, we find that $\Link_{X_{R+1}}^{\mathrm{good}}(\Delta) = \Link_{\Irel}^{<R+1}(\Delta)$ for $\Delta \in B$. 
  By \cref{lem:Irel-cohen-macaulay}, we know that $\Link_{\Irel}(\Delta)$ is $(n - \dim(\Delta) - 3)$-connected. Hence, it follows from \cref{I-retraction} that $\Link_{X_{R+1}}^{\mathrm{good}}(\Delta) = \Link_{\Irel}^{<R+1}(\Delta)$ is $(n - \dim(\Delta) - 3)$-connected as well. Therefore, \cref{it_link_argument_Y_to_X} of \cref{lem:standard-link-argument} implies that $X_{R+1} = (\Irel)^{\leq R+1}$ is $(n-2)$-connected.
\end{proof}

The next lemma refines \cref{lem:Irel-cohen-macaulay} further and verifies that $(\Irel)^{\leq R}$ is Cohen--Macaulay of dimension $(n-1)$ for all $R\in \mbN \cup \ls \infty \rs$.

\begin{lemma}
\label{lem_connectivity_Isig_link}
Let $n,m \geq 0$. Let $\Delta$ be a simplex of $\Irel$ and $R\in \mbN$ the maximum of $\rkfn(-)$ among all vertices of $\Delta$. Then for all $R'\in \mbN$ with $0< R'\geq R$, the complex $\Link_{\Irel}^{< R'}(\Delta)$ is $(n - \dim(\Delta) - 3)$-connected.

In particular, for all $R'\in \mbN \cup \ls \infty \rs$, the complex $(\Irel)^{\leq R'}$ is Cohen--Macaulay of dimension $(n-1)$.
\end{lemma}
\begin{proof}
We show inductively that for all $R' \geq R$, the complex 
$Y_{R'} \coloneqq \Link_{\Irel}^{< R'}(\Delta)$
is $(n - \dim(\Delta) - 3)$-connected.

We start by considering two cases: If $R = 0$ and $R'=1$, then $\Delta$ is a simplex of $\Irel(W)$ and \cref{it_irelw_conn} of \cref{connectivity_subcomplex_W} implies that the subcomplex $Y_{R'} = \Link_{\Irel(W)}(\Delta)$ is $(n - \dim(\Delta) - 3)$-connected. If on the other hand $R' = R >0$, the retraction for $\Irel$ (see \cref{I-retraction}) and the fact that $\Link_{\Irel}(\Delta)$ is $(n - \dim(\Delta) - 3)$-connected by \cref{lem:Irel-cohen-macaulay} imply that $Y_{R'}$ is $(n - \dim(\Delta) - 3)$-connected.

Now assume that $R' > R$, $R' > 1$ and that $Y_{R'-1}$ is $(n-\dim(\Delta) -3)$-connected. Let $B \subset Y_{R'} - Y_{R'-1}$ be the set of all simplices $\Theta$ in $Y_{R'}$ such that all vertices $w \in \Theta$ satisfy $\rkfn(w) = R'-1$. This is a set of bad simplices in the sense of \cref{def:standard-link-argument-bad-simplices}. Following \cref{def:standard-link-argument-good-link}, we find that 
\begin{equation*}
	\Link^{\mathrm{good}}_{Y_{R'}}(\Theta) = \Link_{Y_{R'}}^{< R'-1}(\Theta) \cong \Link_{\Irel}^{< R'-1}(\Delta \ast \Theta)
\end{equation*}
for $\Theta \in B$. The join $\Delta \ast \Theta$ is a standard simplex of dimension $(\dim(\Delta)+\dim(\Theta)+1)$, therefore \cref{lem:Irel-cohen-macaulay} implies that $\Link_{\Irel}(\Delta \ast \Theta)$ is $(n - (\dim(\Delta) + \dim(\Theta) + 2) - 2)$-connected. As the maximal rank of a vertex in $\Delta \ast \Theta$ is $R'-1$, the retraction for $\Irel$ (see \cref{I-retraction}) implies that $\Link^{\mathrm{good}}_{Y_{R'}}(\Theta)$ is $((n - \dim(\Delta) - 3) - \dim(\Theta) - 1)$-connected. Hence, \cref{it_link_argument_Y_to_X} of \cref{lem:standard-link-argument} implies that $Y_{R'}$ is $(n- \dim(\Delta) - 3)$-connected.

It remains to check why this implies that for all $R'\in \mbN \cup \ls \infty \rs$, the complex $(\Irel)^{\leq R'}$ is Cohen--Macaulay of dimension $(n-1)$. For $R' = \infty$ and $R' = 0$, we already showed Cohen--Macaulayness in \cref{lem:Irel-cohen-macaulay} and \cref{it_irelw_conn} of \cref{connectivity_subcomplex_W}. So let $0< R' \in \mbN$ and $\Delta$ be a simplex in $(\Irel)^{\leq R'}$. Then the maximal rank among all vertices of $\Delta$ is given by some $R\leq R'$.
Hence by the first part of the claim, we have that $\Link_{\Irel}^{< R' +1}(\Delta) = \Link_{(\Irel)^{\leq R'}}(\Delta)$ is $(n - \dim(\Delta) - 3)$-connected.
\end{proof}

Finally, we need a similar result for $\Idelrel$.

\begin{lemma}
\label{lem_connectivity_Idelrel}
Let $n \geq 1$ and $m \geq 0$. For all $R\in \mbN \cup \ls \infty \rs$, the complex $(\Idelrel)^{\leq R}$ is $(n-2)$-connected.
\end{lemma}
\begin{proof}
We perform a standard link argument as described in \cref{subsec:standard-link-argument}. Let $X_1 = (\Idelrel)^{\leq R}$ and consider the $(n-2)$-connected subcomplex $X_0 = (\Irel)^{\leq R}$ (see \cref{lem_connectivity_Irel}). Let $B$ be the set of all minimal 2-additive simplices in $X_1$. This set is a set of bad simplices in the sense of \cref{def:standard-link-argument-bad-simplices}. Following \cref{def:standard-link-argument-good-link}, we find that $\Link^{\mathrm{good}}_{X_1}(\Delta) = \Link_{\Idelrel}^{\leq R}(\Delta)$ for $\Delta \in B$. Let $\Delta = \{v_0, \dots, v_k\}$ such that $v_0 = \langle \vec v_1 + \vec v_2 \rangle$ or $v_0 = \langle \vec e_i + \vec v_1 \rangle$ for some $1 \leq i \leq m$. Then by \cref{link_Idelrel_2add}, the simplex $\Delta' = \{v_1 \dots, v_k\}$ is standard and $\Link_{\Idelrel}^{\leq R}(\Delta) = \Link_{\Irel}^{\leq R}(\Delta')$. 
    This complex is $(n - \dim(\Delta') - 3) = (n-\dim(\Delta) -2)$-connected
    by \cref{lem_connectivity_Isig_link}. Hence, \cref{it_link_argument_Y_to_X} of \cref{lem:standard-link-argument} implies that $X_1$ is $(n-2)$-connected.
\end{proof}

\subsection{\texorpdfstring{$\IArel$}{IA} and links of certain vertices}

In this subsection, we collect results about $\IArel$ and links of certain vertices. The following is \cite[Corollary 3.16]{bruecksroka2023} and a consequence of \cite[Proposition 6.11]{put2009}.

\begin{lemma}
\label{lem_conn_IA}
Let $n \geq 1$ and $m \geq 0$. Then $\IArel$ is $(n-1)$-connected.
\end{lemma}

Using the retraction for $\IArel$ (see \cref{IA-retraction}), the previous lemma has the following consequence.

\begin{lemma}
\label{lem_con_linkhat_IA}
\label{lem_con_linkhat_less_IA}
Let $n \geq 1$ and $m \geq 0$.
Let $v \in \IArel$ be a vertex of rank $R = \rkfn(v)>0$.
\begin{enumerate}
\item $\Linkhat_{\IArel}(v)$ is $(n-2)$-connected.
\item $\Linkhat_{\IArel}^<(v)$ is $(n-2)$-connected. 
\end{enumerate}
\end{lemma}
\begin{proof}
By \cref{link_IArel_vertex}, there is an isomorphism $\Linkhat_{\IArel}(v)\cong \IArel[n-1][m+1]$. This complex is $(n-2)$-connected by \cref{lem_conn_IA}, which shows the first item.
The second item then follows by applying \cref{IA-retraction}. 
\end{proof}

\subsection{Subcomplexes of \texorpdfstring{$\IAArel$}{IAA}}

In this final subsection, we study the subcomplex $\IAAstrel$ of $\IAArel$ and show that the inclusion $\IAAst \hookrightarrow \IAA$ is a highly connected map.

\begin{lemma}
\label{lem_connectivity_IAAstar}
	Let $n \geq 1$ and $m \geq 0$. Then $\IAAstrel$ is $(n-1)$-connected.
\end{lemma}
\begin{proof}
  We apply the standard link argument explained in \cref{subsec:standard-link-argument} twice.

Firstly, let $X_1$ be the simplicial complex that is obtained from $\IArel$ by attaching all 3-additive simplices. It has $X_0 = \IArel$ as a subcomplex, which $(n-1)$-connected by \cref{lem_conn_IA}. Let $B$ be the set of minimal 3-additive simplices contained in $X_1$. This is a set of bad simplices in the sense of \cref{def:standard-link-argument-bad-simplices}. Following \cref{def:standard-link-argument-good-link}, we find that $\Link_{X_1}^{\mathrm{good}}(\Delta) = \Link_{X_1}(\Delta)$ for $\Delta \in B$. 
It is not hard to check that for every minimal 3-additive simplex $\Delta$ we have $\Link_{X_1}(\Delta) \cong \Irel[n-\dim(\Delta)][m+\dim(\Delta)]$. This complex is $(n-\dim(\Delta)-2)$-connected by \cref{lem:Irel-cohen-macaulay}. Because $X_0 = \IArel$ is $(n-1)$-connected, \cref{it_link_argument_Y_to_X} of \cref{lem:standard-link-argument} implies that $X_1$ is $(n-1)$-connected as well.

Secondly, let $X_2 = \IAAstrel[n][m]$ and consider its subcomplex $X_1$, which is $(n-1)$-connected by the previous argument. Let $B$ be the set of minimal double-triple and double-double simplices contained in $X_2$. This is a set of bad simplices in the sense of \cref{def:standard-link-argument-bad-simplices}. Following \cref{def:standard-link-argument-good-link}, we find that $\Link_{X_2}^{\mathrm{good}}(\Delta) = \Link_{X_2}(\Delta)$ for $\Delta \in B$. 
For every minimal double-triple or double-double simplex $\Delta$, we have $\Link_{X_2}(\Delta) \cong \Irel[n-(\dim(\Delta) - 1)][m+(\dim (\Delta) - 1)]$ (see \cref{link_IAAstrel_doubletriple_doubledouble}). This complex  is $(n-\dim(\Delta)-1)$-connected by \cref{lem:Irel-cohen-macaulay}. Because $X_1$ is $(n-1)$-connected, \cref{it_link_argument_Y_to_X} of \cref{lem:standard-link-argument} implies that $X_2 = \IAAstrel$ is $(n-1)$-connected as well.
\end{proof}

We now start working towards the proof that the inclusion $\IAAst \hookrightarrow \IAA$ is a highly connected map. For this the following observation is key.

\begin{lemma}
\label{lem_unique_skew_additive_face}
Let $\Delta$ be a skew-additive simplex in $\IAA$ of dimension $k$. Then there is a unique 2-skew-additive simplex in $\IAA$ of dimension $(k+1)$ that has $\Delta$ as a face. 
\end{lemma}
\begin{proof}
Let $\Delta$ be skew-additive. We can write $\Delta= \ls v_0, v_1, \ldots, v_k\rs$, where (after choosing appropriate representatives) $\omega(\vec v_0, \vec v_{k})= 1$, $ \omega(\vec v_1, \vec v_{k}) = -1$. 
It is easy to see that $\ls\langle \vec v_{0}+\vec v_{1}\rangle,  v_0, v_1, \ldots, v_k \rs$ is a $(k+1)$-dimensional 2-skew-additive simplex containing $\Delta$.

Now assume that $\Delta' = \Delta \cup \ls v_{k+1} \rs$ is any such 2-skew-additive simplex. We claim that $v_{k+1} = \langle \vec v_{0}+\vec v_{1}\rangle$.

As $\Delta'$ is 2-skew-additive, there are exactly three vertices 
of it that are not isotropic to every other vertex (see \cref{def:IAA-simplices}).
These must be $v_0, v_1, v_k$ using the notation in the first paragraph.
By \cref{def:IAA-simplices}, we also know that one vertex of $\Delta'$ is equal to $\langle \pm \vec v_0 \pm \vec v_1 \rangle$ and isotropic to $v_k$. As the set $\ls \vec v_0, \ldots, \vec v_k \rs$ is linearly independent, this vertex must be $v_{k+1} = \langle \pm \vec v_0 \pm \vec v_1 \rangle$. What is left to show is that indeed $v_{k+1}$ is equal to $\langle \vec v_0 + \vec v_1 \rangle = \langle - \vec v_0 - \vec v_1 \rangle$ and not to $\langle \vec v_0 - \vec v_1 \rangle = \langle -\vec v_0 + \vec v_1 \rangle$. This follows easily from the assumption that $\omega(v_k,v_{k+1}) = 0$.
\end{proof}

\cref{lem_unique_skew_additive_face} has the following two consequences.

\begin{lemma}
\label{lem_deformation_retract_2skew_additive}
Let $n\geq 2$ and $\IAAst \subset X_1\subset \IAA$ be the complex obtained from $\IAAst$ by adding all skew-additive and 2-skew-additive simplices of $\IAA$. Then there is a deformation retraction $X_1 \to \IAAst$.
\end{lemma}
\begin{proof}
	The skew-additive and 2-skew-additive simplices in $X_1$ are the only simplices that contain a skew-additive face. Let $\Delta$ be a skew-additive simplex of maximal dimension $n$ in $X_1$. \cref{lem_unique_skew_additive_face} implies that there is a \emph{unique} $(n+1)$-dimensional simplex $\Delta'$ that contains $\Delta$ as as face. (I.e.~$\Delta$ is a ``free face'' in $X_1$.) ``Pushing all such free faces $\Delta$ through their $\Delta'$'' gives a deformation retraction from $X_1$ to a complex whose 2-skew-additive faces have dimension at most $n-1$ and whose skew-additive faces have dimension at most $n$. Iterating over the dimension of skew-additive simplices yields the desired deformation retraction $X_1 \to \IAAstrel$.
\end{proof}

\begin{lemma}
\label{lem_inclusion_surjective_pik}
The inclusion $\IAAst \hookrightarrow \IAA$ is $n$-connected.
\end{lemma}
\begin{proof}
We show that there is a sequence of complexes 
$\IAAst \subset X_1 \subset X_2 \subset X_3 \subset \IAA$,
such that each inclusion in this sequence is $n$-connected.

Let $X_1$ be obtained from $\IAAst$ by adding all skew-additive and 2-skew-additive simplices of $\IAA$. Then by \cref{lem_deformation_retract_2skew_additive}, there is a deformation retraction $X_1\to \IAAst$, so the inclusion $\IAAst \hookrightarrow X_1$ is $n$-connected.

We now apply the standard link argument explained in \cref{subsec:standard-link-argument} three times.

Firstly, let $X_2$ be the complex that is obtained from $X_1$ by attaching all $\sigma^2$ simplices. Let $B$ be the set of minimal $\sigma^2$ simplices contained in $X_2$. This is a set of bad simplices in the sense of \cref{def:standard-link-argument-bad-simplices}. Following \cref{def:standard-link-argument-good-link}, we find that for all $\Delta\in B$, we have $\Link_{X_2}^{\mathrm{good}}(\Delta) = \Link_{X_2}(\Delta) = \Link_{\IAA}(\Delta)$.
	By \cref{link_IAArel_sigma2}, this complex is isomorphic to $\I[n-\dim(\Delta)+1]$ and hence $(n-\dim(\Delta)-1)$-connected by \cref{lem:Irel-cohen-macaulay}. Hence, $X_1\hookrightarrow X_2$ is $n$-connected by \cref{it_link_argument_connected_map} of \cref{lem:standard-link-argument}.
	
Secondly, let $X_3$ be obtained from $X_2$ by attaching all skew-$\sigma^2$ simplices. Then, a set $B$ of bad simplices in $X_3 \setminus X_2$ is given by all minimal skew-$\sigma^2$ simplices. Using \cref{link_IAArel_skew_sigma2}, we obtain that the map $X_2\hookrightarrow X_3$ is $n$-connected by \cref{it_link_argument_connected_map} of \cref{lem:standard-link-argument}.

Lastly, $\IAA$ is obtained from $X_3$ by attaching all $\sigma$-additive simplices. A set $B$ of bad simplices in $\IAA \setminus X_3$ is given by all minimal $\sigma$-additive simplices. The desired result follows from \cref{link_IAArel_sigma_additive} and \cref{it_link_argument_connected_map} of \cref{lem:standard-link-argument}.
\end{proof}

Note that \cref{lem_connectivity_IAAstar} and \cref{lem_inclusion_surjective_pik} already imply that $\IAA$ is $(n-1)$-connected. In the next section, we improve upon this and show that $\IAA$ is $n$-connected.

\section{Theorem C: A highly connected complex} \label{sec:thm_connectivity_IAA}
We now collected everything that is needed to show that $\IAA$ is $n$-connected. 

By \cref{lem_inclusion_surjective_pik}, we know that the inclusion $\IAAst \hookrightarrow \IAA$ induces a surjection on $\pi_k$ for $k\leq n$. Hence, \autoref{thm_connectivity_IAA} follows if we can show that this map is zero on $\pi_k$.
By \cref{it_simplicial_approximation_combinatorial_sphere} of \cref{lem_simplicial_approximation_combinatorial}, every element of $\pi_k(\IAAst)$ is represented by a map from a combinatorial $k$-sphere to $\IAAst$. Hence, it suffices to show the following result that uses the notions of regularity and weak regularity as defined in \cref{def_regular_maps}:

\begin{theorem}
\label{thm_inclusion_zero_pik_inductive}
Let $n \geq 1$, $m \geq 0$ and $k\leq n$. Let $S$ be a combinatorial $k$-sphere and $\phi\colon  S\to \IAAstrel$ a simplicial map.
\begin{enumerate}
\item If $n = 1$, the map  $\phi$ is weakly regularly nullhomotopic in $\IAArel[n] = \IAArel[1]$.
\item If $n\geq 2$, the map $\phi$ is regularly nullhomotopic in $\IAArel$.
\end{enumerate}
\end{theorem}

We prove \cref{thm_inclusion_zero_pik_inductive} in this section by an inductive procedure. In \cref{sec_induction_beginning}, we cover the case $n=1$. In \cref{sec_induction_step}, we then show that if every map $S^{k-1} \to \IAAstrel[n-1][m+1]$ is weakly regularly nullhomotopic in $\IAArel[n-1][m+1]$, then every map $S^k \to \IAAstrel$ is regularly nullhomotopic in $\IAArel[n][m]$.

\subsection{Induction beginning: \texorpdfstring{$n = 1$}{n = 1}}
\label{sec_induction_beginning}
We start by proving \cref{thm_inclusion_zero_pik_inductive} for the case $n=1$. 

\begin{lemma}
\label{lem_induction_beginning}
Let $m \geq 0$, $k\in \ls 0,1 \rs$, let $S$ be a combinatorial $k$-sphere and $\phi\colon  S \to \IAAstrel[1]$ a simplicial map. Then $\phi$ is weakly regularly nullhomotopic in $\IAArel[1]$.
\end{lemma}

\begin{proof}
We start by describing the complex $\IAAstrel[1] = \Linkhat_{\IAAst[m+1]}(\ls e_1, \ldots, e_m \rs)$, spelling out this case of \cref{def_linkhat}.
Let $V$ denote its set of vertices. Every vertex $v\in V$ is spanned by a vector
\begin{equation}
\label{eq_vertices_induction_beginning}
\vec v = \sum_{i=1}^m a_i \vec e_i +\underbrace{ a_{m+1} \vec e_{m+1} + b_{m+1} \vec f_{m+1}}_{\eqqcolon \vec v'},
\end{equation}
where $a_1, \ldots, a_m\in \mbZ$ and $\vec v' \coloneqq a_{m+1}\vec e_{m+1} + b_{m+1} \vec f_{m+1}$ spans a rank-1 summand of $\langle \vec e_{m+1}, \vec f_{m+1} \rangle \cong \mbZ^2$.
Note that for a pair $v_0, v_1$ of such vertices, we have $\omega(v_0,v_1) = \omega(v'_0,v'_1)$. 
As $\langle \vec e_{m+1}, \vec f_{m+1} \rangle$ defines a genus-1 summand of $(\mbZ^{2(m+1)}, \omega)$, this implies in particular that there are no standard simplices of dimension greater than zero. Following \cref{def:IAAst}, the complex $\IAAstrel[1]$ has dimension two and contains the following simplices of dimension greater than zero:
\begin{enumerate}
\item \label{it_2add_IAA1} For every $v_0\in V$ and $1\leq i \leq m$, there are the 2-additive simplices of the form $\ls v_0,  \langle \vec v_0 \pm \vec e_i\rangle \rs$.
\item \label{it_3add_IAA1} For every $v_0\in V$ and $1\leq i\not=j \leq m$, there are the 3-additive simplices of the form $\ls v_0,  \langle \vec v_0 \pm \vec e_i\pm \vec e_j\rangle \rs$.
\item \label{it_doubletriple_IAA1} For every $v_0\in V$ and $1\leq i\not=j \leq m$, there are the double-triple simplices of the form $\ls v_0, \langle \vec v_0 \pm \vec e_i\rangle, \langle \vec v_0 \pm \vec e_i \pm \vec e_j\rangle \rs$.
\item \label{it_sigma_IAA1}For every pair $v_0, v_1 \in V$ such that $\omega(v_0, v_1)= \pm 1$, there is a $\sigma$ simplex of the form $\ls v_0,  v_1 \rs$.
\end{enumerate}
These are all simplices of $\IAAstrel[1]$, no double-double or mixed simplices occur in this low-rank case.
In addition to these, the complex $\IAArel[1]$ (see \cref{def:IAA}) has the following simplices:
\begin{enumerate}
\setcounter{enumi}{4}
\item \label{it_2_skew_IAA1} For every pair $v_0, v_1 \in V$ such that $\omega(v_0,v_1) = \pm 1$ and every $1\leq i\not=j \leq m$, there are the 2-skew-additive simplices of the form $\ls v_0, \langle \vec v_0 \pm \vec e_i\rangle, v_1 \rs$.

\item \label{it_sigma_add_IAA1} For every pair $v_0, v_1 \in V$ such that $\omega(v_0,v_1) = \pm 1$, there are the $\sigma$-additive simplices of the form $\ls v_0, v_1, \langle \vec v_0 \pm \vec v_1\rangle\rs$.
\end{enumerate}
No skew-additive, $\sigma^2$ or skew-$\sigma^2$ simplices occur in this low-rank case.
Spelling out \cref{def_regular_maps}, this implies that a map $\Psi\colon  B \to \IAArel[1]$ from a combinatorial ball to $\IAArel[1]$ is weakly regular if and only if 
\begin{description}[labelindent=0.5cm]
  \item[(Weakly regular)] \label{it_weakly_regular_low_dim} $\Psi$ is injective on every simplex mapping to a $\sigma$-additive or 2-skew-additive simplex.
\end{description}
This is because the only requirement on a cross map in this low dimension is that it be an isomorphism onto its image.

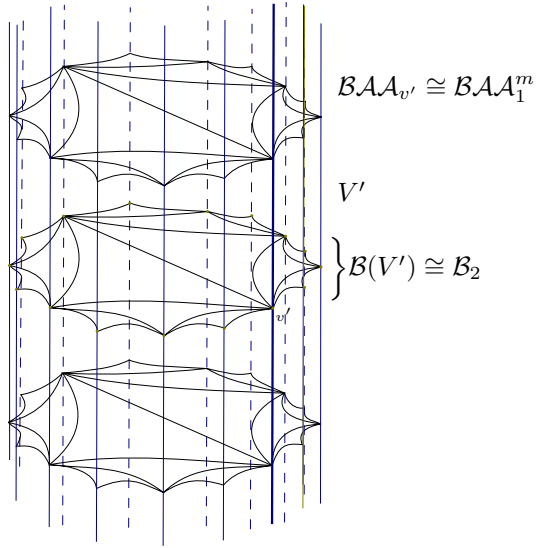
\begin{figure}
\begin{center}\vspace{40px}
    \begin{picture}(140,180)
    \put(0,0){\includegraphics[scale=.3]{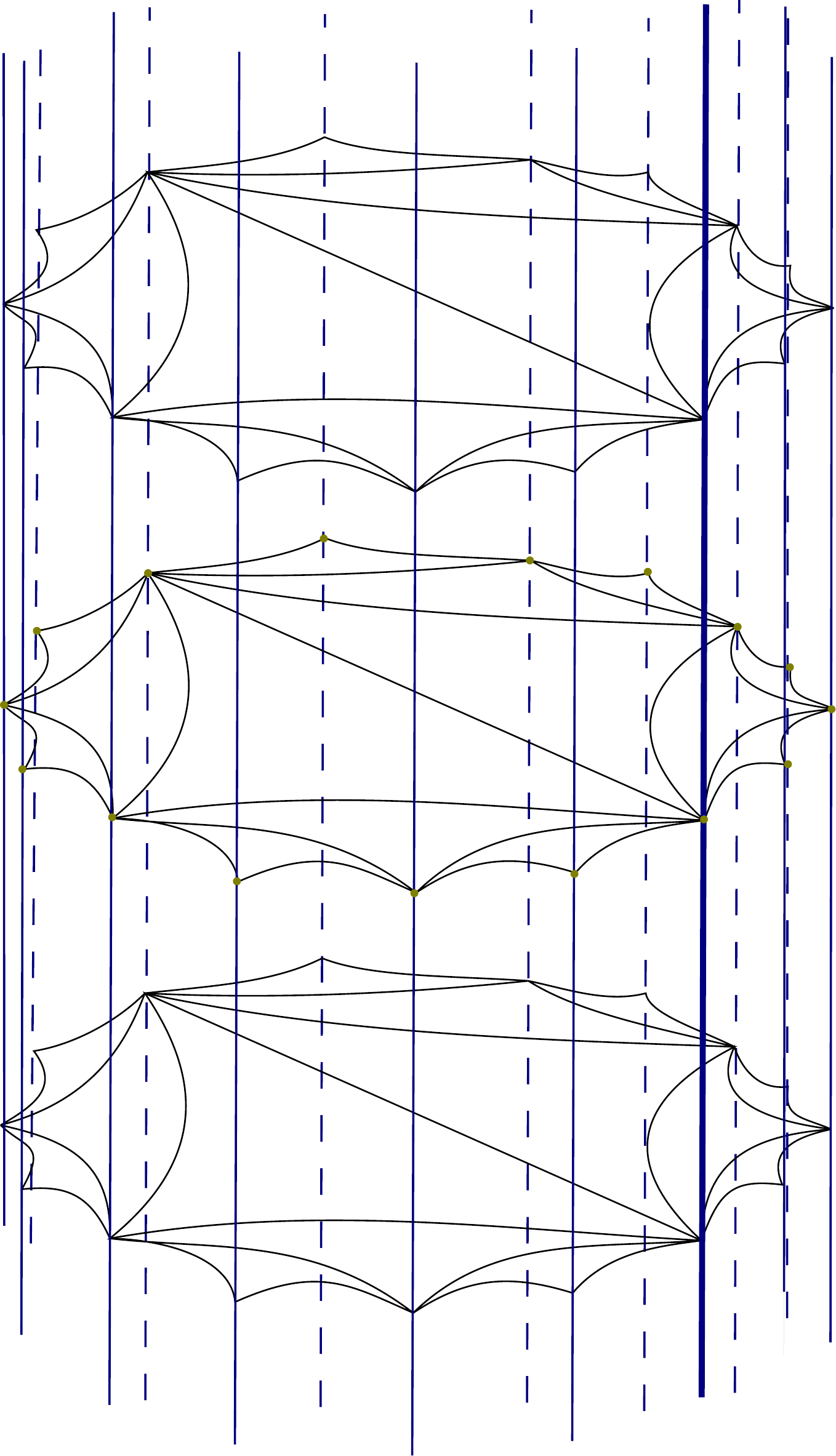}}
    
     \put(100,85){\tiny $v'$}
     \put(120,130){ $V'$}
     \put(120,102){$\bigg\} \B(V') \cong \B_2$}
     \put(120,170){ $\Bv_{v'} \cong \BAA_1^m$}

     \end{picture}
\end{center}
\caption{``Orthogonal'' copies of $\BAA_1^m$ and $\B_2$ in $\IAAstrel[1]$. Not depicted in this schematic are $\sigma$ edges between the different ``$\B_2$ layers''.}
\label{figure_induction_beginning_BA_BAA}
\end{figure}
The complex $\IAAstrel[1]$ contains ``orthogonal'' copies of the complexes $\BAA_1^m$ and $\B_2$, which we describe next (for a schematic overview, see \cref{figure_induction_beginning_BA_BAA}).
Let $V'\coloneqq V\cap \langle \vec e_{m+1}, \vec f_{m+1} \rangle$
 (i.e.~$V'$ is the set of rank-1 direct summands of $\langle \vec e_{m+1}, \vec f_{m+1} \rangle$). For $v'\in V'$, we define $\Bv_{v'}$ to be the full subcomplex of $\IAAstrel[1]$ on the set of vertices
\begin{equation*}
\ls \left\langle \vec v' + \sum_{i=1}^m a_i \vec e_i \right \rangle \,\middle| \, (a_1, \ldots, a_m)\in \mbZ^m \rs \subset V.
\end{equation*}
The simplices in \cref{it_2add_IAA1,it_3add_IAA1,it_doubletriple_IAA1} above imply that $\Bv_{v'}$ is isomorphic to the complex $\BAA_1^m$. This complex is described in detail in \cite[Proof of Lemma 5.10]{Brueck2022}, where it is shown that it is $1$-connected (for the cases $m < 2$, which are not explicitly mentioned in \cite[Proof of Lemma 5.10]{Brueck2022}, this holds as well as is explained in \cite[Paragraph after Theorem 2.10]{Brueck2022}).
Every vertex $v\in V$ is contained in precisely one such complex (namely in $\Bv_{v'}$ with the notation introduced in \cref{eq_vertices_induction_beginning}). The same is true for all simplices of \cref{it_2add_IAA1,it_3add_IAA1,it_doubletriple_IAA1}.

The only simplices of $\IAAstrel[1]$ that are not contained in any such $\Bv_{v'}$ are the $\sigma$ edges described in \cref{it_sigma_IAA1}. 
The full subcomplex of $\IAAstrel[1]$ on the set $V'$, which we write as $\baseB$, is 1-dimensional and only has edges of this type. It is isomorphic to the complex $\B_2$, which in turn is isomorphic to the Farey graph and in particular connected (see \cref{thm_connectivity_B_BA}).
If one adds to $\baseB$ the $\sigma$-additive simplices described in \cref{it_sigma_add_IAA1}, one obtains $\baseBA$, the full subcomplex of $\IAArel[1]$ on the set $V'$. This complex is isomorphic to $\BA_2$ and hence contractible (see \cref{thm_connectivity_B_BA}).

This description in particular implies that $\IAAstrel[1]$ is connected, which proves the case $k=0$ of our claim: Being connected means that we can extend every map $\phi\colon  S^0\to \IAAstrel[1]$ to a map $\Psi\colon  D^1\to \IAAstrel[1]$. As the image of $\Psi$ is contained in $\IAAstrel[1]$, it in particular is \hyperref[it_weakly_regular_low_dim]{weakly regular}, so $\phi$ is weakly regularly nullhomotopic.
\newline

Now let $S$ be a combinatorial $1$-sphere and $\phi\colon S\to \IAAstrel[1]$ a simplicial map. By \cref{lem_regular_homotopic_and_nullhomotopic}, it suffices to show that $\phi$ is weakly regularly homotopic to a map $\newproofphi\colon  \newproofsphere \to \IAAstrel[1]$ that is weakly regular nullhomotopic. (Recall from \cref{rem_regular_homotopies_notation_combinatorial} that in this setting, we always mean that $\newproofsphere$ is a combinatorial $k$-sphere and $\newproofphi$ is simplicial.) 
For this, we can first homotope $\phi$ to a map $\newproofphi\colon  \newproofsphere\to \IAAstrel[1]$, where $\newproofsphere$ is a combinatorial $k$-sphere and $\newproofphi$ is injective on every edge of $\newproofsphere$.\footnote{For a detailed account of how one can construct such a homotopy, see the work of Himes--Miller--Nariman--Putman \cite[Proof of Lemma 4.4]{Himes2022}. The key observation that allows to perform this homotopy here is that every edge $\Delta$ in $\IAAstrel[1]$ has a non-empty link $\Link_{\IAAstrel[1]}(\Delta)\not = \emptyset$. We use very similar arguments for spheres of arbitrary dimensions in \cref{sec_induction_step} and \cref{sec_normal_form_spheres}, where we give more details.} This process only involves simplices in $\IAAstrel[1]$, so defines a homotopy that is \hyperref[it_weakly_regular_low_dim]{weakly regular}. Hence, we can assume that $\phi$ already satisfies this local injectivity property. 
Let $\gamma_1, \ldots, \gamma_l$ be the cyclically indexed sequence of (cyclically) maximal subpaths of $S$ that map via $\phi$ to $\Bv_{v'}$ for some $v'\in V'$. These are separated by $\sigma$ edges in $S$ if $l> 1$. The subpath $\gamma_i$ can consist of a single vertex, namely if two $\sigma$ edges are adjacent in $S$.
If $l=1$, the image of $\phi$ is entirely contained in some $\Bv_{v'}$. This is a $1$-connected subcomplex of $\IAAstrel[1]$, so by \cref{it_simplicial_approximation_combinatorial_ball} of \cref{lem_simplicial_approximation_combinatorial}, there is a combinatorial $2$-ball $B$ with $\partial B = S$ and a map $\Psi\colon  B\to \Bv_{v'} \subset \IAAstrel[1]$ such that $\Psi|_S = \phi$. As the image of such a $\Psi$ is contained in $\IAAstrel[1]$, it is \hyperref[it_weakly_regular_low_dim]{weakly regular}, which means that $\phi$ is weakly regularly nullhomotopic. 
Hence, we can assume that $l>1$.
We now homotope $\phi$ by replacing step by step every $\phi|_{\gamma_i} \subset \Bv_{v'}$ by a map whose image is the single vertex $v'\in V'$, while preserving the other segments $\gamma_j$.

Take any $\gamma_i$ such that $\phi(\gamma_i)$ is not a single vertex that lies in $V'$. (Note that such a path can only occur if $m\geq 1$; otherwise, every $\Bv_{v'}$ is a single vertex.) There is $v'\in V'$ such that $\phi(\gamma_i)\subset  \Bv_{v'}$. This implies that $\phi(\gamma_i)$ gives a 
path in $\Bv_{v'}$ with end points of the form $\alpha = v'+ \sum_{i=1}^m a^{\alpha}_i \vec e_i $ and $\omega = v'+ \sum_{i=1}^m a^{\omega}_i \vec e_i $. 

We start by replacing $\phi|_{\gamma_i}$ by a map whose image is a sequence of 2-additive simplices (see \cref{figure_induction_beginning1}): 
\begin{figure}
\begin{center}\vspace{40px}
    \begin{picture}(400,100)
    \put(0,0){\includegraphics[scale=.3]{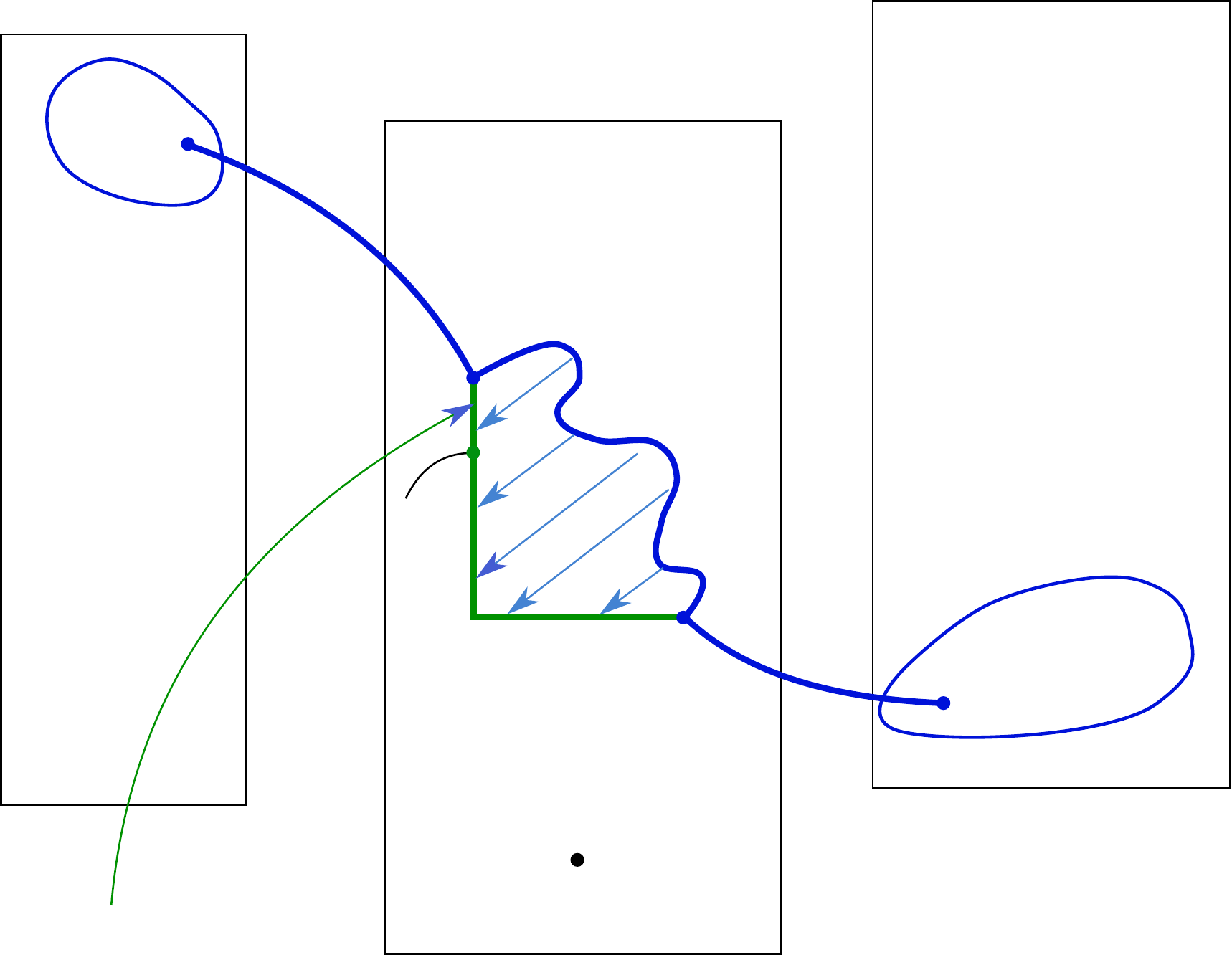}}
    \put(185,0){\includegraphics[scale=.3]{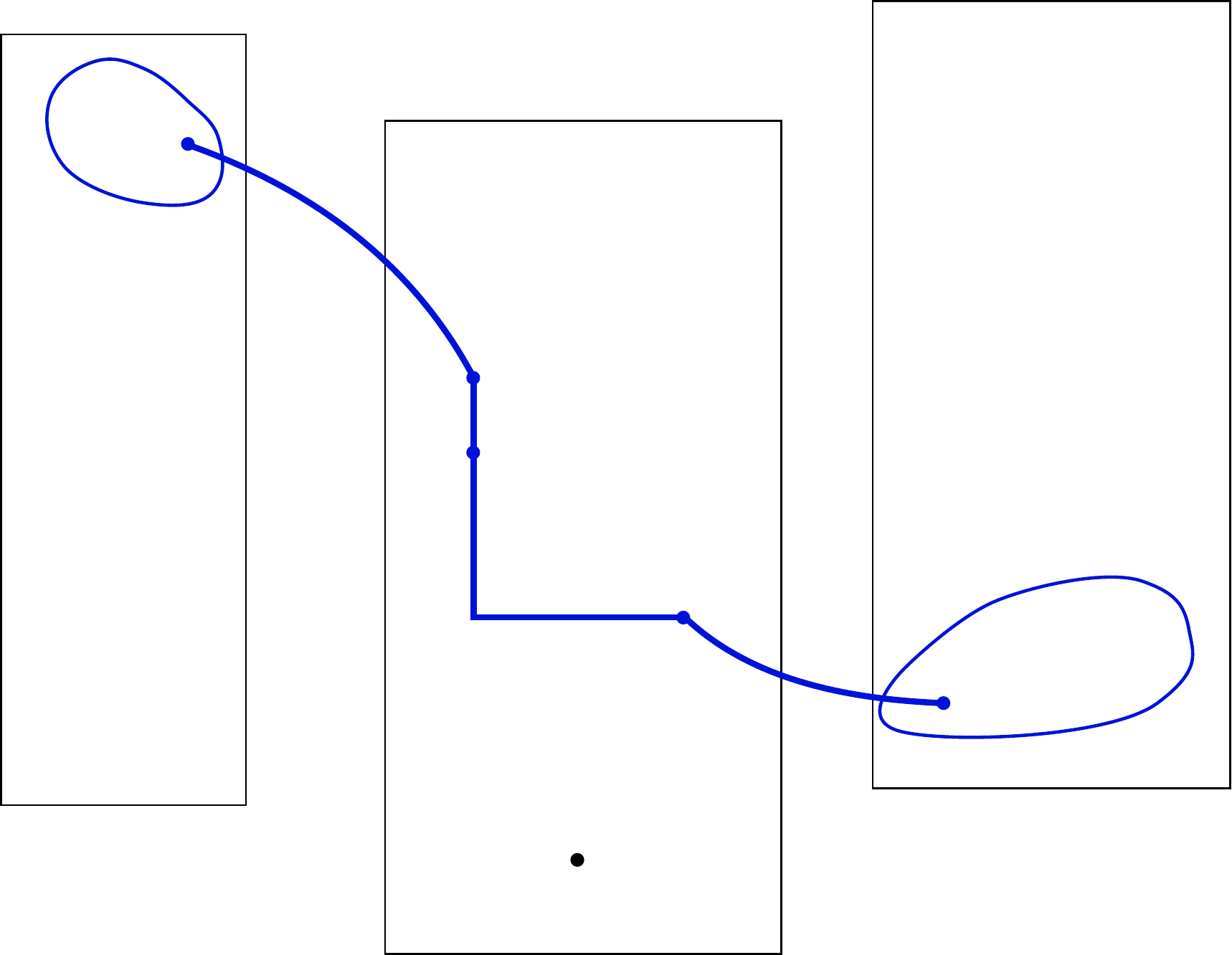}}
    \put(0,0){\tiny 2-additive}
     \put(53,35){\tiny $\newproofphi(\tilde \gamma_i)$}
     \put(73,70){\tiny $\phi(\gamma_i)$}
     \put(56,73){\tiny $\alpha$}
     \put(36,49){\tiny $\alpha\pm e_i$}
     \put(36,93){\tiny $\sigma$}
     \put(98,34){\tiny $\sigma$}
     \put(20,93){\tiny $s$}
     \put(2,85){\tiny $\phi(\gamma_{i-1})$}
     \put(85,40){\tiny $\omega$}
     \put(70,5){\tiny $v'$}
     \put(12,98){\tiny $\ddots$}
     \put(115,50){\tiny $\phi(\gamma_{i+1})$}
     \put(115,30){\tiny $t$}
     \put(125,35){\tiny $\cdots$}
     \put(47,105){\tiny $\B_{v'} \cong \Bv_n^m$}
     
     \put(160,60){ $\rightsquigarrow$}
     
     \put(232,105){\tiny $\B_{v'} \cong \Bv_n^m$}
     \put(238,35){\tiny $\newproofphi(\tilde \gamma_i)$}
     \put(241,73){\tiny $\alpha$}
     \put(220,60){\tiny $\alpha\pm e_i$}
     \put(221,93){\tiny $\sigma$}
     \put(283,34){\tiny $\sigma$}
     \put(205,93){\tiny $s$}
     \put(188,85){\tiny $\phi(\gamma_{i-1})$}
     \put(270,40){\tiny $\omega$}
     \put(255,5){\tiny $v'$}
     \put(197,98){\tiny $\ddots$}
     \put(300,50){\tiny $\phi(\gamma_{i+1})$}
     \put(300,30){\tiny $t$}
     \put(310,35){\tiny $\cdots$}
\end{picture}
\end{center}
\caption{Replacing $\gamma_i$ by $\tilde{\gamma}_i$.}
\label{figure_induction_beginning1}
\end{figure}
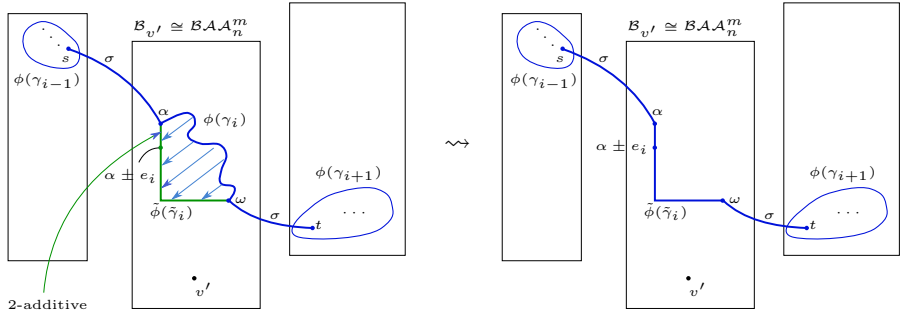
Using the simplices described in \cref{it_2add_IAA1}, we can find a 
path in $\Bv_{v'}$ from $\alpha$ to $\omega$ that only consists of 2-additive simplices. Let $\newproofphi|_{\tilde{\gamma}_i}\colon \tilde{\gamma}_i \to \IAAstrel[1]$ be an injective map that identifies a combinatorial 1-ball (i.e.~a path) $\tilde{\gamma}_i$ with this path consisting of 2-additive simplices. As $\Bv_{v'}$ is 1-connected, $\phi|_{\gamma_i}$ is relative to its endpoint homotopic to $\newproofphi|_{\tilde{\gamma}_i}$. By \cref{it_simplicial_approximation_combinatorial_ball} of \cref{lem_simplicial_approximation_combinatorial}, we can realise this homotopy by a map $\Psi\colon  B\to \Bv_{v'}\subset \IAAstrel[1]$ from a combinatorial $2$-ball $B$.
Such a map is \hyperref[it_weakly_regular_low_dim]{weakly regular} because $\im(\Psi)\subset \IAAstrel[1]$. 
Hence by \cref{lem_regular_homotopy_by_balls}, $\phi$ is weakly regularly homotopic to a simplicial map 
\begin{equation*}
	\newproofphi\colon \newproofsphere \to \IAAstrel[1]
\end{equation*}
that is obtained by replacing $\phi|_{\gamma_i}$ by $\Psi|_{\tilde{\gamma}_i} = \newproofphi|_{\tilde{\gamma}_i}$. The resulting map $\newproofphi$ is still injective on every edge of $\newproofsphere$.

We next replace $\phi|_{\tilde{\gamma}_i}$ by a map whose image is the single vertex $\omega$ (see \cref{figure_induction_beginning2}).
\begin{figure}
\begin{center}
\vspace{40px}
    \begin{picture}(400,100)
    \put(0,0){\includegraphics[scale=.3]{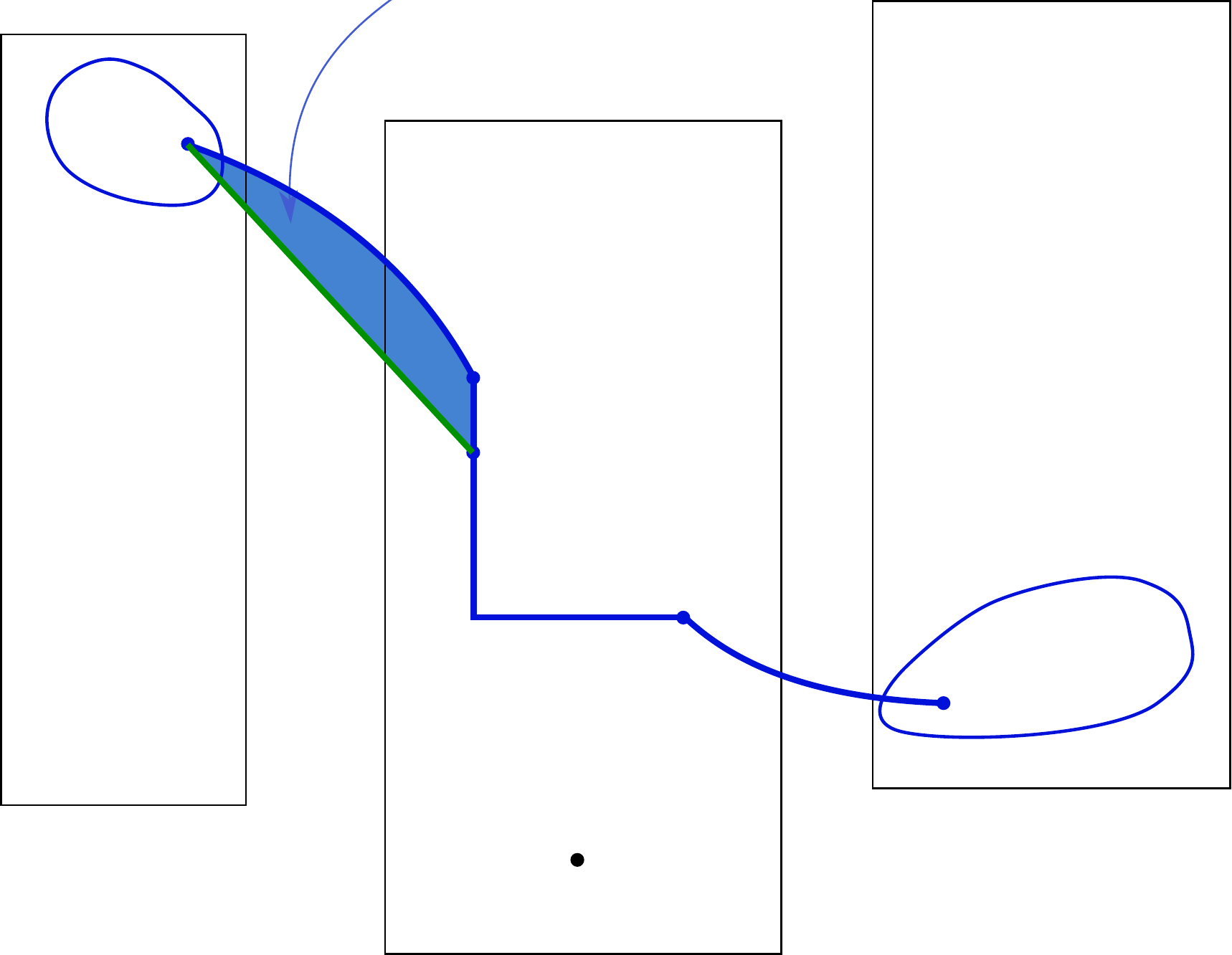}}
    \put(185,0){\includegraphics[scale=.3]{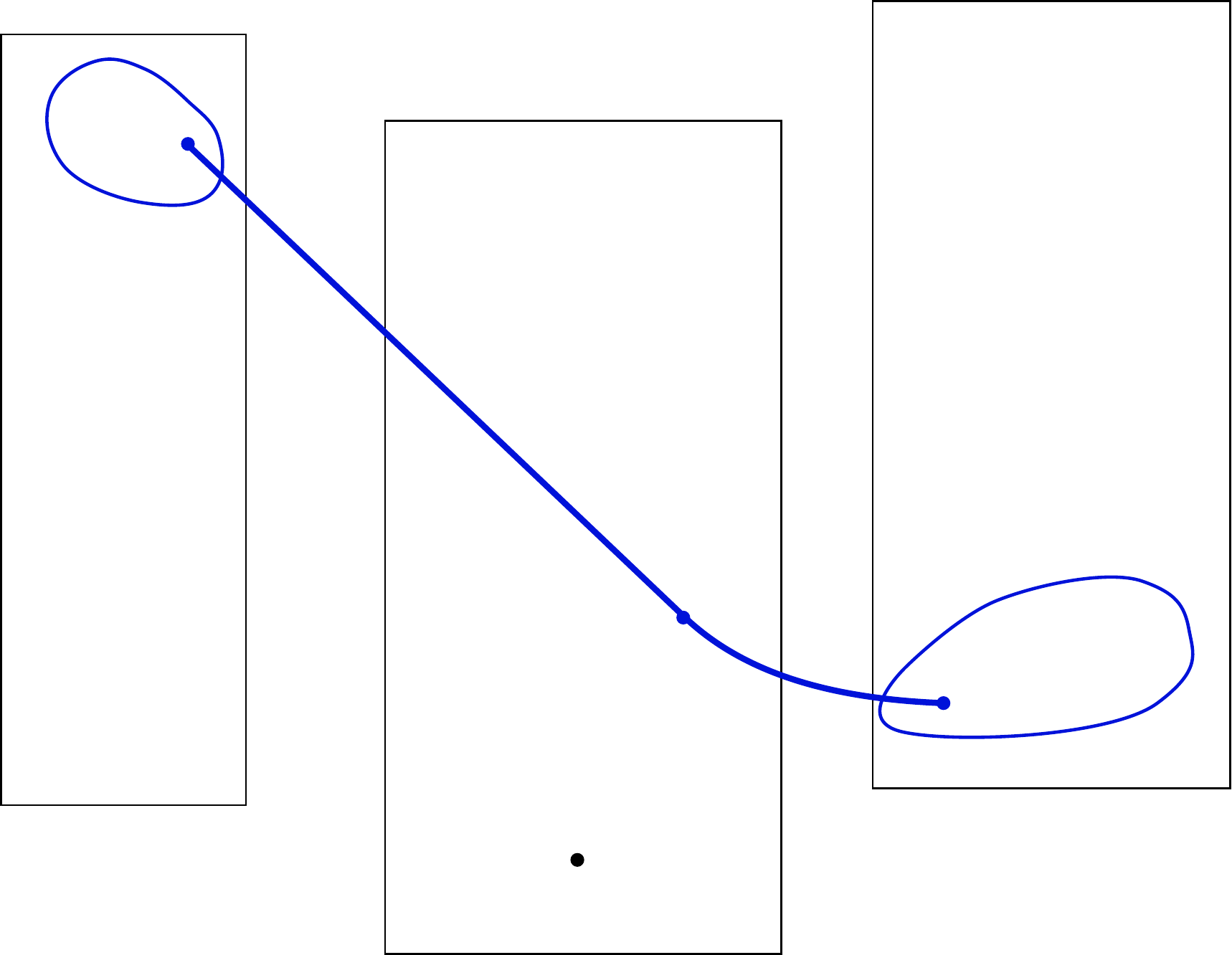}}

     \put(53,35){\tiny $\newproofphi(\tilde \gamma_i)$}
     \put(57,71){\tiny $\alpha$}
     \put(35,60){\tiny $\alpha\pm e_i$}
     \put(36,93){\tiny $\sigma$}
     \put(36,75){\tiny $\sigma$}
     \put(98,34){\tiny $\sigma$}
     \put(18,93){\tiny $s$}
     \put(2,85){\tiny $\phi(\gamma_{i-1})$}
     \put(85,40){\tiny $\omega$}
     \put(70,5){\tiny $v'$}
     \put(12,98){\tiny $\ddots$}
     \put(115,50){\tiny $\phi(\gamma_{i+1})$}
     \put(115,30){\tiny $t$}
     \put(125,35){\tiny $\cdots$}
     \put(47,105){\tiny $\B_{v'} \cong \Bv_n^m$}
     \put(47,120){\tiny 2-skew-additive}
     
     \put(160,60){ $\rightsquigarrow$}
     
     \put(232,105){\tiny $\B_{v'} \cong \Bv_n^m$}
     \put(231,75){\tiny $\sigma$}
     \put(283,34){\tiny $\sigma$}
     \put(205,93){\tiny $s$}
     \put(188,85){\tiny $\phi(\gamma_{i-1})$}
     \put(270,40){\tiny $\omega$}
     \put(255,5){\tiny $v'$}
     \put(197,98){\tiny $\ddots$}
     \put(300,50){\tiny $\phi(\gamma_{i+1})$}
     \put(300,30){\tiny $t$}
     \put(310,35){\tiny $\cdots$}
\end{picture}
\end{center}
\caption{Replacing $\tilde{\gamma}_i$ by $\omega$.}
\label{figure_induction_beginning2}
\end{figure}
For this, we use a homotopy in $\IAArel[1]$ whose image is not contained in $\IAAstrel[1]$:
Assume that $\phi(\tilde{\gamma}_i)$ has more than one vertex and let $s$ be the vertex that precedes $\alpha$ in $\phi(S)$, i.e.~the last vertex of $\phi(\gamma_{i-1})$. By assumption, $\ls s, \alpha \rs$ is a $\sigma$ simplex (this uses the assumption that $l>1$). As all edges in $\phi(\tilde{\gamma}_i)$ are 2-additive, the vertex of $\phi(\tilde{\gamma}_i)$ following $\alpha$ is of the form $\langle \vec \alpha \pm \vec e_i \rangle$ for some $1\leq i \leq m$. Hence, the set $\ls s, \alpha, \langle \vec \alpha \pm \vec e_i \rangle \rs$ is a 2-skew-additive simplex in $\IAArel[1]$, as in \cref{it_2_skew_IAA1}. We will use this to replace the edges $\ls s, \alpha \rs$ and $\ls \alpha, \langle \vec \alpha \pm \vec e_i \rangle \rs$ in the image of $\phi$ by the $\sigma$ edge $\ls s, \langle \vec \alpha \pm \vec e_i \rangle \rs$. Denote by $\ls x,y,z \rs$ the corresponding vertices of $S$,
\begin{equation*}
	\phi(x) = s,\, \phi(y) = \alpha,\, \phi(z) = \langle \vec \alpha \pm \vec e_i \rangle.
\end{equation*}
Define a map $\Psi\colon  B\to \IAAstrel$, where $B$ is the 2-simplex given by $\ls x,y,z \rs$ and $\Psi$ is equal to $\phi$ on each vertex. As 
\begin{equation*}
	\Psi(\ls x,y,z \rs) = \ls s, \alpha, \langle \vec \alpha \pm \vec e_i \rangle \rs
\end{equation*}
is a 2-skew-additive simplex, $\Psi$ is a well-defined and in fact weakly regular map.
Hence by \cref{lem_regular_homotopy_by_balls}, $\phi$ is weakly regularly homotopic to a simplicial map 
\begin{equation*}
	\newproofphi\colon \newproofsphere \to \IAAstrel[1]
\end{equation*}
that is obtained by replacing $\phi|_{\ls x,y \rs, \ls y,z \rs}$ by $\Psi|_{\ls x,z \rs}$. The map $\newproofphi$ is still injective on every edge of $\newproofsphere$.
This removes $\alpha$ from $\phi(\tilde{\gamma}_i)$ without changing the other vertices.
Iterating this, we find a weakly regular homotopy that replaces $\phi|_{\tilde{\gamma}_i}$ by a map that sends a singleton to the vertex $\omega$.

We lastly replace $\omega$ by $v'\in V'$ (see \cref{figure_induction_beginning3}), using a homotopy in $\IAArel[1]$ whose image is not contained in $\IAAstrel[1]$: Again, let $s$ be the last vertex of $\phi(\gamma_{i-1})$ and let $t$ be the first vertex of $\phi(\gamma_{i+1})$. Both $\ls s, \omega \rs $ and $\ls \omega, t \rs$ are $\sigma$ edges (that might be equal to one another). There is a path in $\Bv_{v'}$ from $\omega$ to $v'$ that consists only of 2-additive edges. Similarly to the procedure described in the previous paragraph, we can use this path to successively replace $\omega$ by $v'$, using two 2-skew-additive simplices in every step.
\begin{figure}
\begin{center}
\vspace{40px}
    \begin{picture}(400,100)
    \put(0,0){\includegraphics[scale=.3]{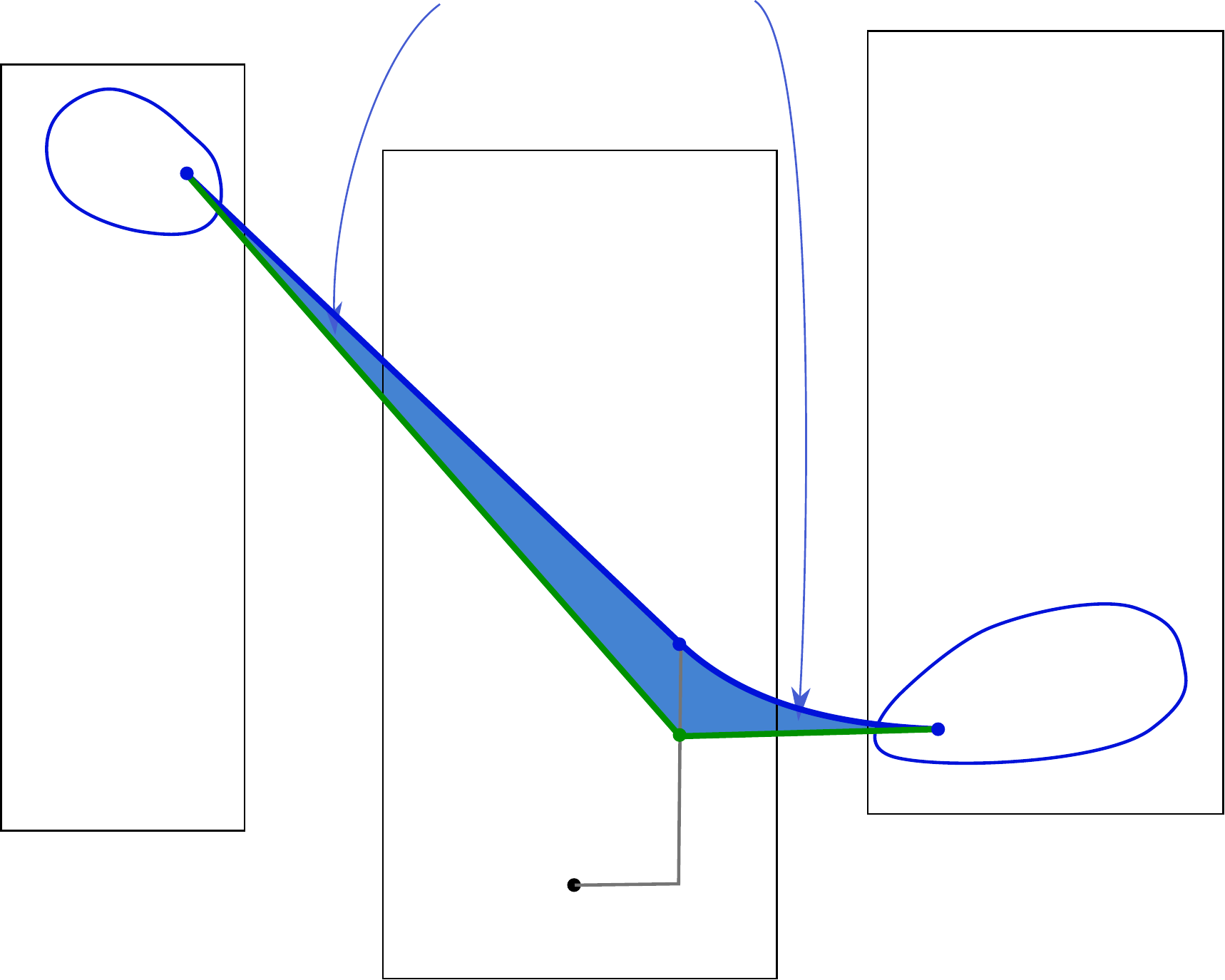}}
    \put(185,0){\includegraphics[scale=.3]{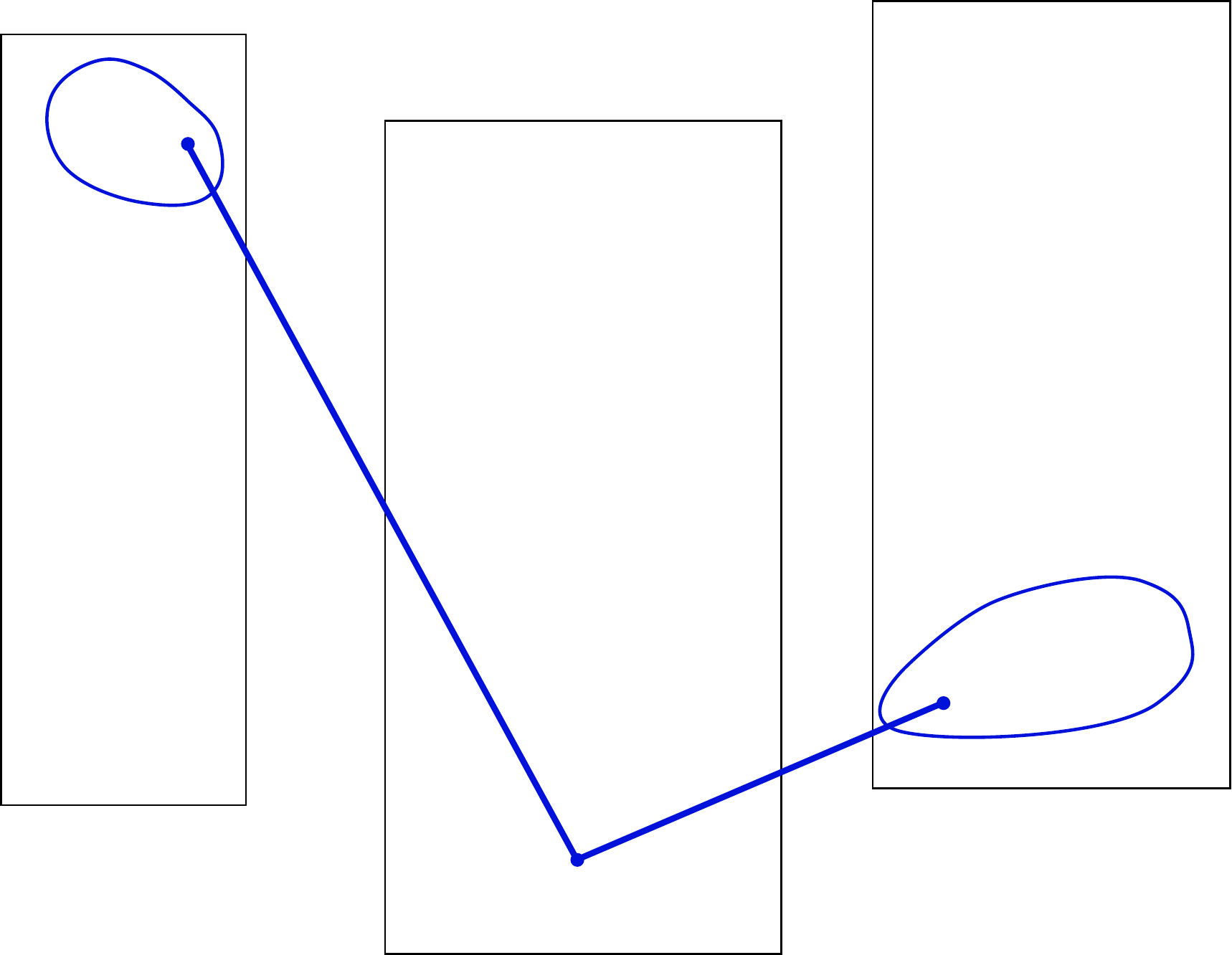}}

     \put(36,86){\tiny $\sigma$}
     \put(98,34){\tiny $\sigma$}
     \put(18,93){\tiny $s$}
     \put(2,85){\tiny $\phi(\gamma_{i-1})$}
     \put(85,40){\tiny $\omega$}
     \put(60,28){\tiny $\omega\pm e_i$}
     \put(70,5){\tiny $v'$}
     \put(12,98){\tiny $\ddots$}
     \put(115,50){\tiny $\phi(\gamma_{i+1})$}
     \put(115,30){\tiny $t$}
     \put(125,35){\tiny $\cdots$}
     \put(47,105){\tiny $\B_{v'} \cong \Bv_n^m$}
     \put(47,120){\tiny 2-skew-additive}
     
     \put(160,60){ $\rightsquigarrow$}
     
     \put(232,105){\tiny $\B_{v'} \cong \Bv_n^m$}
     \put(221,75){\tiny $\sigma$}
     \put(283,28){\tiny $\sigma$}
     \put(205,93){\tiny $s$}
     \put(188,85){\tiny $\phi(\gamma_{i-1})$}
     \put(255,5){\tiny $v'$}
     \put(197,98){\tiny $\ddots$}
     \put(300,50){\tiny $\phi(\gamma_{i+1})$}
     \put(300,30){\tiny $t$}
     \put(310,35){\tiny $\cdots$}
\end{picture}
\end{center}
\caption{Replacing $\omega$ by $v'$.}
\label{figure_induction_beginning3}
\end{figure}

Performing this procedure on all segments $\gamma_j$, we obtain $\newproofphi\colon \newproofsphere \to \IAAstrel[1]$ that is injective on every edge of $\newproofsphere$ and whose image is entirely contained in the subcomplex $\baseB \subseteq \IAAstrel[1]$.
As noted above, this complex is contained in the contractible complex $\baseBA \subseteq \IAArel[1]$. 
By \cref{it_simplicial_approximation_combinatorial_ball} of \cref{lem_simplicial_approximation_combinatorial}, there is a combinatorial $2$-ball $B$ with $\partial B = \newproofsphere$ and a map $\Psi\colon  B\to \baseBA \subset \IAAstrel[1]$ such that $\Psi|_{\newproofsphere} = \newproofphi$. As $\newproofphi$ is injective on every edge of $\newproofsphere$, we can choose $\Psi$ such that it is also injective on every simplex of $B$.\footnote{This follows from the work of Himes--Miller--Nariman--Putman \cite{Himes2022}: By \cref{thm_connectivity_B_BA}, $\baseBA\cong \BA_2$ is Cohen--Macaulay of dimension $2$. So by \cite[Lemma 4.4]{Himes2022}, it has the ``disc local injectivity property up to dimension $1$'' (see \cite[Definition 3.3]{Himes2022}). This means exactly that we can choose $\Psi$ as above to be locally injective.}
From this local injectivity, it follows that $\Psi$ is \hyperref[it_weakly_regular_low_dim]{weakly regular}, which means that $\newproofphi$ is weakly regularly nullhomotopic. 
\end{proof}

\subsection{Induction step}
\label{sec_induction_step}
We now let $n \geq 2$, $m\geq 0$, $ k \leq n$ and assume that either $k=0$ or by induction, any simplicial map $S^{k-1} \to \IAAstrel[n-1][m+1]$ from a combinatorial $(k-1)$-sphere into $\IAAstrel[n-1][m+1]$  is weakly regularly nullhomotopic in $\IAArel[n-1][m+1]$. (Recall that, as noted in \cref{lem:regular-implies-weakly-regular}, every regular map is also weakly regular.)
Let $S$ be a combinatorial $k$-sphere and $\phi\colon  S \to \IAAstrel$ be simplicial. We need to show that $\phi$ is regularly nullhomotopic in $\IAArel$.

Let
\begin{equation*}
R = \max\ls \rkfn(\phi(x)) \,\middle|\, x \text{ vertex of } S \rs,
\end{equation*}
where $\rkfn$ is defined as in \cref{def_rank_and_ranked_complexes}.
If $R=0$, we have $\im(\phi)\in \IAAstrel(W)$, so the claim follows from \cref{IAAW_highly_connected}. Hence by \cref{lem_regular_homotopic_and_nullhomotopic}, it suffices to show that $\phi$ is in $\IAArel$ regularly homotopic to a map $\newlemmaphi\colon  \newlemmasphere \to \IAAstrel$ with this property. (Recall by the convention set up in \cref{rem_regular_homotopies_notation_combinatorial}, $\newlemmasphere$ is a combinatorial $k$-sphere and $\newlemmaphi$ is simplicial.)

For this, assume that $R>0$ and call a simplex $\Delta$ of $S$ \emph{bad} if $\phi(\Delta)=\ls v \rs$ with $\rkfn(v) = R$.
We will explain how one can regularly homotope $\phi$ to a map $\phi'$ with no bad simplices. 
Iterating this procedure, we successively reduce $R$ to $0$, which finishes the induction step.
So from now and throughout \cref{sec_induction_step}, we make the standing assumption that 
\begin{equation}
\label{eq_standing_assumption_nmkR}
n\geq 2,\, m\geq 0, k\leq n \text{ and } R\geq 1.
\tag{\text{Standing assumption}}
\end{equation}

In order to remove bad simplices from $\phi$, we need to assume the following condition on their links.

\begin{definition}
Let $S$ be a combinatorial $k$-sphere, $\phi\colon  S \to \IAAstrel$ a simplicial map and $R = \max\ls \rkfn(\phi(x)) \,\middle|\, x \text{ vertex of } S \rs$. We say that $\phi$ has \hyperref[it_isolation]{Isolation} if it satisfies the following property. 
\begin{description}[labelindent=0.5cm]
  \item[(Isolation)] \label{it_isolation} If $\Delta$ is a bad simplex of $S$, $\phi(\Delta) = \ls v \rs$, and $x\in \Link_{S}(\Delta)$, then 
  \begin{equation*}
  	\phi(x) \in \ls v \rs \cup \Linkhat^{<R}_{\IAAstrel}(v).
  \end{equation*}
\end{description}
\end{definition}

This is an ``isolation'' criterion in the sense that it implies that there cannot be two adjacent vertices in $S$ that map to different vertices of rank $R$.
In fact, we can always assume that our maps have this property:

\begin{proposition}
\label{prop_isolating-rank-r-vertices}
Let $S$ be a combinatorial $k$-sphere and $\phi\colon  S \to (\IAAstrel)^{\leq R}$ a simplicial map. Then $\phi$ is in $\IAArel$ regularly homotopic to a map $\newlemmaphi\colon  \newlemmasphere \to (\IAAstrel)^{\leq R}$ such that $\newlemmaphi$ has \hyperref[it_isolation]{Isolation}.
\end{proposition}

\cref{prop_isolating-rank-r-vertices} is proved in \cref{sec_normal_form_spheres}. We assume it for now and only show that it has the following consequence.

\begin{lemma}
\label{lem_isolation_links}
Let $S$ be a combinatorial $k$-sphere and $\phi\colon  S \to (\IAAstrel)^{\leq R}$ a simplicial map that has \hyperref[it_isolation]{Isolation}. Let $\Delta$ be a bad simplex of $S$ that is maximal with respect to inclusion among all bad simplices, $\phi(\Delta) = \ls v \rs$. Then 
\begin{equation*}
\phi(\Link_{S}(\Delta)) \subset \Linkhat^{<R}_{\IAAstrel}(v).
\end{equation*}
\end{lemma}
\begin{proof}
As $\phi$ is simplicial and $\Delta$ is maximal among bad simplices, we have 
\begin{equation*}
	\phi(\Link_{S^k}(\Delta)) \subseteq \Link_{\IAAstrel}(v).
\end{equation*}
By \hyperref[it_isolation]{Isolation}, this implies that every vertex $x$ in $\Link_{S^k}(\Delta)$ is sent to a vertex $\phi(x)$ in $\Linkhat^{<R}_{\IAAstrel}(v)$. But $\Linkhat^{<R}_{\IAAstrel}(v)$ is a full subcomplex of $\Link_{\IAAstrel}(v)$ (see \cref{def_linkhat} and \cref{def_rank_and_ranked_complexes}). Hence, all higher dimensional simplices in $\phi(\Link_{S^k}(\Delta))$ are contained in  $\Linkhat^{<R}_{\IAAstrel}(v)$ as well, which is what we wanted to show.
\end{proof}

The following proposition shows that we can remove all bad simplices from $\phi$. By the  discussion at the beginning of this section, its proof finishes the induction step. \label{need_weaker_connectivity_IAAst}

\begin{proposition}
\label{prop_cut_out_bad_vertices}
Let $S$ be a combinatorial $k$-sphere and $\phi\colon  S \to (\IAAstrel)^{\leq R}$ a simplicial map. Then $\phi$ is in $\IAArel$ regularly homotopic to a map $\phi'\colon  S' \to (\IAAstrel)^{< R}$.
\end{proposition}
A key observation for the proof of \cref{prop_cut_out_bad_vertices} is that the links of vertices in $\IAAstrel$ and $\IAArel$ are similar to complexes that we already studied in the previous step of the induction (\cref{lem_link_of_vertex}).

\begin{proof}[Proof of \cref{prop_cut_out_bad_vertices}]
Our aim is to replace $\phi$ by a map whose image has only vertices of rank less than $R$. Hence, we are done if we can find a homotopic map without any bad simplices. 
After applying \cref{prop_isolating-rank-r-vertices}, we can assume that $\phi$ has \hyperref[it_isolation]{Isolation}.
Let $\Delta$ be a bad simplex that is maximal with respect to inclusion among all bad simplices. We will explain how to remove $\Delta$ from $\phi$ without introducing further bad simplices and such that the resulting map still has \hyperref[it_isolation]{Isolation}. Iterating this procedure will prove the claim.
\newline

Let $\phi(\Delta) = \{v\}$ be the image of the bad simplex $\Delta$. As $\phi$ has \hyperref[it_isolation]{Isolation}, it maps $\Link_{S}(\Delta)$ to $\Linkhat^{<R}_{\IAAstrel}(v)$ by \cref{lem_isolation_links}.
By \cref{lem_link_of_vertex}, there is a commutative diagram
\begin{equation*}
\xymatrix{
	\Linkhat_{\IAAstrel}(v) \ar[r]^{\cong} \ar@{^{(}->}[d] &  \ar@{^{(}->}[d] \IAAstrel[n-1][m+1] \\
	\Linkhat_{\IAArel}( v )  \ar[r]^{\cong} & \IAArel[n-1][m+1].
}
\end{equation*}
As $S$ is a combinatorial $k$-sphere, the complex $\Link_{S}(\Delta)$ is a combinatorial $(k-\dim(\Delta)-1)$-sphere.
\newline

First assume that $k-\dim(\Delta)-1\leq n-2$. This is the case if $k< n$ or $\dim(\Delta)>0$, i.e.~$\phi|_{\Delta}$ is not injective.
By \cref{lem_connectivity_IAAstar}, the complex $\Linkhat_{\IAAstrel}(v)\cong \IAAstrel[n-1][m+1]$ is $(n-2)$-connected. Hence by \cref{it_simplicial_approximation_combinatorial_sphere} of \cref{lem_simplicial_approximation_combinatorial}, there is a combinatorial $(k-\dim(\Delta))$-ball $\tilde{D}$ with $\partial \tilde{D} = \Link_{S}(\Delta)$ and there is a map 
\begin{equation*}
	\tilde{\psi}\colon  \tilde{D} \to \Linkhat_{\IAAstrel}(v)
\end{equation*}
such that $\tilde{\psi}|_{\Link_{S}(\Delta)} = \phi|_{\Link_{S}(\Delta)}$. 
By \cref{IAAst-retraction}, there are then also a combinatorial ball $D$ with $\partial D = \Link_{S}(\Delta)$ and a map 
\begin{equation*}
	\psi \colon  D \to \Linkhat^{<R}_{\IAAstrel}(v)
\end{equation*}
such that $\psi|_{\Link_{S}(\Delta)} = \tilde{\psi}|_{\Link_{S}(\Delta)} =  \phi|_{\Link_{S}(\Delta)}$.
By \cref{lem_boundary_join}, $B\coloneqq \Delta * D$ is a combinatorial $(k+1)$-ball whose boundary can be decomposed as
\begin{equation*}
\partial B = \Star_{S}(\Delta) \cup (\partial \Delta * D).
\end{equation*}
By \cref{cor_boundary_star}, both $\Star_{S}(\Delta)$ and $(\partial \Delta * D)$ are combinatorial $k$-balls and their intersection is given by 
\begin{equation*}
	\partial \Star_{S}(\Delta) = \partial \Delta \ast \Link_S(\Delta) = \partial \Delta * \partial D =  \partial (\partial \Delta * D). 
\end{equation*}

We define a map 
\begin{equation*}
	\Psi\colon  B = \Delta\ast D\to \IAAstrel
\end{equation*}
by letting $\Psi|_{\Delta} = \phi|_{\Delta}$ and $ \Psi|_D = \psi|_D$. As $\Psi$ has image in $\IAAstrel$, it is regular.
Hence by \cref{lem_regular_homotopy_by_balls}, $\phi$ is regularly homotopic to a simplicial map 
\begin{equation*}
	\newproofphi\colon \newproofsphere \to \IAAstrel
\end{equation*}
that is obtained by replacing $\phi|_{\Star_{S}(\Delta)}$ by $\Psi|_{\partial \Delta * D}$.

As $\phi$ and $\Psi$ have image in $(\IAAstrel)^{\leq R}$, so does $\newproofphi$. 
We claim that every bad simplex $\Theta$ of $\newproofphi$ is also a bad simplex of $\phi$. By \cref{lem_replacement_stars_maps}, it suffices to consider $\Theta \subseteq \partial \Delta * D$.
As 
\begin{equation*}
\newproofphi(D) = \psi (D) \subseteq \Linkhat^{<R}_{\IAAstrel}(v),
\end{equation*}
we then have $\Theta\subset \partial \Delta \subseteq \partial\Star_S(\Delta) \subset S$. So $\Theta$ is a bad simplex of $\phi$.
This implies that $\newproofphi$ has one less bad simplex than $\phi$ (namely the simplex $\Delta$ that was removed). 
It also implies that $\newproofphi$ has \hyperref[it_isolation]{Isolation}: If a bad simplex $\Theta$ of $\newproofphi$ is not contained in $\partial \Delta$, then by \cref{lem_replacement_stars_maps}, $\newproofphi|_{\Link_{\newproofsphere}(\Theta)} = \phi|_{\Link_{S}(\Theta)}$, so the condition for \hyperref[it_isolation]{Isolation} is obviously satisfied because $\phi$ has \hyperref[it_isolation]{Isolation}. But if $\Theta\subseteq \partial \Delta$, then 
\begin{equation*}
	\Link_{\newproofsphere}(\Theta)\subseteq \Link_{S}(\Theta) \cup D.
\end{equation*}
As observed above, $\newproofphi(D) \subseteq \Linkhat^{<R}_{\IAAstrel}(v)$, so for every $x \in \Link_{\newproofsphere}(\Theta)$, we have
\begin{equation*}
	\newproofphi (x) \in \phi(\Link_{S}(\Theta)) \cup \newproofphi(D) \subseteq \ls v \rs \cup \Linkhat^{<R}_{\IAAstrel}(v)
\end{equation*}
This is what we need to show for establishing \hyperref[it_isolation]{Isolation}.
\newline

We can therefore assume that $k = n$ and that $\dim(\Delta)=0$, i.e.~$\phi|_{\Delta}$ is injective.
As we assumed $n\geq 2$, this implies that $k>0$ and, as $S$ is a combinatorial $k$-sphere, $\Link_{S}(\Delta)$ is a combinatorial sphere of dimension $k-0-1 = n-1 \geq 1$. To simplify notation, we in what follows identify $\Delta = \ls x \rs$ with the single vertex $x$ it contains.

By \cref{lem_link_of_vertex} and the induction hypothesis, the map
\begin{equation*}
	\phi|_{\Link_{S}(x)}\colon  \Link_{S}(x) \to \Linkhat_{\IAAstrel}(v) \cong \IAAstrel[n-1][m+1]
\end{equation*}
is weakly regularly nullhomotopic in $\Linkhat_{\IAArel}(v) \cong \IAArel[n-1][m+1]$ .
I.e.~there is a combinatorial $n$-ball $\tilde{D}$ such that $\partial \tilde{D} = \Link_{S}(x)$ and there is a weakly regular map 
\begin{equation*}
	\tilde{\psi}\colon  \tilde{D} \to \Linkhat_{\IAArel}(v)
\end{equation*}
that agrees with $\phi|_{\Link_{S}(x)}$ on $\partial \tilde{D}$ (see \cref{def:weak-regularity-in-link} for the definition of weakly regular maps into $\Linkhat_{\IAArel}(v)$).
By \cref{IAA-retraction}, there is also a combinatorial ball $D$ with $\partial D =\partial \tilde{D} = \Link_{S}(x)$ and a weakly regular map
\begin{equation*}
	\psi\colon  D \to \Linkhat^{<R}_{\IAArel}(v)
\end{equation*}
with the property that $\psi|_{\Link_{S}(x)} = \tilde{\psi}|_{\Link_{S}(x)} = \phi|_{\Link_{S}(x)}$.
Define
\begin{equation*}
	B \coloneqq x \ast D.
\end{equation*}

By \cref{lem_boundary_join}, the complex $B$ is a combinatorial $(n+1)$-ball whose boundary can be decomposed as $
\partial B = \Star_{S}(x) \cup D$.
By \cref{cor_boundary_star}, both $\Star_{S}(x)$ and $D$ are combinatorial $k$-balls and their intersection is given by 
\begin{equation*}
	\partial \Star_{S}(x) = \Link_S(\Delta) = \partial D.  
\end{equation*}
As $\psi(D) \subset \Linkhat^{<R}_{\IAArel}(v)$, we can define a simplicial map $\Psi\colon  B\to \IAArel$
by setting $\Psi(x) = \phi(x) = v$ and $ \Psi|_D = \psi|_D$.

If $\Psi$ was regular and mapped $D$ to $\IAAstrel$ -- both of which need not be satisfied a priori -- we could similarly to the situation above use $\Psi$ to remove the bad simplex $\Delta = \ls x \rs$.

In our next step, we find a combinatorial $(n+1)$-ball $B'$ and a new map 
\begin{equation*}
	\Psi'\colon  B' \to \IAArel
\end{equation*}
such that 
\label{induction_step_properties_Dprime}
\begin{itemize}
\item $\partial B' = \Star_{S}(x) \cup D'$, where $D'$ and $\Star_{S}(x)$ are combinatorial $n$-balls and it holds that $\partial D' = \partial \Star_{S}(x) = \Link_{S}(x)$, 
\item $\Psi'$ is regular, 
\item $\Psi'|_{\Star_{S}(x)} = \phi|_{\Star_{S}(x)}$, and 
\item $\Psi'(D') \subset (\IAAstrel)^{<R}$.
\end{itemize}
As described above, we can then use \cref{lem_regular_homotopy_by_balls} to regularly homotope $\phi$ inside $\IAArel$ to another map $\newproofphi\colon  \newproofsphere \to \IAAstrel$, where $\phi|_{\Star_{S}(x)}$ is replaced by $\Psi'|_{D'}$. This procedures removes the bad simplex $\Delta$ and does not introduce new bad simplices because $\Psi'(D') \subset (\IAAstrel)^{<R}$. As $\dim(\Delta) = 0$, no bad simplex other than $\Delta$ intersects $\Star_{S}(\Delta)$. Hence, it is even true that for every bad simplex $\Theta$ of $\newproofphi$, we have $\Link_{S}(\Theta) = \Link_{\newproofsphere}(\Theta)$. It follows that the replacement preserves \hyperref[it_isolation]{Isolation}.
\newline

\begin{figure}
\begin{center}
\includegraphics{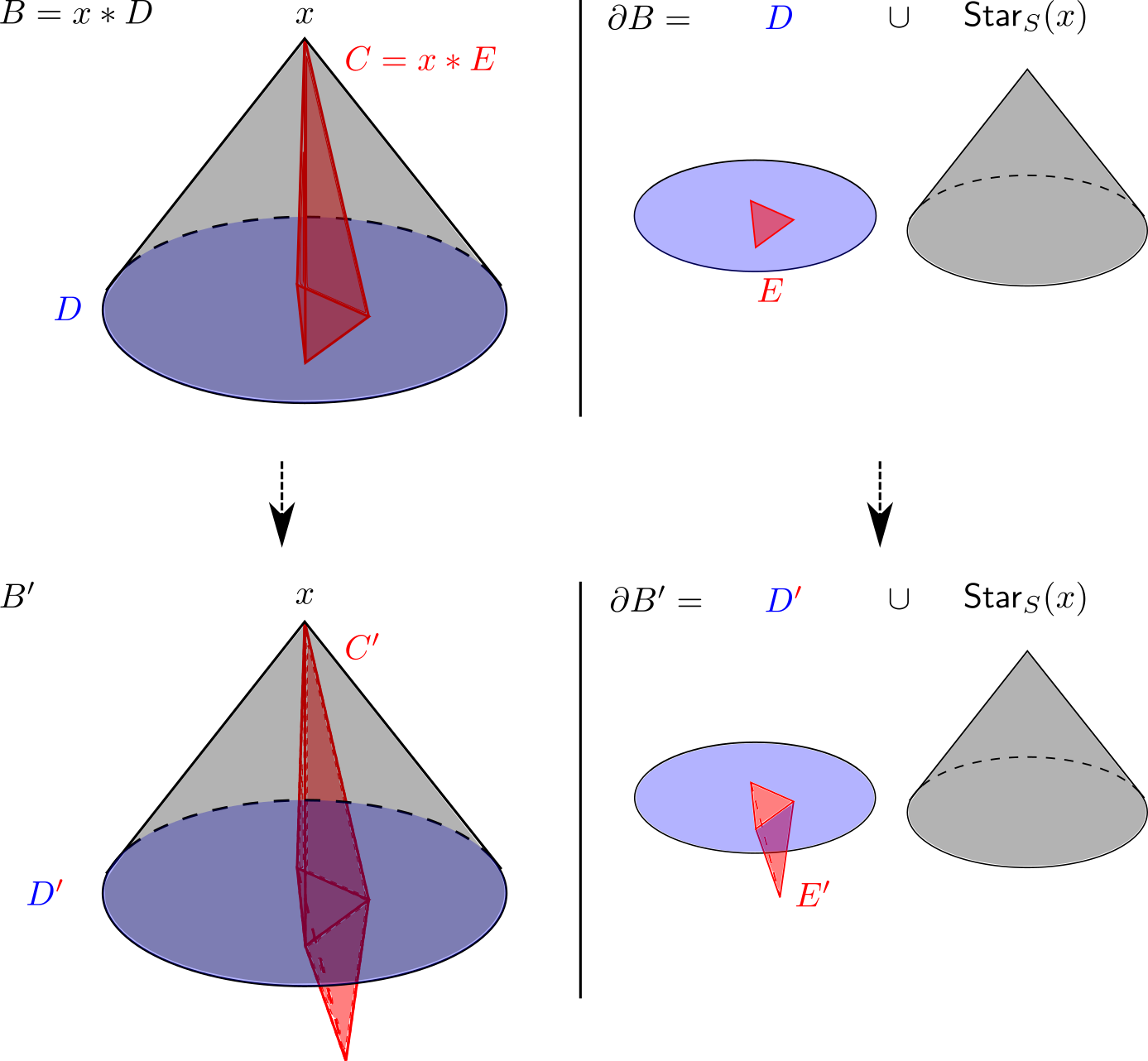}
\end{center}
\caption{An overview of the involved complexes and replacements for obtaining $\Psi'\colon B'\to \IAArel$ from $\Psi\colon B\to \IAArel$.}
\label{figure_induction_step_overview_replacements}
\end{figure}

We start by finding such a map $\Psi'$ for $n=2$. In this case, the domain $B$ of $\Psi$ is a combinatorial 3-ball and we obtain $B'$ by replacing certain $3$-dimensional simplices $C$ in $B$ with combinatorial $3$-balls $C'$. See \cref{figure_induction_step_overview_replacements} for a schematic overview of this process.
All $C$ that we replace are of the form $C = x * E$, where $E$ is a $2$-dimensional simplex in $D$ and $\Psi(E)$ is a simplex of $\IAArel[2] \setminus \IAAstrel[2]$.  We define $\Psi'$ on their replacements $C'$ by sending every vertex  to $(\IAAstrel[2])^{<R}$ and making sure that this defines a simplicial map on every simplex of $C'$. 
We replace a $3$-dimensional simplex $C$ in $\Psi$ whenever it is of one of the following three types:
\begin{enumerate}
\item $\Psi(C)$ is a $\sigma$-additive simplex of the form $\ls v, v_1, w_1, \langle \vec v_1 + \vec w_1 \rangle \rs$, where $\omega(v_1,w_1) = \pm 1$;
\item $\Psi(C)$ is a 2-skew-additive simplex of the form $\ls v, v_1, \langle \vec v_1 + \vec e_i\rangle, w_1 \rs$, where $\omega(v_1,w_1) = \pm 1$ and $1\leq i \leq m$;
\item $\Psi(C)$ is a 2-skew-additive simplex of the form $\ls v, v_1, \langle \vec v_1 + \vec v\rangle, w_1 \rs$, where $\omega(v_1,w_1) = \pm 1$.
\end{enumerate}
The map $\Psi$ is obtained by extending the weakly regular map $\psi$, which has image in a subcomplex of $\Linkhat_{\IAArel[2]}(v) \cong \IAArel[1][m+1]$.
Using the definition of weak regularity (\cref{def_regular_maps}), we see that every simplex $C$ in $B$ such that $\Psi(C)$ is not contained in $\IAAstrel[2]$ is of one of the three types listed above (see \cref{def:IAA-simplices} for the types of simplices in $\IAArel[2] \setminus \IAAstrel[2]$).

Using the notation for $\Psi(C)$ introduced above, we now replace each of these simplices as follows:
\begin{enumerate}
\item \label{it_induction_step_n2_sigma_additive}If the image of $C = x*E$ is $\sigma$-additive of the form $\ls v, v_1, w_1, \langle \vec v_1 + \vec w_1 \rangle \rs$, we define $C'$ by attaching a second 3-dimensional simplex to $C$ along its facet $E$ and name the new vertex $t_C$ as depicted in \cref{fig:induction_step_replacement_sigma_additive}. It is easy to see that $C'$ is a combinatorial $3$-ball.
We extend the map $\Psi'$ to $C'$ by setting $\Psi'(t_C) = w$, where $w$ is chosen such that $\ls \vec v,\vec w,\vec v_1,\vec w_1 \rs $ is a symplectic basis and $\rkfn(w)<R$. To choose such a $w$, first pick any $w'$ such that $\ls v,w',v_1,w_1 \rs $ is a symplectic basis, then take $\vec w = \vec w' + a\vec v$ with appropriate $a \in \Z$ such that $0\leq \rkfn(\vec w) < R$.\footnote{Note that this is the \emph{integral}-valued rank for vectors defined at the beginning of \cref{def_rank_and_ranked_complexes}.} This is possible using the Euclidean algorithm and in particular implies that $\rkfn(w) = |\rkfn(\vec w)| <R$. It is not hard to check that this defines a simplicial map $C'\to \IAArel[2]$, where both maximal simplices of $C'$ get mapped to $\sigma$-additive simplices.
\begin{figure}[h]
\begin{center}
\includegraphics{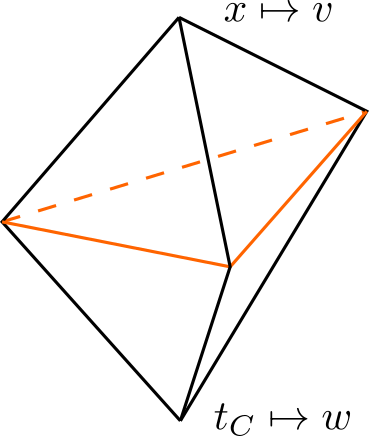}
\end{center}
\caption{A subcomplex $C'$ for the first type $\Psi(C) = \ls v, v_1, w_1, \langle \vec v_1 + \vec w_1 \rangle \rs$. Orange marks edges that are mapped to $\sigma$ simplices.}
    \label{fig:induction_step_replacement_sigma_additive}
\end{figure}

\item If the image of $C = x*E$ is 2-skew-additive of the form $\ls v, v_1, \langle \vec v_1 + \vec e_i\rangle, w_1 \rs$ for some $1\leq i \leq m$, we also attach a second 3-dimensional simplex to $C$ along its facet $E$ and name the new vertex $t_C$. We call the resulting simplicial complex $C''$. Choose a $w$ such that $\ls v,w,v_1,w_1 \rs $ is a symplectic basis and $\rkfn(w)<R$ again. We want to send $t_C$ to $w$ but this would not result in a regular map on $C''$ as it would lead to two 2-skew-additive simplices in the image of $C''$.
We remedy this issue by replacing the interior of $C''$ by three $3$ dimensional simplices that have an edge $\ls x, t_C \rs$ in common as indicated in \autoref{fig:9.5replaceC2}. This simplicial complex shall be $C'$. Again, it is easy to verify that $C'$ is a combinatorial $3$-ball. We can now define $\Psi'(t_C) = w$ as above and in fact this defines a simplicial map $C'\to \IAArel[2]$. It maps the three maximal simplices to two $\sigma^2$ simplices and one mixed simplex.
\begin{figure}[h]
\begin{center}
\includegraphics{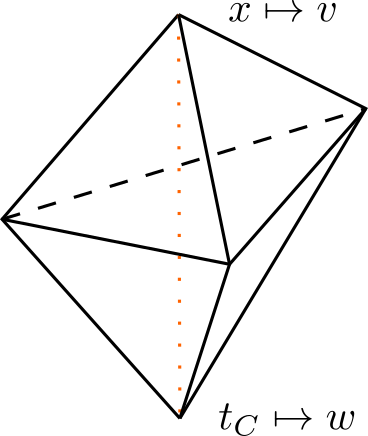}
\end{center}
\caption{A subcomplex $C'$ for the second type $\Psi(C) = \ls v, v_1, \langle \vec v_1 + \vec e_i\rangle, w_1\rs$. Orange marks the unique edge that is mapped to a $\sigma$ simplex.}
    \label{fig:9.5replaceC2}
\end{figure}

\item \label{it_induction_step_n2_2skew_additive}If the image of $C = x*E$ is 2-skew-additive of the form $\ls v, v_1, \langle \vec v_1 + \vec v\rangle, w_1 \rs$, we obtain $C'$ by adding to $C$ two 3-simplices with two new vertices $t_C$ and $s_C$ as depicted in \autoref{fig:9.5replaceC3}. As observed in \cref{lem_prism_combinatorial_manifold}, the prism $C'$ is a combinatorial 3-ball. We define the image of $t_C$ to be $w\in \IAArel[2]$ such that $\ls \vec v, \vec w,\vec v_1,\vec w_1 \rs $ is a symplectic basis with $0\leq \rkfn(\vec w),\rkfn(\vec w_1)<R$ and the image of $s_C$ to be $\langle \vec w - \vec w_1 \rangle \in \IAArel[2]$. This yields a simplicial map $C'\to \IAArel[2]$ whose image is a ``prism'' obtained as the union of two 2-skew-additive simplices and one skew-$\sigma^2$ simplex. Note that as $0\leq \rkfn(\vec w),\rkfn(\vec w_1)<R$, we have $\rkfn(\Psi'(s_C)) = |\rkfn(\vec w - \vec w_1)|<R$.
\begin{figure}[h]
\begin{center}
\includegraphics{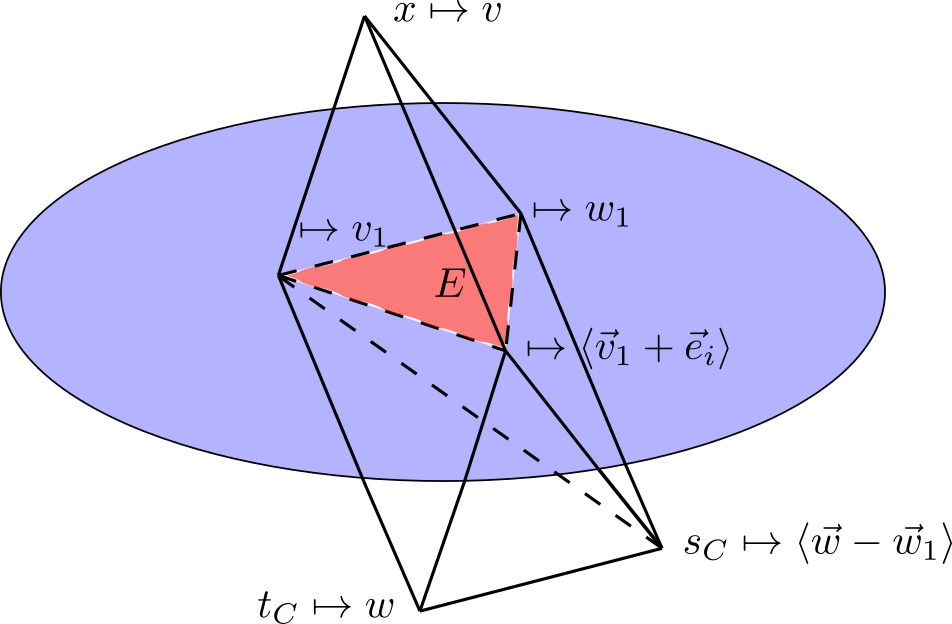}
\end{center}
\caption{A subcomplex $C'$ for the third type $\Psi(C) = \ls v, v_1, \langle \vec v_1 + \vec v\rangle, w_1\rs$.}
    \label{fig:9.5replaceC3}
\end{figure}
\end{enumerate}

Now, as described above and indicated in \cref{figure_induction_step_overview_replacements}, we let $B'$ be the complex obtained from $B$ by replacing the simplices $C$ by the 3-balls $C'$.
On the level of geometric realisations, each such replacement amounts in attaching to $|B|$ another 3-ball (the ``lower part'' of the complexes $C'$ in \cref{fig:induction_step_replacement_sigma_additive}, \cref{fig:9.5replaceC2} or \cref{fig:9.5replaceC3}) along a $2$-ball (the geometric realisation of the simplex $E$). By \cite[Corollary 1.28]{Hudson1969}, the result is again a (topological) $3$-ball. Using \cref{lem_combinatorial_ball_replace_bigger_boundary} and the fact that that the $C'$ are combinatorial $3$-balls, it follows that $B'$ is again a \emph{combinatorial} 3-ball.
Note that all the new $C'$ are chosen such that their boundary is partitioned into $\Star_{\partial C}(x)$ and a combinatorial $2$-ball $E'$. 
It follows that the boundary of $B'$ decomposes into $\Star_{\partial B'}(x) = \Star_{S}(x)$ and a combinatorial $3$-ball $D'$ that is obtained by replacing the simplices $E$ of $D$ by the combinatorial balls $E'$.
That $\Psi'$ is indeed a well-defined simplicial map follows because this is true on each $C'$ discussed above. 
To see that $\Psi'(D') \subseteq (\IAAstrel[2])^{<R}$, it suffices to note that in all of the three cases above, the subcomplex $E'\subset \partial C'$ gets mapped to $(\IAAstrel[2])^{<R}$ -- as mentioned above, all simplices of $D'$ that were already contained in $D$ get mapped to $(\IAAstrel[2])^{<R}$.

Next we need to verify that $\Psi'$ is regular. If $\Sigma$ is a simplex of $B'$ such that $\Psi'(\Sigma)$ is 2-dimensional $\sigma$-additive, then $\Sigma$ must be equal to $E$ for some of the $C$ of the first type described above. In this case, by construction $\Star_{B'}(\Sigma) = C'$ and $\Psi'|_{C'}$ is a $\sigma$-additive cross map.
Similarly, every $\Sigma\subset B'$ mapping to a 3-dimensional $\sigma^2$ simplex is contained in a $C'$ of the second type described above. So for $\sigma^2$-regularity, it suffices to note that $\Psi'|_{C'}$ is injective.
Lastly, every simplex mapping to a skew-additive, 2-skew-additive or skew-$\sigma^2$ simplex of $B'$ is contained in a $C'$ of the third type described above. Here, regularity is fulfilled because $\Psi'|_{C'}$ is a prism cross map.
\newline

We now consider the case $n \geq 3$. Recall that we have a simplicial map $\Psi\colon  B\to \IAArel$, where $B$ is a combinatorial $(n+1)$-ball with $\partial B = D \cup \Star_{S}(x)$, $\psi = \Psi|_D$ and we want to use this to construct a new combinatorial ball $B'$ with a map $\Psi'\colon B'\to \IAArel$ that satisfies the properties listed on \cpageref{induction_step_properties_Dprime}. Similarly to the $n=2$ case, we replace certain combinatorial $(n+1)$-balls $C$ of $B$ with $\partial C = \Star_{\partial C} (x) \cup E$ and $\Psi(E) \not \subseteq \IAAstrel$ by combinatorial balls $C'$ with $\partial C' = \Star_{\partial C'} (x) \cup E'$ and $\Psi'(E') \subseteq (\IAAstrel)^{<R}$ (see \cref{figure_induction_step_overview_replacements} for a schematic overview).

We perform such a replacement whenever there is a subcomplex $\Sigma$ in $B$ such that $\Psi(\Sigma)$ is of one of the following forms
\begin{itemize}
\item a 2-dimensional $\sigma$-additive simplex $\ls v_1, w_1, \langle \vec v_1 + \vec w_1 \rangle \rs$;
\item a 3-dimensional $\sigma^2$ simplex $\ls v_1, w_1, v_2, w_2 \rs$;
\item a 3-dimensional prism with vertex set $\ls v_1, w_1, v_2, w_2, \langle \vec v_1 + \vec v_2 \rangle, \langle \vec w_1 - \vec w_2 \rangle \rs$;
\item a 2-dimensional skew-additive simplex $\ls v_1, \langle \vec v_1 + \vec v\rangle, w_1 \rs$.
\end{itemize}
(In all of the above, the symplectic pairings are as suggested by the notation, i.e.~$\omega(v_1,w_1) = \pm 1$ and $\omega(v_2,w_2) = \pm 1$.)
As $\psi$ maps $x$ to $v$ and $D$ into $\Linkhat_{\IAArel}(v)$, any such $\Sigma$ is necessarily contained in $D$.
Let $\Sigma$ be such a subcomplex, let
\begin{equation*}
E\coloneqq \Star_{D}(\Sigma) \text{ and } C \coloneqq x \ast E \subseteq B.
\end{equation*}
As $\psi\colon  D \to \Linkhat^{<R}_{\IAArel}(v)$ is weakly regular, we have that $\Psi|_{\Star_{D}(\Sigma)} = \psi|_{\Star_{D}(\Sigma)}$ is a cross map of the corresponding type. Let $\ls \vec v_1, \vec w_1, \ldots, \vec v_{n-1}, \vec w_{n-1} \rs$ denote the symplectic basis in its image.

\begin{figure}
\begin{center}
\includegraphics{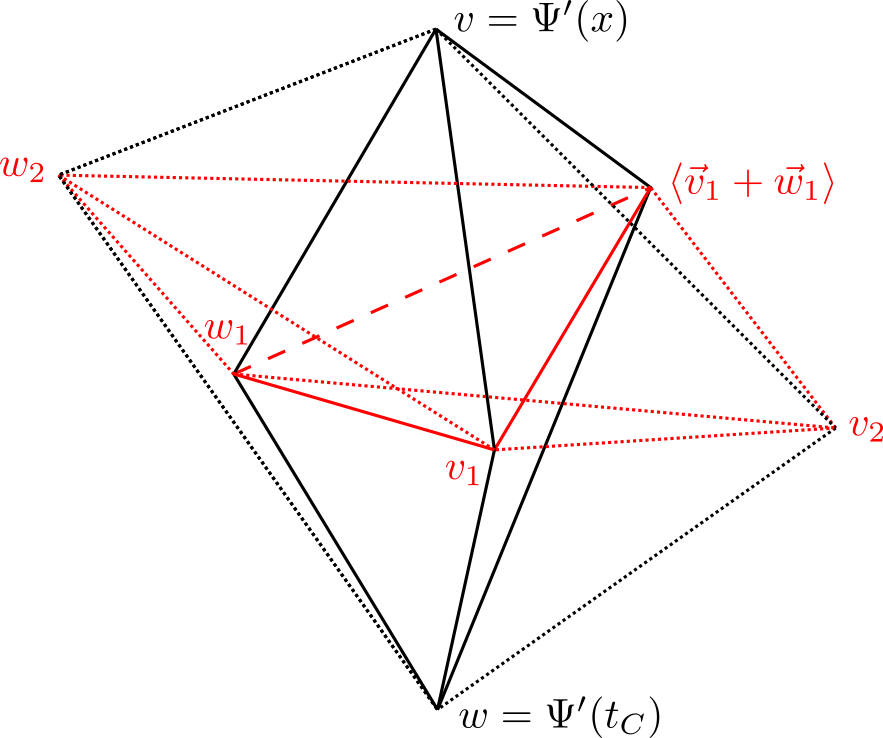}
\end{center}
\caption{The image under $\Psi'$ of the complex $C'$ in the first case for the example of a $\sigma$-additive simplex and $n=3$. $\Psi'|_{C'}$ is an isomorphism, so $C'$ itself has the same form. It is $4$-dimensional, all its edges are depicted in the figure. The red edges are those that are contained in $\Psi'(E)$. The ``upper half'' of the figure, i.e.~the subcomplex spanned by all vertices except $w = \Psi'(t_C)$, is $\Psi'(C) = \Psi(C)$ that we want to replace. It contains in its centre the simplex $\Psi(\Sigma) = \ls v_1, w_1, \langle \vec v_1 + \vec w_1 \rangle \rs$. Note that this is the suspension of the situation described in \cref{fig:induction_step_replacement_sigma_additive}.}
\label{figure_induction_step_ngreater2_sigma_add}
\end{figure}
We distinguish two cases. The first is that $\Psi(\Sigma)$ is a $\sigma$-additive simplex, a $\sigma^2$ simplex or a prism. In this case, we perform a construction similar to the case of $\sigma$-additive simplices for $n=2$ (\cref{it_induction_step_n2_sigma_additive} on \cpageref{it_induction_step_n2_sigma_additive}).
To obtain $C'$ from $C$ in this case, we cone off $E$ by a new vertex $t_C$ (as depicted in \cref{figure_induction_step_ngreater2_sigma_add}). We set $\Psi'(t_C) = w$, where $w\in \IAArel$ is a vertex such that $\ls \vec v, \vec w, \vec v_1, \vec w_1, \ldots, \vec v_{n-1}, \vec w_{n-1} \rs$ is a symplectic basis and $\rkfn(w)<R$. (We obtain such a $w$ just as in the case $n=2$ discussed above.) 
The complex $C'$ is a suspension of $E$ along the points $x$ and $t_C$ and hence a combinatorial ball by \cref{lem_boundary_join}. It is easy to see that $\Psi|_{C'}$ is a cross map of the appropriate type. 
We have $\partial C' = \Star_{\partial C'} (x) \cup E'$, where $E' = t_C \ast \partial E$.
As $\Psi'|_{C'}$ is a cross map, we know by \cref{lem_boundary_cross_map} that $\Psi'(E')$ is contained in $\IAAstrel$. Also by construction, every vertex of $E'$ maps to a line of rank less than $R$, so in fact $\Psi'(E')\subseteq (\IAAstrel)^{<R}$.

The second case is that $\Psi(\Sigma)$ is a skew-additive of the form $\ls v_1, \langle \vec v_1 + \vec v\rangle, w_1 \rs$. 
In other words, $\Psi(x \ast\Sigma)$ is the 2-skew-additive simplex $\ls v, v_1, \langle \vec v_1 + \vec v\rangle, w_1 \rs$.
As $\psi|_{\Star_{D}(\Sigma)}$ is an external 2-skew-additive cross map (see \cref{def:cross-maps}), $\Star_{D}(\Sigma)$ is of the form $\Sigma \ast C_{n-1}$. We extend the 3-simplex $x \ast \Sigma$ as in the corresponding third case for $n=2$ (\cref{it_induction_step_n2_2skew_additive} on page \cpageref{it_induction_step_n2_2skew_additive}): We add two new vertices $t_C$ and $s_C$ as indicated in \cref{fig:9.5replaceC3} and obtain a 3-dimensional prism $\Sigma'$.
We then define
\begin{equation*}
	C' \coloneqq \Sigma' \ast C_{n-1}.
\end{equation*}
Again, this is a combinatorial ball by \cref{lem_boundary_join}.
We set $\Psi'|_{C_{n-1}} = \Psi|_{C_{n-1}}$, let $\Psi'(t_C) = w$ as above and $\Psi'(s_C) = \langle \vec w - \vec w_1 \rangle$, where the sign of $\vec w_1$ is chosen such that $\rkfn(\langle \vec w - \vec w_1 \rangle)<R$ (such a choice exists by the same argument as in \cref{it_induction_step_n2_2skew_additive} on  \cpageref{it_induction_step_n2_2skew_additive}). It is not hard to check that $\Psi'|_{C'}$ is a prism cross map, so as above, it follows that $\Psi'(E')\subseteq (\IAAstrel)^{<R}$. 

Performing these replacements for all\footnote{We can perform these replacements independently of one another as two distinct such subcomplexes $E$ can only intersect in their boundaries by \cref{lem_cross_maps_intersections}. These boundaries are not modified by the replacements.} such $\Sigma$, we obtain a simplicial map $\Psi'\colon B'\to \IAArel$, where $B'$ is a combinatorial\footnote{To see that $B'$ is indeed a combinatorial ball, one again first verifies that $|B'|$ is a topological ball and then uses \cref{lem_combinatorial_ball_replace_bigger_boundary}.} $(n+1)$-ball with $\partial B' =  \Star_{S}(x) \cup D'$ and $\Psi'(D')\subseteq (\IAAstrel)^{<R}$. 
To see that this map is regular, it suffices to note again that every simplex mapping to a $\sigma$-additive, $\sigma^2$, skew-additive, 2-skew-additive or skew-$\sigma^2$ simplex is contained in some of the $C'$ above. Regularity then follows because $\Psi'|_{C'}$ is a cross map. Applying \cref{lem_regular_homotopy_by_balls} as above then finishes the cases $n \geq 3$ and concludes the proof.
\end{proof}

\section{Isolating bad simplices}
\label{sec_normal_form_spheres}
The aim of this section is to prove \cref{prop_isolating-rank-r-vertices}, i.e.~to show that every maximal bad simplex in a map $\phi\colon S^k \to \IAAstrel$ can be isolated.
Throughout this section, we keep the \hyperref[eq_standing_assumption_nmkR]{Standing assumption} from \cref{sec_induction_step}, which give the context in which \cref{prop_isolating-rank-r-vertices} is stated. That is, $n$, $k$, $m$ and $R$ are natural numbers such that $n\geq 2$, $m\geq 0$, $k\leq n$ and $R\geq 1$. 
We keep the convention introduced in \cref{rem_regular_homotopies_notation_combinatorial}.

\subsection{\texorpdfstring{$\sigma$}{\unichar{"03C3}}-regularity}
\label{sec_sigma_regularity}

We need another notion of cross map and associated regularity, similar to the one in \cref{def:cross-maps} and \cref{def_regular_maps}. This is the concept of $\sigma$-regularity that plays an important role in Putman's work \cite{put2009}:

\begin{definition}[{\cite[Definitions 6.16  and 6.17]{put2009}}]
\label{def_sigma_cross_map}
\,

\begin{enumerate}
\item A \emph{$\sigma$ cross map} is a simplicial map $\phi\colon  \Delta^1\ast C_{k-1} \to \IAArel$ with the following property: Let $x_1,y_1 \ldots, x_k,y_k$ be the vertices of $\Delta^1\ast C_{k-1}$. Then there is a symplectic summand of $\mbZ^{2(m+n)}$ with a symplectic basis $\ls \vec v_1, \vec w_1, \ldots, \vec v_k, \vec w_k \rs$ such that $\phi(x_i)= v_i$ and $\phi(y_i) = w_i$ for all $i$. See \cref{figure_cross_polytope_sigma_regular}.
\item Let $M$ be a combinatorial manifold. A simplicial map $\phi\colon  M \to \IAArel$ is called \emph{$\sigma$-regular} if the following holds: If $\Delta$ is simplex of $M$ such that $\phi(\Delta)$ is a minimal (i.e.~1-dimensional) $\sigma$ simplex, then $\phi|_{\Star_M({\Delta})}$ is a $\sigma$ cross map.
\item Let $\phi\colon  S \to\Idelrel$ be a simplicial map from a combinatorial $k$-sphere $S$. We say that $\phi$ is \emph{$\sigma$-regularly nullhomotopic (in $\Isigdelrel$)} if there is a combinatorial ball $B$ with $\partial B = S$ and a regular map $\Psi\colon  B \to \Isigdelrel$ such that $\Psi|_{S} = \phi$.
\end{enumerate}
\end{definition}

\begin{figure}
\begin{center}
\includegraphics{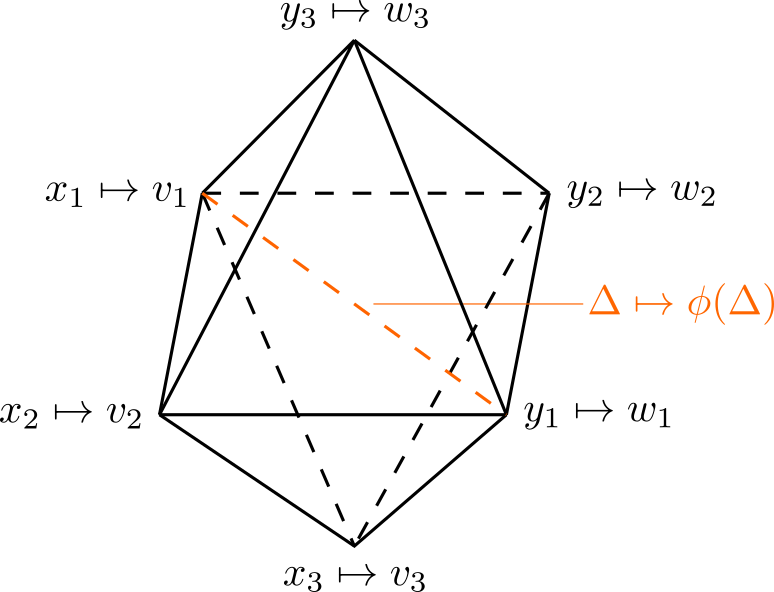}
\end{center}
\caption{The domain $C$ of a $\sigma$ cross map $\phi$ for $k=3$. The orange edge is the unique edge $\Delta$ mapping to a $\sigma$ edge.}
\label{figure_cross_polytope_sigma_regular}
\end{figure}

The notion of a $\sigma$ cross map is closely related to the one of a $\sigma^2$ cross map:
\begin{remark}
\label{ex_sigma_to_sigma2_regular}
Let $\phi\colon  C \to \IAAstrel$ be a $\sigma$ cross map and $D^1$ the 1-ball with its standard simplicial structure consisting of two vertices and one edge. Let $\ls v_{k+1}, w_{k+1} \rs$ be a $\sigma$ simplex such that $\ls \vec v_1, \vec w_1, \ldots, \vec v_k, \vec w_k , \vec v_{k+1}, \vec w_{k+1}\rs$ is a partial symplectic basis of $\mbZ^{m+n}$.
Then we obtain a $\sigma^2$ cross map
\begin{equation*}
	\Phi\colon  C \ast D^1 \to \IAArel
\end{equation*}
by setting $\Phi|_{C} = \phi$ and $\Phi(D^1) = \ls v_{k+1}, w_{k+1} \rs$. 
\end{remark}

We now set up some notation and elementary observations about $\sigma$-regular maps that is used in this section.
Let $\phi\colon S \to \IAArel$ be $\sigma$-regular and $\Delta$ a simplex of $S$ such that $\phi(\Delta)$ is a $\sigma$ edge. Then by definition, the restriction of $\phi$ to $C \coloneqq \Star_{S}(\Delta)$ is a $\sigma$ cross map $C\to \IAAstrel$.
We denote the vertices of $C$ by $x_1, y_1, \ldots, x_k,y_k$ and their images by $v_1, w_1, \ldots, v_k, w_k$, where for all $i$, $\ls v_i,w_i \rs$ is a symplectic pair, i.e.~$\omega(v_i, w_i)=\pm 1$. We always assume that $\ls v_1,w_1 \rs = \phi(\Delta)$ is the (unique) $\sigma$-edge contained in $\phi(C)$. 
Note that $C$ is a $k$-ball that is the union of $2^{k-1}$ simplices of dimension $k$, each of which gets mapped by $\phi$ to 
\begin{equation*}
	\ls v_1, w_1 \rs \cup \ls v_i \mid i\in I\rs \cup \ls w_j \mid j \in J \rs,
\end{equation*} 
for some disjoint, possibly empty, sets $I,J$ such that $I\cup J = \ls 2, \ldots, k \rs$. Its boundary $\partial C$ is given by all those $(k-1)$ simplices whose image does not contain both $v_1$ and $w_1$, i.e.~is of the form
\begin{equation*}
	\ls v_1 \rs \cup \ls v_i \mid i\in I\rs \cup \ls w_j \mid j \in J \rs \text{ or } \ls w_1 \rs \cup \ls v_i \mid i\in I\rs \cup \ls w_j \mid j \in J \rs.
\end{equation*}

As mentioned above, we always assume that the $\sigma$ edge in $\phi(C)$ is given by $\ls v_1, w_1 \rs$. However, as we are only interested in $\phi$ up to regular homotopy, this does not matter too much because there is the following ``flip'' that allows us to exchange $\ls v_1, w_1 \rs$ with any other symplectic pair in $\phi(C)$:
Let $2\leq i \leq k$ and let $C'$ be the simplicial complex obtained from $C$ by ``replacing the edge $\ls x_1, y_1 \rs$ with the edge $\ls x_i, y_i \rs$''. That is, $C'$ is the $k$-ball defined as follows: Let $\Delta_i$ be the 1-simplex with vertices $x_i, y_i$ and 
\begin{equation*}
	C_{k-2} \coloneqq \ast_{j\in\ls 1,\ldots, k \rs \setminus \ls 1,i \rs} \partial \ls x_j, y_j\rs \cong S^{k-3},
\end{equation*}
where $\partial \ls x_j, y_j\rs$ is a copy of $S^0$ with vertices $x_j$, $y_j$.
We define $C'$ as the join
\begin{equation*}
	C' = (\partial \Delta) \ast \Delta_i \ast C_{k-2}.
\end{equation*}
There is an obviously $\sigma$-regular map $C'\to \IAAstrel$ that agrees with $\phi|_C$ on the vertex set of $C$ (which is also the vertex set of $C'$). Furthermore, the following \cref{lem_flip} shows that there is a regular homotopy that allows us to replace $\phi|_C$ by this map, see also \cref{figure_sigma_flip}.
\begin{figure}
\begin{center}
\includegraphics{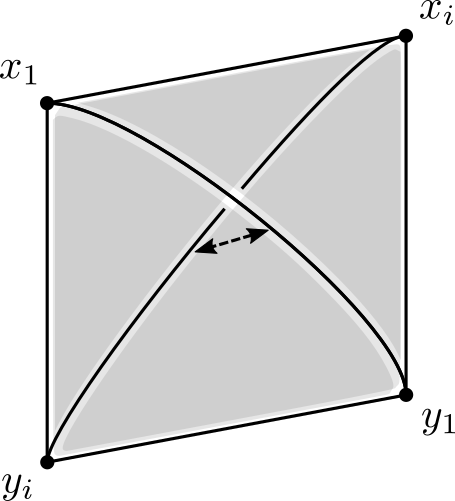}
\end{center}
\caption{A 3-simplex $\ls x_1, y_1, x_i,y_i \rs = \Delta \ast \Delta_i$. The map $\Psi$ in \cref{lem_flip} sends it to a $\sigma^2$ simplex. The arrow indicates the ``flip'' performed by the corresponding homotopy.}
\label{figure_sigma_flip}
\end{figure}
\begin{lemma}
\label{lem_flip}
There are a combinatorial $(k+1)$-ball $B$ with vertices $x_1$, $y_1$,$\ldots$, $x_k$, $y_k$ and a regular map $\Psi\colon  B\to \IAArel$ that agrees with $\phi$ on $x_1, y_1, \ldots, x_k, y_k $ and such that $\partial B  = C \cup C'$.
\end{lemma}
\begin{proof}
We define $B$ as the join
\begin{equation*}
	B = \Delta \ast \Delta_i \ast C_{k-2}.
\end{equation*}
So $B$ has the same vertices as $C$ and $C'$, but ``an additional edge between $x_i$ and $y_i$''; that is $B \cong D^1 \ast D^1 \ast S^{k-3}$, whereas $C \cong D^1 \ast S^0 \ast S^{k-3}$ and $C' \cong S^0 \ast D^1 \ast S^{k-3}$.
Define $\Psi\colon  B \to \IAArel$ to be the map that agrees with $\phi$ on the vertices of $B$ (which are also the vertices of $C$). Using the description of $\phi(C)$ at the beginning of this subsection, the image $\Psi(B)$ is a union of $2^{k-2}$-many $\sigma^2$ simplices of dimension $(k+1)$.
The map $\Psi$ is a $\sigma^2$ cross map and in particular regular. 
By \cref{lem_boundary_join},
\begin{equation*}
	\partial B = (\partial \Delta \ast \Delta_i \ast C_{k-2}) \cup (\Delta \ast \partial \Delta_i \ast C_{k-2}) = C' \cup C,
\end{equation*}
so we have the claimed decomposition of the boundary of $B$.
\end{proof}

This depicts how regular maps arise as homotopies between $\sigma$-regular maps. With the intuition given in \cref{rem_polyhedral_complexes}, the cross maps $\Psi|_C \to \IAArel$ and $\Psi|_{C'} \to \IAArel$ can be thought of as two polyhedral cells (in the shape of cross polytopes). Together, they form the boundary of the higher-dimensional polyhedral cell $\Psi\colon B \to \IAArel$ that yields a homotopy between them.

\subsection{Establishing \texorpdfstring{$\sigma$}{\unichar{"03C3}}-regularity}
\label{sec_sigma_regularity_obtained}
We keep the \hyperref[eq_standing_assumption_nmkR]{Standing assumption} that $n\geq 2$, $m\geq 0$, $k\leq n$ and $R\geq 1$.
In this first step to proving \cref{prop_isolating-rank-r-vertices}, we show that every map  $\phi\colon  S^k \to \IAAstrel$ is regularly homotopic to a $\sigma$-regular map.
\begin{lemma}
\label{lem_sigma_regular_can_be_assumed}
Let $S$ be a combinatorial $k$-sphere and $\phi\colon  S \to (\IAAstrel)^{\leq R}$ a simplicial map. Then $\phi$ is in $\IAArel$ regularly homotopic to a $\sigma$-regular map $\newlemmaphi\colon \newlemmasphere\to (\IAAstrel)^{\leq R}$.
\end{lemma}

Our proof of \cref{lem_sigma_regular_can_be_assumed} relies on a result of Putman. 
In \cite[Proposition 6.13.4]{put2009}, he shows that the inclusion $\Idelrel \hookrightarrow \Isigdelrel$ induces a trivial map on $\pi_k$ for all $k\leq n-1$. In fact, his proof shows more, namely:

\begin{theorem}[{\cite[Proof of Proposition 6.13, fourth conclusion]{put2009}}]
\label{thm_andy_fourth_conclusion}
Let $0\leq k \leq n-1$ and $R\in \mbN \cup \ls\infty\rs$. Let $S$ be a combinatorial $k$-sphere and $\phi\colon  S \to (\Idelrel)^{\leq R}$ a simplicial map. Then $\phi$ is $\sigma$-regularly nullhomotopic in $(\Isigdelrel)^{\leq R}$. 
\end{theorem}

From this, we obtain $\sigma$-regularity using a ``cut-out'' argument very similar to the one in the proof \cref{prop_cut_out_bad_vertices}:

\begin{proof}[Proof of \cref{lem_sigma_regular_can_be_assumed}]
Call a simplex $\Delta$ of $S$ \emph{non-regular  in $\phi$}  if $\phi(\Delta)$ is a $\sigma$ edge and $\phi|_{\Star_{S}(\Delta)}$ is not a $\sigma$ cross map.

Assume that there are non-regular simplices in $\phi$ and let $\Delta$ be one that is maximal with respect to inclusion among all non-regular simplices.
We will show how to homotope $\phi$ to a map $\newproofphi$ that has one less non-regular simplex. Iterating this procedure then leads to a map without non-regular simplices, which proves the claim.
As $S$ is a combinatorial $k$-sphere, the link $\Link_{S}(\Delta)$ is a combinatorial $(k-\dim(\Delta)-1)$-sphere.
As $\phi$ is simplicial and $\Delta$ is maximal among non-regular simplices, we have $\phi(\Link_{S}(\Delta))\subseteq (\Link_{\IAAstrel}(\phi(\Delta)))^{\leq R}$. By \cref{lem_isom_link_for_making_regular}, the latter is isomorphic to $(\Idelrel[n-1][m])^{\leq R'}$ for some $R'\in \mbN \cup \ls \infty \rs$.
\newline

We consider two cases: If $\phi|_\Delta$ is not injective, then $\dim(\Delta)\geq 2$, so $\Link_{S}(\Delta)$ is a combinatorial sphere of dimension $ k-\dim(\Delta)-1 \leq n-3$. By \cref{lem_connectivity_Idelrel}, the complex $(\Link_{\IAAstrel}(\phi(\Delta)))^{\leq R}\cong (\Idelrel[n-1][m])^{\leq R'}$ is $(n-3)$-connected. Hence by \cref{lem_simplicial_approximation_combinatorial}, there are a combinatorial $(k-\dim(\Delta))$-ball $D$ with $\partial D = \Link_{S}(\Delta)$ and a map 
\begin{equation*}
	\psi\colon  D \to (\Link_{\IAAstrel}(\phi(\Delta)))^{\leq R}
\end{equation*}
such that $\psi|_{\Link_{S}(\Delta)} = \phi|_{\Link_{S}(\Delta)}$.
By \cref{lem_boundary_join}, the complex $B\coloneqq \Delta * D$ is a combinatorial $(k+1)$-ball  whose boundary can be decomposed as
\begin{equation*}
\partial B = \Star_{S}(\Delta)  \cup (\partial \Delta * D) .
\end{equation*}
By \cref{cor_boundary_star}, both $\Star_{S}(\Delta)$ and $(\partial \Delta * D)$  are combinatorial $k$-balls and their intersection is given by 
\begin{equation*}
	\partial \Star_{S}(\Delta) = \partial \Delta \ast \Link_S(\Delta) = \partial \Delta * \partial D= \partial (\partial \Delta * D). 
\end{equation*}

We define a map 
\begin{equation*}
	\Psi\colon  B\to (\IAAstrel)^{\leq R}
\end{equation*}
by letting $\Psi|_{\Delta} = \phi|_{\Delta}$ and $ \Psi|_D = \psi|_D$. As $\Psi$ has image in $\IAAstrel$, it is regular.
Hence by \cref{lem_regular_homotopy_by_balls}, $\phi$ is regularly homotopic to a simplicial map 
\begin{equation*}
	\newproofphi\colon \newproofsphere \to \IAAstrel
\end{equation*}
that is obtained by replacing $\phi|_{\Star_{S}(\Delta)}$ by 
\begin{equation*}
	\Psi|_{\partial \Delta * D} \colon   \partial \Delta * D \to \phi(\Delta) * (\Link_{\IAAstrel}(\phi(\Delta)))^{\leq R}\subseteq (\IAAstrel)^{\leq R}.
\end{equation*}

As $\phi$ and $\Psi$ have image in $(\IAAstrel)^{\leq R}$, so does $\newproofphi$.
We claim that every non-regular simplex $\Theta$ of $\newproofphi$ is also a non-regular simplex of $\phi$. By \cref{lem_replacement_stars_maps}, it suffices to consider $\Theta \subseteq \partial \Delta * D$.
As
\begin{equation*}
\newproofphi(D) = \psi (D) \subseteq \Linkhat^{<R}_{\IAAstrel}(v),
\end{equation*}
we then have $\Theta\subset \partial \Delta \subseteq \partial\Star_S(\Delta) \subseteq S$. So $\Theta$ is a non-regular simplex of $\phi$.
This implies that $\newproofphi$ has one less non-regular simplex than $\phi$ (namely the simplex $\Delta$ that was removed).
\newline

The more difficult case is the one where $\phi|_\Delta$ is injective. Here, we know that $\Link_{S}(\Delta)$ is a combinatorial sphere of dimension $k-\dim(\Delta)-1 \leq n-2$, so we cannot necessarily do the same replacement.
But now by \cref{thm_andy_fourth_conclusion} and \cref{lem_isom_link_for_making_regular}, the map 
\begin{equation*}
	\Link_{S}(\Delta) \to (\Link_{\IAAstrel}(\phi(\Delta)))^{\leq R} \cong (\Idelrel)^{\leq R'}
\end{equation*}
is $\sigma$-regularly nullhomotopic in $(\Linkhat_{\IAArel}(\phi(\Delta)))^{\leq R} \cong (\Isigdelrel))^{\leq R'}$. I.e.~there are a combinatorial ball $D$ with $\partial D = \Link_{S}(\Delta)$ and a $\sigma$-regular map 
\begin{equation*}
	\psi\colon  D \to (\Linkhat_{\IAArel}(\phi(\Delta)))^{\leq R}
\end{equation*}
such that $\psi|_{\Link_{S}(\Delta)} = \phi|_{\Linkhat_{S}(\Delta)}$.

Just as before, we can extend this to a map 
\begin{equation*}
	\Psi\colon  B = \Delta \ast D \to (\IAArel)^{\leq R}
\end{equation*}
by setting $\Psi|_{\Delta} = \phi|_{\Delta}$ and $\Psi|_D = \psi|_D$.  Again, $B$ is a combinatorial ball with boundary $\partial B = \Star_{S}(\Delta) \cup (\partial \Delta * D)$ and $\Psi$ agrees with $\phi$ on $\Star_{S}(\Delta)$. 
The image of $\Psi$ contains no skew-additive, 2-skew-additive or $\sigma$-additive simplices. The map $\psi$ is $\sigma$-regular and the isomorphisms in \cref{lem_isom_link_for_making_regular} identify $\sigma$ simplices in $(\Isigdelrel)^{\leq R'}$ with $\sigma^2$ simplices in $(\Link_{\IAArel}(\phi(\Delta)))^{\leq R}$. This implies that $\Psi$ is regular (cf.~\cref{ex_sigma_to_sigma2_regular}). 
Furthermore, as $\phi|_{\Delta}$ is injective, we have $\phi(\partial \Delta) \subseteq \partial \phi(\Delta)$. Hence, $\Psi|_{\partial \Delta \ast D}$ has image in  
\begin{equation*}
	\im(\Psi|_{\partial \Delta \ast D}) \subseteq \partial\phi(\Delta) * (\Linkhat_{\IAArel}(\phi(\Delta)))^{\leq R} \subseteq (\IAAstrel)^{\leq R}.
\end{equation*}
Invoking \cref{lem_regular_homotopy_by_balls}, we see that $\phi$ is regularly homotopic to a simplicial map
\begin{equation*}
	\newproofphi\colon  \newproofsphere \to (\IAArel)^{\leq R}
\end{equation*}
that is obtained by replacing $\phi|_{\Star_{S}(\Delta)}$ by 
\begin{equation*}
	\Psi|_{\partial \Delta \ast D}\colon  \partial \Delta * D \to (\IAAstrel)^{\leq R}.
\end{equation*}
We claim that every non-regular simplex $\Theta$ of $\newproofphi$ is also a non-regular simplex of $\phi$. By \cref{lem_replacement_stars_maps}, it again suffices to consider $\Theta \subseteq \partial \Delta \ast D$. We show that every such simplex is regular. By construction, every simplex $\Theta$ of $\partial \Delta \ast D$ that maps to a $\sigma$ edge must be contained in $D$. 
The map $\Psi|_D$ is equal to $\psi$, which is $\sigma$-regular. So $\psi|_{\Star_{D}(\Theta)}$ is a $\sigma$ cross map. Furthermore, as $\Theta \subseteq D$, we have
\begin{equation*}
\Star_{\partial \Delta \ast D}(\Theta) = \partial \Delta \ast \Star_{D}(\Theta)
\end{equation*}
and $\Psi$ maps $\partial \Delta$ to a symplectic pair $\ls v,w \rs$ that extends the partial symplectic basis given by the vertices of $\psi(\Star_{D}(\Theta))$. Hence, $\newproofphi|_{\Star_{\partial \Delta \ast D}(\Theta)}$ is a $\sigma$ cross map as well, so $\Theta$ is a regular simplex. This implies that $\newproofphi$ has one less non-regular simplex than $\phi$ (namely the simplex $\Delta$ that was removed).
\end{proof}

\subsection{Establishing \texorpdfstring{$\sigma$}{\unichar{"03C3}}-smallness}
\label{sec_separating_cross_polytopes}
We keep the \hyperref[eq_standing_assumption_nmkR]{Standing assumption} that $n\geq 2$, $m\geq 0$, $k\leq n$ and $R\geq 1$.
In this section, we show how to regularly homotope a $\sigma$-regular map $\phi\colon  S^{k}\to \IAAstrel$ to a map $\newlemmaphi$ such that in the image of $\newlemmaphi$, no vertex of rank $R$ is contained in a $\sigma$ edge. In fact, we establish the following slightly stronger condition.
\begin{definition}
\label{def_sigma_small}
Let $S$ be a combinatorial $k$-sphere.
A $\sigma$-regular map $\phi\colon  S\to \IAAstrel$ is called \emph{$\sigma$-small} if for every simplex $\Delta$ of $S$ such that $\phi(\Delta)$ is a $\sigma$ edge, the following properties hold:
\begin{enumerate}
\item The image of the $\sigma$ cross map $\phi|_{\Star_{S}(\Delta)}$ contains at most one vertex of rank $R$.
\item For all $x\in \Delta$, we have $\rkfn(\phi(x))<R$.
\end{enumerate}
\end{definition} 
Recall that as $\phi$ is $\sigma$-regular, the simplex $\Delta$ in the above definition is necessarily an edge.

We establish $\sigma$-smallness with three lemmas in this subsection. In each one, we show how to obtain a map $S^k \to \IAAstrel$ such that every $\Delta$ mapping to a $\sigma$ edge has a certain desired property that brings the map closer to being $\sigma$-small. To do so, we each time remove step by step all ``critical'' simplices, i.e.~those $\Delta$ that do not have the desired property. The following is the first of these lemmas.

\begin{lemma}
\label{lem_avoid_edgy_crosspolytopes}
Let $S$ be a combinatorial $k$-sphere and $\phi\colon  S \to (\IAAstrel)^{\leq R}$ a $\sigma$-regular map. Then $\phi$ is in $\IAArel$ regularly homotopic to a $\sigma$-regular map $\newlemmaphi\colon \newlemmasphere\to (\IAAstrel)^{\leq R}$ such that the following property holds.
  \begin{itemize}
  \item If $\Delta\subset \newlemmasphere$ is a simplex such that $\phi'(\Delta)$ is a $\sigma$ edge, then in the image of the $\sigma$ cross map $\newlemmaphi|_{\Star_{\newlemmasphere}(\Delta)}$, every symplectic pair $\ls v_i, w_i \rs$ satisfies $\rkfn(v_i)< R$ or $\rkfn(w_i)< R$.
  \end{itemize}
\end{lemma}
\begin{proof}
Let $\Delta$ be a simplex such that $\phi(\Delta)$ is a $\sigma$ edge. We keep the notation set up in \cref{sec_sigma_regularity}, so $C = \Star_{S}(\Delta)$, the vertices in $\phi(C)$ are $v_1, w_1, \ldots, v_k, w_k$ and $\phi(\Delta)=\ls v_1,w_1 \rs$ is the $\sigma$-edge contained in $\phi(C)$. 

Assume that $\Delta$ is critical, i.e.~there is $1\leq i \leq k$ such that $\rkfn(v_i)=R=\rkfn(w_i)$. 
Using \cref{lem_flip} to replace the edge $\ls v_1,w_1 \rs$ by $\ls v_i,w_i \rs$, we can assume that $i=1$. 
Choose representatives $\vec v_1$ and $\vec w_1$ such that $\rkfn(\vec v_1) = R = \rkfn(\vec w_1)$ and define $w_1' \coloneqq \langle \vec w_1 - \vec v_1 \rangle$. This is a vertex in $\IAAstrel$ and $\rkfn(w_1') = 0$. 
Let 
\begin{equation*}
	B\coloneqq t \ast C 
\end{equation*}
be the combinatorial $(k+1)$-ball obtained as a join of $C$ with a new vertex $t$.
Define a map 
\begin{equation*}
	\Psi\colon  B \to \IAArel
\end{equation*}
by letting $\Psi|_C = \phi|_C$ and $h(t) \coloneqq w_1'$.
We claim that $\Psi$ is well-defined and regular: It follows from the description of the simplicial structure of $C$ at the beginning of this subsection that $B = t \ast C $ is a union of $2^{k-1}$ simplices of dimension $(k+1)$, each of which gets mapped to 
\begin{equation*}
	\ls v_1, w_1, w_1' = \langle \vec w_1 - \vec v_1 \rangle \rs \cup \ls v_i \mid i\in I\rs \cup \ls w_j \mid j \in J \rs
\end{equation*} 
for some disjoint sets $I,J$ such that $I\cup J = \ls 2, \ldots, k \rs$.
See \cref{figure_separating_cross_polytope1} for a low-dimensional picture.
\begin{figure}
\begin{center}
\includegraphics{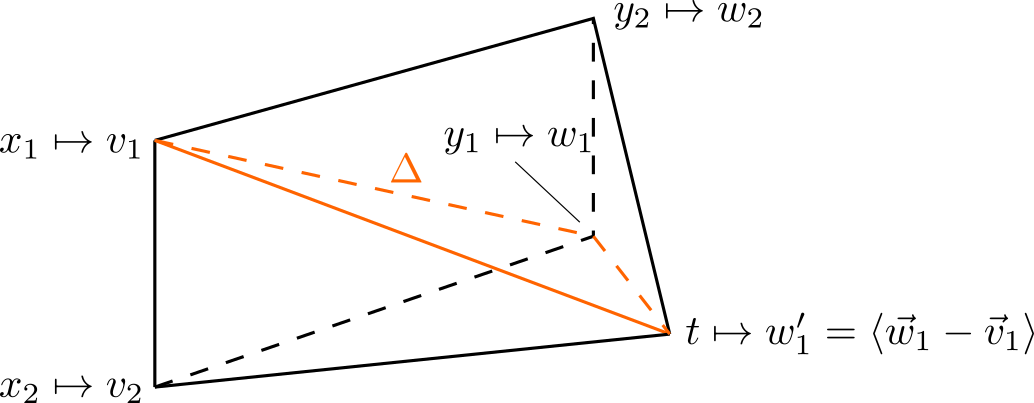}
\end{center}
\caption{The complex $B = t \ast C$ for the $i=1$ case of \cref{lem_avoid_edgy_crosspolytopes}. Here, $k=2$ and orange marks edges that are mapped to $\sigma$ simplices.}
\label{figure_separating_cross_polytope1}
\end{figure}
It is easy to check that these are all $\sigma$-additive simplices. It follows that $\Psi$ has image in $\IAArel$, is a $\sigma$-additive cross map and hence is indeed regular. By \cref{lem_boundary_join}, the boundary of $B$ decomposes as
\begin{equation*}
	\partial B = (t \ast \partial C  ) \cup C
\end{equation*}
Both $t \ast \partial C$ and $C$ are combinatorial $k$-balls and their intersection is given by 
\begin{equation*}
	\partial (t \ast \partial C )  = \partial C. 
\end{equation*}
Using the description of $\partial C$ at the beginning of this subsection, we have $\Psi(t \ast \partial C  ) \subset (\IAAstrel)^{\leq R}$.
Hence by \cref{lem_regular_homotopy_by_balls}, $\phi$ is regularly homotopic to a simplicial map 
\begin{equation*}
	\newproofphi\colon \newproofsphere \to (\IAAstrel)^{\leq R}
\end{equation*}
that is obtained by replacing $\phi|_{\Star_{S}(\Delta)}$ by $\Psi|_{t \ast \partial C}$.

Using the description of $\partial C$ above, one sees that $t \ast \partial C  $ is the union of $\Star_{t \ast \partial C} (\ls x_1, t \rs)$ and $\Star_{t \ast \partial C} (\ls y_1, t \rs)$, where $\phi(x_1)=v_1$ and $\phi(y_1) = w_1$. The restriction of $\Psi$ to each of those is a $\sigma$ cross map and the vertices in their image are $v_1, w_1', v_2, w_2, \ldots, v_k, w_k$ and $w_1, w_1', v_2, w_2, \ldots, v_k, w_k$, respectively. In particular, both of these images contain one less symplectic pair with two vertices of rank $R$ than $\phi(C)$. 

Hence, replacing $\phi$ with $\newproofphi$ removes the critical simplex $\Delta$ from $\phi$ and only produces critical simplices that have less symplectic pairs with two vertices of rank $R$ in the images of their stars.
Iterating this procedure yields a map $\newlemmaphi$ with the desired property.
\end{proof}

\begin{lemma}
\label{lem_avoid_edgy_crosspolytopes2}
Let $S$ be a combinatorial $k$-sphere and $\phi\colon  S \to (\IAAstrel)^{\leq R}$ a $\sigma$-regular map. Then $\phi$ is in $\IAArel$ regularly homotopic to a $\sigma$-regular map $\newlemmaphi\colon \newlemmasphere\to (\IAAstrel)^{\leq R}$ such that the following property holds.
\begin{itemize}
  \item If $\Delta\subset \newlemmasphere$ is a simplex such that $\newlemmaphi(\Delta)$ is a $\sigma$ edge, then in the image of the $\sigma$ cross map $\newlemmaphi|_{\Star_{\newlemmasphere}(\Delta)}$, there is at most one vertex of rank $R$.
  \end{itemize}
\end{lemma}
\begin{proof}

Again, let $\Delta$ be a simplex such that $\phi(\Delta)$ is a $\sigma$ edge and keep the notation from \cref{sec_sigma_regularity}. So $C = \Star_{S}(\Delta)$ has vertices $x_1, y_1, \ldots, x_k, y_k$ and their images are $v_1, w_1, \ldots, v_k, w_k$, where $\phi(\Delta)=\ls v_1,w_1 \rs$. 
Using \cref{lem_avoid_edgy_crosspolytopes}, we can assume that $\rkfn(w_i) < R$ for all $i$.
Now assume that $\Delta$ is critical, i.e.~there are $1\leq i \not = j \leq m$ such that $\rkfn(v_i) = \rkfn(v_j) = R$.  
Using \cref{lem_flip} to replace the edge $\ls v_1,w_1 \rs$ by $\ls v_i,w_i \rs$, we can assume that $i=1$ and $j=2$. 
We will show how to remove $C$ from $\phi$ while only creating $\sigma$ cross maps that have less vertices of rank $R$ in their image than $\phi|_C$.

The set $\ls v_1,v_2 \rs$ is a standard edge in $\phi(C)$ both of whose vertices have rank $R$.
Choose representatives $\vec v_1, \vec v_2$ such that $\rkfn(\vec v_1) = R = \rkfn(\vec v_2)$ and $\vec w_1, \vec w_2$ such that $\omega(\vec v_1, \vec w_1) = 1 = \omega(\vec v_2, \vec w_2)$. 
Let 
\begin{equation}
\label{eq_avoid_edgy_crosspolytopes2_def_vi_wi}
	\vec v \coloneqq \vec v_1 - \vec v_2  \text{ and } \vec w \coloneqq \vec w_1+ \vec w_2.
\end{equation}
Both $v$ and $w$ are vertices in $\IAAstrel$ and $\omega(v,w) = 0$. We have $\rkfn(v) = 0$. 

\paragraph*{Step 1:}
We first show that we can assume that 
\begin{equation}
\label{eq_avoid_edgy_crosspolytopes2_sum_w1w2}
	\rkfn(w) = \rkfn(\langle \vec w_1 + \vec w_2\rangle)<R.
\end{equation}
Assume that \cref{eq_avoid_edgy_crosspolytopes2_sum_w1w2} does not hold. Then $1 \leq \rkfn(w_1)$, $\rkfn(w_2) < R$ and
\begin{equation*}
	R \leq \rkfn(\langle \vec w_1 + \vec w_2\rangle) < 2R.
\end{equation*}
As $\rkfn(v_1)= R$, this implies that for some $\epsilon\in \ls -1,1 \rs$,
\begin{equation*}
	\rkfn(\langle \vec w_1 + \vec w_2 + \epsilon \vec v_1\rangle) < R.
\end{equation*}
Define $\vec w_1' \coloneqq \vec w_1 + \epsilon \vec v_1$. We then also have $\rkfn(w_1')<R$. 
Let $B\coloneqq t \ast C$ and define 
\begin{equation*}
	\Psi\colon  B = t \ast C \to \IAArel
\end{equation*}
by setting $\Psi|_C = \phi|_C$ and $\Psi(t) \coloneqq w_1'$. It follows as in the proof of \cref{lem_avoid_edgy_crosspolytopes} that $\Psi$ is a $\sigma$-additive cross map and in particular regular. 

Hence by \cref{lem_regular_homotopy_by_balls}, $\phi$ is regularly homotopic to a simplicial map 
\begin{equation*}
	\newproofphi\colon \newproofsphere \to (\IAAstrel)^{\leq R}
\end{equation*}
that is obtained by replacing $\phi|_C$ by $\Psi|_{t \ast \partial C}$. 

We need to verify that replacing $\phi$ by $\newproofphi$ only creates cross maps that are ``better'' than $\phi|_C$. By \cref{lem_replacement_stars_maps}, it suffices to consider cross maps in $\Psi|_{t \ast \partial C}$.
For this, first note that the image of $\Psi|_{t \ast \partial C}$ contains two $\sigma$ edges, namely $\Psi(\ls x_1,t \rs)= \ls v_1, w_1' \rs$ and $\Psi(\ls y_1,t \rs)= \ls w_1, w_1' \rs$  (see \cref{figure_separating_cross_polytope1}). Hence, there are also two corresponding $\sigma$ cross maps. The vertices in the images of these $\sigma$ cross maps are $v_1, w_1', v_2, w_2, \ldots, v_k, w_k$ and $w_1, w_1', v_2, w_2, \ldots, v_k, w_k$, respectively. In particular, as $\rkfn(w_1')<R$, these still satisfy the property that at least one vertex of every symplectic pair has rank less than $R$ (this is the condition we achieved in \cref{lem_avoid_edgy_crosspolytopes}; it shows that the cross maps are not ``worse'' than $\phi|_C$ in this sense). 
As both $w_1$ and $w_1'$ have rank less than $R$, the $\sigma$ cross map in $\Psi|_{\Star_{t \ast \partial C}}$ that corresponds to the edge $\ls y_1,t \rs$ has less vertices of rank $R$ in its image than $\phi|_C$, so it is ``better'' in this sense.
The other $\sigma$ cross map in $\Psi|_{\Star_{t \ast \partial C}}$ corresponds to the edge $\ls x_1,t \rs$. It  has the same number of rank $R$ vertices in its image as $\phi|_C$. However, the vertices in the image of this map are $v_1, w_1', v_2, w_2, \ldots, v_k, w_k$. Here, we still have $\rkfn(\langle \vec v_1- \vec v_2 \rangle) = 0$, but also 
\begin{equation*}
	\rkfn(\langle \vec w_1'+\vec  w_2\rangle) = \rkfn(\langle \vec w_1 + \vec w_2 + \epsilon \vec v_1\rangle) < R.
\end{equation*}
Hence, this cross map satisfies the condition of \cref{eq_avoid_edgy_crosspolytopes2_sum_w1w2}, which was violated by $\phi|_C$ and it is ``better'' in this sense.

\paragraph{Step 2:}
This allows us to assume that \cref{eq_avoid_edgy_crosspolytopes2_sum_w1w2} holds, i.e.~$\rkfn(w)<R$ for $w = \langle \vec w_1+ \vec w_2 \rangle $ as defined in \cref{eq_avoid_edgy_crosspolytopes2_def_vi_wi}.
Let $t$ and $s$ be new vertices and define a combinatorial 3-ball $B$ that looks as follows: It has six vertices $x_1,x_2, y_1, y_2, t, s $ and is the union of three 3-simplices, namely $\ls x_1, x_2, y_1, t \rs$, $\ls x_1, y_1, y_2, t\rs$ and $\ls y_1, y_2, t, s \rs$, see \cref{figure_separating_cross_polytope2}.
\begin{figure}
\begin{center}\vspace{40px}
    \begin{picture}(240,120)
    \put(0,10){\includegraphics[scale=.8]{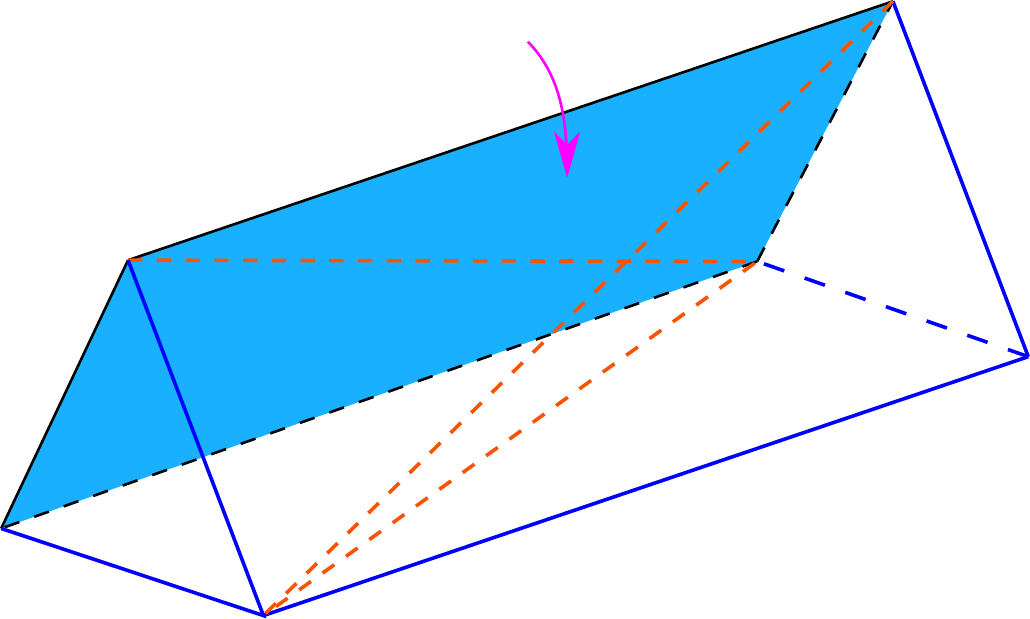}}
    
     \put(-10,97){$x_{1}\mapsto v_1$}
     \put(45,2){$t\mapsto v$}
     \put(-40,30){$x_2\mapsto v_2$}     
    
    \put(200,160){$y_{2}\mapsto w_2$}
     \put(180,95){$y_1\mapsto w_1$}
     \put(245,70){$s\mapsto w$}   
     
    \put(110,150){$D_1$}      
        \put(100,85){$\Delta$}      
    
     \end{picture}
\end{center}
\caption{The complex $B$ used in \cref{lem_avoid_edgy_crosspolytopes2} after assuming that $\rkfn(w)<R$. Orange marks edges that are mapped to $\sigma$ simplices. The subcomplex $D_1$ of the boundary is marked, the 2-simplices not contained in $D_1$ are those that form $D_2$.}
\label{figure_separating_cross_polytope2}
\end{figure}
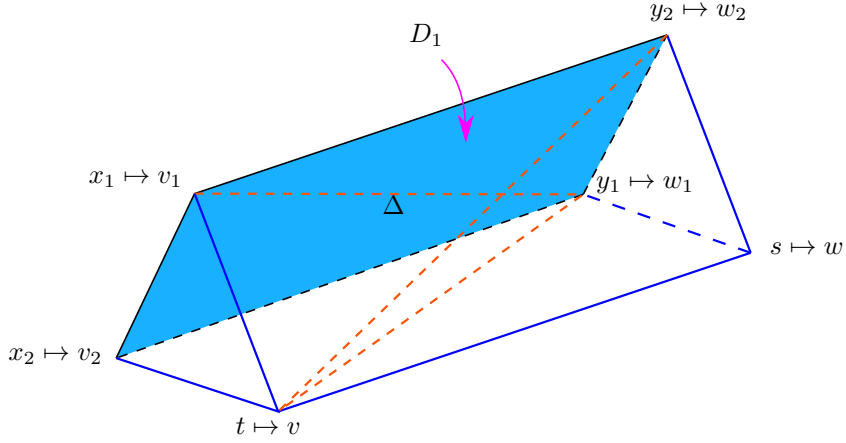
Let 
\begin{equation*}
	\tilde{B} \coloneqq B \ast C_{k-2}
\end{equation*}
be the combinatorial $(k+1)$-ball that is obtained by joining $B$ with
\begin{equation*}
	C_{k-2} \coloneqq \ast_{i=3}^k \ls x_i, y_i\rs \cong S^{k-3}.
\end{equation*} 
The vertex set of $\tilde{B}$ consists of the vertices of $C = \Star_{S}(\Delta)$ and the newly added vertices $t, s$. 

We define a map 
\begin{equation*}
	\Psi\colon \tilde{B}\to \IAArel
\end{equation*}
on the vertices by letting $\Psi|_{C} = \phi|_C$, $\Psi(t) = v$ and $\Psi(s) = w$ for $v$ and $w$ as defined in \cref{eq_avoid_edgy_crosspolytopes2_def_vi_wi}.
Using the description of $C$ at the beginning of this subsection, the maximal simplices of $\tilde{B}$ are all of the form
\begin{equation*}
\Theta \cup \ls x_i \mid i\in I\rs \cup \ls y_j \mid j \in J \rs,
\end{equation*} 
where $\Theta$ is a maximal simplex of $B$ and $I,J$ are disjoint sets such that $I\cup J = \ls 3, \ldots, k \rs$. The images of these simplices look as follows:
\begin{multline*}
\Psi(\ls x_1,x_2, y_1, t \rs \cup \ls x_i \mid i\in I\rs \cup \ls y_j \mid j \in J \rs) = \\
	\ls v_1,v_2, w_2, v = \langle \vec v_1 - \vec v_2 \rangle \rs \cup \ls v_i \mid i\in I\rs \cup \ls w_j \mid j \in J \rs
\end{multline*}
--- 2-skew-additive simplices;
\begin{multline*}
\Psi(\ls x_1, y_1, y_2, t\rs \cup \ls x_i \mid i\in I\rs \cup \ls y_j \mid j \in J \rs) = \\
	\ls v_1, w_1, w_2, v = \langle \vec v_1 - \vec v_2 \rangle \rs \cup \ls v_i \mid i\in I\rs \cup \ls w_j \mid j \in J \rs
\end{multline*}
--- skew-$\sigma^2$ simplices;
\begin{multline*}
\Psi(\ls y_1, y_2, t, s \rs \cup \ls x_i \mid i\in I\rs \cup \ls y_j \mid j \in J \rs) = \\
	\ls w_1, w_2, v = \langle \vec v_1 - \vec v_2 \rangle, w  = \langle \vec w_1 + \vec w_2 \rangle \rs \cup \ls v_i \mid i\in I\rs \cup \ls w_j \mid j \in J \rs)
\end{multline*}
--- 2-skew-additive simplices.
This shows that the definition we gave for $\Psi$ on the vertices actually defines a simplicial map with image in $\IAArel$.
 It also shows that this map is regular, namely a prism cross map.

The boundary of $B$ is a union of the subcomplex $D_1$ spanned by $\ls x_1, y_1, x_2, y_2 \rs$, which is a union of two 2-dimensional simplices that are both mapped to $\sigma$ simplices, and a subcomplex $D_2$ consisting of six 2-dimensional simplices, four of which are mapped to $\sigma$ simplices and two of which are mapped to 2-additive simplices (see \cref{figure_separating_cross_polytope2}). We have $C = \Star_{S}(\Delta) = D_1 \ast C_{k-2}$, so using \cref{lem_boundary_join}, the boundary of $\tilde{B}$ decomposes as
\begin{equation*}
	\partial \tilde{B} = \partial B \ast C_{k-2}  = D_1 \ast C_{k-2} \cup(D_2 \ast C_{k-2}) =  C \cup (D_2 \ast C_{k-2}).
\end{equation*}
Both $C = D_1 \ast C_{k-2}$ and $D_2 \ast C_{k-2}$ are combinatorial $k$-balls and their intersection is given by
\begin{equation*}
	 \partial C = (\partial D_1) \ast C_{k-2} = (\partial D_2) \ast C_{k-2} = \partial (D_2 \ast C_{k-2}). 
\end{equation*}

Hence by \cref{lem_regular_homotopy_by_balls}, $\phi$ is regularly homotopic to a simplicial map 
\begin{equation*}
	\newproofphi\colon \newproofsphere \to \IAAstrel
\end{equation*}
that is obtained by replacing $\phi|_C$ by $\Psi|_{D_2 \ast C_{k-2}}$. 

From the description of $D_2$ given above, it is easy to verify that $\Psi|_{D_2 \ast C_{k-2}}$ has image in $(\IAAstrel)^{\leq R}$, so the same is true for $\newproofphi$.
We need to verify that $\newproofphi$ is still $\sigma$-regular and replacing $\phi$ by $\newproofphi$ only creates cross maps that are ``better'' than $\phi|_C$. By \cref{lem_replacement_stars_maps}, it suffices to consider cross maps in $\Psi|_{D_2 \ast C_{k-2}}$.
The complex $D_2 \ast C_{k-2}$ contains two simplices mapping to $\sigma$ edges, namely $\ls y_1, t \rs$, which maps to $\ls w_1, v \rs$, and $\ls y_2, t \rs$, which maps to $\ls w_2, v \rs$. The restriction of $\Psi$ to the stars of these simplices in $D_2 \ast C_{k-2}$ is a $\sigma$ cross map. This implies that $\Psi|_{D_2 \ast C_{k-2}}$, and hence $\newproofphi$, is $\sigma$-regular. Furthermore, these maps still satisfy the property of \cref{lem_avoid_edgy_crosspolytopes}, i.e.~at most one element of each symplectic pair in the image has rank $R$. But now all vertices of the $\sigma$ edges $\ls w_1, v \rs$ and $\ls w_2, v \rs$ have rank less than $R$.
This implies that both $\Psi|_{\Star_{D_2 \ast C_{k-2}}(\ls y_1, t  \rs)}$ and $\Psi|_{\Star_{D_2 \ast C_{k-2}}(\ls y_2, t \rs)}$ have less vertices of rank $R$ in their image than $\phi|_C$. 

So replacing $\phi$ by $\newproofphi$, we get closer to obtain the removing all critical simplices. Iterating this procedure yields the desired map $\newlemmaphi$.
\end{proof}

We are now ready to show the final result of this subsection.

\begin{lemma}
\label{lem_sigma_small_obtained}
Let $S$ be a combinatorial $k$-sphere and $\phi\colon  S \to (\IAAstrel)^{\leq R}$ a $\sigma$-regular map. Then $\phi$ is in $\IAArel$ regularly homotopic to a $\sigma$-regular map $\newlemmaphi\colon \newlemmasphere\to (\IAAstrel)^{\leq R}$ such that the following property holds.
\begin{itemize}
  \item $\newlemmaphi$ is $\sigma$-small.
  \end{itemize}
\end{lemma}
\begin{proof}
We can assume that $\phi$ satisfies the property of \cref{lem_avoid_edgy_crosspolytopes2}, i.e.~the first point in \cref{def_sigma_small}. 
Let $\Delta = \ls x_1, y_1 \rs$ be a simplex such that $\phi(\Delta) = \ls v_1,w_1 \rs$ is a $\sigma$ edge. We keep again the notation from \cref{sec_sigma_regularity} and write $C = \Star_{S}(\Delta)$, denote its vertices by $x_1, y_1, \ldots, x_k, y_k$ and their images by $v_1, w_1, \ldots, v_k, w_k$.

Assume that $\Delta$ is critical, i.e.~its image $\ls v_1,w_1 \rs$ contains (exactly) one vertex of rank $R$.
As $\phi$ satisfies the first point of \cref{def_sigma_small}, both $\rkfn(v_2)$ and $\rkfn(w_2)$ are less than $R$ (this uses the \hyperref[eq_standing_assumption_nmkR]{Standing assumption} that $n\geq 2$).
Now using \cref{lem_flip}, we find a regular homotopy that replaces $\ls v_1, w_1 \rs$ by $\ls v_2, w_2 \rs$. This modifies $\phi|_C$ such that the result is a $\sigma$ cross map $\Psi|_{C'}$ whose image contains the $\sigma$ edge $\ls v_2, w_2 \rs$. As both of these vertices have rank less than $R$, this removes a critical simplex of $\phi$ without creating a new one. As $\phi|_C$ and $\Psi|_{C'}$ agree on the vertex set, this preserves the first point in \cref{def_sigma_small}. Iterating this, we obtain the desired map $\phi'$.
\end{proof}

\subsection{Removing edgy simplices}
\label{sec_edgy_simplices}
We keep the \hyperref[eq_standing_assumption_nmkR]{Standing assumption} that $n\geq 2$, $m\geq 0$, $k\leq n$ and $R\geq 1$.
We now prove \cref{prop_isolating-rank-r-vertices}. This is done in several steps. Much of it is similar to \cite[Section 6]{Brueck2022}. For those parts where the proofs carry over almost verbatim, we provide outlines here and refer the reader to \cite[Section 6]{Brueck2022} for more details.

\begin{definition}
\label{def_edgy_simplices}
Let $S$ be a combinatorial $k$-sphere and $\phi \colon  S  \to \IAAstrel$ a simplicial map.
A simplex $\Delta$ of $S$ is called \emph{edgy} if $\phi(\Delta)=\{ v_0, v_1\}$, $v_0\not = v_1$, is an edge such that $\rkfn(v_0) = \rkfn(v_1) = R$.
\end{definition}
\noindent 
This coincides with the definition of edgy simplices in \cite[Definition 6.1]{Brueck2022}.
A necessary condition to make sure that bad vertices are isolated in the sense of \cref{prop_isolating-rank-r-vertices} is that $\phi \colon  S  \to \IAAstrel$ has no edgy simplices. Assuring that this holds is our main aim before we before we prove   \cref{prop_isolating-rank-r-vertices} at the end of this subsection.

It is not hard to see that if $\Delta$ is edgy, then $\phi(\Delta)=\ls v_0, v_1\rs$ is either a standard simplex, a 2-additive simplex where $\vec v_0 = \vec v_1 \pm \vec e_i$ for some $1 \leq i \leq m$, a 3-additive $\vec v_0 = \vec v_1 \pm \vec e_i \pm \vec e_j$ for some $1 \leq i \not = j \leq m$ or a $\sigma$ simplex where $\omega(v_0,v_1) = \pm 1$ (cf.~\cref{def:IAAst}).

In the previous \cref{sec_separating_cross_polytopes}, we already showed how to remove edgy simplices from $\phi$ whose image is a $\sigma$ simplex.
The remaining three cases all corresponds to simplex types that are also present in the complex $\BAA$ (cf.~\cref{def:BAA}). This allows us to closely follows the arguments in \cite[Section 6, in particular Step 1 on page p.53 et seq.]{Brueck2022}.
Just as in the setting of that article, we need to further control the stars of such  edgy simplices before we can remove these. More precisely, we need to make sure that there are no simplices of the following type:

\begin{definition}
\label{def_overly_augmented}
A simplex $\Delta$ of $S$ is called \emph{overly augmented}, if 
\begin{itemize}
\item $\phi(\Delta)$ is a 3-additive, double-triple or double-double simplex,
\item every vertex of $\phi(\Delta)$ has rank $R$ or is contained in the augmentation core,
\item $\Delta$ contains at least one vertex $x$ such that $\rkfn(\phi(x)) = R$,
\item if $\phi(\Delta)$ is $3$-additive, then for all $v_0 \in \phi(\Delta) $ of rank $\rkfn(v_0) = R$, there does \emph{not} exist $v_1\in \phi(\Delta) $  and $1 \leq i \not= j \leq m$ such that $\vec v_0 = \vec v_1 \pm \vec e_i \pm \vec e_j $.
\end{itemize}
\end{definition}
\noindent This is the analogue of \cite[Definition 6.2]{Brueck2022}.

In order to remove overly augmented simplices from $\phi$, we will apply the following lemma several times:
\begin{lemma}
\label{lem_avoid_overly_augmented}
Let $S$ be a combinatorial $k$-sphere and $\phi\colon  S \to (\IAAstrel)^{\leq R}$ a simplicial map. Then $\phi$ is in $\IAArel$ regularly homotopic to a map $\newlemmaphi\colon \newlemmasphere\to (\IAAstrel)^{\leq R}$  such that the following properties hold.
\begin{enumerate}
\item If $\Delta\subset \newlemmasphere$ is an edgy simplex of $\newlemmaphi$, then $\Delta$ is a simplex of $S$ and $\phi|_\Delta = \newlemmaphi|_\Delta$. In particular, every edgy simplex of $\newlemmaphi$ is also an edgy simplex of $\phi$. 
\item If $\Delta \subset \newlemmasphere$ is a simplex such that $\newlemmaphi(\Delta)$ is a $\sigma$ simplex, then $\Delta$ is a simplex of $S$ and $\phi|_{\Star_{S}(\Delta)} = \newlemmaphi|_{\Star_{\newlemmasphere}(\Delta)}$. In particular, if $\phi$ is $\sigma$-regular or $\sigma$-small, then so is $\newlemmaphi$.
\item The map $\newlemmaphi$ has no overly augmented simplices.
\end{enumerate}
\end{lemma}
\begin{proof}

The proof works very similarly to the one of \cite[Section 6, Procedure 1]{Brueck2022}. The idea is as follows:
First define a measure of ``badness'' for overly augmented simplices. In \cite[Definition 6.2]{Brueck2022}, this measure is given by three integers $a,b,c$. Then successively remove overly augmented simplices, starting with the ``worst'' ones.

Let $\Delta$ be $(a,b,c)$-over augmented in the sense of \cite[Definition 6.2]{Brueck2022}, with $(a,b,c)$ as large as possible (in the lexicographical order). In order to remove $\Delta$ from $\phi$, one modifies $\phi|_{\Star_{S^k}(\Delta)}$ such that the image of the result is contained in $\phi(\partial \Delta) \ast K(\Delta)$, where $K(\Delta)$ is a certain subcomplex of $\IAAstrel$ whose vertices have better properties than those of $\phi(\Delta)$.

The complex $K(\Delta)$ is defined as follows: If $\phi(\Delta)$ is a double-triple or double-double simplex, we have $\phi(\Delta) = \ls v_0, \ldots, v_l \rs$, where $\ls v_2, \ldots, v_l \rs$ is a standard simplex. Define
\begin{equation*}
	K(\Delta) \coloneqq \Link^{<R}_{\Irel}(\ls v_2, \ldots, v_l \rs).
\end{equation*}
If $\phi(\Delta)$ is 3-additive, we can write $\phi(\Delta) = \ls v_0, \ldots, v_l \rs$, where $\ls v_1, \ldots, v_l \rs$ is a standard simplex and $\vec v_0 = \vec w_1 + \vec w_2 + \vec w_3$ for $w_1, w_2, w_3\in \ls v_1, \ldots, v_l, e_1, \ldots, e_m \rs$. Let $J^<$ be the set of those lines in $\ls \langle \vec w_1 + \vec w_2 \rangle, \langle \vec w_1 +  \vec w_3 \rangle, \langle \vec w_2 + \vec w_3 \rangle \rs$ that have rank less than $R$. This set is non-empty by the last condition in \cref{def_overly_augmented}. Define
\begin{equation*}
	K(\Delta) \coloneqq J^{<R} \ast \Link^{<R}_{\Irel}(\ls v_1, \ldots, v_l \rs).
\end{equation*}

Similarly to \cite[Claim 6.4]{Brueck2022}, one can check that $K(\Delta)$ is a subcomplex of $\Link_{\IAAstrel}(\phi(\Delta))$ and that $\phi(\Link_{S}(\Delta)) \subseteq K(\Delta)$.
Next, one verifies the analogue of \cite[Claim 6.5]{Brueck2022}, namely that $K(\Delta)$ is $(\dim \Link_{S}(\Delta))$-connected. Using that $J^{<R}\not= \emptyset$, this is a consequence of \cref{lem_connectivity_Isig_link}. 

Just as in \cite{Brueck2022}, one can now use $K(\Delta)$ to replace $\phi|_{\Star_{S}(\Delta)}$ by a map with image in $\phi(\partial \Delta) \ast K(\Delta)$ (this is again a cut out argument similar to the proof of \cref{prop_cut_out_bad_vertices}). It follows from \cref{lem_regular_homotopy_by_balls} that this defines a regular homotopy. (In fact, the entire homotopy takes place in $\IAAstrel$).  Following \cite[Claim 6.5]{Brueck2022}, one can verify that the result has less overly augmented simplices than $\phi$ and that no new edgy simplices are introduced in the process.

This replacement takes place on $\Star_{S}(\Delta)$, where $\phi(\Delta)$ is a double-triple, double-double or 3-additive simplex. As no such simplex contains a $\sigma$ simplex in its star, the process does not affect the stars of simplices of $S$ that map to $\sigma$ simplices.
Iterating this procedure removes all overly augmented simplices from $\phi$, which proves the claim.
\end{proof}

We now begin to remove simplices that are edgy in the sense of \cref{def_edgy_simplices}. We start with edgy simplices with 3-additive image, i.e.~simplices $\Delta$ such that $\phi(\Delta) = \ls v_0, v_1 \rs$, where $\vec v_0 = \vec v_1 \pm \vec e_i \pm \vec e_j$ for some $1 \leq i \not = j \leq m$.

\begin{lemma}
\label{lem_avoid_edgy_3_additive}
Let $S$ be a combinatorial $k$-sphere and $\phi\colon  S \to (\IAAstrel)^{\leq R}$ a map that is $\sigma$-regular and $\sigma$-small. Then $\phi$ is in $\IAArel$ regularly homotopic to a map $\newlemmaphi\colon \newlemmasphere\to (\IAAstrel)^{\leq R}$ such that the following properties hold.
  \begin{enumerate}
  \item $\newlemmaphi$ is $\sigma$-regular.
  \item $\newlemmaphi$ is $\sigma$-small.
  \item $\newlemmaphi$ has no edgy simplices with 3-additive image.
  \end{enumerate}
\end{lemma}
\begin{proof}
The proof of this lemma is entirely parallel to the one of \cite[Section 6, Step 1.1]{Brueck2022}, so we are brief here. By \cref{lem_avoid_overly_augmented}, we can assume that $\phi$ has no overly augmented simplices. Let $\Delta$ be an edgy simplex such that $\phi(\Delta) = \ls v_0, v_1 \rs$ is 3-additive. There are representatives $\vec v_0$ and $\vec v_1$ with $\rkfn(\vec v_0) = R = \rkfn(\vec v_1)$ and $\vec v_0 = \vec v_1 \pm \vec e_i \pm \vec e_j$ for some $1 \leq i \not = j \leq m$. Set $v\coloneqq \langle \vec v_1 \pm \vec e_i \rangle$, such that $\{v_0, v_1, v\}$ is a double-triple simplex.
Using that $\phi$ has no overly augmented simplices, one can verify that $\phi(\Star_{S}(\Delta))$ is contained in $\Star_{\IAAstrel}(\ls v_0, v_1, v \rs)$.

Now alter $S$ on $\Star_{S}(\Delta)$ by a adding a new point $t$ to the barycentre of $\Delta$ and subdividing the simplices in $\Star_{S}(\Delta)$ accordingly (as described in \cite[Section 6, Step 1.1]{Brueck2022}). The result is a combinatorial sphere $\newproofsphere$, see \cref{rem_subdivision_combinatorial_manifolds}. We define a map $\newproofphi$ by sending $t$ to $v$. The fact that $\phi(\Star_{S}(\Delta))\subseteq \Star_{\IAAstrel}(\ls v_0, v_1, v \rs)$ implies that this gives indeed a simplicial map $\newproofphi\colon  \newproofsphere \to \IAAstrel$ that is regularly homotopic in $\IAAstrel$ to $\phi$ by \cref{lem_regular_homotopy_by_balls}. (Again, the homotopy has image in $\IAAstrel$.) This process removes the edgy simplex $\Delta$ without introducing new edgy simplices with 3-additive image. 
It preserves $\sigma$-regularity (\cref{def_regular_maps}) and $\sigma$-smallness (\cref{def_sigma_small}) because no $\sigma$ simplex is contained in the star of the 3-additive simplex $\phi(\Delta)$.
The process might have introduced new overly augmented simplices, but using \cref{lem_avoid_overly_augmented}, these can again be removed without introducing new edgy simplices. Iterating this procedure removes all edgy simplices with 3-additive image from $\phi$.
\end{proof}

\begin{lemma}
\label{lem_avoid_edgy_standard}
Let $S$ be a combinatorial $k$-sphere and $\phi\colon  S \to (\IAAstrel)^{\leq R}$ a map that is $\sigma$-regular and $\sigma$-small. Then $\phi$ is in $\IAArel$ regularly homotopic to a map $\newlemmaphi\colon \newlemmasphere\to (\IAAstrel)^{\leq R}$ such that the following properties hold.
  \begin{enumerate}
  \item $\phi'$ is $\sigma$-regular.
  \item $\phi'$ is $\sigma$-small.
  \item $\phi'$ has no edgy simplices with 3-additive image.
  \item $\phi'$ has no edgy simplices with standard image.
  \end{enumerate}
\end{lemma}
\begin{proof}
The proof works similarly to that of \cite[Section 6, Step 1.2]{Brueck2022}. However, in contrast to the proofs of \cref{lem_avoid_overly_augmented} and \cref{lem_avoid_edgy_3_additive}, it also uses the assumptions that $\phi$ be $\sigma$-regular (\cref{def_regular_maps}) and $\sigma$-small (\cref{def_sigma_small}). This is why we give a few more details here than in the previous proofs.

Let $\Delta$ be an edgy simplex of $S$ such that $\phi(\Delta) = \ls v_0,v_1 \rs$ is a standard simplex. Choose representatives $\vec v_0, \vec v_1$ such that $\rkfn(\vec v_0) = R = \rkfn(\vec v_1)$ and let $v \coloneqq \langle \vec v_0 - \vec v_1 \rangle$. This is a vertex in $\IAAstrel$ and $\rkfn(v) = 0$. 

If $\tilde{\Delta} \supseteq \Delta$ is a simplex containing $\Delta$, then $\phi(\tilde{\Delta})$ contains the vertices $v_0$ and $v_1$, which have rank $R$. As $\phi$ has no overly augmented simplices and no edgy simplices with 3-additive image, it follows that $\phi(\tilde{\Delta})$ cannot be a 3-additive, double-triple or double-double simplex (see \cite[Section 6, Step 1.2]{Brueck2022}). As $\phi$ is $\sigma$-regular, $\phi(\tilde{\Delta})$ also cannot be a mixed simplex (see the description of $\sigma$-regular maps at the beginning of \cref{sec_separating_cross_polytopes}). Lastly, as $\phi$ is $\sigma$-small, the first condition of \cref{def_sigma_small} implies that $\phi(\tilde{\Delta})$ cannot be a $\sigma$ simplex. Hence $\phi(\Delta)$ has to be a standard or $2$-additive simplex.

Define $B$ to be the combinatorial $(k+1)$-ball 
\begin{equation*}
	B \coloneqq t \ast \Star_{S}(\Delta)
\end{equation*}
that is obtained by coning off $\Star_{S}(\Delta)$ with a new vertex $t$. Define 
\begin{equation*}
	\Psi\colon  B \to (\IAAstrel)^{\leq R}
\end{equation*}
by setting $\Psi|_{\Star_{S}(\Delta)} = \phi|_{\Star_{S}(\Delta)}$ and $\Psi(t) = v$. To see that $\Psi$ gives a well-defined map with image in $\IAAstrel$, note that by the previous paragraph, every $\tilde{\Delta}$ in $ \Star_{S}(\Delta)$ gets mapped to a standard or 2-additive simplex. In either case, it forms a simplex with $v$. These simplices are either 2-additive, double-triple or double-double simplices, so they are contained in $\IAAstrel$.
In particular, $\Psi$ is regular.
By \cref{lem_boundary_join}, the boundary of $B$ decomposes as
\begin{equation*}
	\partial B = \Star_{S}(\Delta) \cup (t \ast \partial \Star_{S}(\Delta))  = \Star_{S}(\Delta) \cup  (t \ast \Link_{S}(\Delta) \ast \partial \Delta ).
\end{equation*}
By \cref{cor_boundary_star}, both $\Star_{S}(\Delta)$ and $t \ast \Link_{S}(\Delta) \ast \partial \Delta$  are combinatorial $k$-balls and their intersection is given by
\begin{equation*}
	\partial \Star_{S}(\Delta) =  \Link_{S}(\Delta) \ast \partial \Delta = \partial (t \ast \Link_{S}(\Delta) \ast \partial \Delta ). 
\end{equation*}
Hence by \cref{lem_regular_homotopy_by_balls}, $\phi$ is regularly homotopic to a simplicial map 
\begin{equation*}
	\newproofphi\colon \newproofsphere \to (\IAAstrel)^{\leq R}
\end{equation*}
that is obtained by replacing $\phi|_{\Star_{S}(\Delta)} = \Psi|_{\Star_{S}(\Delta)}$ by $\Psi|_{(t \ast \Link_{S}(\Delta) \ast \partial \Delta )}$. 
As $\Psi(t \ast \Link_{S}(\Delta) \ast \partial \Delta )$ contains no $\sigma$ simplices, the map $\newproofphi$ is still $\sigma$-regular and $\sigma$-small. As $\Psi(t \ast \Link_{S}(\Delta) \ast \partial \Delta )$ contains no 3-additive simplices and $\phi$ had no edgy simplices with 3-additive image, neither does $\newproofphi$. 
After possibly applying \cref{lem_avoid_overly_augmented}, we can remove further edgy simplices with standard image in order to obtain the map $\newlemmaphi$.
\end{proof}

What remains to be done is to remove edgy simplices with 2-additive image. We do this in the following lemma. In contrast to the previous steps, the removal process here does not preserve $\sigma$ regularity. Hence, the first condition for $\sigma$-smallness in \cref{def_sigma_small} makes no sense for the resulting map. However, the second condition does and it will be preserved by this process.

\begin{lemma}
\label{lem_avoid_edgy_2_additive}
Let $S$ be a combinatorial $k$-sphere and $\phi\colon  S \to (\IAAstrel)^{\leq R}$ a map that is $\sigma$-regular and $\sigma$-small. Then $\phi$ is in $\IAArel$ regularly homotopic to a map $\newlemmaphi\colon \newlemmasphere\to (\IAAstrel)^{\leq R}$ such that the following properties hold.
  \begin{enumerate}
  \item If $\Delta\subset S$ is a simplex such that $\newlemmaphi(\Delta)$ is a $\sigma$ edge, then for all $x\in \Delta$, we have $\rkfn(\phi'(x))<R$.
  \item $\newlemmaphi$ has no edgy simplices with 3-additive image.
  \item $\newlemmaphi$ has no edgy simplices with standard image.
  \item $\newlemmaphi$ has no edgy simplices with 2-additive image.
  \end{enumerate}
In particular, $\phi'$ has no edgy simplices.
\end{lemma}
\begin{proof}
The proof of this lemma works similarly to the one of \cite[Section 6, Step 1.3]{Brueck2022}, we provide an outline in what follows.
We can assume that all conditions of \cref{lem_avoid_edgy_standard} are satisfied and that (using \cref{lem_avoid_overly_augmented}) $\phi$ has no overly augmented simplices. These in particular imply the first property in the statement of this lemma. In what follows, we do not use the assumption that $\phi$ be $\sigma$-regular and $\sigma$-small, which are only true at the first step of our iterative removal procedure.

Let $\Delta$ be a maximal edgy simplex with 2-additive image, i.e.~$\phi(\Delta) = \ls v_0, v_1 \rs$, where $\rkfn(v_0) = \rkfn(v_1)$ and $\vec v_1 = \vec v_0\pm \vec e_i$ for some $1\leq i \leq m$.
The general strategy is the same as that of the proof of \cref{lem_avoid_overly_augmented}: We define a complex $K(\Delta)$ and modify $\phi|_{\Star_{S}(\Delta)}$ such that the result has image in $\phi(\partial \Delta) \ast K(\Delta)$.
Define $K(\Delta) \coloneqq \Linkhat_{\IArel}^{<R}(v_0)$.

Using \cref{lem_compare_linkhatIA_and_linhatIAA}, it is not hard to see that
$\phi(\Link_{S}(\Delta)) \subseteq \Link_{\IAAstrel}^{<R}(\phi(\Delta))$.
(This uses that $\phi$ has no overly augmented simplices and no edgy simplices whose image is standard.)
By \cref{it_eq_link_linkhat_IAA_2add} of \cref{equality_link_linkhat},
we have 
\begin{equation*}
	\Link_{\IAAstrel}^{<R}(\phi(\Delta)) = \Link_{\IAAstrel}^{<R}(\ls v_0, \langle \vec v_0\pm \vec e_i\rangle \rs) = \Linkhat_{\IAAstrel}^{<R}(\phi(\Delta)).
\end{equation*}
By \cref{lem_compare_linkhatIA_and_linhatIAA}, every simplex of $\Linkhat_{\IAAstrel}^{<R}(\phi(\Delta))$ is either contained in 
\begin{equation*}
	\Linkhat_{\IArel}^{<R}(v_0) = K(\Delta)
\end{equation*}
or is of type double-triple. But as $\rkfn(v_0) = R$ and $\phi$ has no overly augmented simplices, there are no double-triple simplices in $\phi(\Link_{S}(\Delta))$ (see \cite[Observation 6.3]{Brueck2022}). Hence, we have 
\begin{equation*}
	\phi(\Link_{S}(\Delta)) \subseteq K(\Delta).
\end{equation*}

The complex $K(\Delta)$ is $(n-2)$-connected by \cref{lem_con_linkhat_less_IA}. So as 
\begin{equation*}
\dim \Link_{S}(\Delta) = k-\dim(\Delta)-1 \leq n - 2,
\end{equation*}
the complex $K(\Delta)$ is $(\dim \Link_{S}(\Delta))$-connected.
By \cref{lem_regular_homotopy_by_balls}, $\phi$ is regularly homotopic to a simplicial map 
\begin{equation*}
	\newproofphi\colon \newproofsphere \to (\IAAstrel)^{\leq R}
\end{equation*}
that is obtained by replacing $\phi|_{\Star_{S}(\Delta)}$ by a map that has image in $\phi(\partial \Delta) \ast K(\Delta)$. It is easy to see that this removes the edgy simplex $\Delta$ without introducing new edgy simplices. Again, the homotopy takes place in $\IAAstrel$.
The complex $\phi(\partial \Delta) \ast K(\Delta)$ can contain $\sigma$ edges. However, by \cref{lem_compare_linkhatIA_and_linhatIAA}, we have 
\begin{equation*}
	K(\Delta)\subseteq \Linkhat_{\IAAstrel}(\Delta),
\end{equation*}
so every such $\sigma$ edge is contained in $K(\Delta)$. As $K(\Delta)\subset (\IAAstrel)^{<R}$ by definition, all vertices of such a $\sigma$ edge have rank less than $R$. Hence, the map $\newproofphi$ still satisfies the first property in the statement of this lemma. It need not be $\sigma$-regular any more though because its restrictions to the stars of the potential new $\sigma$ simplices in $K(\Delta)$ need not be $\sigma$ cross maps.
This allows us to iterate this procedure (after possibly applying \cref{lem_avoid_overly_augmented} again) to obtain a map $\newlemmaphi$ that satisfies all four properties in the statement of the lemma. This implies that $\newlemmaphi$ has no edgy simplices as the image of every edgy simplex is a standard, 2-additive, 3-additive or $\sigma$ simplex.
\end{proof}

We now combine the results of \cref{sec_sigma_regularity_obtained}, \cref{sec_separating_cross_polytopes} and \cref{sec_edgy_simplices} for proving \cref{prop_isolating-rank-r-vertices}.

\begin{proof}[Proof of \cref{prop_isolating-rank-r-vertices}]
Let $S$ be a combinatorial $k$-sphere and 
\begin{equation*}
	\phi\colon  S \to (\IAAstrel)^{\leq R}
\end{equation*}
a simplicial map.
Using \cref{lem_sigma_regular_can_be_assumed}, we can assume that $\phi$ is $\sigma$-regular. 
Using \cref{lem_sigma_small_obtained}, we can then assume that it is also $\sigma$-small. This allows us to apply \cref{lem_avoid_edgy_2_additive} and assume that $\phi$ satisfies all properties of this lemma. In particular, we can assume that $\phi$ has no edgy simplices.

To show that under these conditions, $\phi$ actually has \hyperref[it_isolation]{Isolation}, let $\Delta$ be a bad simplex of $S$ and set $v\coloneqq \phi(\Delta)$. We need to see that for all $x \in \Link_{S}(\Delta)$, we have  
\begin{equation}
\label{eq_isolation_obtained}
	\phi(x)\in \ls v \rs \cup \Linkhat^{<R}_{\IAAstrel}(v).
\end{equation}
As $\phi$ is simplicial, we have 
\begin{equation*}
	\phi(\Link_{S}(\Delta)) \subseteq \Star_{\IAAstrel}(v) = \ls v \rs \cup \Link_{\IAAstrel}(v).
\end{equation*}
Assume that there was $x\in \Link_{S}(\Delta)$ that gets mapped to a vertex of rank equal to $R$. Then, as there are no edgy simplices, we have $\phi(x) = v $. Consequently,
\begin{equation*}
 \phi(\Link_{S}(\Delta)) \subseteq \ls v \rs \cup \Link^{<R}_{\IAAstrel}(v).
\end{equation*}
To show that \cref{eq_isolation_obtained} holds, by \cref{it_eq_link_linkhat_IAA_v} of \cref{equality_link_linkhat}, it only remains to verify that no vertex in $\phi(\Link_{S}(\Delta))$ forms a $\sigma$ simplex with $\phi(\Delta) = v$. But as $\rkfn(v)  = R$, this follows from the first condition of \cref{lem_avoid_edgy_2_additive}.
\end{proof}

\section{Theorem B: A presentation of \texorpdfstring{$\St^\omega_n$}{St\unichar{"005E}\unichar{"1D714}}}
\label{sec:presentation}

Recall that by the Solomon--Tits Theorem (\cref{Solomon-Tits}), the reduced homology of the symplectic Tits building $T^\omega_n$ is concentrated in dimension $n-1$. This  homology group is what we call the \emph{symplectic Steinberg module}
\begin{equation*}
	\St^\omega_n = \Stsymp_n(\mbQ) \coloneqq \widetilde H_{n-1}(T^\omega_n;\mbZ).
\end{equation*}
In this section, we obtain a presentation of this module by proving \autoref{thmB}. The proof is an induction on the genus $n$. We start by giving 
an overview and outline the four steps of the argument, which correspond to the 
four subsections.\medskip

\textbf{Outline of the proof:} \cref{subsec:highly-connected-complex-IAA-zero-one-two} contains topological preparations.  We enlarge the three highly connected, nested simplicial complexes $(\IAA, \IA, \I)$ that we studied in the previous sections to three highly connected, nested simplicial complexes $(\IAA^{(2)}, \IAA^{(1)}, \IAA^{(0)})$. The high connectivity results for the latter are a corollary of the connectivity results for the former. The key difference is that the local structure of the complexes $(\IAA^{(2)}, \IAA^{(1)}, \IAA^{(0)})$ is ``nice''; e.g.\ links of (certain) simplices in $\IAA^{(2)}$ are isomorphic to copies of $\IAA[k]^{(0)}$ with $k < n$. In \cref{subsec:intermediate-exact-sequence}, we use this local structure and the long exact homology sequence of the triple $(\IAA^{(2)}, \IAA^{(1)}, \IAA^{(0)})$ to construct an ``intermediate'' exact sequence relating the symplectic Steinberg module $\St_n^\omega$ to direct sums 
of $\St^{\omega}_k$ with $k < n$. The exactness of this ``intermediate'' sequence is a key ingredient of the induction argument on the genus $n$ described in the final subsection. Indeed, in \cref{subsec:reduction-of-theorem-b-to-an-exact-sequence}, we reduce \autoref{thmB} to checking that a certain sequence
\begin{equation}
	\label{eq:presentation-right-exact-sequence}
	\SigmaAdditiveApartmentModule_n \oplus \SkewApartmentModule_n \xrightarrow{\AdditiveMap_n \oplus \SkewMap_n} \ApartmentModule_n \xrightarrow{\ApartmentMap_n} \St^\omega_n \longrightarrow 0
\end{equation}
of ``(augmented) apartment'' $\Sp{2n}{\mbZ}$-modules is exact. In \cref{subsec:commutative-diagram}, we construct a commutative diagram that relates the ``intermediate'' exact sequence with the sequence in 
\cref{eq:presentation-right-exact-sequence}. The key point is that this diagram also contains a direct sum of copies of \cref{eq:presentation-right-exact-sequence} but for ``smaller'' Steinberg modules 
$\St^{\omega}_k$ with $k < n$. This is where the induction hypothesis is applied in the last step of the argument. In the final \cref{subsec:proof-of-presentation-right-exact-sequence}, 
we use induction on the genus $n$ and a diagram 
chasing argument to show that the sequence in 
\cref{eq:presentation-right-exact-sequence} is exact. Throughout this section, we make the standing assumption that
\begin{equation*}
	n \geq 1.
	\tag{Standing assumption}
\end{equation*}

\subsection{\texorpdfstring{$(\IAA^{(2)}, \IAA^{(1)}, \IAA^{(0)})$}{(IAA\unichar{"005E}(2), IAA\unichar{"005E}(1), IAA\unichar{"005E}(0))} and an intermediate exact sequence}
\label{subsec:highly-connected-complex-IAA-zero-one-two}

This subsection contains topological preparations for the proof of 
\autoref{thmB}. We enlarge the nested complexes $(\IAA, \IA, 
\I)$ to nested complexes $(\IAA^{(2)}, \IAA^{(1)}, 
\IAA^{(0)})$, prove that each is highly connected and describe their local 
structure. From this we construct an ``intermediate'' exact sequence, which is a key ingredient of the proof of \autoref{thmB}.

\subsubsection{The complex \texorpdfstring{$\IAA^{(2)}$}{IAA\unichar{"005E}(2)} and its subcomplexes}
\label{sec_int_cpx}

To define $\IAA^{(2)}$ and its subcomplexes, we introduce some new types of ``mixed'' simplices.

\begin{definition}
\label{def_general_mixed_simplices}
Let 
\begin{equation*}
\VertexSet{n} \coloneqq \ls v \subseteq \mbZ^{2n} \,\middle|\, v \text{ is a rank-1 summand of } \mbZ^{2n}\rs
\end{equation*}
be as in \cref{eq_V_n}. Let $\tau$ and $\tau'$ be simplex types of the form
\begin{align*}
	\tau &\in  \ls \text{standard}, \text{ 2-additive}, \text{ 3-additive}, \text{ double-triple}, \text{ double-double}\rs,\\
	\tau' &\in \ls \sigma, \text{ skew-additive}, \text{ 2-skew-additive}, \,\sigma^2, \text{ skew-}\sigma^2, \,\sigma\text{-additive} \rs,
\end{align*}
where the simplex types are defined as in \cref{def:IA-simplices}, \cref{def:BAA-simplices} and \cref{def:IAA-simplices}.
A subset 
$\Delta = \{v_0, \dots, v_k\} \subset \VertexSet{n}$
of $(k+1)$ lines is called a \emph{$(\tau, \tau')$ simplex} if there is a subset $\Delta'\subset \Delta$ such that $\Delta'$ is a minimal simplex of type $\tau'$ and $\Delta \setminus \Delta'$ is a simplex of type $\tau$ such that every $v \in \Delta \setminus \Delta'$ is contained in $\langle\ls \vec v' \mid v'\in \Delta'\rs \rangle^{\perp}$.
\end{definition}

Note that a mixed simplex as defined in \cref{def:IA-simplices} is a $(\text{2-additive}, \sigma)$ simplex in this notation. A $(\text{standard}, \tau')$ simplex is just the same as a $\tau'$ simplex. (This follows using \cref{lem_split_off_symplectic_summand}.)

\begin{example}
Let $\{ \vec e_1, \ldots, \vec e_n, \vec f_1, \ldots, \vec f_n \}$ be a symplectic basis of $\mbZ^{2n}$ and let $1\leq k \leq n$. Then
\begin{enumerate} 
\item $\ls \langle \vec e_1 + \vec e_2 + \vec e_3 \rangle, e_1, e_2, e_3, \ldots, e_{k-1}, e_k, f_k  \rs$ is a $(\text{3-additive}, \sigma)$ simplex, where the minimal simplex of type $\sigma$ is $ \Delta' = \{ e_k, f_k\}$;
\item $\{ \langle \vec e_1 + \vec e_2 \rangle, e_1, e_2, e_3, \ldots, e_{k-1}, e_k, f_k, \langle \vec e_k + \vec f_k \rangle  \}$ is a $(\text{2-additive}, \sigma\text{-additive})$ simplex, where the minimal $\sigma$-additive simplex is $\Delta' = \{ e_k, f_k, \langle \vec e_k + \vec f_k \rangle\}$.
\end{enumerate}
\end{example}

The nested complexes $(\IAA^{(2)}, \IAA^{(1)}, \IAA^{(0)})$ are now defined in such a way that they allow for more and more symplectic pairs in their simplices. I.e.\ simplices in $\IAA^{(0)}$ do not contain any symplectic pair (such as standard or double-triple simplices), simplices in $\IAA^{(1)}$ contain at most one symplectic pair (such as $\sigma$ or $( \text{double-double},\sigma)$ simplices) and simplices in $\IAA^{(2)}$ contain at most two symplectic pairs (such as $\sigma^2$ but also $\sigma$-additive simplices). For technical reasons, we also define a complex $\IAA^{(1.5)}$ that sits between $\IAA^{(2)}$ and $\IAA^{(1)}$. \cref{fig_intermediate_complexes} illustrates the inclusion relations of these complexes as well as their relation to the complexes $\IAAst$ and $\IAA$ studied in earlier sections.

\begin{definition}
\label{def_intermediate_complexes}
The simplicial complexes $\IAA^{(i)}$ all have vertex set $\VertexSet{n}$.
\begin{enumerate}
\item The simplices of $\IAA^{(0)}$ are all either standard, 2-additive, 3-additive, double-triple or double-double.
\item The simplices of $\IAA^{(1)}$ are all either standard, 2-additive, 3-additive, double-triple, double-double or $(\tau, \tau')$ with
\begin{align*}
	\tau &\in  \ls \text{standard}, \text{ 2-additive}, \text{ 3-additive}, \text{ double-triple}, \text{ double-double}\rs,\\
	\tau' &\in \ls \sigma\rs.
\end{align*}
\item The simplices of $\IAA^{(1.5)}$ are all either standard, 2-additive, 3-additive, double-triple, double-double or $(\tau, \tau')$ with
\begin{align*}
	\tau &\in  \ls \text{standard}, \text{ 2-additive}, \text{ 3-additive}, \text{ double-triple}, \text{ double-double}\rs,\\
	\tau' &\in \ls \sigma, \text{ skew-additive}, \text{ 2-skew-additive}\rs.
\end{align*}
\item The simplices of $\IAA^{(2)}$ are all either standard, 2-additive, 3-additive, double-triple, double-double or $(\tau, \tau')$ with
\begin{align*}
	\tau &\in  \ls \text{standard}, \text{ 2-additive}, \text{ 3-additive}, \text{ double-triple}, \text{ double-double}\rs,\\
	\tau' &\in \ls \sigma, \text{ skew-additive}, \text{ 2-skew-additive}, \,\sigma^2, \text{ skew-}\sigma^2, \,\sigma\text{-additive} \rs.
\end{align*}
\end{enumerate}
Finally, if $\SymplecticSummand \subseteq \mbZ^{2n}$ is a summand of genus $k \leq n$, we let $\IAA[]^{(i)}(\SymplecticSummand)$ denote the full subcomplex of $\IAA^{(i)}$ on the set $\VertexSet{n} \cap \SymplecticSummand$ of rank-1 summands of $\SymplecticSummand$.
\end{definition}

\begin{figure}
\begin{equation*}
\xymatrix{
	\IAA^{(0)} \ar@{^{(}->}[r] &   \IAA^{(1)} \ar@{^{(}->}^<<<<<{\simeq}[r] &    \IAA^{(1.5)} \ar@{^{(}->}[r] &   \IAA^{(2)}\\
	& \IAAst \ar@{^{(}->}[u] \ar@{^{(}->}[rr] & & \IAA \ar@{^{(}->}[u]
}
\end{equation*}
\caption{The complexes $\IAA^{(i)}$ and their relation to the complexes $\IAAst$ and $\IAA$.}
\label{fig_intermediate_complexes}
\end{figure}

The complexes $\IAA^{(1)}$ and $\IAA^{(1.5)}$ are, in fact, homotopy equivalent as the following lemma shows.

\begin{lemma}
\label{lem_hom_eq_1_15}
There is a deformation retraction $\IAA^{(1.5)} \to \IAA^{(1)}$.
\end{lemma}
\begin{proof}
	
The complex $\IAA^{(1.5)}$ is obtained from the complex $\IAA^{(1)}$ by attaching the simplices of type $(\tau, \text{2-additive})$ and $(\tau, \text{2-skew-additive})$. 
The proof now works exactly as that of \cref{lem_deformation_retract_2skew_additive}: 

As in \cref{lem_unique_skew_additive_face}, one can see that for each $k$-dimensional $(\tau, \text{2-additive})$ simplex $\Delta$, there is a unique $(\tau, \text{2-skew-additive})$ simplex of dimension $(k+1)$ that has $\Delta$ as a face. Starting with the top-dimensional simplices, we can use this to push in all $(\tau, \text{2-additive})$ simplices, which removes both the $(\tau, \text{2-additive})$ and the $(\tau, \text{2-skew-additive})$ simplices.
\end{proof}

\subsubsection{The local structure of \texorpdfstring{$(\IAA^{(2)}, \IAA^{(1)}, \IAA^{(0)})$}{(IAA\unichar{"005E}(2), IAA\unichar{"005E}(1), IAA\unichar{"005E}(0))}}
\label{sec_intermediate_complexes_links}

In this short subsection, we describe the local structure of the complexes $\IAA^{(2)}$ and $\IAA^{(1)}$ by relating the link of certain simplices to $\IAA[k]^{(0)}$ for $k < n$.

\begin{lemma}
	\label{lem_links_mixed_sigma}
	Let $\Delta$ be a minimal $\sigma$ simplex in $\IAA^{(1)}$. Then 
	$$\Link_{\IAA^{(1)}}(\Delta) = \IAA[]^{(0)}(\langle \Delta \rangle^\perp) \cong \IAA[n-1]^{(0)}.$$
\end{lemma}
\begin{proof}
	Note that $\langle \Delta \rangle^\perp \subseteq \mbZ^{2n}$ is a genus $(n-1)$-summand. The equality on the left immediately follows from the definitions of $\IAA^{(1)}$ and $\IAA[]^{(0)}(\langle \Delta \rangle^\perp)$ (see \cref{def_intermediate_complexes}). An isomorphism on the right is induced by the choice of some identification $\langle \Delta \rangle^\perp \xrightarrow{\cong} \mbZ^{2(n-1)}$.
\end{proof}

Similarly, one checks the following:

\begin{lemma}
	\label{lem_links_mixed_sigma2_sigmaadd}
	Let $\Delta$ be a simplex in $\IAA^{(2)}$.
	\begin{enumerate}
		\item If $\Delta$ is minimal $\sigma^2$ or skew-$\sigma^2$, then $\Link_{\IAA^{(2)}}(\Delta) = \IAA[]^{(0)}(\langle \Delta \rangle^\perp) \cong \IAA[n-2]^{(0)}.$
		\item If $\Delta$ is minimal $\sigma$-additive, then $\Link_{\IAA^{(2)}}(\Delta) =	 \IAA[]^{(0)}(\langle \Delta \rangle^\perp) \cong \IAA[n-1]^{(0)}.$
	\end{enumerate}
\end{lemma}

\subsubsection{Connectivity properties of \texorpdfstring{$(\IAA^{(2)}, \IAA^{(1)}, \IAA^{(0)})$}{(IAA\unichar{"005E}(2), IAA\unichar{"005E}(1), IAA\unichar{"005E}(0))}}
\label{sec_connectivity_intermediate}

We now use the connectivity results obtained in earlier sections to show that the complexes $\IAA^{(i)}$ are $(n-2+i)$-connected for $i=0,1,2$.

We first note that there is a poset map
$s\colon  \Simp(\IAA^{(0)}) \to T^\omega_n$
given by sending a simplex $\Delta$ to the span $\ll \Delta \rr$ of its vertices.

\begin{proposition}\label{prop:spanmap}
	The span map $s\colon  \Simp(\IAA^{(0)}) \to T^\omega_n$ is $(n+1)$-connected.
\end{proposition}

\begin{proof}
As in the proof of \cref{lem:intwbaa-connectivity}, this can be shown using {\cite[Corollary 2.2]{vanderkallenlooijenga2011sphericalcomplexesattachedtosymplecticlattices}} by van der Kallen--Looijenga\footnote{For the $n$ in \cite[Corollary 2.2]{vanderkallenlooijenga2011sphericalcomplexesattachedtosymplecticlattices}, choose what is $(n+1)$ in the notation of the present article; define their map $t$ by $t(V)=\dim(V)+ 2$}. 
To apply it in the setting here, one uses that the target $T^{\omega}_n$ is Cohen--Macaulay of dimension $n-1$ (\cref{Solomon-Tits}) and that for $V\in T^{\omega}_n$, the poset fibre  $s_{\leq V}$ is isomorphic to $\Simp(\BAA(V))$. The latter is $\dim(V)$-connected by \autoref{connectivity_BAA}.
\end{proof}

As $T^\omega_n$ is $(n-2)$-connected by the Solomon--Tits Theorem (\cref{Solomon-Tits}), this implies:

\begin{corollary}
\label{lem_connectivity_IAA0}
$\IAA^{(0)}$ is $(n-2)$-connected.
\end{corollary}

Furthermore, we obtain the following corollary, which relates $\IAA^{(0)}$ to the symplectic Steinberg module and which will be frequently used later.

\begin{corollary}
	\label{lem_homology_IAA0_Steinberg}
	We have $\widetilde{H}_{n-1}(\IAA^{(0)}) \cong \Stsymp_n$.
\end{corollary}
\begin{proof}
	This follows from \cref{prop:spanmap} because $\widetilde{H}_{n-1}(T^\omega_n) = \Stsymp_n$.
\end{proof}

We now turn our attention to the complex $\IAA^{(1)}$, which can be seen as a 
version of $\IA$ with a nicer local structure (see 
\cref{sec_intermediate_complexes_links}). Recall that $\IA$ is 
$(n-1)$-connected (see \cref{lem_conn_IA}). In \cref{lem_connectivity_IAAstar}, 
we deduced from this that $\IAAst$ is $(n-1)$-connected, using that it is obtained from $\IA$ 
by attaching new simplices along highly connected links. Similarly, 
$\IAA^{(1)}$ is obtained from $\IAAst$ and therefore $\IAA^{(1)}$ is 
$(n-1)$-connected as well. This connectivity argument is formalised in the 
following lemma and corollary.

\begin{lemma}
\label{lem_IAAst_IAA1}
The inclusion $\IAAst \hookrightarrow \IAA^{(1)}$ is $n$-connected.
\end{lemma}
\begin{proof}
Comparing \cref{def_intermediate_complexes} to \cref{def:IAAst} and recalling that a mixed simplex has type $(\text{2-additive}, \sigma)$, one sees that the simplices of $\IAA^{(1)}$ that are not contained in $\IAAst$ are exactly those of type $(\tau, \sigma)$, where
$\tau \in  \ls \text{3-additive}, \text{ double-triple}, \text{ double-double}\rs$.
We apply the standard link argument explained in \cref{subsec:standard-link-argument} twice.

Firstly, let $X_1$ be the complex that is obtained from $\IAAst$ by attaching all $(\text{3-additive}, \sigma)$ simplices. It has $X_0 = \IAAst$ as a subcomplex. Let $B$ be the set of minimal $(\text{3-additive}, \sigma)$ simplices contained in $X_1$. This is a set of bad simplices in the sense of \cref{def:standard-link-argument-bad-simplices}. Following \cref{def:standard-link-argument-good-link}, we find that $\Link_{X_1}^{\mathrm{good}}(\Delta) = \Link_{X_1}(\Delta)$ for $\Delta \in B$. One then checks that for any $\Delta \in B$, there is an isomorphism $\Link_{X_1}(\Delta) \cong \Irel[n-4][3]$. This complex is $(n-\dim(\Delta)-1)=(n-6)$-connected by \cref{lem_connectivity_Irel}. Hence, the inclusion $\IAAst \hookrightarrow X_1$ is $n$-connected by \cref{it_link_argument_connected_map} of \cref{lem:standard-link-argument}.

Secondly, let $X_2 = \IAA^{(1)}$. By the previous paragraph, the lemma follows if we show that the inclusion $X_1 \hookrightarrow X_2$ is $n$-connected.
Let $B$ be the set of  minimal $(\text{double-triple}, \sigma)$ and $(\text{double-double}, \sigma)$ simplices contained in $X_2$. This is a set of bad simplices in the sense of \cref{def:standard-link-argument-bad-simplices}. Following \cref{def:standard-link-argument-good-link}, we find that $\Link_{X_2}^{\mathrm{good}}(\Delta) = \Link_{X_2}(\Delta)$ for $\Delta \in B$. For $\Delta \in B$ one then checks: If $\Delta$ is a minimal $(\text{double-triple}, \sigma)$ simplex, then $\Delta$ has dimension $6$ and there is an isomorphism $\Link_{X_2}(\Delta) \cong \Irel[n-4][3]$. If $\Delta$ is a minimal $(\text{double-double}, \sigma)$ simplex, then $\Delta$ has dimension $7$ and there is an isomorphism $\Link_{X_2}(\Delta) \cong \Irel[n-5][4]$. In either case, \cref{lem_connectivity_Irel} implies that $\Link_{X_2}(\Delta)$ is $((n-\dim(\Delta)-1)+1)$-connected. Using \cref{it_link_argument_connected_map} of \cref{lem:standard-link-argument}, it therefore follows that $X_1 \hookrightarrow X_2$ is $n$-connected.
\end{proof}

\begin{corollary}\label{cor_IAA1}
$\IAA^{(1)}$ is $(n-1)$-connected.
\end{corollary}
\begin{proof}
This follows from \cref{lem_IAAst_IAA1} because $\IAAst$ is $(n-1)$-connected by \cref{lem_connectivity_IAAstar}.
\end{proof}

Just as $\IAA^{(1)}$ can be seen as a variant of $\IAAst$, the complex $\IAA^{(2)}$ can be understood as a version of $\IAA$ with a nicer local structure (see \cref{sec_intermediate_complexes_links}) and the same connectivity properties. We close this subsection by proving that $\IAA^{(2)}$ is $n$-connected.

\begin{lemma}
\label{lem_IAA1_IAA2_surjection}
The inclusion $\IAA^{(1)} \hookrightarrow \IAA^{(2)}$ is $n$-connected.
\end{lemma}
\begin{proof}
By \cref{lem_hom_eq_1_15}, the inclusion $\IAA^{(1)}\hookrightarrow \IAA^{(1.5)}$ is $n$-connected. So it suffices to show that the same is true for the inclusion $\IAA^{(1.5)}\hookrightarrow \IAA^{(2)}$

The simplices of $X_1\coloneqq \IAA^{(2)}$ that are not contained in $X_0\coloneqq \IAA^{(1.5)}$ are exactly those of type $(\tau, \tau')$, where
\begin{equation}
\label{eq_IAA1_IAA2_surjection}
	\tau' \in  \ls \sigma^2, \text{skew-}\sigma^2, \sigma\text{-additive}\rs.
\end{equation}

We apply the standard link argument explained in \cref{subsec:standard-link-argument}.
Let $B$ be the set of minimal $\tau'$ simplices with $\tau'$ as in \cref{eq_IAA1_IAA2_surjection}. This is a set of bad simplices in the sense of \cref{def:standard-link-argument-bad-simplices}. Following \cref{def:standard-link-argument-good-link}, we find that $\Link_{X_1}^{\mathrm{good}}(\Delta) = \Link_{X_1}(\Delta)$ for $\Delta \in B$. 
We need to show that each such link is $(n-\dim(\Delta)-1)$-connected. 
But this follows immediately from  the results in \cref{sec_intermediate_complexes_links}: Minimal $\sigma^2$ and  skew-$\sigma^2$ simplices have dimension $3$ and their link is isomorphic to $\IAA[n-2]^{(0)}$ by \cref{lem_links_mixed_sigma2_sigmaadd}, hence $(n-4) = (n-3-1)$-connected by \cref{lem_connectivity_IAA0}. Similarly, minimal $\sigma$-additive simplices have dimension $2$ and their link is isomorphic to $\IAA[n-1]^{(0)}$ by \cref{lem_links_mixed_sigma2_sigmaadd}, hence $(n-3) = (n-2-1)$-connected by \cref{lem_connectivity_IAA0}
\end{proof}

These results suffice to deduce that $\IAA^{(2)}$ is highly connected as well:

\begin{lemma}
\label{lem_IAA_IAA2_surjection}
The inclusion $\IAA \hookrightarrow \IAA^{(2)}$ is $n$-connected.
\end{lemma}
\begin{proof}
By \cref{lem_IAAst_IAA1} and \cref{lem_IAA1_IAA2_surjection}, the inclusion $\IAAst \hookrightarrow  \IAA^{(2)}$ is $n$-connected and by \cref{lem_inclusion_surjective_pik}, the inclusion $\IAAst \hookrightarrow  \IAA$ is $n$-connected as well. This implies the claim.
\end{proof}

\begin{corollary}\label{lem_IAA2_highlyconnected}
$\IAA^{(2)}$ is $n$-connected.
\end{corollary}
\begin{proof}
This follows from \cref{lem_IAA_IAA2_surjection} because $\IAA$ is $n$-connected by \autoref{thm_connectivity_IAA}.
\end{proof}

We remark that combining \cref{lem_links_mixed_sigma} and \cref{lem_links_mixed_sigma2_sigmaadd} with \cref{lem_rel_homology_IAAi_IAAj} and \cref{lem_homology_IAA0_Steinberg} leads to a description of $H_{n}(\IAA^{(1)} , \IAA^{(0)})$ and $H_{n+1}(\IAA^{(2)} , \IAA^{(1)})$ as sums of smaller Steinberg modules $\Stsymp_k$ with $k < n$.

\subsubsection{An intermediate exact sequence}
\label{subsec:intermediate-exact-sequence}

We now use the long exact homology sequence of the triple $(\IAA^{(2)}, \IAA^{(1)}, \IAA^{(0)})$, the connectivity theorems and the observations about the local structure above to construct an ``intermediate'' exact sequence that is akin but not equal to the sequence in \cref{eq:presentation-right-exact-sequence}. We start by identifying one of the relative homology groups with $\St^\omega_n$.

\begin{lemma}
	\label{lem:steinberg-as-relative-homology}
	The following is a sequence of $\Sp{2n}{\mbZ}$-equivariant isomorphisms:
	\[ H_n(\IAA^{(2)}, \IAA^{(0)}) \xrightarrow{\partial_n^{(2,0)}} \widetilde H_{n-1}( \IAA^{(0)}) \xrightarrow{s_{*}}  \widetilde H_{n-1}( T^\omega_n) = \St^\omega_n.\]
\end{lemma}

\begin{proof}
	The first isomorphism $\partial_n^{(2,0)} = \partial_{(\IAA^{(2)}, \IAA^{(0)})}$ is the connecting morphism in the long exact sequence of the pair $(\IAA^{(2)}, \IAA^{(0)})$. This uses the fact that $\IAA^{(2)}$ is $n$-connected (see \cref{lem_IAA2_highlyconnected}). The second isomorphism $s_{*}$ is induced by the composition of the canonical homeomorphism between $\IAA^{(0)}$ and its barycentric subdivision $\Simp(\IAA^{(0)})$ and the span map from \cref{prop:spanmap} (defined in the paragraph before \cref{prop:spanmap}).
\end{proof}

In the next step, we extract an exact sequence from the long exact sequence of the triple $(\IAA^{(2)}, \IAA^{(1)}, \IAA^{(0)})$.

\begin{lemma}
	\label{lem:right-exact-homology-sequence}
	The following is an exact sequence of $\Sp{2n}{\mbZ}$-modules:
	\[ H_{n+1}(\IAA^{(2)}, \IAA^{(1)}) \xrightarrow{\partial_{n+1}^{(2,1,0)}} H_n(\IAA^{(1)}, \IAA^{(0)}) \xrightarrow{s_{*} \circ \partial_{n}^{(1,0)}} \St^\omega_n \longrightarrow 0.\]
\end{lemma}

\begin{proof}
	Consider the long exact sequence of the triple  $(\IAA^{(2)}, \IAA^{(1)}, \IAA^{(0})$
	\begin{multline*}
		\label{eq:long-exact-triple-sequence}  
		H_{n+1}(\IAA^{(2)}, \IAA^{(1)}) \xrightarrow{\partial_{n+1}^{(2,1,0)}} H_n(\IAA^{(1)}, \IAA^{(0)})\\ \longrightarrow H_n(\IAA^{(2)}, \IAA^{(0)}) \longrightarrow H_n(\IAA^{(2)}, \IAA^{(1)}) = 0. 
	\end{multline*}
	We observe that $H_n(\IAA^{(2)}, \IAA^{(1)}) = 0$ because $\IAA^{(2)}$ is $n$-connected and $\IAA^{(1)}$ is $(n-1)$-connected (see \cref{lem_IAA2_highlyconnected} and \cref{cor_IAA1}).  The result then follows by invoking \cref{lem:steinberg-as-relative-homology}, and using that the connecting morphism $\partial_n^{(1,0)}$ of the pair $(\IAA^{(1)}, \IAA^{(0)})$ satisfies that
	\begin{center}
		\begin{tikzcd}
			H_n(\IAA^{(1)}, \IAA^{(0)}) \arrow[r] \arrow[rd, "\partial_n^{(1,0)}", swap] & H_n(\IAA^{(2)}, \IAA^{(0)}) \arrow[d, "\partial_n^{(2,0)}"]\\
			& \widetilde{H}_{n-1}(\IAA^{(0)}).
		\end{tikzcd}
	\end{center}
	commutes.
\end{proof}

Finally, we explain how one can decompose the two remaining relative homology groups in the exact sequence appearing in \cref{lem:right-exact-homology-sequence} into direct sums of ``smaller'' symplectic Steinberg modules $\St_k^\omega$ for $k < n$.

\begin{lemma}
	\label{lem_rel_homology_IAAi_IAAj}
	Fix a total ordering on the set of vertices of $\IAA^{(2)}$ such that every simplex $\Delta$ is oriented. We have the following $\Sp{2n}{\mbZ}$-equivariant isomorphisms:
	\begin{equation*}
		H_{n}(\IAA^{(1)} , \IAA^{(0)}) \cong \bigoplus_{\substack{\Delta = \{v_1,w_1\}\\ \text{$\sigma$ 		simplex}}}\St^\omega(\ll\Delta\rr^\perp)
	\end{equation*}
	and
	\begin{equation*}
		H_{n+1}(\IAA^{(2)} , \IAA^{(1)})
		\cong \bigoplus
		\begin{rcases*}
			\begin{dcases*}
				\bigoplus_{\substack{\Delta = \{v_1,w_1,v_2,w_2\}\\ \text{$\sigma^2$ simplex}}}\St^\omega(\ll\Delta\rr^\perp)\\
				\bigoplus_{\substack{\Delta = \{z_0, z_1, z_2, z_3\}\\ \text{skew-$\sigma^2$ simplex}}} \St^\omega(\ll\Delta\rr^\perp)\\
				\bigoplus_{\substack{\Delta = \{z_0, z_1, z_2\}\\ \text{$\sigma$-additive simplex}}} \St^\omega(\ll\Delta\rr^\perp)
			\end{dcases*}
		\end{rcases*}.
	\end{equation*}
where we use the convention that $\St^{\omega}(\{0\}) \coloneqq \mbZ$.
\end{lemma}
\begin{remark}
	\label{rem:module-structure-right-hand-term}
	In \cref{lem_rel_homology_IAAi_IAAj}, the $\Sp{2n}{\mbZ}$-module structure of the sum terms appearing on the right is as follows: We make it precise for the $\Sp{2n}{\mbZ}$-module
	$$\bigoplus_{\substack{\Delta = \{v_1,w_1,v_2,w_2\}\\ \text{$\sigma^2$ simplex}}}\St^\omega(\ll\Delta\rr^\perp),$$
	for the other terms it is defined similarly. Let $\Delta$ be a minimal $\sigma^2$ simplex in $\IAA^{(2)}$ and consider a class $\zeta \in \St^\omega(\ll\Delta\rr^\perp)$. Then, an element $\GroupElement \in \Sp{2n}{\mbZ}$ acts by mapping $\zeta$ to the class $\pm (\phi \cdot \zeta)$ in the summand $\St^\omega(\ll \phi \cdot \Delta\rr^\perp)$ indexed by the $\sigma^2$ simplex $\phi \cdot \Delta$, where the sign is positive if $\phi\colon  \Delta \to \phi \cdot\Delta$ is orientation-preserving with respect to the total ordering of the vertices and negative if it is orientation-reversing.
\end{remark}
\begin{proof}
	For the second isomorphism, we first use the identification
	$$H_{n+1}(\IAA^{(2)} , \IAA^{(1)})\cong H_{n+1}(\IAA^{(2)} , \IAA^{(1.5)})$$
	induced by the deformation retraction in \cref{lem_hom_eq_1_15}. Apart from this extra step, the construction of the two claimed isomorphisms is analogous and can be described simultaneously: Let $k = n$, $i = 1$ and $i' = 0$ or $k = n+1$, $i = 2$ and $i' = 1.5$. Using excision and that 
	$$\Star_{\IAA^{(i)}}(\Delta) \cap \IAA^{(i')} = \partial \Delta \ast \Link_{\IAA^{(i)}}(\Delta)$$
	for $\Delta$ one of the minimal simplices listed above,
	we obtain isomorphisms
	\begin{equation*}
		H_{n}(\IAA^{(1)} , \IAA^{(0)}) \cong \bigoplus_{\substack{\Delta \text{ minimal }\\ \text{$\sigma$ simplex}}} H_{n}( \Star_{\IAA^{(1)}}(\Delta), \partial \Delta \ast \Link_{\IAA^{(1)}}(\Delta))
	\end{equation*}
	and
	\begin{equation*}
			H_{n+1}(\IAA^{(1.5)} , \IAA^{(1)})
			\cong \bigoplus
			\begin{rcases*}
				\begin{dcases*}
					\bigoplus_{\substack{\Delta \text{ minimal }\\ \text{$\sigma^2$ simplex}}} H_{n+1}( \Star_{\IAA^{(2)}}(\Delta), \partial \Delta \ast \Link_{\IAA^{(2)}}(\Delta))\\
					\bigoplus_{\substack{\Delta \text{ minimal }\\ \text{skew-$\sigma^2$ simplex}}} H_{n+1}(\Star_{\IAA^{(2)}}(\Delta), \partial \Delta \ast \Link_{\IAA^{(2)}}(\Delta))\\
					\bigoplus_{\substack{\Delta \text{ minimal }\\ \text{$\sigma$-additive simplex}}} H_{n+1}(\Star_{\IAA^{(2)}}(\Delta), \partial \Delta \ast \Link_{\IAA^{(2)}}(\Delta))
				\end{dcases*}
			\end{rcases*}.
	\end{equation*}
	The terms appearing on the right can then be further simplified: The fact that $\Star_{\IAA^{(i)}}(\Delta)$ is contractible implies that the connecting morphism of the long exact sequence of the pairs $(\Star(\Delta)_{\IAA^{(i)}}, \partial \Delta \ast \Link_{\IAA^{(i)}}(\Delta))$ is an isomorphism in reduced homology, i.e.\
	$$H_{k}(\Star_{\IAA^{(i)}}(\Delta), \partial \Delta \ast \Link_{\IAA^{(i)}}(\Delta)) \cong \widetilde{H}_{k-1}(\partial \Delta \ast \Link_{\IAA^{(i)}}(\Delta)).$$
	We then observe that $\partial \Delta$ is a $(\dim \Delta - 1)$-sphere and use the $\partial \Delta$-suspension isomorphism associated to $\Link_{\IAA^{(i)}}(\Delta)$ to obtain an identification
	\begin{equation*}
		\widetilde{H}_{k-1}(\partial \Delta \ast \Link_{\IAA^{(i)}}(\Delta)) \cong
		\widetilde{H}_{(k-1)-\dim \Delta}(\Link_{\IAA^{(i)}}(\Delta)).
	\end{equation*}
	We note that this isomorphism is determined by the orientation of $\Delta$, i.e.\ it is induced by the cross product with the orientation class $\eta_{\dim \Delta - 1} \in \widetilde{H}_{\dim(\Delta)-1}(\partial \Delta)$ (see e.g.\ \cite[p.276 et seq.]{Hatcher2002}).
	Finally, we use \cref{lem_links_mixed_sigma} to identify $\Link_{\IAA^{(i)}}(\Delta)$ with $\IAA[n-1]^{(0)}$ if $i = 1$ and \cref{lem_links_mixed_sigma2_sigmaadd} to identify $\Link_{\IAA^{(i)}}(\Delta)$ with $\IAA[n-2]^{(0)}$ or $\IAA[n-1]^{(0)}$ (depending on $\Delta$) if $i = 2$. Invoking \cref{lem_homology_IAA0_Steinberg} then yields an isomorphism $\widetilde{H}_{(k-1)-\dim \Delta}(\Link_{\IAA^{(i)}}(\Delta)) \cong \St^\omega(\ll\Delta\rr^\perp)$. \qedhere
\end{proof}

\begin{corollary}
	\label{cor:right-exact-sequence-smaller-steinberg-modules}
	The exact sequence obtained in \cref{lem:right-exact-homology-sequence} identifies with an exact sequence
	\begin{equation*}
		\bigoplus
		\begin{rcases*}
			\begin{dcases*}
				\bigoplus_{\substack{\Delta = \{v_1,w_1,v_2,w_2\}\\ \text{$\sigma^2$ simplex}}}\St^\omega(\ll\Delta\rr^\perp)\\
				\bigoplus_{\substack{\Delta = \{z_0, z_1, z_2, z_3\}\\ \text{skew-$\sigma^2$ simplex}}} \St^\omega(\ll\Delta\rr^\perp)\\
				\bigoplus_{\substack{\Delta = \{z_0, z_1, z_2\}\\ \text{$\sigma$-additive simplex}}} \St^\omega(\ll\Delta\rr^\perp)
			\end{dcases*}
		\end{rcases*}
		\xrightarrow{\partial_{n+1}^{\St^\omega}}
		\bigoplus_{\substack{\Delta = \{v_1,w_1\}\\ \text{$\sigma$ simplex}}}\St^\omega(\ll\Delta\rr^\perp)
		\xrightarrow{\partial_{n}^{\St^\omega}}
		\St^\omega_n
		\to
		0		
	\end{equation*}
	where we use the convention that $\St^{\omega}(\{0\}) \coloneqq \mbZ$.
\end{corollary}

\subsection{Reduction of Theorem B to checking exactness}
\label{subsec:reduction-of-theorem-b-to-an-exact-sequence}

The goal of this subsection is to define $\Sp{2n}{\mbZ}$-modules $\ApartmentModule_n, \SigmaAdditiveApartmentModule_n$ and $\SkewApartmentModule_n$ as well as maps $\ApartmentMap_n, \AdditiveMap_n$ and $\SkewMap_n$ such that \autoref{thmB} is a consequence of the following statement.

\begin{theorem}\label{thm:stpres}
	For $n\ge 1$, the sequence
	$\SigmaAdditiveApartmentModule_n \oplus \SkewApartmentModule_n \xrightarrow{\AdditiveMap_n \oplus \SkewMap_n} \ApartmentModule_n \xrightarrow{\ApartmentMap_n} \St^\omega_n \longrightarrow 0$
	is exact.
\end{theorem}

Throughout this section we use the following notation: Let $n \geq 1$ and $\SymplecticSummand$ be a symplectic summand of $(\mbZ^{2n}, \omega)$ of genus $k$. We denote by $\on{Sp}(\SymplecticSummand) \subseteq \Sp{2n}{\mbZ}$ the symplectic automorphisms of $\SymplecticSummand$ and let $\St^\omega(\SymplecticSummand)$ be the symplectic Steinberg module of $\on{Sp}(\SymplecticSummand)$. If $\SymplecticSummand = \{0\}$ we define $\St^\omega(\SymplecticSummand) \coloneqq \mbZ$.\medskip

\noindent \textbf{The apartment module $\ApartmentModule_n$ and the apartment class map $\ApartmentMap: \ApartmentModule_n \to \St_n^\omega$:} A theorem of Gunnells \cite{gun2000} states that $\St_n^\omega$ is a cyclic $\Sp{2n}{\mbZ}$-module, i.e.\ there is an equivariant surjection $\ApartmentMap'_n\colon  \mbZ[\Sp{2n}{\mbZ}] \twoheadrightarrow \St_n^\omega$. The $\Sp{2n}{\mbZ}$-module $\ApartmentModule_n$ appearing in \cref{thm:stpres} is a quotient of $\mbZ[\Sp{2n}{\mbZ}]$ through which Gunnells' apartment class map $\ApartmentMap'_n$ factors, i.e.\ there is a commutative diagram of $\Sp{2n}{\mbZ}$-equivariant maps
\begin{equation}
	\label{eq:apartment-class-map}
\begin{tikzcd}
	\mbZ[\Sp{2n}{\Z}] \arrow[d, two heads] \arrow[dr, two heads, "\ApartmentMap'_n"] & \\
	\ApartmentModule_n \arrow[r, two heads, "\ApartmentMap_n"] & \St_n^\omega.
\end{tikzcd}
\end{equation}
Gunnells' map $\ApartmentMap'_n\colon  \mbZ[\Sp{2n}{\mbZ}] \twoheadrightarrow 
\St_n^\omega$ is defined as follows: A symplectic matrix $M \in \Sp{2n}{\mbZ}$ 
is, by considering the column vectors of $M$, equivalent to the data of an 
ordered symplectic basis $M = (\vec v_1, \vec w_1,\dots, \vec v_n,\vec w_n)$ of 
$\mbZ^{2n}$. We now describe how each such symplectic basis gives rise to a 
unique simplicial embedding 
$$M \colon  S^{n-1} \hookrightarrow T_n^\omega$$
and hence, after fixing a fundamental class $\xi_{n-1} \in 
\widetilde{H}_{n-1}(S^{n-1})$, defines a unique class $\ApartmentMap'_n(M) = 
M_{*}(\xi_{n-1}) \in \St_n^{\omega}$, the \emph{apartment class} of $M$. We 
refer the reader to \cite{bruecksroka2023} for a more detailed account.

\begin{definition}
\label{def:combinatorial-apartment-model}
	Let $\sym{n} \coloneqq \{ 1, \bar{1}, \dots, n, \bar{n}\}$. A nonempty subset $I \subseteq \sym{n}$ is called a \emph{standard subset} if for all $1 \leq a \leq n$ it is true that $\{a,\bar{a}\} \not\subset I$. We denote by $C_n$ the simplicial complex whose vertex set is $\sym{n}$ and whose $k$-simplices are the standard subsets $I \subset \sym{n}$ of size $k+1$.
\end{definition}
Note that the simplicial complex $C_n$ in \cref{def:combinatorial-apartment-model} encodes exactly the simplicial structure of the boundary sphere of the $n$-dimensional cross polytope appearing in \cref{def_prism} and that there is a 
homeomorphism $C_{n+1} \cong C_n \ast \partial\{n+1, 
\overline{n+1}\}$. Hence, $C_n = \ast_{i=1}^{n} S^0$ is a 
simplicial $(n-1)$-sphere. Given any ordered symplectic basis $M = (\vec v_1, 
\vec w_1,\dots, \vec v_n,\vec w_n)$ of $\mbZ^{2n}$, we can label its basis 
vectors from left to right using $\{ 1, \bar{1}, \dots, n, \bar{n}\}$ (i.e.\ 
$\vec M_a = \vec v_a$ and $\vec M_{\bar{a}} = \vec w_a$). Denoting by 
$\Simp(C_n)$ the poset of simplices of $C_n$, we obtain an embedding of posets $M\colon  \Simp(C_n) \hookrightarrow T^\omega_n$ by 
mapping a standard subset $I$ to $M_I = \langle \vec M_z \mid z \in I \rangle$. 
Furthermore, we fix a fundamental class $\xi_{-1} \in \widetilde{H}_{-1}(C_0; \mbZ) = \widetilde{H}_{-1}(\emptyset; \mbZ) = \mbZ$ and define $\xi_{n}$ for $n \geq 0$ as the class in $\widetilde{H}_{n}(C_{n+1}; \mbZ)$ obtained from $\xi_{n-1}$ using the suspension isomorphism $C_{n+1} \cong C_n \ast \partial\{n+1, \overline{n+1}\}$.
This completes our discussion apartment classes and Gunnells' map 
$\ApartmentMap'_n$.

We now turn to the definition of $\ApartmentModule_n$ and  $\ApartmentMap_n$. 
For this we recall that, for $n\geq 1$, the set of all  bijections 
\begin{equation}
	\label{eq:coxeter_details}
	\pi\colon  \sym{n} \to \sym{n}
\end{equation}
with $\pi(\bar a) = \overline {\pi(a)}$ for all $a\in \sym n$  is the group of 
signed permutations. Here $\bar{\bar a} =a$ for $a \in \{1, \dots, n\}$. We denote this group by $\CoxeterGroupTypeB_n$. Let 
$\CoxeterGenerators_n = \ls s_1, \ldots, s_n \rs\subset \CoxeterGroupTypeB_n$ 
denote the subset containing the following permutations: For $1 \leq i < n$, 
$s_i$ swaps $i$ and $(i+1)$ while keeping all other elements fixed, and $s_n$ 
swaps $n$ and $\bar{n}$ while keeping all other elements fixed. Then 
$(\CoxeterGroupTypeB_n,\CoxeterGenerators_n)$ is a 
Coxeter system of type $\mathtt{C}_n = \mathtt{B}_n$ (see \cite{AB:Buildings} or 
\cite{BB:CombinatoricsCoxetergroups} for more details). For $\pi\in 
\CoxeterGroupTypeB_n$, we write $\on{len}(\pi) = \on{len}_\CoxeterGenerators(\pi)$ for the word 
length of $\pi$ with respect to the generating set $\CoxeterGenerators$.\footnote{A combinatorial description of this length function can be found in \cite[Proposition 8.1.1]{BB:CombinatoricsCoxetergroups}.}

\begin{definition}
	\label{def:apartment-module}
	Let $\SymplecticSummand$ be a genus-$k$ symplectic summand of $(\mbZ^{2n}, \omega)$. 
	The symplectic \emph{apartment module} 
	$\ApartmentModule(\SymplecticSummand)$ is the 
	$\on{Sp}(\SymplecticSummand)$-module whose underlying group is free 
	abelian with generators the set of formal symbols $[v_1,w_1, \dots, v_k, w_k]$,
	where
	\begin{itemize}
		\item $(v_1,w_1, \dots, v_k, w_k)$ is a tuple of lines in 
		$\SymplecticSummand$ such that, for 
		some choice of primitive representatives, $(\vec v_1,\vec w_1, \dots, 
		\vec v_k, \vec w_k)$ is a symplectic basis of 
		$\SymplecticSummand$,
	\end{itemize} 
	and where for all $\pi \in \mathcal W_k$, we have
			$$[v_1, v_{\bar{1}}, \dots, v_k, v_{\bar{k}}] = 
			(-1)^{\on{len}(\pi)}\cdot 
			[v_{\pi(1)},v_{\pi(\bar{1})},\dots, v_{\pi(k)}, v_{\pi(\bar{k})}].$$ 
	We write $\ApartmentModule_n = 
	\ApartmentModule(\mbZ^{2n})$ and set $\ApartmentModule(\{0\}) \coloneqq \mbZ$. The $\on{Sp}(\SymplecticSummand)$-action is defined by
	\begin{equation*}
		\GroupElement \cdot [v_1, w_1, \dots, v_k, w_k] = [\GroupElement 
		(v_1),\GroupElement (w_1), \dots, \GroupElement (v_k), \GroupElement 
		(w_k)] \text{ for all } \phi \in \on{Sp}(\SymplecticSummand).
	\end{equation*}
\end{definition}

Note that $\ApartmentModule(\SymplecticSummand)$ is a quotient of 
$\mbZ[\on{Sp}(\SymplecticSummand)]$. The apartment class map on 
$\ApartmentModule(\SymplecticSummand)$ is defined as follows:

\begin{definition}
	\label{def:apartment-map}
	Let $\SymplecticSummand$ be a symplectic summand of $(\mbZ^{2n}, \omega)$ 
	of genus $k$. The symplectic \emph{apartment class map} 
	$\ApartmentMap_\SymplecticSummand\colon  \ApartmentModule(\SymplecticSummand) \to 
	\St(\SymplecticSummand)$ is the unique map such that the following is a 
	commutative diagram of $\on{Sp(\SymplecticSummand)}$-equivariant maps
	\begin{equation*}
		\begin{tikzcd}
			\mbZ[\on{Sp}(\SymplecticSummand)] \arrow[d, two heads] \arrow[dr, two heads, "\ApartmentMap_\SymplecticSummand'"] & \\
			\ApartmentModule(\SymplecticSummand) \arrow[r, two heads, "\ApartmentMap_\SymplecticSummand"] & \St^\omega(\SymplecticSummand)
		\end{tikzcd}
	\end{equation*}
	where $\ApartmentMap_\SymplecticSummand'$ is Gunnells' map that we described above. If $\SymplecticSummand = \mbZ^{2n}$, we write $\ApartmentMap_n = \ApartmentMap_{\mbZ^{2n}}$.
\end{definition}

It follows from \cite[Proposition 3.7.1a and 3.7.1b]{gun2000}\footnote{\cite[Proposition 3.7.1a]{gun2000} contains a misprint. Using the notation of \cite{gun2000}, the corrected statement is: $[v_1, \dots, v_n; v_{\bar n}, \dots, v_{\bar 1}] = \on{sign}(\tau)\cdot [\tau(v_1), \dots, \tau(v_n); \tau(v_{\bar n}), \dots, \tau(v_{\bar 1})].$} that Gunnells' map $\ApartmentMap_\SymplecticSummand$ factors over $\ApartmentModule(\SymplecticSummand)$, i.e.\ that \cref{def:apartment-map} is well-defined. \medskip

\noindent \textbf{Mapping the augmented apartment modules $\SigmaAdditiveApartmentModule_n$ and $\SkewApartmentModule_n$ to $\ApartmentModule_n$:} We now define the two augmented apartment modules $\SigmaAdditiveApartmentModule_n$ and $\SkewApartmentModule_n$, as well as the two maps $\AdditiveMap_n\colon  \SigmaAdditiveApartmentModule_n \to \ApartmentModule_n$ and $\SkewMap_n\colon  \SkewApartmentModule_n \to \ApartmentModule_n$ occurring in \cref{thm:stpres}. We start with the pair $(\SigmaAdditiveApartmentModule_n, \AdditiveMap_n)$, which is related to the $\sigma$-additive simplices appearing in $\IAArel$.

\begin{definition}
	\label{def:additive-apartment-module}
	Let $\SymplecticSummand$ be a genus-$k$ symplectic summand of $(\mbZ^{2n}, \omega)$. The \emph{$\sigma$-additive apartment module} 
	$\SigmaAdditiveApartmentModule(\SymplecticSummand)$ is the 
	$\on{Sp}(\SymplecticSummand)$-module whose underlying group is free 
	abelian with generators the set of formal symbols $[z_0, z_1, z_2] \ast [v_2, w_2, 
	\dots, v_k, w_k]$, where
	\begin{itemize}
		\item $(v_2, w_2, \dots, v_k, w_k)$ is a tuple of 
		lines in $\SymplecticSummand$ such that, for some choice of primitive 
		representatives, $(\vec v_2, \vec w_2, \dots, \vec v_k, \vec w_k)$ is 
		a symplectic basis of a genus-$(k-1)$ summand $\SymplecticSummand_{k-1} 
		\subset \SymplecticSummand$;
		\item $(z_0, z_1, z_2)$ is a tuple of lines in 
		$\SymplecticSummand$ such that $\{z_0, z_1, z_2\}$ is a 
		$\sigma$-additive simplex\footnote{The conditions given in this item are equivalent to saying that for some choice of primitive representatives $\vec z_0, \vec z_1, \vec z_2$, these three vectors span the symplectic complement of $U_{k-1}$ and satisfy $\omega(\vec z_0, \vec z_1) = \omega(\vec z_1, \vec z_2) = \omega(\vec z_0, \vec z_2) = 1$.} in the symplectic complement of 
		$\SymplecticSummand_{k-1}$ in $\SymplecticSummand$;
	\end{itemize}
	and where for all $\tau \in \Sym(\{0,1,2\})$ and $\pi \in 
	\CoxeterGroupTypeB_k$ such that $\pi(1)=1$, we have
	\begin{align*}
	[z_0, z_1, z_2] \ast [v_1, 
		v_{\bar{1}}, \dots, v_k, v_{\bar{k}}] =
	\on{sign}(\tau)(-1)^{\on{len}(\pi)}\cdot 
	[z_{\tau(0)}, z_{\tau(1)}, z_{\tau(2)}] \ast [v_{\pi(2)}, 
	v_{\pi(\bar{2})},\dots, v_{\pi(k)}, 
	v_{\pi(\bar{k})}].
	\end{align*}
	We write $\SigmaAdditiveApartmentModule_n = \SigmaAdditiveApartmentModule(\mbZ^{2n})$ and set $\SigmaAdditiveApartmentModule(\{0\}) \coloneqq \{0\}$. The $\on{Sp}(\SymplecticSummand)$-action is defined by 
	\begin{align*}
	\phi \cdot [z_0, z_1, z_2] \ast [v_2, \dots, w_k] =
	[\phi(z_0), \phi(z_1), \phi(z_2)] \ast [\phi(v_2), \dots, 
 \phi(w_k)] \text{ for all } \phi 
	\in \on{Sp}(\SymplecticSummand).
	\end{align*}
	The $\on{Sp}(\SymplecticSummand)$-equivariant map
	$\AdditiveMap_\SymplecticSummand\colon  \SigmaAdditiveApartmentModule(\SymplecticSummand) \to \ApartmentModule(\SymplecticSummand)$
	is defined by
	\begin{align*}
		\AdditiveMap_\SymplecticSummand([z_0, z_1, z_2] \ast [v_2, \dots, w_k])\coloneqq
		[z_1, z_2, v_2, \dots,  w_k] - [z_1, z_0, v_2,\dots, 
		w_k] - [z_0, z_2, v_2, \dots,   w_k].
	\end{align*}
	If $\SymplecticSummand = \mbZ^{2n}$, we write $\AdditiveMap_n = \AdditiveMap_{\mbZ^{2n}}$.
\end{definition}

\begin{remark}
	In the setting of \cref{def:additive-apartment-module}, an elementary 
	calculation using the symplectic form $\omega$ shows that the terms 
	appearing in the image under $\AdditiveMap_\SymplecticSummand$ of a formal 
	symbol in $\SigmaAdditiveApartmentModule(\SymplecticSummand)$ are indeed 
	formal symbols in $\ApartmentModule(\SymplecticSummand)$. Using the 
	relations between formal symbols in $\ApartmentModule(\SymplecticSummand)$ 
	(see \cref{def:apartment-module}), one then verifies that 
	$\AdditiveMap_\SymplecticSummand$ is well-defined.
\end{remark}

We now introduce $(\SkewApartmentModule_n, \SkewMap_n)$, which is related to the skew-apartment simplices in $\IAArel$.

\begin{definition}
	\label{def:skew-apartment-module}
	Let $\SymplecticSummand$ be a genus-$k$ symplectic summand of $(\mbZ^{2n}, \omega)$. The \emph{skew-apartment module} 
	$\SkewApartmentModule(\SymplecticSummand)$ is the 
	$\on{Sp}(\SymplecticSummand)$-module whose underlying group is free 
	abelian with generators the set of formal symbols $[z_0, z_1, z_2, z_3] \ast [v_3, w_3, 
	\dots, v_k, w_k]$, where
	\begin{itemize}
		\item $(v_3, w_3, \dots, v_k, w_k)$ is a tuple of 
		lines in $\SymplecticSummand$ such that, for some choice of primitive 
		representatives, $(\vec v_3, \vec w_3, \dots, \vec v_k, \vec w_k)$ is 
		a symplectic basis of a summand $\SymplecticSummand_{k-2} 
		\subset \SymplecticSummand$ of genus $k-2$;
		\item $(z_0, z_1, z_2, z_3)$ is a tuple of lines in 
		$\SymplecticSummand$ such that $\{z_0, z_1, z_2, z_3\}$ is a 
		skew-apartment simplex\footnote{The conditions given in this item are equivalent to saying that for some choice of primitive representatives $\vec z_0, \vec z_1, \vec z_2, \vec z_3$, these four vectors span the symplectic complement of $U_{k-2}$ and satisfy $\omega(\vec 		z_0, \vec z_1) = \omega(\vec z_1, \vec z_2) = \omega(\vec z_2, \vec z_3) = 1$.} in the symplectic complement of 
		$\SymplecticSummand_{k-2}$ in $\SymplecticSummand$ and $\omega(\vec 
		z_0, \vec z_1) = \omega(\vec z_1, \vec z_2) = \omega(\vec z_2, \vec 
		z_3) = 1$ for some choice of primitive representatives;
	\end{itemize}
	and where for all $\pi \in \CoxeterGroupTypeB_k$ such that $\pi(1)=1$ and 
		$\pi(2)= 2$, we have
		\begin{align*}
			[z_0, z_1, z_2,z_3] \ast [v_3, v_{\bar{3}}, \dots, v_k, 
		v_{\bar{k}}]
			=&(-1)^{\on{len}(\pi)}\cdot [z_0, z_1, z_2, 
			z_3] \ast [v_{\pi(3)}, v_{\pi(\bar{3})},\dots, v_{\pi(k)},  
			v_{\pi(\bar{k})}]\\
			=&(-1)^{\on{len}(\pi)}\cdot [z_3, z_2, z_1, 
			z_0] \ast [v_{\pi(3)}, v_{\pi(\bar{3})},\dots, v_{\pi(k)}, 	
			v_{\pi(\bar{k})}].
		\end{align*}
	We write $\SkewApartmentModule_n = \SkewApartmentModule(\mbZ^{2n})$ and set $\SkewApartmentModule(\{0\}) \coloneqq \{0\}$. The $\on{Sp}(\SymplecticSummand)$-action is defined by 
	\begin{align*}
		\phi \cdot [z_0, z_1, z_2, z_3] \ast [v_3, \dots, w_k] =
		[\phi(z_0), \phi(z_1), \phi(z_2), \phi(z_3)] \ast [\phi(v_3), 
		\dots,  \phi(w_k)] \text{ for all } \phi 
		\in \on{Sp}(\SymplecticSummand).
	\end{align*}
	The $\on{Sp}(\SymplecticSummand)$-equivariant map
	$\SkewMap_\SymplecticSummand\colon  \SkewApartmentModule(\SymplecticSummand) \to \ApartmentModule(\SymplecticSummand)$
	is defined by
	\begin{multline*}
		\SkewMap_\SymplecticSummand([z_0, z_1, z_2, z_3] \ast [v_3,\dots,w_k])\coloneqq
		[ z_0, z_1, \langle \vec z_0 + \vec z_2 \rangle, z_3, v_3, \dots, 
		w_k] \\
		+ [ z_1, z_2, \langle \vec z_0 + \vec z_2 \rangle, \langle 
		\vec z_1 + \vec z_3 \rangle, v_3,\dots,  w_k]
		+ [z_2, z_3, z_0, \langle \vec z_1 + \vec z_3 \rangle, v_3, 
		\dots, w_k],
	\end{multline*}
	where $(\vec z_0, \vec z_1, \vec z_2, \vec z_3)$ are primitive 
	representatives such that $\omega(\vec z_0, \vec z_1) = \omega(\vec z_1, 
	\vec z_2) = \omega(\vec z_2, \vec z_3) = 1$.
	If $\SymplecticSummand = \mbZ^{2n}$, we write $\SkewMap_n = \SkewMap_{\mbZ^{2n}}$.
\end{definition}

\begin{remark} 
	In the setting of \cref{def:skew-apartment-module}, an elementary 
	calculation using the symplectic form $\omega$ shows that 
	$\SkewMap_\SymplecticSummand$ in the terms in the image of a formal symbol 
	in $\SkewApartmentModule(\SymplecticSummand)$ are indeed in 
	$\ApartmentModule(\SymplecticSummand)$. Using the relations between formal 
	symbols in $\ApartmentModule(\SymplecticSummand)$ (see 
	\cref{def:apartment-module}) one then verifies that 
	$\SkewMap_\SymplecticSummand$ is well-defined.
\end{remark}

We are now ready to prove that \cref{thm:stpres} implies \autoref{thmB}.

\begin{proof}[Reduction of \autoref{thmB} to \cref{thm:stpres}]

	Assuming \autoref{thm:stpres}, $\St^{\omega}_n$ is a quotient of 
	$\ApartmentModule_n$ with two additional relations: 
	The first relation is given by
	\begin{equation}
		\label{eq:relation-sigma-additive}
		0 = [z_1, z_2, v_2, \dots, w_n] - [z_1, z_0, v_2, \dots, w_n]	 - [z_0, z_2, v_2, \dots, w_n],
	\end{equation}
	where $[z_0, z_1, z_2] \ast [v_2, \dots, w_n]$ is a generator of 
	$\SigmaAdditiveApartmentModule_n$, i.e.\ for some choice of primitive 
	vectors, $(\vec z_1, \vec z_2, \vec v_2, \dots, \vec w_n)$ is a 
	symplectic basis of $\mbZ^{2n}$ and $\vec z_0 = \vec z_1 + \vec z_2$ 
	(compare \cref{def:additive-apartment-module} and \cref{def:IAA-simplices}).
	The second relation is given by
	\begin{align}
	\begin{split}
		\label{eq:relation-skew-apartment}
		0 = &[ z_0, z_1, \langle \vec z_0 + \vec z_2 \rangle, z_3, v_3,  
		\dots, w_n] 
		 + [ z_1, z_2, \langle \vec z_0 + \vec z_2 \rangle, \langle 
		 \vec z_1 + \vec z_3 \rangle, v_3,\dots, w_n]\\
		&+ [z_2, z_3, z_0, \langle \vec z_1 + \vec z_3 \rangle, v_3, 
		\dots, w_n],
		\end{split}
	\end{align}
	where $[z_0, z_1, z_2, z_3] \ast [v_3,\dots, w_n]$ is a generator of 
	$\SkewApartmentModule_n$, i.e.\ for some choice of primitive 
	vectors $(\vec z_0, \vec z_1, \vec z_2, \vec z_3, \vec v_3,
	\dots, \vec w_n)$, it holds that 
	\begin{equation}
		\label{eq:translate-to-basis}
		(\vec v_1', \vec w_1', \vec v_2', \vec w_2', \vec v_3', 
		\dots, \vec w_n') \coloneqq (\vec z_0, \vec z_1 + \vec z_3, 
		\vec z_2, \vec z_3, \vec v_3, \dots, \vec w_n)
	\end{equation}
	 is a symplectic basis of $\mbZ^{2n}$ (compare 
	 \cref{def:skew-apartment-module} and \cref{def:IAA-simplices}).
	
	The relation in the definition of $\ApartmentModule_n$ (see 
	\cref{def:apartment-module}) is the first in \autoref{thmB}. The relation 
	imposed by \cref{eq:relation-sigma-additive} directly translates to the 
	second relation claimed in \autoref{thmB}. Finally, the relation imposed by 
	\cref{eq:relation-skew-apartment} can be translated into the third relation 
	in \autoref{thmB}: Using \cref{eq:translate-to-basis}, we obtain
	\begin{align}  
	\begin{split}
		\label{eq:relation-skew-apartment-rewritten}
		0 =& [ v'_1, \langle \vec w'_1-\vec w'_2 \rangle , \langle \vec v'_1 + 
		\vec v'_2 \rangle, w'_2,  v'_3, \dots,	w'_n]
		 + [ \langle \vec w'_1 - \vec w'_2 \rangle, v'_2, \langle \vec v'_1 + 
		 \vec v'_2 \rangle, w'_1, v'_3, \dots, w'_n]\\
		 &+ [ v'_2, w'_2, v'_1,  w'_1, v'_3,  \dots,w'_n].
		 \end{split}
	\end{align}
	\cref{eq:relation-skew-apartment-rewritten} is exactly the third relation  
	of \autoref{thmB}, since the first relation implies
	\begin{equation*} [\vec v'_2, \vec w'_2,\vec v'_1, \vec w'_1,\vec v'_3,\dots, \vec w'_n] = - [ \vec v'_1, \vec w'_1, \vec v'_2, 
	\vec w'_2, \vec v'_3,\dots, \vec w'_n]. \qedhere \end{equation*}
\end{proof}

\subsection{A commutative diagram}
\label{subsec:commutative-diagram}

The goal of this subsection is to explain how the sequences and modules described in the previous subsections are related. This relation is expressed in terms of a commutative diagram and the content of the next proposition.

\begin{proposition}
	\label{prop:commutative-diagram}
	There exist $\Sp{2n}{\mbZ}$-equivariant maps $\pi_{\sigma^2}$, $\pi_{\skewname}$, $\pi_{\addname}$, $\partial_{\sigma^2}$, $\partial_{\skewname}$, $\partial_{\addname}$, $\partial_\sigma$, $\partial_\sigma^{\addname}$ and $\partial_\sigma^{\skewname}$ 
	such that the two sequences occurring in \cref{thm:stpres} and 
	\cref{cor:right-exact-sequence-smaller-steinberg-modules} fit into the 
	commutative diagram depicted in \cref{fig_commdiag}. Furthermore, the maps
	$\pi_{\sigma^2} \oplus \pi_{\skewname} \oplus \pi_{\addname}$ and	$\partial_\sigma$ are surjective.
\end{proposition}

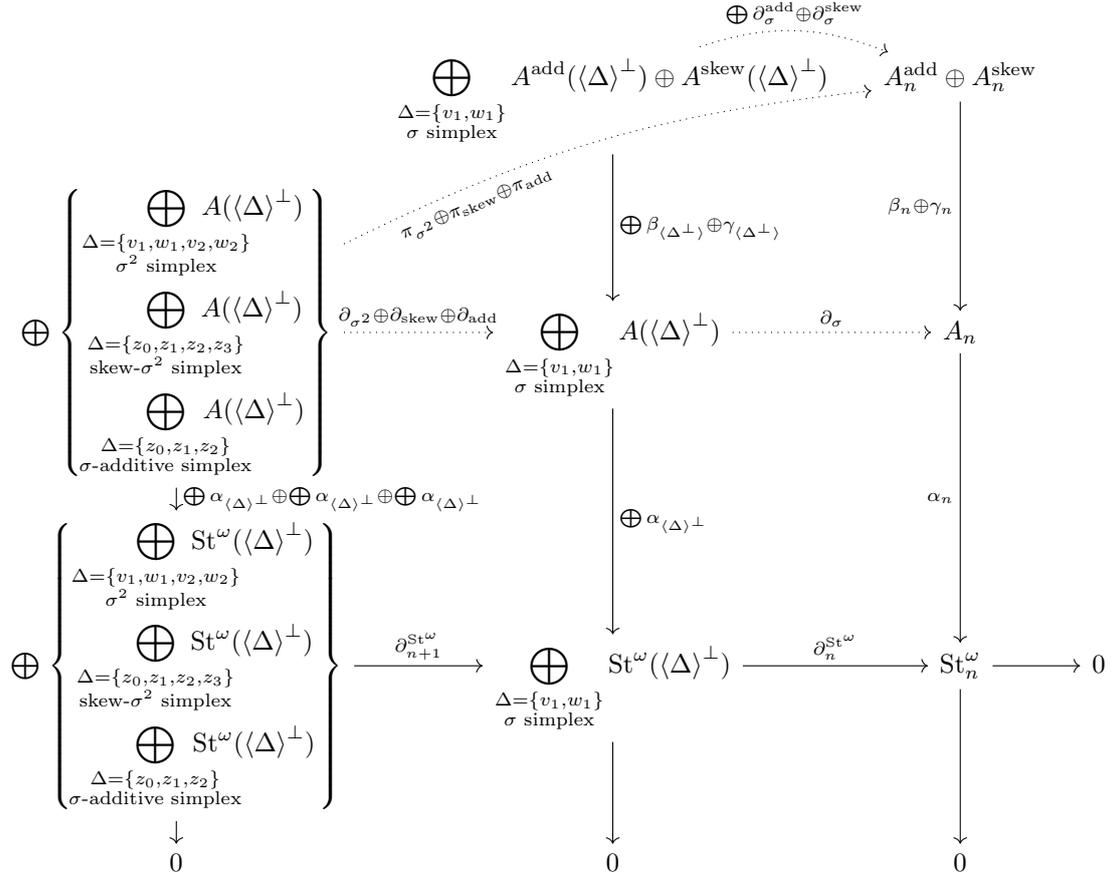
\begin{figure}
	\adjustbox{center}{
		\begin{tikzcd}[column sep=small, row sep=small]		
			& \displaystyle\bigoplus_{\substack{\Delta = \{v_1,w_1\}\\ \text{$\sigma$ simplex}}} \SigmaAdditiveApartmentModule(\ll\Delta\rr^\perp) \oplus \SkewApartmentModule(\ll\Delta\rr^\perp)\arrow[d, "\bigoplus \AdditiveMap_{\langle \Delta^\perp \rangle} \oplus \SkewMap_{\langle \Delta^\perp \rangle}"] \arrow[r, "\bigoplus \partial_\sigma^{\addname} \oplus \partial_\sigma^{\skewname}", dotted, bend left = 20]& \SigmaAdditiveApartmentModule_n \oplus \SkewApartmentModule_n\arrow[d, "\AdditiveMap_n \oplus \SkewMap_n", swap]&\\
			\bigoplus
			\begin{rcases*}
				\begin{dcases*}
					\bigoplus_{\substack{\Delta = \{v_1,w_1,v_2,w_2\}\\ \text{$\sigma^2$ simplex}}} \hspace{-20pt}\ApartmentModule(\ll\Delta\rr^\perp)\\
					\hspace{2pt}\bigoplus_{\substack{\Delta = \{z_0, z_1, z_2, z_3\}\\ \text{skew-$\sigma^2$ simplex}}} \hspace{-18pt}\ApartmentModule(\ll\Delta\rr^\perp)\\
					\hspace{-1pt}\bigoplus_{\substack{\Delta = \{z_0, z_1, z_2\}\\ \text{$\sigma$-additive simplex}}} \hspace{-20pt}\ApartmentModule(\ll\Delta\rr^\perp)
				\end{dcases*}
			\end{rcases*}
			\arrow[r, "\partial_{\sigma^2} \oplus \partial_{\skewname} \oplus \partial_{\addname}", dotted]\arrow[d, "\bigoplus \ApartmentMap_{\langle \Delta \rangle^\perp} \oplus \bigoplus \ApartmentMap_{\langle \Delta \rangle^\perp} \oplus \bigoplus \ApartmentMap_{\langle \Delta \rangle^\perp}"] \arrow[rru, "\pi_{\sigma^2} \oplus \pi_{\skewname} \oplus \pi_{\addname}" {anchor=center, rotate=23, xshift=-12ex, yshift=-1ex}, dotted, bend left = 10, swap] &\displaystyle\bigoplus_{\substack{\Delta = \{v_1,w_1\}\\ \text{$\sigma$ simplex}}}\ApartmentModule(\ll\Delta\rr^\perp)\arrow[d, "\bigoplus \ApartmentMap_{\langle \Delta \rangle^\perp}"]\arrow[r, "\partial_{\sigma}", dotted] & \ApartmentModule_n\arrow[d, "\ApartmentMap_n", swap]&\\
			\bigoplus
			\begin{rcases*}
				\begin{dcases*}
					\bigoplus_{\substack{\Delta = \{v_1,w_1,v_2,w_2\}\\ \text{$\sigma^2$ simplex}}} \hspace{-20pt}\St^\omega(\ll\Delta\rr^\perp)\\
					\hspace{2pt}\bigoplus_{\substack{\Delta = \{z_0, z_1, z_2, z_3\}\\ \text{skew-$\sigma^2$ simplex}}} \hspace{-18pt}\St^\omega(\ll\Delta\rr^\perp)\\
					\hspace{-1pt}\bigoplus_{\substack{\Delta = \{z_0, z_1, z_2\}\\ \text{$\sigma$-additive simplex}}} \hspace{-20pt}\St^\omega(\ll\Delta\rr^\perp)
				\end{dcases*}
			\end{rcases*}
			\arrow[r, "\partial_{n+1}^{\St^\omega}"] \arrow[d]
			& \displaystyle\bigoplus_{\substack{\Delta = \{v_1,w_1\}\\ 
			\text{$\sigma$ simplex}}}\St^\omega(\ll\Delta\rr^\perp) \arrow[r, 
			"\partial_{n}^{\St^\omega}"] \arrow[d] & \St^\omega_n  
			\arrow[d] \arrow[r]& 0\\ 
			0 & 0 & 0 &
		\end{tikzcd}
	}
\caption{Commutative diagram}
\label{fig_commdiag}
\end{figure}

For the proof of \cref{prop:commutative-diagram} and throughout this subsection, we fix a total ordering on the vertex set of $\IAA^{(2)}$. We write $z < z'$ if in this ordering the vertex $z$ is smaller than the vertex $z'$ of $\IAA^{(2)}$. The $\Sp{2n}{\mbZ}$-module structure of the sum terms appearing in \cref{fig_commdiag} is defined analogous to \cref{rem:module-structure-right-hand-term}. Keeping this in mind, we now define the equivariant maps appearing \cref{fig_commdiag}.

\begin{definition}
	\label{def:partial-sigma2-skew-add}
	Let $\Delta \in \IAA^{(2)}$ be a minimal simplex of type $\tau \in \{ \sigma^2, \text{skew-}\sigma^2, \sigma\text{-additive}\}.$
	Then the value of
	$$\partial_{\tau}\colon  \ApartmentModule(\langle \Delta \rangle^\perp) \to 
	\bigoplus_{\substack{\Delta' = \{v_1,w_1\}\\ \text{$\sigma$ 
	simplex}}}\ApartmentModule(\ll\Delta'\rr^\perp)$$
	at a formal symbol 
	$[v_k, \dots, w_n] \in 
	\ApartmentModule(\langle \Delta \rangle^\perp),$
	where $k = 2$ if $\tau = \sigma\text{-additive}$ and $k = 3$ otherwise,
	is defined as follows.
	\begin{itemize}[leftmargin=*]
		\item If $\Delta = \{v_1, w_1, v_2, w_2\}$ is a $\sigma^2$ simplex such that $\omega(\vec v_1, \vec w_1) = \omega(\vec v_2, \vec w_2) = 1$, $v_1 < w_1$ and $v_2 < w_2$, then
		\begin{equation*}
			\partial_{\sigma^2}([ v_3, \dots, w_n])  = 
			[v_1, w_1, v_3,\dots, w_n] \oplus [v_2, w_2, v_3, \dots, w_n]
		\end{equation*}
		 in $\ApartmentModule(\langle v_2, w_2 \rangle^\perp) \oplus 
		 \ApartmentModule(\langle v_1, w_1 \rangle^\perp)$ indexed by the $\sigma$ simplices $ \{v_2, w_2\}$ and $\{v_1, w_1\}$, respectively.
		
		\item If $\Delta = \{z_0, z_1, z_2, z_3\}$ is 
		a skew-$\sigma^2$ simplex such that $\omega(\vec 
		z_0, \vec z_1) = \omega(\vec z_1, \vec z_2) = \omega(\vec z_2, \vec 
		z_3) = 1$,
		let 
\begin{eqnarray*}
\ls v_0, w_0 \rs &= \{\langle \vec z_0 + \vec z_2 \rangle, z_3 \}, &\ls v_1, w_1 \rs = \{\langle \vec z_0 + \vec z_2 \rangle, \langle \vec z_1 + \vec z_3 \rangle \}, \\
\ls v_2, w_2 \rs &= \{z_0 ,\langle \vec z_1 + \vec z_3 \rangle\}, &\text{where } v_i<w_i \text{ for }0\leq i \leq 2.
\end{eqnarray*}		
		Define
		\begin{equation*}
			\partial_{\skewname}([v_3, \dots, w_n])  = [v_0, w_0, v_3, \dots, w_n] \oplus
			[v_1, w_1, v_3, \dots, w_n] \oplus [v_2, w_2, v_3, \dots, w_n]
		\end{equation*}
		in $\ApartmentModule(\langle z_0, z_1  \rangle^\perp) \oplus 
		\ApartmentModule(\langle  z_1, z_2  \rangle^\perp) \oplus 
		\ApartmentModule(\langle  z_2, z_3 \rangle^\perp)$  indexed by the $\sigma$ simplices $\{ z_0, z_1 \}$, $ \{ z_1, z_2 \}$ and $ \{ z_2, z_3 \}$, respectively.
		\item If $\Delta = \{z_0, z_1, z_2\}$ is a 
		$\sigma$-additive simplex such that $z_0 < z_1 <z_2$, then
		\begin{equation*}
			\partial_{\addname}([ v_2,  \dots,   w_n])  = [ v_2, 
			\dots,  v_n,  w_n] \oplus
			-[ v_2,  w_2, \dots,   w_n] \oplus [ v_2, \dots,  w_n]
		\end{equation*} in $\ApartmentModule(\langle  z_1, z_2 \rangle^\perp) \oplus \ApartmentModule(\langle  z_0, z_2 \rangle^\perp) \oplus \ApartmentModule(\langle  z_0, z_1 \rangle^\perp)$ indexed by the $\sigma$ simplices $ \{ z_1, z_2 \}$, $ \{ z_0, z_2 \}$ and $\{ z_0, z_1 \}$, respectively.
	\end{itemize}
\end{definition}

\begin{definition}
	\label{def:partial-sigma}
	Let $\Delta = \{v_1, w_1\}$ be a minimal $\sigma$ simplex in $\IAA^{(2)}$ such that $v_1 < w_1$.
	\begin{itemize}[leftmargin=*]
		\item The map $\partial_{\sigma}\colon  \ApartmentModule(\langle \Delta \rangle^\perp) \to \ApartmentModule_n$ is defined by
		$$\partial_{\sigma}([ v_2,  \dots,  w_n])  = [ v_1,  w_1, v_2, \dots,  w_n].$$
		\item The map $\partial_{\sigma}^{\addname}\colon  \SigmaAdditiveApartmentModule(\langle \Delta \rangle^\perp) \to \SigmaAdditiveApartmentModule_n$ is defined by
		\begin{equation*}
			\partial_{\sigma}([z_0, z_1, z_2] \ast [v_3, 
		\dots, w_n])  =-[z_0, z_1, z_2] \ast [ v_1,  w_1, v_3,\dots,  w_n]
		\end{equation*}
		\item The map $\partial_{\sigma}^{\skewname}\colon  \SkewApartmentModule(\langle \Delta \rangle^\perp) \to \SkewApartmentModule_n$ is defined by
		\begin{equation*}
			\partial_{\sigma}([z_0, z_1, z_2, z_3] \ast [v_4,\dots, w_n])  =
			[z_0, z_1, z_2, z_3] \ast [v_1, w_1, v_4, \dots, w_n].
		\end{equation*}
	\end{itemize}
\end{definition}

\begin{definition}
	\label{def:pi-sigma2-skew-add}
	Let $\Delta \in \IAA^{(2)}$ be a minimal simplex.
	\begin{itemize}[leftmargin=*]
		\item If $\Delta$ is a $\sigma^2$ simplex, the value of 
		$$\pi_{\sigma^2} = 0\colon  \ApartmentModule(\langle \Delta \rangle^\perp) \to \SigmaAdditiveApartmentModule_n \oplus \SkewApartmentModule_n$$
		at all formal symbols $[ v_3, \dots, w_n] \in 
		\ApartmentModule(\langle \Delta \rangle^\perp)$ is zero.
		\item If $\Delta = \{z_0, z_1, z_2, z_3\}$ is a skew-$\sigma^2$ 
		simplex such that $\omega(\vec 
		z_0, \vec z_1) = \omega(\vec z_1, \vec z_2) = \omega(\vec z_2, \vec 
		z_3) = 1$, then the map $\pi_{\skewname}\colon  \ApartmentModule(\langle \Delta \rangle^\perp) \to \SkewApartmentModule_n \hookrightarrow \SigmaAdditiveApartmentModule_n \oplus \SkewApartmentModule_n$ is defined by
		$$\pi_{\skewname}([ v_3,  \dots,  w_n])  = [z_0, z_1, z_2, z_3] 
		\ast [v_3, \dots,  w_n].$$
		\item If $\Delta = \{z_0, z_1, z_2\}$ is a $\sigma$-additive simplex such that $z_0 < z_1 <z_2$, 
		then the map $\pi_{\addname}\colon  \ApartmentModule(\langle \Delta \rangle^\perp) \to \SigmaAdditiveApartmentModule_n \hookrightarrow\SigmaAdditiveApartmentModule_n \oplus \SkewApartmentModule_n$ is defined by
		$$\pi_{\addname}([ v_2,  \dots,  w_n])  = [z_0, z_1, z_2] \ast 
		[v_2,  \dots,  w_n].$$
	\end{itemize}
\end{definition}

\subsubsection{The diagram commutes and surjectivity properties}

This subsection contains the proof of \cref{prop:commutative-diagram}, which is split into several lemmas. We first note that the claimed surjectivies hold because by definition, every generator of $\ApartmentModule_n$ and $\SigmaAdditiveApartmentModule_n$ is in the image of $\pi_{\sigma^2} \oplus \pi_{\skewname} \oplus \pi_{\addname}$ and $\partial_\sigma$, respectively. Hence, we have:

\begin{lemma}
	\label{lem:pi-partial-sigma-surjective}
	The maps $\pi_{\sigma^2} \oplus \pi_{\skewname} \oplus \pi_{\addname}$ and $\partial_\sigma$ are surjective.
\end{lemma}

We now discuss the commutativity of the diagram depicted in \cref{fig_commdiag}.

\begin{lemma}
	\label{lem:commutativity-one}
	\label{lem:commutativity-two}
	The following hold.
	\begin{enumerate}
	\item $(\AdditiveMap_n \oplus \SkewMap_n) \circ (\pi_{\sigma^2} \oplus \pi_{\skewname} \oplus \pi_{\addname}) = \partial_\sigma \circ (\partial_{\sigma^2} \oplus \partial_{\skewname} \oplus \partial_{\addname})$.
	\item $(\AdditiveMap_n \oplus \SkewMap_n) \circ (\partial^{\addname}_{\sigma} \oplus \partial^{\skewname}_{\sigma}) = \partial_\sigma \circ (\AdditiveMap_{\langle \Delta \rangle^\perp} \oplus \SkewMap_{\langle \Delta \rangle^\perp})$.
	\end{enumerate}
\end{lemma}

Checking the validity of this lemma is elementary; one simply picks 
a generator, evaluates the map on the left and right side using the definitions 
above and uses the first relation in the apartment module $\ApartmentModule_n$ 
(see \cref{def:apartment-module}) to see that the two values agree. We leave the 
details to the reader, and now focus on the remaining two commuting squares for 
which the argument is more involved. 

\begin{lemma}
	\label{lem:commutativity-three}
	It holds that 
	$$\left(\bigoplus \ApartmentMap_{\langle \Delta \rangle^\perp}\right) \circ 
	(\partial_{\sigma^2} 
	\oplus \partial_{\skewname} \oplus \partial_{\addname}) = \partial_{n+1}^{\St^\omega} 
	\circ \left(\bigoplus \ApartmentMap_{\langle \Delta \rangle^\perp} \oplus 
	\bigoplus \ApartmentMap_{\langle \Delta \rangle^\perp} \oplus \bigoplus 
	\ApartmentMap_{\langle \Delta \rangle^\perp}\right).$$
\end{lemma}

\begin{lemma}
	\label{lem:commutativity-four}
	It holds that 
	$\ApartmentMap_n \circ \partial_{\sigma} = \partial_{n}^{\St^\omega} \circ 
	\left(\bigoplus \ApartmentMap_{\langle \Delta \rangle^\perp}\right).$
\end{lemma}

To prove these two lemmas, we need to unravel the identifications made in the construction of the exact sequence in \cref{cor:right-exact-sequence-smaller-steinberg-modules}, which appears in the bottom row of the diagram in \cref{fig_commdiag}. We use the following construction.\medskip

\textbf{Construction:} Let $\Delta \in \IAA$ be a minimal simplex of type 
$\sigma^2$, skew-$\sigma^2$, $\sigma$-additive or $\sigma$ and let 
$M^{\Delta^\perp} = (\vec v_k, \vec w_k, \dots, \vec v_n, \vec w_n)$ be an 
ordered symplectic basis of $\langle \Delta \rangle^\perp \subseteq \mbZ^{2n}$, 
where $k = 2$ if $\Delta$ is $\sigma$-additive or of type $\sigma$, and $k = 3$ 
otherwise. Using the notation introduced in 
\cref{def:combinatorial-apartment-model}, we can label the basis vectors in 
$M^{\Delta^\perp}$ by $\{ 1, \bar{1}, \dots, n-k+1, \overbar{n-k+1}\}$ from 
left to right (i.e.\ $\vec M_i^{\Delta^\perp} = \vec v_{i+k-1}$ and 
$\vec M_{\bar{i}}^{\Delta^\perp} = \vec w_{i+k-1}$).  Therefore this data determines 
a simplicial embedding
$$M^{\Delta^\perp}\colon  \Delta \ast C_{n-k+1} \hookrightarrow \IAA$$
mapping $\Delta$ to $\Delta$, $i$ to $\vec M^{\Delta^\perp}_i$ and $\bar i$ to 
$\vec M_{\bar{i}}^{\Delta^\perp}$. Note that $\Delta \ast C_{n-k+1}$ is a combinatorial $(n+1)$- or $n$-ball 
with boundary sphere $\partial \Delta \ast C_{n-k+1}$. Taking the 
simplex types into account, we hence obtain a pair of simplicial 
embeddings
$$\left(M^{\Delta^\perp}, \partial M^{\Delta^\perp}\right)\colon  (\Delta \ast C_{n-k+1}, \partial \Delta \ast C_{n-k+1}) \hookrightarrow 
(\IAA^{(2)}, \IAA^{(1)})$$
if $\Delta$ is of type $\sigma^2$, skew-$\sigma^2$ or $\sigma$-additive and
$$\left(M^{\Delta^\perp}, \partial M^{\Delta^\perp}\right)\colon  (\Delta \ast C_{n-k+1}, \partial \Delta \ast C_{n-k+1}) \hookrightarrow 
(\IAA^{(1)}, \IAA^{(0)})$$
if $\Delta$ is of type $\sigma$. Now, $M^{\Delta^\perp}$ is a cross map (see \cref{def:cross-maps} and \cref{def_sigma_cross_map}) except if $\Delta$ is a skew-$\sigma^2$ simplex (since we did not introduce a notion of cross map for this type). However, if $\Delta = \{z_0, z_1, z_2, z_3\}$ is a skew-$\sigma^2$ simplex and $\omega(\vec 
z_0, \vec z_1) = \omega(\vec z_1, \vec z_2) = \omega(\vec z_2, \vec 
z_3) = 1$, then \cref{lem_unique_skew_additive_face} implies that there is a unique prism $P$ 
in $\IAA$ containing $\Delta$. The two additional vertices in $P$ are $\langle 
\vec z_0 + \vec z_2 \rangle$ and $\langle \vec z_1 + \vec z_3 
\rangle$. Therefore the map $M^{\Delta^{\perp}}$ above 
extends to a unique prism-regular map
$$M^{\Delta^\perp}\colon  P \ast C_{n-k+1} \hookrightarrow \IAA$$
and we obtain a pair of embeddings
$$\left(M^{\Delta^\perp}, \partial M^{\Delta^\perp}\right)\colon  (P \ast C_{n-k+1}, \partial P \ast C_{n-k+1}) \hookrightarrow 
(\IAA^{(2)}, \IAA^{(1)}).$$

Let $d = \dim(\Delta)$ and let $\eta_{d-1} \in \widetilde{H}_{d-1}(\partial \Delta)$ denote the fundamental class associated to the ordering of the vertex set of $\Delta$. Since the $\partial \Delta$-suspension isomorphism is induced by taking the cross product with $\eta_{d - 1}$ (see e.g.\ \cite[p.276 et seq.]{Hatcher2002}), we obtain a unique fundamental class $\eta_{d-1} \ast \xi_{n-1} \in \widetilde{H}_{d+n-1}(\partial \Delta \ast C_n; \mbZ)$ with the property that for all $n \geq 0$ it holds that
\begin{align}
	\label{eq:generators-for-construction}
	\begin{split}
	\widetilde{H}_{d+n-1}(\partial \Delta \ast C_n; \mbZ)\stackrel{\cong}{\leftarrow} & \,\widetilde{H}_{n-1}(C_n; \mbZ) \\
	\eta_{d-1} \ast \xi_{n-1} \rotatebox[origin=c]{180}{$\mapsto$}& \, \xi_{n-1}
	\end{split}
\end{align}
(compare with \cref{def:combinatorial-apartment-model} et seq.). If $\Delta$ is a minimal skew-$\sigma^2$ simplex, there is a unique class $\eta_2 \in \widetilde{H}_2(\partial P_3; \mbZ)$ that is homologous to $\eta_2 \in \widetilde{H}_2(\partial \Delta; \mbZ)$ in $P_3^\diamond$ (see \cref{def:prism-diamond} for the definition of $P_3^\diamond$, and \cref{lem_deformation_retract_2skew_additive} for a closely related argument). Hence, we obtain a fundamental class $\eta_{2} \ast \xi_{n-1} \in \widetilde{H}_{n+2}(\partial P_3 \ast C_n; \mbZ)$ exactly as above. Finally, we note that if $\Delta = \{v_1, w_1\}$ is a minimal $\sigma$ simplex and $v_1 < w_1$, there is an identification $\partial \Delta \ast C_{n} \cong C_{n+1}$ mapping
\begin{align*}
v_1&\text{ to }1, &w_1&\text{ to }\overline 1,& i&\text{ to }i+1, &\overline i &\text{ to }\overline{i+1}
\end{align*}
$i \in \{1, \dots, n\}$.
Under this identification, it holds that $\eta_{0} \ast \xi_{n-1} \mapsto \xi_{n}$.

To relate the maps above to the apartment class maps (see \cref{def:apartment-map}) in \cref{fig_commdiag}, we furthermore use that for any $n \geq 1$, the long exact sequence of the pairs $(\Delta \ast C_{n-k+1}, \partial \Delta \ast C_{n-k+1})$ and $(P \ast C_{n-k+1}, \partial P \ast C_{n-k+1})$ yields identifications and relative homology classes
\begin{align*}
	\widetilde{H}_{d+n-k}(\partial \Delta \ast C_{n-k+1}; \mbZ) &\cong H_{d+n-k+1}(\Delta \ast C_{n-k+1}, \partial \Delta \ast C_{n-k+1}; \mbZ)\\
	\widetilde{H}_{n}(\partial P \ast C_{n-k+1}; \mbZ) &\cong H_{n+1}(P \ast C_{n-k+1}, \partial P \ast C_{n-k+1}; \mbZ)\\
	\eta_{d-1} \ast \xi_{n-k} &\mapsto (0, \eta_{d-1} \ast \xi_{n-k}).
\end{align*}

From this construction, we obtain the following ``topological'' description of the image of the ``relative'' apartment class maps occurring in \cref{fig_commdiag}.

\begin{corollary}
	\label{cor:relative-apartment-map-geometric}
	The following correspondence holds under the identifications in 
	\cref{lem_rel_homology_IAAi_IAAj}:
	\begin{itemize}[leftmargin=*]
		\item If $\Delta$ is of type $\sigma^2$, then $\ApartmentMap_{\langle \Delta \rangle^\perp}([ v_3, 
		\dots,  w_n]) \in \St^\omega(\langle \Delta \rangle^\perp)$
		identifies with
		\begin{align*}
			\left(\Delta \ast [ v_3,  \dots,  w_n], \partial \Delta \ast [ v_3,   \dots, w_n]\right) \coloneqq
			 \left(M^{\Delta^\perp}, \partial M^{\Delta^\perp}\right)_*(0, \eta_2 \ast \xi_{n-3}) \in H_{n+1}(\IAA^{(2)}, \IAA^{(1)}).
		\end{align*}
		\item If $\Delta$ is of type skew-$\sigma^2$, then $\ApartmentMap_{\langle \Delta \rangle^\perp}([ v_3, \dots,   w_n]) \in \St^\omega(\langle \Delta \rangle^\perp)$ identifies with
		\begin{align*}
			\left(P \ast [ v_3,  \dots,  w_n], \partial P \ast 
			[ v_3,  \dots,   w_n]\right) \coloneqq
			& \left(M^{\Delta^\perp}, \partial  M^{\Delta^\perp}\right)_*(0, \eta_2 \ast \xi_{n-3}) \in H_{n+1}(\IAA^{(2)}, \IAA^{(1)}).
		\end{align*}
		\item If $\Delta$ is of type $\sigma$-additive, then
		$\ApartmentMap_{\langle \Delta \rangle^\perp}([ v_2,
		\dots,  w_n]) \in \St^\omega(\langle \Delta \rangle^\perp)$ 
		identifies with
		\begin{align*}
			\left(\Delta \ast [ v_2,  \dots,  w_n], \partial 
			\Delta \ast [ v_2,  
			\dots, w_n]\right) \coloneqq
			 \left(M^{\Delta^\perp}, \partial M^{\Delta^\perp}\right)_*(0, \eta_1 \ast \xi_{n-2}) \in H_{n+1}(\IAA^{(2)}, \IAA^{(1)}).
		\end{align*}
	\end{itemize}
	Finally, under the identifications made in	
	\cref{lem_rel_homology_IAAi_IAAj} and assuming that 
	$\Delta$ is of type $\sigma$, the class $\ApartmentMap_{\langle \Delta 
		\rangle^\perp}([ v_2,  \dots,   w_n]) \in 
	\St^\omega(\langle \Delta \rangle^\perp)$ identifies with
	\begin{align*}
		\left(\Delta \ast [ v_2, \dots, w_n], \partial 
		\Delta \ast [ v_2, \dots,  w_n]\right) \coloneqq
		 \left(M^{\Delta^\perp}, \partial M^{\Delta^\perp}\right)_*(0, \eta_0 \ast \xi_{n-2}) \in H_{n}(\IAA^{(1)}, \IAA^{(0)}).
	\end{align*}
\end{corollary}

With these observations and definitions in place, we are now ready to prove 
\cref{lem:commutativity-three} and \cref{lem:commutativity-four}.

\begin{proof}[Proof of \cref{lem:commutativity-three} for summands indexed by 
$\sigma^2$ simplices]
Let $\Delta = \{v_1, w_1, v_2, w_2\}$ be a minimal $\sigma^2$ simplex in $\IAA^{(2)}$ such that $\omega(\vec v_1, \vec w_1) = \omega(\vec v_2, \vec w_2) = 1$, $v_1 < w_1$ and $v_2 < w_2$. Let $[ v_3,  \dots, 
 w_n] \in \ApartmentModule(\langle \Delta \rangle^\perp)$ denote a formal 
symbol. We claim that
\begin{equation}
	\label{eq:sigmasquare}
	\left(\left(\ApartmentMap_{\langle v_2,w_2 \rangle^\perp} \oplus \ApartmentMap_{\langle v_1, 
	w_1 \rangle^\perp}\right) \circ \partial_{\sigma^2}\right) ([ 
	v_3, \dots,  w_n]) = \left(\partial_{n+1}^{\St^\omega} 
	\circ \ApartmentMap_{\langle v_1,w_1,v_2,w_2 \rangle^\perp}\right)([ v_3,  \dots,   w_n]).
\end{equation}
Using the definition of $\partial_{\sigma^2}$ (see 
\cref{def:partial-sigma2-skew-add}) and the last part of 
\cref{cor:relative-apartment-map-geometric}, it follows that, under the 
identifications made in the 
\cref{cor:right-exact-sequence-smaller-steinberg-modules}, the
left side term is equal to the following element in $H_n(\IAA^{(1)}, 
\IAA^{(0)})$:
\begin{multline}
	\label{eq:sigmasquareleftterm}
	\left(\{v_2,w_2\} \ast [ v_1,  w_1,  v_3,  \dots,  
	 w_n], \partial \{v_2,w_2\} \ast [ v_1,  w_1,  
	v_3, \dots,  w_n]\right) \\
	+ \left(\{v_1,w_1\} \ast [ v_2,  w_2,  
	v_3,  \dots,   w_n], \partial \{v_1,w_1\} 
	\ast [ v_2,  w_2,  v_3, \dots,   w_n]\right).
\end{multline}
Using the definition of the connection morphism $\partial_{n+1}^{(2,1,0)}$ 
and the first part of \cref{cor:relative-apartment-map-geometric}, it 
similarly follows that the 
right hand term is equal to
\begin{equation}
	\label{eq:sigmasquarerightterm}
	 \left(\partial \{v_1,w_1,v_2,w_2\} \ast  [ v_3, \dots, 
	  w_n], \emptyset\right) =
	  \left(\partial M^{\Delta^\perp}, \emptyset\right)_*(\eta_2 \ast \xi_{n-3}, \emptyset) \in H_{n+1}(\IAA^{(1)}, \IAA^{(0)}).
\end{equation}
Recall that the domain of $\partial M^{\Delta^\perp}$ is the combinatorial $(n+1)$-sphere $\partial \Delta \ast C_{n-2}$. We decompose the sphere $\partial \Delta = \partial \{v_1,w_1,v_2,w_2\}$ into two combinatorial 2-balls of the form $\Delta^1 \ast S^0$ given by $\{v_2,w_2\} \ast \partial 
\{v_1,w_1\}$ and $\{v_1,w_1\} \ast \partial \{v_2,w_2\}$, as illustrated in 
\cref{figure_homology_sigma2}. This induces a decomposition of $\partial \Delta \ast C_{n-2}$ into two combinatorial $(n+1)$-balls of the form $(\Delta^1 \ast S^0) \ast C_{n-2} \cong \Delta^1 \ast C_{n-1}$. This allows us to express the homology class in \cref{eq:sigmasquarerightterm} as a sum of two terms which are equal to \cref{eq:sigmasquareleftterm}. We conclude the left and right hand term of \cref{eq:sigmasquare} agree.
\begin{figure}
\begin{center}
\includegraphics{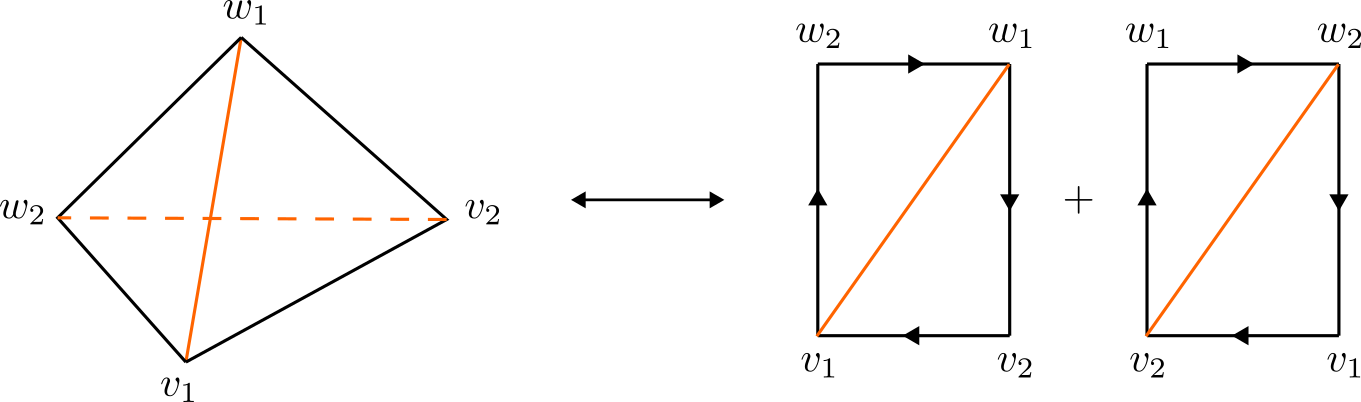}
\end{center}
\caption{Decomposing the boundary sphere $\partial \Delta$ of a $\sigma^2$ simplex $\Delta$ into two balls $\Delta^1 \ast S^0$. The induced orientation is indicated by arrows. Orange marks $\sigma$ edges.}
\label{figure_homology_sigma2}
\end{figure}
\end{proof}

\begin{proof}[Proof of \cref{lem:commutativity-three} for summands indexed by 
skew-$\sigma^2$ simplices] 
Consider a minimal skew-$\sigma^2$ simplex $\Delta = \{z_0, z_1, z_2, z_3\} \in \IAA^{(2)}$ such that $\omega(\vec z_0, \vec z_1) = \omega(\vec z_1, \vec z_2) = \omega(\vec z_2, \vec z_3) = 1$. Let 
\begin{eqnarray*}
	\ls v_0, w_0 \rs &= \{\langle \vec z_0 + \vec z_2 \rangle, z_3 \}, &\ls v_1, w_1 \rs = \{\langle \vec z_0 + \vec z_2 \rangle, \langle \vec z_1 + \vec z_3 \rangle \}, \\
	\ls v_2, w_2 \rs &= \{z_0 ,\langle \vec z_1 + \vec z_3 \rangle\}, &\text{where } v_i<w_i \text{ for }0\leq i \leq 2,
\end{eqnarray*}
and let $[v_3,  w_3, \dots, v_n,  w_n] \in \ApartmentModule(\langle \Delta \rangle^\perp)$ be a formal symbol. We claim that
\begin{multline}
	\label{eq:skewsigma}
	\left(\left(\ApartmentMap_{\langle z_0,z_1 \rangle^\perp} \oplus \ApartmentMap_{\langle z_1, z_2 \rangle^\perp} \oplus \ApartmentMap_{\langle z_2, z_3 \rangle^\perp}\right) \circ \partial_{\skewname}\right) ([v_3, \dots, w_n]) =\\
	\left(\partial_{n+1}^{\St^\omega} \circ \ApartmentMap_{\langle z_0,z_1,z_2,z_3 \rangle^\perp}\right)([v_3, \dots, w_n]).
\end{multline}
Using the definition of $\partial_{\skewname}$ (see \cref{def:partial-sigma2-skew-add}) and the last part of \cref{cor:relative-apartment-map-geometric}, it follows that, under the identifications made in the \cref{cor:right-exact-sequence-smaller-steinberg-modules}, the left side term is equal to the following element in $H_n(\IAA^{(1)}, \IAA^{(0)})$:
\begin{multline}
	\label{eq:skewsigmaleftterm}
	\bigl(\{z_0,z_1\} \ast [v_0,  w_0, v_3, \dots, w_n], \partial \{z_0,z_1\} \ast [v_0,  w_0, v_3, \dots, w_n]\bigr)\\
	+ \bigl(\{z_1,z_2\} \ast [v_1, w_1, v_3,\dots,  w_n], \partial \{z_1,z_2\} \ast [v_1, w_1, v_3, \dots, w_n]\bigr)\\
	+ \bigl(\{z_2,z_3\} \ast [v_2, w_2, v_3, \dots,  w_n], \partial \{z_2,z_3\}\ast [v_2, w_2, v_3,\dots,  w_n]\bigr).
\end{multline}
Using the definition of the connection morphism $\partial_{n+1}^{(2,1,0)}$ and using the second part of \cref{cor:relative-apartment-map-geometric}, it similarly follows that the right hand term is equal to
\begin{equation}
	\label{eq:skewsigmarightterm}
	\left(\partial P \ast  [ v_3,  \dots,  w_n], \emptyset\right) =
	\left(\partial M^{\Delta^\perp}, \emptyset\right)_*(\eta_2 \ast \xi_{n-3}, \emptyset) \in H_{n+1}(\IAA^{(1)}, \IAA^{(0)}).
\end{equation}
Recall that the domain of $\partial M^{\Delta^\perp}$ is the combinatorial $(n+1)$-sphere $\partial P \ast C_{n-2}$. We decompose $\partial P$ into five combinatorial 2-balls; three of the form $\Delta^1 \ast S^0$, given by 
\begin{equation*}
\{z_0, z_1\} \ast \partial \{\langle \vec z_0 + \vec z_1 \rangle, z_3\}, \{z_1, z_2\} \ast \partial \{\langle \vec z_0 + \vec z_1 \rangle, \langle \vec z_1 + \vec z_3 \rangle\}, \{z_2, z_3\} \ast \partial \{z_0, \langle \vec z_1 + \vec z_3 \rangle\},
\end{equation*}
and two 2-simplices $\Delta^2$, given by 
\begin{equation*}
\{z_0, z_1, \langle \vec z_0 + \vec z_1 \rangle \}, \{z_1, z_3, \langle \vec z_1 + \vec z_3 \rangle\}.
\end{equation*}
This induces a decomposition of $\partial P \ast C_{n-2}$ into five combinatorial $(n+1)$-balls; three are of the form $(\Delta^1 \ast S^0) \ast C_{n-2} \cong \Delta^1 \ast C_{n-1}$ and two are of the form $\Delta^2 \ast C_{n-2}$. Here, we have that $S^0 = \partial \{v_i, w_i\}$ for $i \in \{0,1,2\}$ and we use the identification $(\Delta^1 \ast S^0) \ast C_{n-2} \cong \Delta^1 \ast C_{n-1}$ mapping 
\begin{align*}
\Delta^1 &\text{ to } \Delta^1, &v_i&\text{ to }1, &w_i&\text{ to }\overline 1,& j&\text{ to }j+1, &\overline j &\text{ to }\overline{j+1}
\end{align*}
for $j \in \{1, \dots, n-2\}$. This allows us to express the homology class in \cref{eq:skewsigmarightterm} as a sum of five terms. The first three correspond to the balls $\Delta^1 \ast C_{n-1}$ and are exactly the terms in \cref{eq:skewsigmaleftterm}. The other two correspond to the balls $\Delta^2 \ast C_{n-2}$ and are zero, since the image of the restriction of $\partial M^{\Delta^\perp}$ to these balls is entirely contained in $\IAA^{(0)}$. We conclude that the left and right hand term of \cref{eq:skewsigma} agree.
\begin{figure}
\begin{center}
\includegraphics{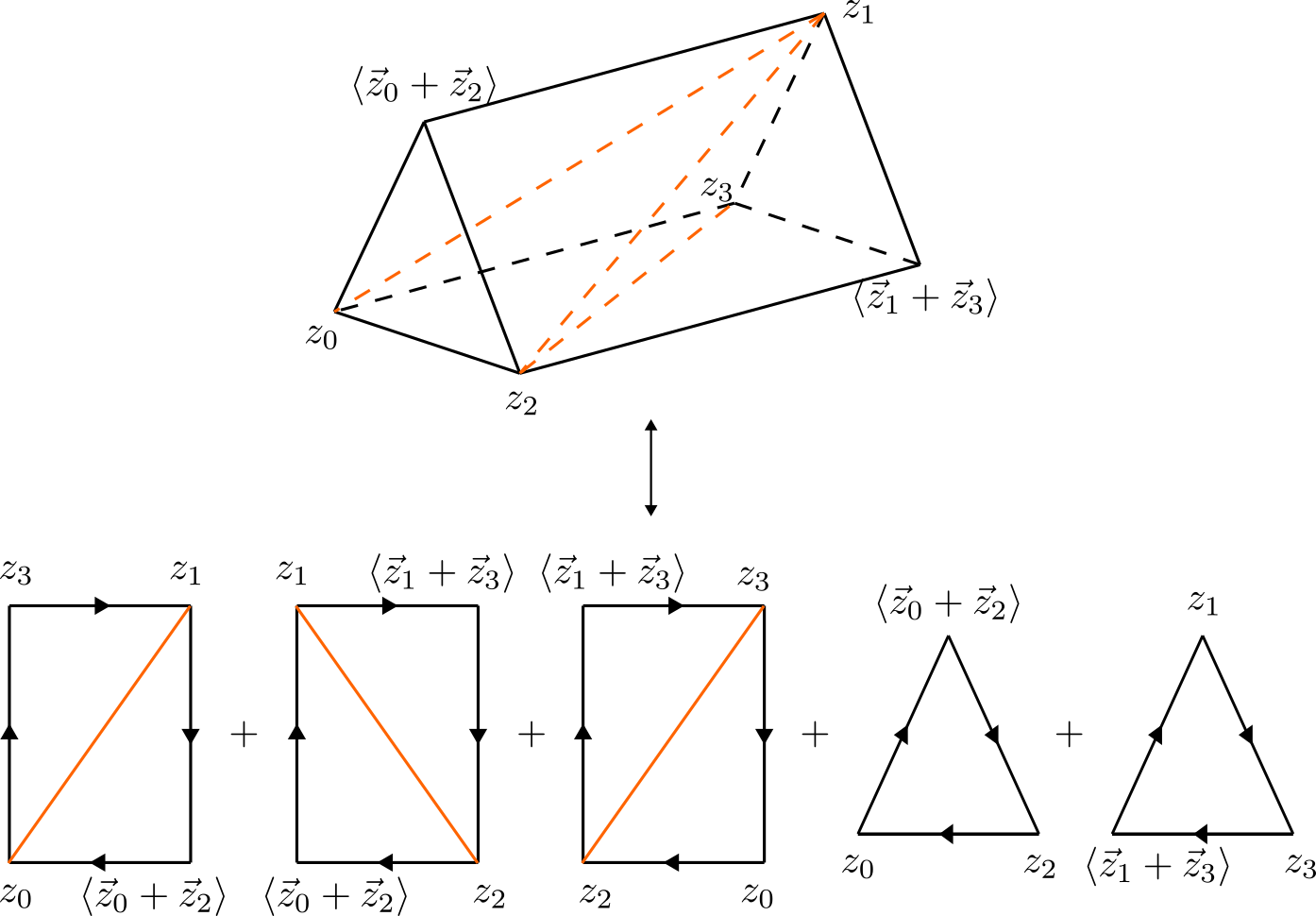}
\end{center}
\caption{Decomposing the $2$-sphere $\partial P$ into three balls $\Delta^1 \ast S^0$ and two triangles. The induced orientation is indicated by arrows. Orange marks $\sigma$ edges.}
\label{figure_homology_skew_sigma2}
\end{figure}
\end{proof}

\begin{proof}[Proof of \cref{lem:commutativity-three} for summands indexed by	$\sigma$-additive simplices]
	Let $\Delta = \{z_0, z_1, z_2\}$ be a minimal $\sigma$-additive simplex in $\IAA^{(2)}$ such that $z_0 < z_1 <z_2$. Let $[ v_2, \dots,  w_n] \in \ApartmentModule(\langle \Delta \rangle^\perp)$. We claim that
	\begin{equation}
		\label{eq:sigmaadd}
		\left((\ApartmentMap_{\langle  z_1, z_2  \rangle^\perp} \oplus \ApartmentMap_{\langle  z_0, z_2  \rangle^\perp} \oplus \ApartmentMap_{\langle  z_0, z_1  \rangle^\perp}) \circ 
		\partial_{\addname}\right) ([v_2, \dots, w_n]) = 
		\left(\partial_{n+1}^{\St^\omega} \circ \ApartmentMap_{\langle z_0,z_1,z_2 \rangle ^\perp}\right)([ v_2, \dots,  w_n]).
	\end{equation}
	Using the definition of $\partial_{\addname}$ (see 
	\cref{def:partial-sigma2-skew-add}) and the last part of 
	\cref{cor:relative-apartment-map-geometric}, it follows that, under the 
	identifications made in the \cref{cor:right-exact-sequence-smaller-steinberg-modules}, the
	left side term is equal to the following element in $H_n(\IAA^{(1)}, \IAA^{(0)})$:
	\begin{multline}
		\label{eq:sigmaadditiveleftterm}
		\left(\{ z_1, z_2 \} \ast [v_2, \dots,  w_n], \partial \{z_1,z_2\} \ast [v_2,   \dots,  w_n]\right)
		- \left(\{z_0, z_2\} \ast [ v_2, \dots,  w_n], \partial \{z_0, z_2\}\ast [ v_2, \dots,  w_n]\right)\\
		+ \left(\{z_0,z_1\} \ast [ v_2, \dots,  w_n], \partial \{z_0,z_1\}\ast [ v_2, \dots,  w_n]\right).
	\end{multline}
	Using the definition of the connection morphism $\partial_{n+1}^{(2,1,0)}$ 
	and the third part of \cref{cor:relative-apartment-map-geometric}, 
	it similarly follows that the 
	right hand term is equal to
	\begin{equation}
		\label{eq:sigmaadditiverightterm}
		\left(\partial \{z_0,z_1,z_2\} \ast  [ v_3, \dots, 
		 w_n], \emptyset\right) =
		 \left(\partial M^{\Delta^\perp}, \emptyset\right)_*(\eta_1 \ast \xi_{n-2}, \emptyset) \in H_{n+1}(\IAA^{(1)}, \IAA^{(0)}).
	\end{equation}
	Recall that the domain of $\partial M^{\Delta^\perp}$ is the combinatorial $(n+1)$-sphere $\partial \Delta \ast C_{n-1}$. We decompose $\partial \Delta = \partial \{z_0,z_1,z_2\}$ into three 1-simplices $\Delta^1$ (given by $\{z_1, z_2\}$, $\{z_0, z_2\}$ and $\{z_0, 
	z_1\}$), as illustrated in \cref{figure_homology_sigma_add}. This induces a decomposition of $\partial \Delta \ast C_{n-1}$ into three combinatorial $(n+1)$-balls of the form $\Delta^1 \ast C_{n-1}$. This allows us to express the homology class \cref{eq:sigmaadditiverightterm} as a sum 
	of three terms which are equal to \cref{eq:skewsigmaleftterm}. We conclude 
	that the left and right hand term of \cref{eq:sigmaadd} agree.
\begin{figure}
\begin{center}
\includegraphics{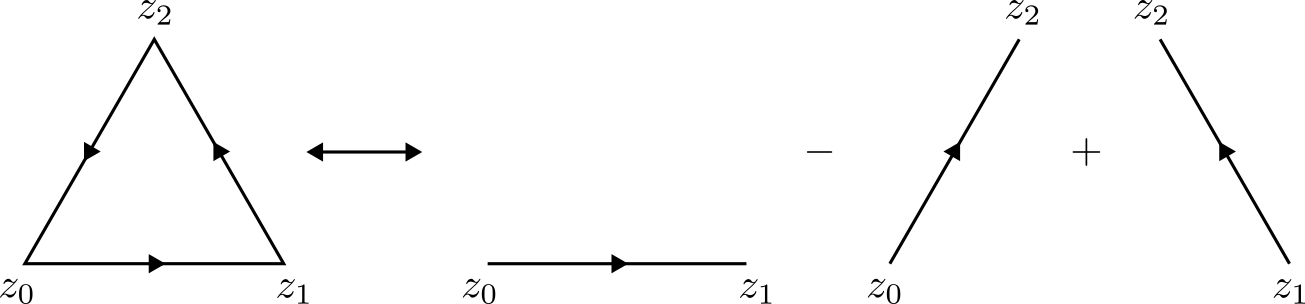}
\end{center}
\caption{Decomposing the boundary sphere $\partial \Delta$ of a $\sigma$-additive simplex $\Delta$ into three 1-simplices $\Delta^1$. The orientations are indicated by arrows.}
\label{figure_homology_sigma_add}
\end{figure}
\end{proof}

\begin{proof}[Proof of \cref{lem:commutativity-four}]
	Consider a minimal $\sigma$ simplex $\Delta = \{v_1, w_1\} \in \IAA^{(2)}$ with $v_1 < w_1$. Let $[ v_2,  w_2, \dots,  v_n,  w_n] \in \ApartmentModule(\langle \Delta \rangle^\perp)$. We claim that
	\begin{equation}
		\label{eq:sigma}
		(\ApartmentMap_n \circ \partial_{\sigma})([ v_2, \dots, 
		  w_n]) =
		(\partial_{n}^{\St^\omega} \circ \ApartmentMap_{\langle v_1, 
		w_1 \rangle^\perp})([ v_2, \dots, w_n]).
	\end{equation}
	Using the definition of $\partial_{\sigma}$ (see 
	\cref{def:partial-sigma}) and the definition of the apartment class map 
	(see \cref{def:apartment-map}), it follows that the
	left side term is equal to the following apartment class in $\St^\omega_n$:
	$$\ApartmentMap_n([v_1, w_2, v_2, w_2, \dots, v_n, w_n]) = M_* (\xi_{n-1}) \in \widetilde{H}_{n-1}(T^\omega_n; \mbZ) =\St^\omega_n,$$
	where $M \colon  \Simp(C_n) \to T^\omega_n$ is the poset map constructed after \cref{eq:apartment-class-map} for the symplectic basis $M = (\vec v_1, \vec w_1, \vec v_2, \vec w_2, \dots, \vec v_n, \vec w_n)$.
	
	Using the last part of \cref{cor:relative-apartment-map-geometric} and that, under the identifications made in \cref{cor:right-exact-sequence-smaller-steinberg-modules}, the map 
	$\partial^{\St^\omega}_n$ identifies with the composition of the connecting 
	morphism $\partial_{n}^{(1,0)}$ and the map $s_{*}$, the right hand term can be computed as follows: We start by observing that
	\begin{multline*}
		\partial_{n}^{(1,0)}\left(\{v_1, w_2\} \ast [v_2, \dots,  
		w_n], \partial \{v_1, w_2\} \ast [v_2, \dots, w_n]\right)
		= \partial \{v_1, w_2\} \ast [v_2, \dots, w_n]\\
		= \partial M^{\Delta^\perp}_*(\eta_0 \ast \xi_{n-2}) \in \widetilde{H}_{n-1}(\IAA^{(0)}).
	\end{multline*}
	The domain of $\partial M^{\Delta^\perp}$ is the combinatorial sphere $\partial \Delta \ast C_{n-1} = \partial \{v_1, w_1\} \ast C_{n-1}$. Using the isomorphism $\partial \{v_1, w_1\} \ast C_{n-1} \cong C_n$ described after \cref{eq:generators-for-construction} and the resulting identification of $\eta_0 \ast \xi_{n-2}$ and $\xi_{n-1}$, we may assume
	$$\partial M^{\Delta^\perp} \colon  C_n \to \IAA^{(0)}$$
	and write 
	\begin{equation*}
		\partial_{n}^{(1,0)}\left(\{v_1, w_2\} \ast [v_2, \dots,  
		w_n], \partial \{v_1, w_2\} \ast [v_2, \dots, w_n]\right)
		= \partial M^{\Delta^\perp}_*(\xi_{n-1}) \in \widetilde{H}_{n-1}(\IAA^{(0)}).
	\end{equation*}
	Now recall from the proof of \cref{lem:steinberg-as-relative-homology} that $$s_{*}\colon  \widetilde{H}_{n-1}(\IAA^{(0)}) \cong \widetilde{H}_{n-1}(\Simp(\IAA^{(0)})) \to \widetilde{H}_{n-1}(T^\omega_n; \mbZ) = \St^\omega_n$$
	is obtained by first passing to the barycentric subdivision and then applying the span map in \cref{prop:spanmap}. It follows that
	\begin{equation*}
		s_*\left(\partial M^{\Delta^\perp}_*(\xi_{n-1})\right) = \left(s \circ \Simp\left(\partial M^{\Delta^\perp}\right)\right)_*(\xi_{n-1}) \in \widetilde{H}_{n-1}(T^\omega_n; \mbZ) = \St^\omega_n,
	\end{equation*}
	where $\Simp(\partial M^{\Delta^\perp})\colon  \Simp(C_{n}) \to \Simp(\IAA^{(0)})$ is the map that $\partial M^{\Delta^\perp}$ induces between the simplex posets of $C_{n}$ and $\IAA^{(0)}$.
	It then suffices to note that, under the previous identification, the map $$s \circ \Simp\left(\partial M^{\Delta^\perp}\right)\colon  \Simp(C_n) \to T^\omega_n$$
	is exactly the map 
	$$M\colon  \Simp(C_n) \to T^\omega_n$$
	used to define the apartment class map $\ApartmentMap_n$ (compare \cref{eq:apartment-class-map} et seq.).	We conclude that 
	\begin{equation*} s_*\left(\partial M^{\Delta^\perp}_*(\xi_{n-1})\right) = M_* (\xi_{n-1}).\qedhere \end{equation*}
\end{proof}

\subsection{A diagram chase and the proof of \texorpdfstring{\cref{thm:stpres}}{Theorem 10.20}}
\label{subsec:proof-of-presentation-right-exact-sequence}

Using the results in the previous subsections, we are now ready to prove 
\cref{thm:stpres}. The induction argument is largely formal and relies on the 
following diagram chasing lemma whose proof we leave to the reader.

\begin{lemma}
	\label{lem:diagram-chasing}
	Assume we are given a commutative diagram of abelian groups
	\begin{center}
		\begin{tikzcd}
			& A_{1,2} \arrow[d] \arrow[r, bend left = 10]& A_{1,3} \arrow[d]&\\
			A_{2,1} \arrow[r]\arrow[d] \arrow[rru, bend left = 10, "\pi" {xshift = -4ex}, swap, sloped] & A_{2,2}\arrow[d] \arrow[r, "\partial_\sigma"] & A_{2,3}\arrow[d]&\\
			A_{3,1} \arrow[r] \arrow[d]	& A_{3,2} \arrow[r] \arrow[d] & A_{3,3} \arrow[d] \arrow[r]& 0\\
			0 & 0 & 0 &
		\end{tikzcd}
	\end{center}
	such that the first two columns $A_{*,1}$ and $A_{*, 2}$ are exact, the third row $A_{3,*}$ is exact, and $\pi$ as well as $\partial_\sigma$ are surjective maps. Then the third column $A_{*,3}$ is exact.
\end{lemma}

The proof of \cref{thm:stpres} is now by induction on the genus $n$ of the symplectic module $\mbZ[\Sp{2n}{\mbZ}]$. \medskip

\noindent \textbf{Induction beginning:}  For $n = 1$, we consider the group $\Sp2\mbZ 
=\SL2\mbZ$. The results for $\Sp2\mbZ$ presented hence here are closely 
related the work of Church--Putman; in particular to \cite[Theorem B]{CP}, 
which was originally proved by Bykovski\u{\i} \cite{byk2003} and generalises work of Manin \cite{manin1972} for $n = 1$.

\begin{claim}
	\label{claim:induction-beginning}
	The sequence 
	$ \SigmaAdditiveApartmentModule_1 \oplus \SkewApartmentModule_1 \xrightarrow{\AdditiveMap_1 \oplus \SkewMap_1} \ApartmentModule_1 \xrightarrow{\ApartmentMap_1} \St^\omega_1 \longrightarrow 0$
	is exact.
\end{claim}

\begin{proof}
	We start by noting that $\IAA[1]^{(2)} = \IAA[1]$ and $\IAA[1]^{(1)} = \IAAst[1]$. Hence, the following are a special case ($m = 0$) of the description of the complexes $\IAArel[1][m]$ and $\IAAstrel[1][m]$ given in the proof of \cref{lem_induction_beginning}:
	\begin{itemize}
		\item $\IAA[1]^{(2)}= \IAA[1]$ only 	contains simplices of type standard of 
		dimension $0$, $\sigma$ of dimension $1$ and $\sigma$-additive of 
		dimension $2$. The complex is isomorphic to $\BA_2$, a 
		contractible 2-dimensional simplicial complex.
		\item $\IAA[1]^{(1)} = \IAAst[1]$ is the subcomplex of $\IAA[1]^{(2)}$ consisting 
		of all simplices of type standard and $\sigma$, i.e.\ the 
		$1$-skeleton of $\IAA[1]^{(2)}$. The complex is isomorphic to the 
		subcomplex $\B_2$ of 
		$\BA_2$, the connected simplicial graph known as the Farey 
		graph.
		\item $\IAA[1]^{(0)}$ is the subcomplex of $\IAA[1]^{(2)}$ consisting 
		of all standard simplices, i.e.\ the $0$-skeleton of 
		$\IAA[1]^{(2)}$. The complex is isomorphic to the Tits building $T^\omega_1$ (which is a 
		discrete set in this case).
	\end{itemize}
	Using this description of $(\IAA[1]^{(2)}, \IAA[1]^{(1)}, \IAA[1]^{(0)})$, 
	the fact that $\langle \Delta \rangle^\perp = \{0\}$ if $\Delta$ is a 
	$\sigma$-additive or $\sigma$ simplex in $\IAA[1]^{(2)}$, the convention 
	that $\St^\omega(\{0\}) = \ApartmentModule(\{0\})  = \mbZ$, that $\SkewApartmentModule(\{0\}) = \SigmaAdditiveApartmentModule(\{0\}) = 0$  and noting that $\SkewApartmentModule_1 = 0$, it follows that the commutative diagram constructed in \cref{prop:commutative-diagram} has the shape depicted in \cref{fig_commdiag_base_case} for $n = 1$.
	\begin{figure}
	\begin{center}
		\begin{tikzcd}
			& 0 \arrow[d] \arrow[r, dotted] & \SigmaAdditiveApartmentModule_1\arrow[d, "\AdditiveMap_1", swap]&\\
			\displaystyle \bigoplus_{\substack{\Delta = \{ z_0, z_1, z_2\}\\ 
			\text{$\sigma$-additive simplex}}}  \mbZ  \arrow[r, 
			"\partial_{\addname}", dotted] \arrow[d, "\ApartmentMap_{\langle \Delta 
			\rangle^\perp}"] \arrow[rru, "\pi_{\addname}" {anchor=center, rotate=15, 
			xshift=-5ex, yshift=-1ex}, bend left = 5, dotted, swap] 
			&\displaystyle\bigoplus_{\substack{\Delta = \{v_1,w_1\}\\ 
			\text{$\sigma$ simplex}}} \mbZ \arrow[d, "\ApartmentMap_{\langle 
			\Delta \rangle^\perp}"]\arrow[r, "\partial_{\sigma}", dotted] & 
			\ApartmentModule_1\arrow[d, "\ApartmentMap_1", swap] &\\
			\displaystyle \bigoplus_{\substack{\Delta = \{ z_0, z_1, z_2\}\\ 
			\text{$\sigma$-additive simplex}}} \mbZ \arrow[r, 
			"\partial_{n+1}^{\St^\omega}"] \arrow[d] & 
			\displaystyle\bigoplus_{\substack{\Delta = \{v_1, w_1\}\\ 
			\text{$\sigma$ simplex}}} \mbZ \arrow[r, 
			"\partial_{n}^{\St^\omega}"] \arrow[d] & \St^\omega_1 \arrow[r] 
			\arrow[d] & 0\\
			0 & 0 & 0 &
		\end{tikzcd}
	\end{center}
	\caption{Commutative diagram for $n = 1$}
	\label{fig_commdiag_base_case}
	\end{figure}
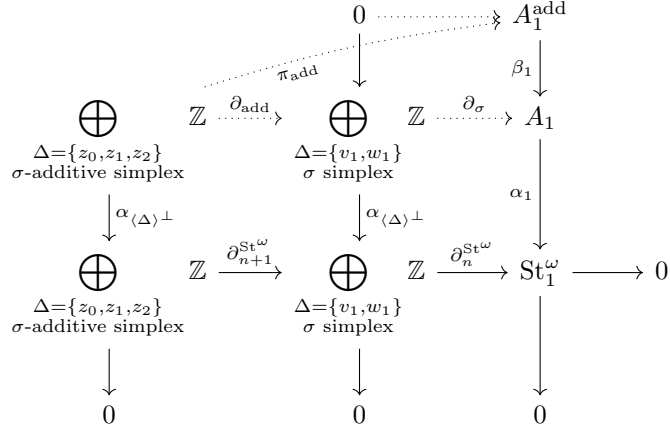
	It is easy to check that the first two columns are 
	exact. By \cref{cor:right-exact-sequence-smaller-steinberg-modules} the bottom row is exact and by \cref{lem:pi-partial-sigma-surjective}, $\pi_{\addname}$ and 
	$\partial_{\sigma}$ are surjections. Therefore, 
	\cref{lem:diagram-chasing} implies the claim.
\end{proof}

\noindent \textbf{Induction hypothesis:} Assume that for any symplectic 
submodule 
$\SymplecticSummand \subsetneq \mbZ^{2n}$ of genus less than $n$, the sequence
\[ \SigmaAdditiveApartmentModule(\SymplecticSummand) \oplus 
\SkewApartmentModule(\SymplecticSummand) 
\xrightarrow{\AdditiveMap_{\SymplecticSummand} \oplus 
\SkewMap_{\SymplecticSummand}} \ApartmentModule(\SymplecticSummand) 
\xrightarrow{\ApartmentMap_{\SymplecticSummand}} \St^\omega(\SymplecticSummand) 
\longrightarrow 
0\]
is exact.\medskip

\noindent \textbf{Induction step:} The proof of induction step is now 
completely formal.

\begin{claim}
	For $n > 1$, the sequence $\SigmaAdditiveApartmentModule_n \oplus \SkewApartmentModule_n \xrightarrow{\AdditiveMap_n \oplus \SkewMap_n} \ApartmentModule_n \xrightarrow{\ApartmentMap_n} \St^\omega_n \longrightarrow 0$ is exact.	
\end{claim}

\begin{proof}
For $n > 1$, it follows from the induction hypothesis that the first 
two columns of the diagram in 
\cref{fig_commdiag} are exact. \cref{cor:right-exact-sequence-smaller-steinberg-modules} says that the bottom row is exact. 
Furthermore, by \cref{lem:pi-partial-sigma-surjective}, $\pi_{\sigma^2} \oplus \pi_{\skewname} 
\oplus \pi_{\addname}$ and $\partial_{\sigma}$ 
are surjections. Therefore, \cref{lem:diagram-chasing} 
implies 
the claim.
\end{proof}

\section{Theorem A: A vanishing theorem}
\label{sec_cohomology_vanishing}
The presentation of $\St^\omega_n$ we obtained in \cref{sec:presentation} now lets us prove 
\autoref{thmA}, which states that the rational cohomology of $\Sp{2n}{\mbZ}$ vanishes in degree $n^2-1$. By Borel--Serre duality, this is equivalent to $H_1(\Sp{2n}\Z; \St^\omega_n(\Q)\otimes\mbQ)$ being trivial (see \cref{eq_Borel_Serre}).
In order to show the latter, we want to use \cref{thm:stpres} and prove that rationally, the three 
modules $\ApartmentModule_n$, $\SigmaAdditiveApartmentModule_n$, and 
$\SkewApartmentModule_n$ (see \cref{def:apartment-module}, \cref{def:additive-apartment-module} and \cref{def:skew-apartment-module}) are flat and have trivial coinvariants.

As before, $\{\vec e_1, \vec f_1, \ldots, \vec e_n, \vec f_n\}$ denotes the symplectic standard basis of $\mbZ^{2n}$.
\begin{lemma}\label{lem:Aflatness}
\label{lem:AAflatness}
\label{lem:Askewflatness}
The $\mbQ[\Sp{2n}{\mbZ}]$-modules $\ApartmentModule_n\otimes \mbQ$, $\SigmaAdditiveApartmentModule_n\otimes \mbQ$ and $\SkewApartmentModule_n\otimes \mbQ$  are flat.
\end{lemma}
\begin{proof}
The proof is similar to that of \cite[Lemma 3.2]{CP}. We first note that $\ApartmentModule_n\otimes \mbQ$ is a cyclic $\mbQ[\Sp{2n}{\mbZ}]$-module, generated by $[e_1,f_1, \dots, e_n, f_n]$.

Let $H$ be the subgroup of $\Sp{2n}{\mbZ}$ that sends $[e_1,f_1, \dots, e_n, f_n] \text{ to } \pm [e_1,f_1, \dots, e_n, f_n]$
and let $M$ be the $\mbQ[H]$-module whose underlying vector space is $\mbQ$ and where $H$ acts by $\pm 1$, depending on its action on $[e_1,f_1, \dots, e_n, f_n]$. It is not hard to see that $H$ is finite\footnote{This can be read off from the presentation of $\ApartmentModule_n$, \cref{def:apartment-module}, using the facts that an element in $\Sp{2n}{\mbZ}$ is uniquely determined by its its image of a basis of $\mbZ^{2n}$ and that every line in $\mbZ^{2n}$ contains at most two elements.}.
This implies that $M$ is a projective $\mbQ[H]$-module.
Furthermore, we have
\begin{equation*}
	\ApartmentModule_n\otimes \mbQ \cong \Ind_{H}^{\Sp{2n}{\mbZ}} M,
\end{equation*}
so the claim follows just as in \cite[Lemma 3.2]{CP}.

The proofs for  $\SigmaAdditiveApartmentModule_n\otimes \mbQ$ and $\SkewApartmentModule_n\otimes \mbQ$ are the same after verifying that these modules are cyclically generated by
$[\langle \vec e_1 + \vec f_1\rangle, e_1, f_1] \ast [e_2, f_2, \dots, e_n, f_n].
$
and $[e_1, f_1, \langle \vec e_2 - \vec e_1 \rangle , f_2] \ast [e_3, f_3, \dots, e_n, f_n]$, respectively.
\end{proof}

\begin{lemma}\label{lem:Acoinv}
For $n\ge1$, the $\Sp{2n}\mbZ$-coinvariants of $\ApartmentModule_n\otimes \mbQ$ vanish.
\end{lemma}
\begin{proof}
There is an element $\phi$ of $\Sp{2n}\mbZ$ defined by $\phi(\vec e_1)  = \vec f_1$, $\phi(\vec f_1) = -\vec e_1$, and $\phi(\vec e_i) = \vec e_i$ , $\phi(\vec f_i) = \vec f_i$ for $i\ge 2$. Using the relations in $\ApartmentModule_n$, we have
\begin{equation*}
	\phi([e_1,f_1, \dots, e_n,f_n]) = [f_1,e_1, \dots, e_n,f_n]  = -[e_1,f_1, \dots, e_n,f_n].
\end{equation*} 
Hence, in the coinvariants $(\ApartmentModule_n\otimes \mbQ)\otimes_{\Sp{2n}{\mbZ}} \mbQ$, we have 
\begin{equation*}
	[e_1,f_1, \dots, e_n,f_n]\otimes q = - [e_1,f_1, \dots, e_n,f_n]\otimes q \text{ for all } q\in \mbQ.
\end{equation*}
This implies that $[e_1,f_1, \dots, e_n,f_n]\otimes q$ is trivial. As noted in the proof of \cref{lem:Aflatness}, the module $\ApartmentModule_n\otimes \mbQ$ is generated by $[e_1,f_1, \dots, e_n,f_n]$, so the claim follows.
\end{proof}

\begin{lemma}\label{lem:AAcoinv}
For $n\ge2$, the $\Sp{2n}\mbZ$-coinvariants of $\SigmaAdditiveApartmentModule_n \otimes \mbQ$ vanish.
\end{lemma}
\begin{proof}
Let $\Delta = [\langle \vec e_1 + \vec f_1\rangle, e_1, f_1] \ast [e_2, f_2, \dots, e_n, f_n]$
be the generator of $\SigmaAdditiveApartmentModule_n \otimes \mbQ$ from the proof of \cref{lem:AAflatness} and let $\phi \in \Sp{2n}{\mbZ}$ be defined by $\phi(\vec e_2) = \vec f_2$, $\phi(\vec f_2) = -\vec e_2$ and $\phi(\vec e_i) = \vec e_i$, $\phi(\vec e_i) = \vec f_i$ for $i\neq 2$. Using the relations in $\SigmaAdditiveApartmentModule_n$, we have
\begin{equation*}
	\phi(\Delta) = [\langle \vec e_1 + \vec f_1\rangle, e_1, f_1] \ast [f_2, e_2, \dots, e_n, f_n] = -\Delta.
\end{equation*}
As in the proof of \cref{lem:AAcoinv}, this implies that $(\SigmaAdditiveApartmentModule_n\otimes \mbQ)\otimes_{\Sp{2n}{\mbZ}} \mbQ = 0.$
\end{proof}

\begin{lemma}\label{lem:Askewcoinv}
For $n\ge3$, the $\Sp{2n}\mbZ$-coinvariants of $\SkewApartmentModule_n \otimes \mbQ$ vanish.
\end{lemma}

\begin{proof}
Let $\Delta = [e_1, f_1, \langle \vec e_2 - \vec e_1 \rangle , f_2] \ast [e_3, f_3, \dots, e_n, f_n]$
be the generator of $\SkewApartmentModule_n \otimes \mbQ$ from the proof of \cref{lem:Askewflatness} and let $\phi \in \Sp{2n}{\mbZ}$ be defined by $\phi(\vec e_3) = \vec f_3$, $\phi(\vec f_3) = -\vec e_3$ and $\phi(\vec e_i) = \vec e_i$, $\phi(\vec e_i) = \vec f_i$ for $i\neq 3$. Using the relations in $\SkewApartmentModule_n$, we have
\begin{equation*}
	\phi(\Delta) = [e_1, f_1, \langle \vec e_2 - \vec e_1 \rangle , f_2] \ast [f_3, e_3, \dots, e_n, f_n] = -\Delta.
\end{equation*}
As in the proof of \cref{lem:AAcoinv}, this implies that $(\SkewApartmentModule_n\otimes \mbQ)\otimes_{\Sp{2n}{\mbZ}} \mbQ = 0$.
\end{proof}

We conclude by proving \autoref{thmA}, which states that $H^{n^2-i}(\Sp{2n}{\mbZ};\mbQ) = 0$ for $i\le 1$ and $n\ge2$.

\begin{proof}[Proof of \autoref{thmA}]
For $n = 2$, this follows from work of Igusa \cite{igusa1962}, see also Lee--Weintraub \cite[Corollary 5.2.3 et seq.]{LW}. As commented after \autoref{thmA}, our methods apply for $n\ge 3$ (and recover known results for $n\in \{3,4 \}$).

Let $n\geq 3$. Using Borel--Serre duality, we see that
\[ H^{n^2-i}(\Sp{2n}{\mbZ};\mbQ)  \cong H_{i}(\Sp{2n}{\mbZ}; \St^\omega_n\otimes \mbQ) .\]
The latter can be computed using a flat resolution of $\St^\omega_n\otimes \mbQ$. From \cref{thm:stpres}, we get a partial resolution that is flat by \cref{lem:Aflatness}, \cref{lem:AAflatness} and \cref{lem:Askewflatness} and hence can be extended to a flat resolution. Taking coinvariants of this flat resolution yields a chain complex whose homology is 
\[H_*(\Sp{2n}{\mbZ}; \St^\omega_n \otimes \mbQ).\] 
The theorem follows from \cref{lem:Acoinv}, \cref{lem:AAcoinv}, and \cref{lem:Askewcoinv}.
\end{proof}

\appendix
\section{Overview of different complexes}
\begin{table}[h]
\begin{center}
\begin{tabular}{l|l|l}
Complex & Defined in & Connectivity  \\
\hline
$T^\omega_n$ & \cref{def_building_over_Z} & $n-2$ (\cref{Solomon-Tits}) \\
$\I$, $\Irel$ 	& \cref{def:I-and-IA}, \cref{def_linkhat} & $n-2$ (\cref{lem:Irel-cohen-macaulay}) \\
$\Idel$, $\Idelrel$ 	& \cref{def:I-and-IA}, \cref{def_linkhat} & $n-2$ (\cref{lem_connectivity_Idelrel}) \\
$\Isigdel$, $\Isigdelrel$ 	& \cref{def:I-and-IA}, \cref{def_linkhat} & $n-1$ \cite[Prop.~6.11]{put2009}\\
$\IA$, $\IArel$ 	& \cref{def:I-and-IA}, \cref{def_linkhat} & $n-1$ (\cref{lem_conn_IA}) \\
$\IAAst$, $\IAAstrel$ 	& \cref{def:IAAst}, \cref{def_linkhat} & $n-1$ (\cref{lem_connectivity_IAAstar})\\
$\IAA$, $\IAArel$ 	& \cref{def:IAA}, \cref{def_linkhat}  & $n$ (\autoref{thm_connectivity_IAA})\\
$\B_n$ 	& \cref{def:B-and-BA} & $n-2$ (\cref{thm_connectivity_B_BA})\\
$\BA_n$ 	& \cref{def:B-and-BA} & $n-1$ (\cref{thm_connectivity_B_BA}) \\
$\BAA_n$ 	& \cref{def:BAA} & $n$ (\cref{connectivity_BAA}) \\
$\Irel(W)$ 	& \cref{def_W_complexes} & $n-2$ (\cref{connectivity_subcomplex_W}) \\
$\Idelrel(W)$ 	& \cref{def_W_complexes} & $n-1$ (\cref{connectivity_subcomplex_W})\\
$\IAAstrel(W)$ 	& \cref{def_W_complexes} & $n$ (\cref{connectivity_subcomplex_W})\\
$\IAA^{(0)}$ 	& \cref{def_intermediate_complexes} & $n-2$ (\cref{lem_connectivity_IAA0}) \\
$\IAA^{(1)}$ 	& \cref{def_intermediate_complexes} & $n-1$ (\cref{cor_IAA1}) \\
$\IAA^{(1.5)}$ 	& \cref{def_intermediate_complexes} & $n-1$ (\cref{lem_hom_eq_1_15}) \\
$\IAA^{(2)}$ 	& \cref{def_intermediate_complexes} & $n$ (\cref{lem_IAA2_highlyconnected})
\end{tabular}
\end{center}
\caption{An overview of complexes appearing in this article. The last column shows the largest $k$ such that the corresponding complex is $k$-connected.}
\label{table_overview_all_complexes}
\end{table}

\addcontentsline{toc}{section}{References}
\emergencystretch=2em
\printbibliography

\bigskip
  \footnotesize
\noindent Benjamin Br\"uck\\ \textsc{Department of Mathematics \\
ETH Z\"urich\\
Rämistrasse 101\\
8092 Z\"urich, Switzerland}\\
\href{mailto:benjamin.brueck@math.ethz.ch}{benjamin.brueck@math.ethz.ch}
\vspace{0.5cm}\\
\noindent Peter Patzt\\ \textsc{Department of Mathematics \\
University of Oklahoma\\
601 Elm Avenue\\
Norman, OK-73019, USA}\\
  \href{mailto:ppatzt@ou.edu}{ppatzt@ou.edu}
\vspace{0.5cm}\\
\noindent Robin J. Sroka\\ \textsc{Department of Mathematics \& Statistics \\
McMaster University\\
1280 Main Street West\\
Hamilton, ON L8S 4K1, Canada}\\
\href{mailto:srokar@mcmaster.ca}{srokar@mcmaster.ca}
\end{document}